\documentclass[numbers=enddot,12pt,final,onecolumn,notitlepage]{scrartcl}%
\usepackage{e-jc}
\usepackage[all,cmtip]{xy}
\usepackage[colorlinks=true,breaklinks=true,
citecolor=black,linkcolor=black,urlcolor=blue, pdfauthor={Darij Grinberg, Tom
Roby}, pdftitle={Birational rowmotion on a rectangle over a noncommutative
ring}]{hyperref}
\usepackage{amsfonts}
\usepackage{amssymb}
\usepackage{amsmath}
\usepackage{amsthm}
\usepackage{framed}
\usepackage{comment}
\usepackage{color}
\usepackage{hyperref}
\usepackage{needspace}
\usepackage{tabls}
\usepackage{tikz}
\providecommand{\U}[1]{\protect\rule{.1in}{.1in}}
\theoremstyle{definition}
\newcommand{\red}{\color{red}}
\newtheorem{theo}{Theorem}[section]
\newenvironment{theorem}[1][]
{\begin{theo}[#1]\begin{leftbar}}
{\end{leftbar}\end{theo}}
\newtheorem{lem}[theo]{Lemma}
\newenvironment{lemma}[1][]
{\begin{lem}[#1]\begin{leftbar}}
{\end{leftbar}\end{lem}}
\newtheorem{prop}[theo]{Proposition}
\newenvironment{proposition}[1][]
{\begin{prop}[#1]\begin{leftbar}}
{\end{leftbar}\end{prop}}
\newtheorem{defi}[theo]{Definition}
\newenvironment{definition}[1][]
{\begin{defi}[#1]\begin{leftbar}}
{\end{leftbar}\end{defi}}
\newtheorem{remk}[theo]{Remark}
\newenvironment{remark}[1][]
{\begin{remk}[#1]\begin{leftbar}}
{\end{leftbar}\end{remk}}
\newtheorem{coro}[theo]{Corollary}
\newenvironment{corollary}[1][]
{\begin{coro}[#1]\begin{leftbar}}
{\end{leftbar}\end{coro}}
\newtheorem{conv}[theo]{Convention}
\newenvironment{convention}[1][]
{\begin{conv}[#1]\begin{leftbar}}
{\end{leftbar}\end{conv}}
\newtheorem{quest}[theo]{Question}
\newenvironment{question}[1][]
{\begin{quest}[#1]\begin{leftbar}}
{\end{leftbar}\end{quest}}
\newtheorem{conj}[theo]{Conjecture}
\newenvironment{conjecture}[1][]
{\begin{conj}[#1]\begin{leftbar}}
{\end{leftbar}\end{conj}}
\newtheorem{exam}[theo]{Example}
\newenvironment{example}[1][]
{\begin{exam}[#1]\begin{leftbar}}
{\end{leftbar}\end{exam}}
\newenvironment{verlong}{}{}
\newenvironment{vershort}{}{}
\newenvironment{noncompile}{}{}
\excludecomment{verlong}
\includecomment{vershort}
\excludecomment{noncompile}
\newenvironment{statement}{\begin{quote}}{\end{quote}}
\newcommand{\undf}{\operatorname*{\bot}}
\newcommand{\upslack}{\mathchoice{\rotatebox[origin=c]{180}{$\displaystyle A$}}{\rotatebox[origin=c]{180}{$\textstyle A$}}{\rotatebox[origin=c]{180}{$\scriptstyle A$}}{\rotatebox[origin=c]{180}{$\scriptscriptstyle A$}}}
\newcommand{\downslack}{A}
\newcommand{\bfupslack}{\mathchoice{\rotatebox[origin=c]{180}{$\displaystyle \mathbf{A}$}}{\rotatebox[origin=c]{180}{$\textstyle \mathbf{A}$}}{\rotatebox[origin=c]{180}{$\scriptstyle \mathbf{A}$}}{\rotatebox[origin=c]{180}{$\scriptscriptstyle \mathbf{A}$}}}
\newcommand{\bfdownslack}{\mathbf{A}}

\DeclareMathOperator{\id}{id}
\DeclareMathOperator{\first}{\mathfrak{first}}
\DeclareMathOperator{\rank}{\mathfrak{rank}}
\DeclareMathOperator{\tilt}{\mathfrak{tilt}}
\let\sumnonlimits\sum
\renewcommand{\sum}{\sumnonlimits\limits}
\newcommand{\are}{\ar@{-}}
\begin{document}

\title{\textbf{Birational rowmotion on a rectangle over a noncommutative ring}}
\author{Darij Grinberg\\{\small Department of Mathematics}\\[-0.8ex] {\small Drexel University}\\[-0.8ex] {\small Philadelphia, U.S.A.}\\{\small \texttt{darijgrinberg@gmail.com}}\\\phantom {!}
\and Tom Roby\\{\small Department of Mathematics}\\[-0.8ex] {\small University of Connecticut}\\[-0.8ex] {\small Connecticut, U.S.A.}\\{\small \texttt{tom.roby@uconn.edu}}}
\date{version {4.0} (\today).}
\maketitle

\begin{abstract}
\textbf{Abstract.} We extend the periodicity of birational rowmotion for
rectangular posets to the case when the base field is replaced by a
noncommutative ring (under appropriate conditions). This resolves a conjecture
from 2014. The proof uses a novel approach and is fully self-contained.

Consider labelings of a finite poset $P$ by $\left\vert P\right\vert + 2$
elements of a ring $\mathbb{K}$: one label associated with each poset element
and two constant labels for the added top and bottom elements in $\widehat{P}%
$. \textit{Birational rowmotion} is a partial map on such labelings. It was
originally defined by Einstein and Propp~\cite{einstein-propp} for
$\mathbb{K}=\mathbb{R}$ as a lifting (via detropicalization) of
\textit{piecewise-linear rowmotion}, a map on the order polytope
$\mathcal{O}(P) := \{\text{order-preserving } f: P \to[0,1]\}$. The latter, in
turn, extends the well-studied rowmotion map on the set of order ideals (or
more properly, the set of order filters) of $P$, which correspond to the
vertices of $\mathcal{O}(P)$. Dynamical properties of these combinatorial maps
sometimes (but not always) extend to the birational level, while results
proven at the birational level always imply their combinatorial counterparts.
Allowing $\mathbb{K}$ to be noncommutative, we generalize the birational level
even further, and some properties are in fact lost at this step.

In 2014, the authors gave the first proof of periodicity for birational
rowmotion on rectangular posets (when $P$ is a product of two chains) for
$\mathbb{K}$ a field, and conjectured that it survives (in an appropriately
twisted form) in the noncommutative case. In this paper, we prove this
noncommutative periodicity and a concomitant antipodal reciprocity formula. We
end with some conjectures about periodicity for other posets, and the question
of whether our results can be extended to (noncommutative) semirings.

It has been observed by Glick and Grinberg that, in the commutative case,
periodicity of birational rowmotion can be used to derive Zamolodchikov
periodicity in the type $AA$ case, and vice-versa. However, for noncommutative
$\mathbb{K}$, Zamolodchikov periodicity fails even in small examples (no
matter what order the factors are multiplied), while noncommutative birational
rowmotion continues to exhibit periodicity. Thus, our result can be viewed as
a lateral generalization of Zamolodchikov periodicity to the noncommutative setting.


\end{abstract}

%

\tableofcontents

\section*{Introduction}



\begin{noncompile}
TODO2: Add some more to the introduction.

Ideally mention:

- The connection to SSYT promotion (\cite[Proposition A.9]{Hopkin20}) and
quiver reps (\cite{GaPaTh18}).



\end{noncompile}

The goal of this paper is to extend the periodicity of birational rowmotion
for rectangular posets to the case when the base field is replaced by a
noncommutative ring (under appropriate conditions). This resolves a conjecture
from 2014. The proof uses a novel approach (even in the commutative case) and
is fully self-contained.

Let $P$ be a finite poset, and let $\widehat{P}$ be the same poset with two
extra elements added: one global minimum and one global maximum. For the time
being, let $\mathbb{K}$ be a field. A \emph{$\mathbb{K}$-labeling of $P$}
means a map from $\widehat{P}$ to $\mathbb{K}$; we view it as a way of
labeling each element of $\widehat{P}$ by an element of $\mathbb{K}$.
Birational rowmotion, as studied conventially, is a rational map $R$ on such
labelings (i.e., a rational map $R:\mathbb{K}^{\widehat{P}}\dashrightarrow
\mathbb{K}^{\widehat{P}}$). It was introduced by Einstein and
Propp~\cite{einstein-propp} for $\mathbb{K}=\mathbb{R}$, generalizing (via the
tropical limit\footnote{See \cite[Section 4.1]{kirillov-intro} for what we
mean by the \textquotedblleft tropical limit\textquotedblright\ here, and
\cite{kirillov-berenstein} for one of the earliest example of
detropicalization (i.e., the generalization of a combinatorial map to a
rational one).}) the well-studied \emph{combinatorial rowmotion} map on order
ideals of $P$~\cite{brouwer-schrijver,striker-williams,propp-roby,row-slow}.

Birational rowmotion can be defined as a composition of \textquotedblleft
toggles\textquotedblright: For each $v\in P$, we define the $v$\emph{-toggle}
as the rational map $T_{v}:\mathbb{K}^{\widehat{P}}\dashrightarrow
\mathbb{K}^{\widehat{P}}$ that modifies a $\mathbb{K}$-labeling $f$ by
changing the label $f\left(  v\right)  $ to\footnote{The notations $\lessdot$
and $\gtrdot$ mean \textquotedblleft covered by\textquotedblright\ and
\textquotedblleft covers\textquotedblright, respectively (see Sections
\ref{sec.linext} and \ref{sec.ncbr} for details).}%
\[
\left(  \sum\limits_{\substack{u\in\widehat{P};\\u\lessdot v}}f\left(
u\right)  \right)  \cdot\left(  f\left(  v\right)  \right)  ^{-1}\cdot\left(
\sum\limits_{\substack{u\in\widehat{P};\\u\gtrdot v}}\left(  f\left(
u\right)  \right)  ^{-1}\right)  ^{-1},
\]
while leaving all the other labels of $f$ unchanged. Now, birational rowmotion
$R$ is the composition of all the $v$-toggles, where $v$ runs over the poset
$P$ from top to bottom. (That is, we pick a linear extension $\left(
v_{1},v_{2},\ldots,v_{n}\right)  $ of $P$, and set $R=T_{v_{1}}\circ T_{v_{2}%
}\circ\cdots\circ T_{v_{n}}$.)

Dynamical properties at the combinatorial level sometimes extend to higher
levels, while results proven at the birational level always imply their
combinatorial counterparts. In particular, while combinatorial rowmotion
always has finite order (since it is an invertible map on a finite set), there
is no reason to expect periodicity at all at the higher levels. Indeed, for
many nice posets, birational rowmotion has infinite order, including for the
Boolean algebra of order 3 (or those in \cite[Fig.~6]{Roby15}), and there are
only a few infinite classes where it appears to have finite order (mostly
posets associated with representation theory, e.g., root or minuscule posets).
In these cases the order of birational rowmotion is generally the \emph{same}
as for combinatorial rowmotion, e.g., $p+q$ for $P=[p]\times[q]$.

In 2014, the authors gave the first proof of periodicity of birational
rowmotion for rectangular posets (i.e., when $P$ is a product of two chains)
and $\mathbb{K}$ a field~\cite{bir-row-arxiv}. The main idea of this proof was
to embed the space of labelings into an appropriate Grassmannian (where in
each ``sufficiently generic'' $\mathbb{K}$-labeling, the labels can be
expressed as ratios of certain minors of a matrix) and use particular
Pl\"ucker relations to derive the result. There were several serious technical
hurdles to overcome.

The definition of birational rowmotion relies entirely on addition,
multiplication and inverses in $\mathbb{K}$. Thus, it is natural to extend it
to the case when $\mathbb{K}$ is a ring (not necessarily commutative), or even
just a semiring. (At this level, birational rowmotion is no longer a rational
map, just a partial map.) However, there is no guarantee that the properties
of birational rowmotion survive at this level for every poset; and indeed,
sometimes they do not (see, e.g., Example~\ref{exa.illbehaved.claw}). However,
in 2014, the authors experimentally observed that the periodicity for
rectangular posets appears to hold even in this noncommutative setting, as
long as it is appropriately modified: After $p+q$ iterations of birational
rowmotion, the labels are not returned to their original states, but rather to
certain ``twisted variants'' thereof (resembling, but not the same as,
conjugates). See Example~\ref{ex.rowmotion.2x2} to get the sense of this.

Strikingly, this noncommutative generalization has resisted all approaches
that have previously succeeded in the commutative case. The determinantal
computations involved in the proof in~\cite{bir-row-arxiv} can be extended to
the noncommutative setting using the \emph{quasideterminants} of Gelfand and
Retakh, but it seems impossible to make a rigorous proof out of it (lacking,
e.g., any useful notation of Zariski topology in this setting, it is not clear
what it means for a $\mathbb{K}$-labeling to be ``generic''). The alternative
proof of commutative periodicity found by Musiker and Roby~\cite{MusRob17}
(via a lattice-path formula for iterates of birational rowmotion) could not be
generalized as well. Thus the noncommutative case remained an open
problem.\footnote{This is not the first time that rational maps in algebraic
combinatorics have been generalized to the noncommutative case; some other
instances are \cite{IyuShk14, BerRet15, Rupel17, GonKon21}. Each time, the
generalizations have been much harder to prove, not least because very little
of the commutative groundwork is (currently?) available at the noncommutative
level. For instance, it is insufficient to work over the ``free skew fields'',
since an identity between rational expressions can be true in all skew fields
yet fail in some noncommutative rings (such as the identity $x \left(
yx\right)  ^{-1} y = 1$). For this reason, while natural from an algebraic
point of view, the noncommutative setting is only recently and slowly getting
explored.}

At some point, Glick and Grinberg noticed that the $Y$-variables in the
type-$AA$ Zamolodchikov periodicity theorem of Volkov~\cite{volkov} could be
written as ratios of labels under iterated birational rowmotion~\cite[\S ~4.4]%
{Roby15}; this allows the periodicity in one setting to be derived from that
in the other (with some work). However, for noncommutative $\mathbb{K}$,
Zamolodchikov periodicity fails even in small examples such as $r = r^{\prime
}= 2$ (no matter what order we multiply the factors), while noncommutative
birational rowmotion continues to exhibit periodicity. This approach is
therefore unavailable in the noncommutative case as well.

In this paper, we prove the periodicity of birational rowmotion and a
concomitant antipodal reciprocity formula over an arbitrary noncommutative
ring. The proof proceeds from first principles, by studying certain values
$\upslack_{\ell}^{v}$ and $\downslack_{\ell}^{v}$ and their products along
paths in the rectangle. At the core of the proof is a \textquotedblleft
conversion lemma\textquotedblright\ (Lemma \ref{lem.rect.conv}), which
provides an identity between a certain sum of $\upslack_{\ell}^{v}$ products
and a certain sum of $\downslack_{\ell}^{v}$ products for the same $\ell$;
this equality does not actually depend on the concept of rowmotion and might
be of interest on its own. Another important step is the reduction of the
reciprocity claim to the labels on the \textquotedblleft lower
boundary\textquotedblright\ of the rectangle (i.e., to the labels at the
elements of the form $\left(  i,1\right)  $ and $\left(  1,j\right)  $). This
reduction requires subtraction, which is why we are only addressing the case
of a ring, not of a semiring; the latter remains open.


A few words are in order about the relation between our birational rowmotion
and a parallel construction. Combinatorial rowmotion seems first to have been
defined not on the set $J\left(  P\right)  $ of order ideals of $P$, but
rather on the set $\mathcal{A}\left(  P\right)  $ of \emph{antichains} of
$P$~\cite{brouwer-schrijver}. The standard bijection between $J(P)$ and
$\mathcal{A}(P)$ (by taking maximal elements of $I\in J(P)$ or saturating down
from an antichain) makes it easy to go between the two maps and to see that
they have the same periodicity. However, some dynamic properties (e.g.,
homomesy) that depend on the sets themselves are not so easily translated.
Just as Einstein and Propp lifted combinatorial rowmotion on $J\left(
P\right)  $ to a birational map and we continued to the noncommutative
context, Joseph and Roby did a parallel lifting on the antichain side: from
antichain rowmotion to piecewise-linear rowmotion on the \emph{chain
polytope}, $\mathcal{C}(P)$, to birational antichain rowmotion, and finally to
noncommutative antichain rowmotion~\cite{JosRob20,JosRob21}. In particular
they lifted ``transfer maps'' (originally defined by Stanley to go between
$\mathcal{O}(P)$ and $\mathcal{C}(P)$~\cite{stanley-polytopes}) from the
piecewise-linear to the birational and noncommutative realms. These serve as
equivariant bijections at each level, thus showing that periodicity at each
level is equivalent for the order-ideal and antichain liftings. But they were
unable to find a new proof of periodicity for the piecewise-linear and higher
levels, relying instead on the periodicity results for birational order-ideal
rowmotion to deduce it for birational antichain rowmotion. They also lifted a
useful invariant, the \emph{Stanley--Thomas word}, which cyclically rotates
with antichain rowmotion at each level. At the combinatorial level, this gives
an equivariant bijection that proves periodicity~\cite[\S ~3.3.2]{propp-roby};
however, it is no longer a bijection at higher levels. Our paper completes the
story in the case of a ring: Via the transfer maps mentioned above, the
periodicity of noncommutative birational order-ideal rowmotion entails the
periodicity of noncommutative birational antichain rowmotion.


The paper is structured in a fairly straightforward way: In the first sections
(Sections \ref{sec.linext} to \ref{sec.ncbr}), we introduce our noncommutative
setup and define birational rowmotion in it. These include technicalities
about partial maps and the definition of noncommutative toggles. In Section
\ref{sec.rect}, we state our main results. In the sections that follow, we
build an arsenal of lemmas to prove these results; the proofs are completed in
Section \ref{sec.gencase}. (The structure of the proof is outlined at the end
of Section \ref{sec.rect}.) In Sections \ref{sec.semiring} and
\ref{sec.otherposets}, we discuss avenues for further work: a possible
generalization to semirings and conjectured periodicity claims for other
posets. In the final Section \ref{sec.Rf11}, we apply our techniques to
arbitrary posets (not just rectangles), obtaining two identities.

A 12-page survey of the results of this paper (with the main steps of the
proof outlined) can be found in the extended abstract \cite{this-fpsac}.

\subsection{Remark on the level of detail}

This paper comes in two versions: a regular one and a more detailed one. The
regular version is optimized for readability, leaving out the more
straightforward parts and technical arguments. The more detailed version has
many of them expanded.

\begin{vershort}
This is the regular version of the paper. The more detailed one can be
obtained by replacing\newline\texttt{%
$\backslash$%
excludecomment\{verlong\}}\newline\texttt{%
$\backslash$%
includecomment\{vershort\}}\newline by\newline\texttt{%
$\backslash$%
includecomment\{verlong\}}\newline\texttt{%
$\backslash$%
excludecomment\{vershort\}}\newline in the preamble of the LaTeX sourcecode
and then compiling to PDF. It is also available as an ancillary file on the
arXiv page of this paper.
\end{vershort}

\begin{verlong}
This is the more detailed version of the paper. The two versions share the
same .tex file, with the only difference that there are two lines in the
preamble of the file which need to be modified in order to switch between the
short and the detailed version. Namely, these lines are \newline\texttt{%
$\backslash$%
excludecomment\{verlong\}}\newline\texttt{%
$\backslash$%
includecomment\{vershort\}}\newline for the short version and \newline\texttt{%
$\backslash$%
includecomment\{verlong\}}\newline\texttt{%
$\backslash$%
excludecomment\{vershort\}}\newline for the detailed one. It is also available
on the arXiv page of this paper.
\end{verlong}

\subsection{Acknowledgments}

We are greatly indebted to the Mathematisches Forschungsinstitut Oberwolfach,
which hosted us for three weeks during Summer 2021. Much of this paper was
conceived during that stay. We thank Gerhard Huisken and Andrea Schillinger in
particular for their flexibility in the scheduling of the visit.

We are also grateful to Banff International Research Station for hosting a
hybrid workshop on dynamical algebraic combinatorics in November 2021 where
these results were first presented.

We further acknowledge our appreciation of Michael Joseph, Tim Campion, Max
Glick, Maxim Kontsevich, Gregg Musiker, Pace Nielsen, James Propp, Pasha
Pylyavskyy, Bruce Sagan, Roland Speicher, David Speyer, Hugh Thomas, and Jurij
Volcic, for useful advice and conversations. We thank two referees for helpful
corrections and advice.

Computations using the SageMath computer algebra system \cite{sage} provided
essential data for us to conjecture some of the results.

\section{\label{sec.linext}Linear extensions of posets}

This section collects a few standard notions concerning posets and their
linear extensions, needed to define the main characters of our paper. Readers
familiar with the subject may wish to skip forward to
Section~\ref{sec.inverses} or Section~\ref{sec.ncbr}. We start by defining
general notations identical with those in \cite{bir-row-arxiv}, to which we
refer the reader for commentary and comparison to other references.

\begin{convention}
We let $\mathbb{N}$ denote the set $\left\{  0,1,2,\ldots\right\}  $.
\end{convention}

\begin{definition}
Let $P$ be a poset, and $u,v\in P$.

\begin{enumerate}
\item[\textbf{(a)}] We will use the symbols $\leq$, $<$, $\geq$ and $>$ to
denote the lesser-or-equal relation, the lesser relation, the greater-or-equal
relation and the greater relation, respectively, of the poset $P$. (Thus, for
example, \textquotedblleft$u<v$\textquotedblright\ means \textquotedblleft$u$
is smaller than $v$ with respect to the partial order on $P$\textquotedblright.)

\item[\textbf{(b)}] The elements $u$ and $v$ of $P$ are said to be
\emph{incomparable} if we have neither $u\leq v$ nor $u\geq v$.

\item[\textbf{(c)}] We write $u\lessdot v$ if we have $u<v$ and there is no
$w\in P$ such that $u<w<v$. One often says that \textquotedblleft$u$ is
covered by $v$\textquotedblright\ to signify that $u\lessdot v$.

\item[\textbf{(d)}] We write $u\gtrdot v$ if we have $u>v$ and there is no
$w\in P$ such that $u>w>v$. (Thus, $u\gtrdot v$ holds if and only if
$v\lessdot u$.) One often says that \textquotedblleft$u$ covers $v$%
\textquotedblright\ to signify that $u\gtrdot v$.

\item[\textbf{(e)}] An element $u$ of $P$ is called \emph{maximal} if every
$w\in P$ satisfying $w\geq u$ satisfies $w=u$. In other words, an element $u$
of $P$ is called \emph{maximal} if there is no $w\in P$ such that $w>u$.

\item[\textbf{(f)}] An element $u$ of $P$ is called \emph{minimal} if every
$w\in P$ satisfying $w\leq u$ satisfies $w=u$. In other words, an element $u$
of $P$ is called \emph{minimal} if there is no $w\in P$ such that $w<u$.
\end{enumerate}

These notations may become ambiguous when an element belongs to several
different posets simultaneously. In such cases, we will disambiguate them by
adding the words \textquotedblleft in $P$\textquotedblright\ (where $P$ is the
poset which we want to use).\footnotemark
\end{definition}

\footnotetext{For instance, if $R$ denotes the poset $\mathbb{Z}$ endowed with
the \textbf{reverse} of its usual order, then we say (for instance) that
\textquotedblleft$0>3$ in $R$\textquotedblright\ rather than just
\textquotedblleft$0>3$\textquotedblright\ (to avoid mistaking our statement
for an absurd claim about the usual order on $\mathbb{Z}$).}

\begin{convention}
From now on, for the rest of the paper, \textbf{we fix a finite poset }$P$.
Most of our results will concern the case when $P$ has a rather specific form
(viz., a rectangular poset, i.e., a Cartesian product of two finite chains),
but we do not assume this straightaway.
\end{convention}

\begin{definition}
\label{def.linext}A \emph{linear extension} of $P$ will mean a list $\left(
v_{1},v_{2},\ldots,v_{m}\right)  $ of the elements of $P$ such that

\begin{itemize}
\item each element of $P$ occurs exactly once in this list, and

\item any $i,j\in\left\{  1,2,\ldots,m\right\}  $ satisfying $v_{i}<v_{j}$ (in
$P$) must satisfy $i<j$ (in $\mathbb{Z}$).
\end{itemize}
\end{definition}

A linear extension of $P$ is also known as a \emph{topological sorting} of $P$.

We will use the following well-known fact:

\begin{theorem}
\label{thm.linext.ex}There exists a linear extension of $P$.
\end{theorem}

\begin{definition}
\label{def.L(P)}The set of all linear extensions of $P$ will be called
$\mathcal{L}\left(  P\right)  $. Thus, $\mathcal{L}\left(  P\right)
\neq\varnothing$ (by Theorem \ref{thm.linext.ex}).
\end{definition}

The reader can easily verify the following proposition:

\begin{proposition}
\label{prop.linext.switch}Let $\left(  v_{1},v_{2},\ldots,v_{m}\right)  $ be a
linear extension of $P$. Let $i\in\left\{  1,2,\ldots,m-1\right\}  $ be such
that the elements $v_{i}$ and $v_{i+1}$ of $P$ are incomparable. Then $\left(
v_{1},v_{2},\ldots,v_{i-1},v_{i+1},v_{i},v_{i+2},v_{i+3},\ldots,v_{m}\right)
$ (this is the tuple obtained from the tuple $\left(  v_{1},v_{2},\ldots
,v_{m}\right)  $ by interchanging the adjacent entries $v_{i}$ and $v_{i+1}$)
is a linear extension of $P$ as well.
\end{proposition}

We will also use the following folklore result:\footnote{Particular cases of
Proposition \ref{prop.linext.transitive} have a tendency to appear in various
parts of combinatorics; see \cite[Proposition 1.3]{DefKra21} for a few such
references.}

\begin{proposition}
\label{prop.linext.transitive}Let $\sim$ denote the equivalence relation on
$\mathcal{L}\left(  P\right)  $ generated by the following requirement: For
any linear extension $\left(  v_{1},v_{2},\ldots,v_{m}\right)  $ of $P$ and
any $i\in\left\{  1,2,\ldots,m-1\right\}  $ such that the elements $v_{i}$ and
$v_{i+1}$ of $P$ are incomparable, we set
\[
\left(  v_{1},v_{2},\ldots,v_{m}\right)  \sim\left(  v_{1},v_{2}%
,\ldots,v_{i-1},v_{i+1},v_{i},v_{i+2},v_{i+3},\ldots,v_{m}\right)  .
\]
Then any two elements of $\mathcal{L}\left(  P\right)  $ are equivalent under
the relation $\sim$.
\end{proposition}

Proofs of Proposition \ref{prop.linext.transitive} can be found in
\cite[Proposition 1.7]{bir-row-arxiv}, in \cite[Proposition 4.1 (for the
$\pi^{\prime}=\pi\tau_{j}$ case)]{ayyer-klee-schilling}, in \cite[Lemma
1]{etienne} and in \cite[Lemma 4.2]{Gyoja86}\footnote{Note that the sources
\cite{ayyer-klee-schilling}, \cite{etienne} and \cite{Gyoja86} define linear
extensions of $P$ as bijections $\beta:\left\{  1,2,\ldots,n\right\}
\rightarrow P$ (where $n=\left\vert P\right\vert $) whose inverse map
$\beta^{-1}$ is order-preserving. This is equivalent to our definition
(indeed, if $\beta:\left\{  1,2,\ldots,n\right\}  \rightarrow P$ is a linear
extension of $P$ in their sense, then the list $\left(  \beta\left(  1\right)
,\ \beta\left(  2\right)  ,\ \ldots,\ \beta\left(  n\right)  \right)  $ is a
linear extension of $P$ in our sense).}. See also \cite[Proposition
2.2]{Naatz00} for a stronger claim (describing a shortest way to transform a
given linear extension into another by successively swapping adjacent
incomparable entries).

Another well-known fact says that any nonempty finite poset has a minimal
element and a maximal element. In other words:

\begin{proposition}
\label{prop.poset-minmax} Assume that $P \neq\varnothing$. Then:

\begin{enumerate}
\item[\textbf{(a)}] The poset $P$ has a minimal element.

\item[\textbf{(b)}] The poset $P$ has a maximal element.
\end{enumerate}
\end{proposition}

\begin{verlong}

\begin{proof}
[Proof sketch.]\textbf{(a)} For any $p \in P$, let $n_{p}$ denote the number
of all $q \in P$ satisfying $q < p$. Argue that if $a, b \in P$ are two
elements satisfying $a < b$, then $n_{a} < n_{b}$ (since any $q \in P$
satisfying $q < a$ must also satisfy $q < b$, and furthermore the element $a$
satisfies $a < b$ but not $a < a$). Use this to conclude that any element $p
\in P$ with minimum $n_{p}$ must be a minimal element of $P$.

\textbf{(b)} The proof is analogous to the proof of part \textbf{(a)}; just
replace some (not all!) ``$<$'' signs by ``$>$'' signs.
\end{proof}
\end{verlong}

\section{\label{sec.inverses}Inverses in rings}

\begin{convention}
From now on, for the rest of this paper, \textbf{we fix a ring} $\mathbb{K}$.
This ring is not required to be commutative, but must have a unity and be associative.
\end{convention}

For example, $\mathbb{K}$ can be $\mathbb{Z}$ or $\mathbb{Q}$ or $\mathbb{C}$
or a polynomial ring or a matrix ring over any of these. In almost all
previous work on birational rowmotion (with the exception of \cite{JosRob20}
and \cite{JosRob21}), only commutative rings (and, occasionally, semirings)
were considered; by removing the commutativity assumption, we are invalidating
many of the methods used in prior research. We suspect that the level of
generality can be increased even further, replacing our ring $\mathbb{K}$ by a
semiring (i.e., a \textquotedblleft ring without subtraction\textquotedblright%
); however, this poses new difficulties which we will not surmount in the
present work. (See Section \ref{sec.semiring} for more about this.)

Even as we do not assume our ring $\mathbb{K}$ to be a division ring, we will
nevertheless take multiplicative inverses of elements of $\mathbb{K}$ on many
occasions. These inverses do not always exist, but when they do exist, they
are unique; thus, we introduce a notation for them:

\begin{definition}
\label{def.inverses.inv}Let $a$ be an element of $\mathbb{K}$.

\begin{enumerate}
\item[\textbf{(a)}] An \emph{inverse} of $a$ means an element $b\in\mathbb{K}$
such that $ab=ba=1$. This inverse is unique when it exists, and will be
denoted by $\overline{a}$. (A more standard notation for it is $a^{-1}$, but
we prefer the notation $\overline{a}$ since it helps keep our formulas short.)

\item[\textbf{(b)}] We say that the element $a$ of $\mathbb{K}$ is
\emph{invertible} if it has an inverse.
\end{enumerate}
\end{definition}

The following well-known properties of inverses will often be used without mention:

\begin{proposition}
\label{prop.inverses.ab}\ \ 

\begin{enumerate}
\item[\textbf{(a)}] If $a$ is an invertible element of $\mathbb{K}$, then its
inverse $\overline{a}$ is invertible as well, and its inverse is
$\overline{\overline{a}}=a$.

\item[\textbf{(b)}] If $a$ and $b$ are two invertible elements of $\mathbb{K}%
$, then their product $ab$ is invertible as well, and its inverse is
$\overline{ab}=\overline{b}\cdot\overline{a}$.

\item[\textbf{(c)}] If $a_{1},a_{2},\ldots,a_{m}$ are several invertible
elements of $\mathbb{K}$, then their product $a_{1}a_{2}\cdots a_{m}$ is
invertible as well, and its inverse is $\overline{a_{1}a_{2}\cdots a_{m}%
}=\overline{a_{m}}\cdot\overline{a_{m-1}}\cdot\; \cdots\; \cdot\overline
{a_{1}}$.
\end{enumerate}
\end{proposition}

The converse of Proposition \ref{prop.inverses.ab} \textbf{(b)} does not
necessarily hold: A product $ab$ of two elements $a$ and $b$ of $\mathbb{K}$
can be invertible even when neither $a$ nor $b$ is\footnote{See
\url{https://math.stackexchange.com/questions/627562} for examples of such
situations.}.

The next property of inverses is less well-known:\footnote{Proposition
\ref{prop.inverses.a+b} \textbf{(a)} will not be used in what follows, but its
proof provides a good warm-up exercise in manipulating inverses in a
noncommutative ring.}

\begin{proposition}
\label{prop.inverses.a+b}Let $a$ and $b$ be two elements of $\mathbb{K}$ such
that $a+b$ is invertible. Then:

\begin{enumerate}
\item[\textbf{(a)}] We have $a\cdot\overline{a+b}\cdot b=b\cdot\overline
{a+b}\cdot a$.

\item[\textbf{(b)}] If both $a$ and $b$ are invertible, then $\overline
{a}+\overline{b}$ is invertible as well and its inverse is%
\[
\overline{\overline{a}+\overline{b}}=a\cdot\overline{a+b}\cdot b.
\]

\end{enumerate}
\end{proposition}

\begin{proof}
\textbf{(a)} Comparing
\[
a\cdot\overline{a+b}\cdot a+a\cdot\overline{a+b}\cdot b=a\cdot
\underbrace{\overline{a+b}\cdot\left(  a+b\right)  }_{=1}=a
\]
with%
\[
a\cdot\overline{a+b}\cdot a+b\cdot\overline{a+b}\cdot a=\underbrace{\left(
a+b\right)  \cdot\overline{a+b}}_{=1}\, \cdot\, a=a,
\]
we obtain%
\[
a\cdot\overline{a+b}\cdot a+a\cdot\overline{a+b}\cdot b=a\cdot\overline
{a+b}\cdot a+b\cdot\overline{a+b}\cdot a.
\]
Subtracting $a\cdot\overline{a+b}\cdot a$ from both sides of this equality, we
obtain $a\cdot\overline{a+b}\cdot b=b\cdot\overline{a+b}\cdot a$. This proves
Proposition \ref{prop.inverses.a+b} \textbf{(a)}. \medskip


\textbf{(b)} Assume that both $a$ and $b$ are invertible. Set $x:=\overline
{a}+\overline{b}$ and $y:=a\cdot\overline{a+b}\cdot b$.

From $x=\overline{a}+\overline{b}$, we obtain $x\cdot a=\left(  \overline
{a}+\overline{b}\right)  \cdot a=\underbrace{\overline{a}a}_{=1}%
+\,\overline{b}a=1+\overline{b}a$. Comparing this with%
\[
\overline{b}\cdot\left(  a+b\right)  =\overline{b}a+\underbrace{\overline{b}%
b}_{=1}=\overline{b}a+1=1+\overline{b}a,
\]
we obtain $x\cdot a=\overline{b}\cdot\left(  a+b\right)  $. Now, from
$y=a\cdot\overline{a+b}\cdot b$, we obtain%
\[
x\cdot y=\underbrace{x\cdot a}_{=\overline{b}\cdot\left(  a+b\right)  }\cdot\,
\overline{a+b}\cdot b=\overline{b}\cdot\underbrace{\left(  a+b\right)
\cdot\overline{a+b}}_{=1}\cdot\, b=\overline{b}\cdot b=1.
\]

\begin{vershort}
A similar argument (starting with $b \cdot x = b \overline{a} + 1 = \left(  a
+ b\right)  \cdot\overline{a}$) shows that $y\cdot x = 1$, so that $y$ is an
inverse of $x$. Hence, $x$ is invertible and its inverse is $\overline{x} =
y$. This is precisely the claim of Proposition \ref{prop.inverses.a+b}
\textbf{(b)}. \qedhere

\end{vershort}


\begin{verlong}
Furthermore, from $x=\overline{a}+\overline{b}$, we obtain $b\cdot
x=b\cdot\left(  \overline{a}+\overline{b}\right)  =b\overline{a}%
+\underbrace{b\overline{b}}_{=1}=b\overline{a}+1$. Comparing this with%
\[
\left(  a+b\right)  \cdot\overline{a}=\underbrace{a\overline{a}}_{=1}+\,
b\overline{a}=1+b\overline{a}=b\overline{a}+1,
\]
we obtain $b\cdot x=\left(  a+b\right)  \cdot\overline{a}$. Now, from
$y=a\cdot\overline{a+b}\cdot b$, we obtain%
\[
y\cdot x=a\cdot\overline{a+b}\cdot\underbrace{b\cdot x}_{=\left(  a+b\right)
\cdot\overline{a}}=a\cdot\underbrace{\overline{a+b}\cdot\left(  a+b\right)
}_{=1}\cdot\, \overline{a}=a\cdot\overline{a}=1.
\]

From $x\cdot y=1$ and $y\cdot x=1$, we conclude that $y$ is an inverse of $x$.
In other words, $a\cdot\overline{a+b}\cdot b$ is an inverse of $\overline
{a}+\overline{b}$ (since $x=\overline{a}+\overline{b}$ and $y=a\cdot
\overline{a+b}\cdot b$). Thus, $\overline{a}+\overline{b}$ is invertible and
its inverse is $\overline{\overline{a}+\overline{b}}=a\cdot\overline{a+b}\cdot
b$. This proves Proposition \ref{prop.inverses.a+b} \textbf{(b)}. \qedhere

\end{verlong}
\end{proof}

\section{\label{sec.ncbr}Noncommutative birational rowmotion}

In this section, we introduce the basic objects whose nature we will
investigate: labelings of a finite poset $P$ by elements of a ring, and a
partial map between them called \textquotedblleft birational
rowmotion\textquotedblright. These labelings generalize the field-valued
labelings studied in \cite{bir-row-arxiv}, which in turn generalize the
piecewise-linear labelings of \cite{einstein-propp}, which in turn generalize
the order ideals of $P$. Many of the definitions that follow will imitate
analogous definitions made (in somewhat lesser generality) in
\cite{bir-row-arxiv}.

\subsection{The extended poset $\protect\widehat{P}$}

\begin{definition}
\label{def.Phat}We define a poset $\widehat{P}$ as follows: As a set, let
$\widehat{P}$ be the disjoint union of the set $P$ with the two-element set
$\left\{  0,1\right\}  $. The smaller-or-equal relation $\leq$ on
$\widehat{P}$ will be given by%
\[
\left(  a\leq b\right)  \Longleftrightarrow\left(  \left(  a\in P\text{ and
}b\in P\text{ and }a\leq b\text{ in }P\right)  \text{ or }a=0\text{ or
}b=1\right)  .
\]
Here and in the following, we regard the canonical injection of the set $P$
into the disjoint union $\widehat{P}$ as an inclusion; thus, $P$ becomes a
subposet of $\widehat{P}$.
\end{definition}

In the terminology of Stanley's \cite[Section 3.2]{Stanley-EC1}, this poset
$\widehat{P}$ is the ordinal sum $\left\{  0\right\}  \oplus P\oplus\left\{
1\right\}  $.

\Needspace{5cm}

\begin{example}
Let us represent posets by their Hasse diagrams. Then:%
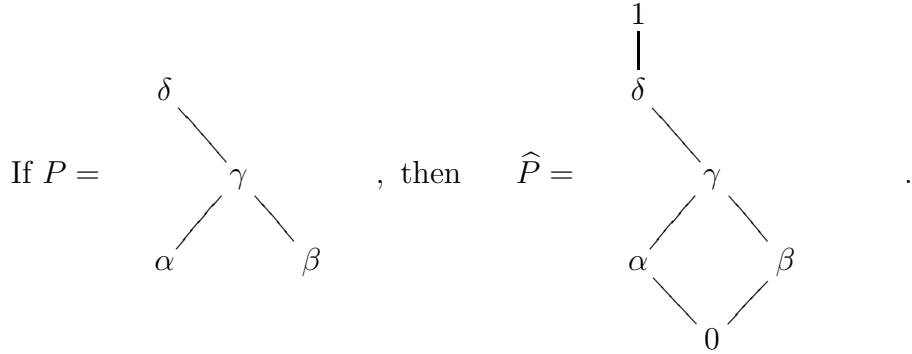
\[%
\xymatrixrowsep{1.2pc} \xymatrixcolsep{1.2pc} \xymatrix{
& & & & & & 1 \ar@{-}[d] & & \\
& \delta\ar@{-}[rd] & & & & & \delta\ar@{-}[rd] & & \\
\text{If } P= & & \gamma\ar@{-}[ld] \ar@{-}[rd] & & , \text{ then}
& \widehat{P} = & & \gamma\ar@{-}[ld] \ar@{-}[rd] & & & .\\
& \alpha& & \beta& & & \alpha\ar@{-}[rd] & & \beta\ar@{-}[ld] \\
& & & & & & & 0
}%
\]

\end{example}

\begin{verlong}

\begin{remark}
\label{rmk.Phat.covers} The following observations are easy to check and will
be used tacitly:

\begin{enumerate}
\item[\textbf{(a)}] An element $p\in P$ satisfies $0\lessdot p$ in
$\widehat{P}$ if and only if $p$ is a minimal element of $P$.

\item[\textbf{(b)}] An element $p\in P$ satisfies $1\gtrdot p$ in
$\widehat{P}$ if and only if $p$ is a maximal element of $P$.

\item[\textbf{(c)}] We have $0\lessdot1$ in $\widehat{P}$ if and only if
$P=\varnothing$.

\item[\textbf{(d)}] Two elements $p, q \in P$ satisfy $p \lessdot q$ in
$\widehat{P}$ if and only if they satisfy $p \lessdot q$ in $P$. An analogous
statement holds with the symbol ``$\lessdot$'' replaced by ``$\gtrdot$''.
\end{enumerate}
\end{remark}

\begin{convention}
Let $u$ and $v$ be two elements of $P$. Then $u$ and $v$ are also elements of
$\widehat{P}$ (since we are regarding $P$ as a subposet of $\widehat{P}$).
Thus, strictly speaking, statements like \textquotedblleft$u<v$%
\textquotedblright\ or \textquotedblleft$u\lessdot v$\textquotedblright\ are
ambiguous because it is not clear whether they are referring to the poset $P$
or to the poset $\widehat{P}$. However, this ambiguity is harmless, because it
is easily seen that the truth of each of the statements \textquotedblleft%
$u<v$\textquotedblright, \textquotedblleft$u\leq v$\textquotedblright,
\textquotedblleft$u>v$\textquotedblright, \textquotedblleft$u\geq
v$\textquotedblright, \textquotedblleft$u\lessdot v$\textquotedblright,
\textquotedblleft$u\gtrdot v$\textquotedblright\ and \textquotedblleft$u$ and
$v$ are incomparable\textquotedblright\ is independent on whether it refers to
the poset $P$ or to the poset $\widehat{P}$. We are going to therefore omit
mentioning the poset in these statements, unless there are other reasons for
us to do so.
\end{convention}
\end{verlong}

\subsection{$\mathbb{K}$-labelings}

Let us now define the type of object on which our maps will act:

\begin{definition}
\label{def.labeling}A $\mathbb{K}$\emph{-labeling of }$P$ will mean a map
$f:\widehat{P}\rightarrow\mathbb{K}$. Thus, $\mathbb{K}^{\widehat{P}}$ is the
set of all $\mathbb{K}$-labelings of $P$. If $f$ is a $\mathbb{K}$-labeling of
$P$ and $v$ is an element of $\widehat{P}$, then $f\left(  v\right)  $ will be
called the \emph{label of }$f$ \emph{at }$v$.
\end{definition}

\Needspace{20cm}

\begin{example}
\label{exa.2x2-rect}Assume that $P$ is the poset $\left\{  1,2\right\}
\times\left\{  1,2\right\}  $ with order relation defined by setting
\[
\left(  i,k\right)  \leq\left(  i^{\prime},k^{\prime}\right)
\ \ \ \ \ \ \ \ \ \ \text{if and only if }\left(  i\leq i^{\prime}\text{ and
}k\leq k^{\prime}\right)  .
\]
This poset will later be called the \textquotedblleft$2\times2$%
-rectangle\textquotedblright\ in Definition \ref{def.rect}. It has Hasse
diagram
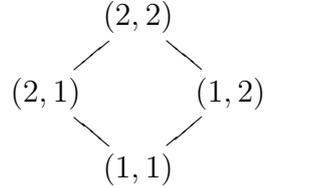
\[
\xymatrixrowsep{0.9pc}\xymatrixcolsep{0.20pc}\xymatrix{
& \left(2,2\right) \ar@{-}[rd] \ar@{-}[ld] & \\
\left(2,1\right) \ar@{-}[rd] & & \left(1,2\right) \ar@{-}[ld] \\
& \left(1,1\right) & & & .
}
\]
The extended poset $\widehat{P}$ has Hasse diagram
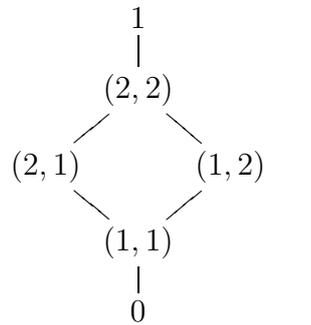
\[
\xymatrixrowsep{0.9pc}\xymatrixcolsep{0.20pc}\xymatrix{
& 1 \ar@{-}[d] & \\
& \left(2,2\right) \ar@{-}[rd] \ar@{-}[ld] & \\
\left(2,1\right) \ar@{-}[rd] & & \left(1,2\right) \ar@{-}[ld] \\
& \left(1,1\right) \ar@{-}[d] & \\
& 0 &  & & .
}
\]

We recall that a $\mathbb{K}$-labeling of $P$ is a map $f:\widehat{P}%
\rightarrow\mathbb{K}$. We can visualize such a $\mathbb{K}$-labeling by
replacing, in the Hasse diagram of $\widehat{P}$, each element $v\in
\widehat{P}$ by the label $f\left(  v\right)  $. For example, the $\mathbb{Z}%
$-labeling of $P$ that sends $0$, $\left(  1,1\right)  $, $\left(  1,2\right)
$, $\left(  2,1\right)  $, $\left(  2,2\right)  $, and $1$ to $12$, $5$, $7$,
$-2$, $10$, and $14$, respectively can be visualized as follows:
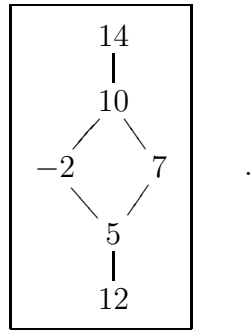
\begin{equation}%
\begin{tabular}
[c]{c|cc}\cline{2-2}
& \multicolumn{1}{|c|}{$\xymatrixrowsep{0.9pc}\xymatrixcolsep{0.20pc}%
\xymatrix{
& 14 \ar@{-}[d] & \\
& 10 \ar@{-}[rd] \ar@{-}[ld] & \\
-2 \ar@{-}[rd] & & 7 \ar@{-}[ld] \\
& 5 \ar@{-}[d] & \\
& 12 &
}%
$} & \\\cline{2-2}%
\end{tabular}
\ \ \ . \label{eq.exa.2x2-rect.labeling}%
\end{equation}
For example, its label at $\left(  1,2\right)  $ is $7$.

\begin{verlong}
(The rectangular box around the drawing (\ref{eq.exa.2x2-rect.labeling}) is
meant to signal that it shows a $\mathbb{K}$-labeling, not a poset. We will
follow this convention throughout this paper.)
\end{verlong}
\end{example}

\subsection{Partial maps}

We will next define the notion of a partial map, to formalize the idea of an
operation whose result may be undefined, such as division on $\mathbb{Q}$
(since division by zero is undefined). We will use $\undf$ as a symbol for
such undefined values:

\begin{convention}
We fix an object called $\undf$. In the following, we tacitly assume that none
of the sets we will consider contains this object $\undf$ (unless otherwise specified).
\end{convention}

The reader can think of $\undf$ as a ``division-by-zero error'' (more
precisely, a ``division-by-a-non-invertible-element error'', since $0$ is
often not the only non-invertible element of $\mathbb{K}$).


\begin{definition}
Let $X$ and $Y$ be two sets. A \emph{partial map} from $X$ to $Y$ means a map
from $X$ to $Y\sqcup\left\{  \undf\right\}  $.

If $f$ is a partial map from $X$ to $Y$, then $f$ can be canonically extended
to a map from $X\sqcup\left\{  \undf\right\}  $ to $Y\sqcup\left\{
\undf\right\}  $ by setting $f\left(  \undf\right)  :=\undf$. We always
consider $f$ to be extended in this way.

If $f$ is a partial map from $X$ to $Y$, then the set $\left\{  x\in
X\ \mid\ f\left(  x\right)  \neq\undf\right\}  $ will be called the
\emph{domain of definition} of $f$.
\end{definition}

We view the element $\undf$ as an \textquotedblleft undefined
output\textquotedblright\ -- i.e., we think of a partial map $f$ from $X$ to
$Y$ as a \textquotedblleft map\textquotedblright\ from $X$ to $Y$ that is
defined only on some elements of $X$ (namely, on those whose image under this
map is not $\undf$). Thus, for example, in $\mathbb{Q}$, division is a partial
map because division by $0$ is undefined:

\begin{example}
The map%
\begin{align*}
\mathbb{Q}  &  \rightarrow\mathbb{Q}\sqcup\left\{  \undf\right\}  ,\\
x  &  \mapsto%
\begin{cases}
1/x, & \text{if }x\neq0;\\
\undf, & \text{if }x=0
\end{cases}
\end{align*}
is a partial map from $\mathbb{Q}$ to $\mathbb{Q}$.
\end{example}

Partial maps can be composed much like usual maps:

\begin{definition}
\ \ 

\begin{enumerate}
\item[\textbf{(a)}] Let $X$, $Y$ and $Z$ be three sets. Let $f$ be a partial
map from $Y$ to $Z$. Let $g$ be a partial map from $X$ to $Y$.

Then $f\circ g$ denotes the partial map from $X$ to $Z$ that sends
\[
\text{each }x\in X\text{ to }%
\begin{cases}
f\left(  g\left(  x\right)  \right)  , & \text{if }g\left(  x\right)
\neq\undf\ ;\\
\undf, & \text{if }g\left(  x\right)  =\undf.
\end{cases}
\]
(Following our convention that $f\left(  \undf\right)  $ is understood to be
$\undf$, we could simplify the right hand side to just $f\left(  g\left(
x\right)  \right)  $, but we nevertheless subdivided it into two cases just to
stress the different branches in our \textquotedblleft control
flow\textquotedblright.)

This partial map $f\circ g$ is called the \emph{composition} of $f$ and $g$.

\item[\textbf{(b)}] This notion of composition lets us define a category whose
objects are sets and whose morphisms are partial maps. (The identity maps in
this category are the obvious ones: i.e., the maps $\operatorname*{id}%
:X\rightarrow X\sqcup\left\{  \undf\right\}  $ that send each $x\in X$ to
$x\in X\subseteq X\sqcup\left\{  \undf\right\}  $.)

\item[\textbf{(c)}] Thus, if $X$ is any set, and if $f$ is any partial map
from $X$ to $X$, then we can define $f^{k}:=\underbrace{f\circ f\circ
\cdots\circ f}_{k\text{ times}}$ for any $k\in\mathbb{N}$.
\end{enumerate}
\end{definition}

\begin{convention}
Let $X$ and $Y$ be two sets. We will write \textquotedblleft%
$f:X\dashrightarrow Y$\textquotedblright\ for \textquotedblleft$f$ is a
partial map from $X$ to $Y$\textquotedblright\ (just as maps from $X$ to $Y$
are denoted \textquotedblleft$f:X\rightarrow Y$\textquotedblright).
\end{convention}


A warning is worth making: While we are using the symbol $\dashrightarrow$ for
partial maps here, the same symbol has been used for rational maps in
\cite{bir-row-arxiv}. The two uses serve similar purposes (they both model
\textquotedblleft maps defined only on those inputs for which the relevant
denominators are invertible\textquotedblright), but they have some technical
differences. Rational maps are defined only when $\mathbb{K}$ is an infinite
field\footnote{It stands to reason that a notion of \textquotedblleft rational
map\textquotedblright\ should exist for a sufficiently wide class of infinite
skew-fields as well, but we have not encountered a satisfactory theory of such
maps in the literature. See \url{https://mathoverflow.net/questions/362724/}
for a discussion of how this theory might start. It appears unlikely, however,
that such \textquotedblleft noncommutative rational maps\textquotedblright%
\ exist in the generality that we are working in (viz., arbitrary rings).},
but are well-behaved in many ways that partial maps are not. (For example, a
rational map is uniquely determined if its values on a Zariski-dense subset of
its domain are known, but no such claims can be made for partial maps.) Thus,
by working with partial maps instead of rational maps, we are freeing
ourselves from technical assumptions on $\mathbb{K}$, but at the same time
forcing ourselves to be explicit about the domains on which our partial maps
are defined.

\subsection{Toggles}

Recall that $\mathbb{K}^{\widehat{P}}$ denotes the set of $\mathbb{K}%
$-labelings of a poset $P$ (that is, the set of all maps $\widehat{P}%
\rightarrow\mathbb{K}$). Next, we define \emph{(noncommutative) toggles}:
certain (fairly simple) partial maps on this set.

\begin{definition}
\label{def.Tv}Let $v\in P$. We define a partial map $T_{v}:\mathbb{K}%
^{\widehat{P}}\dashrightarrow\mathbb{K}^{\widehat{P}}$ as follows: If
$f\in\mathbb{K}^{\widehat{P}}$ is any $\mathbb{K}$-labeling of $P$, then the
$\mathbb{K}$-labeling $T_{v} f \in\mathbb{K}^{\widehat{P}}$ is given by
\begin{align}
\left(  T_{v}f\right)  \left(  w\right)   &  =%
\begin{cases}
f\left(  w\right)  , & \text{if }w\neq v;\\
\left(  \sum\limits_{\substack{u\in\widehat{P};\\u\lessdot v}}f\left(
u\right)  \right)  \cdot\overline{f\left(  v\right)  }\cdot\overline
{\sum\limits_{\substack{u\in\widehat{P};\\u\gtrdot v}}\overline{f\left(
u\right)  }}, & \text{if }w=v
\end{cases}
\label{def.Tv.def}\\
&  \qquad\qquad\qquad\qquad\qquad\qquad\qquad\qquad\qquad\text{for all }%
w\in\widehat{P}.\nonumber
\end{align}
Here, we agree that if any part of the expression $\left(  \sum
\limits_{\substack{u\in\widehat{P};\\u\lessdot v}}f\left(  u\right)  \right)
\cdot\overline{f\left(  v\right)  }\cdot\overline{\sum\limits_{\substack{u\in
\widehat{P};\\u\gtrdot v}}\overline{f\left(  u\right)  }}$ is not well-defined
(i.e., if one of the values $f\left(  u\right)  $ and $f\left(  v\right)  $
appearing in it is undefined, or if $f\left(  v\right)  $ is not invertible,
or if $f\left(  u\right)  $ is not invertible for some $u\in\widehat{P}$
satisfying $u\gtrdot v$, or if the sum $\sum\limits_{\substack{u\in
\widehat{P};\\u\gtrdot v}}\overline{f\left(  u\right)  }$ is not invertible),
then $T_{v}f$ is understood to be $\undf$.

This partial map $T_{v}$ is called the $v$\emph{-toggle} or the \emph{toggle
at }$v$.
\end{definition}

Thus, the partial map $T_{v}$ is a \textquotedblleft local\textquotedblright%
\ transformation: it only changes the label at the element $v$ (unless its
result is $\undf$).

\begin{remark}
You are reading Definition \ref{def.Tv} right: We set $T_{v}f=\undf$ if any of
$\overline{f\left(  v\right)  }$ and $\overline{\sum\limits_{\substack{u\in
\widehat{P};\\u\gtrdot v}}\overline{f\left(  u\right)  }}$ fails to be
well-defined. Thus, in this case, none of the values $\left(  T_{v}f\right)
\left(  w\right)  $ exists. It may appear more natural to leave only the value
$\left(  T_{v}f\right)  \left(  v\right)  $ undefined, while letting all other
values $\left(  T_{v}f\right)  \left(  w\right)  $ equal the respective values
$f\left(  w\right)  $. Our choice to \textquotedblleft panic and
crash\textquotedblright, however, will be more convenient for some of our proofs.
\end{remark}

The $v$-toggle $T_{v}$ is called a \textquotedblleft\emph{noncommutative order
toggle}\textquotedblright\ in \cite[Definition 5.6]{JosRob20}. When the ring
$\mathbb{K}$ is commutative, this $v$-toggle $T_{v}$ is an \textquotedblleft
involution\textquotedblright\ in the sense that each $\mathbb{K}$-labeling
$f\in\mathbb{K}^{\widehat{P}}$ satisfying $T_{v}\left(  T_{v}f\right)
\neq\undf$ satisfies $T_{v}\left(  T_{v}f\right)  =f$. For noncommutative
$\mathbb{K}$, this is usually not the case; an \textquotedblleft
inverse\textquotedblright\ partial map\footnote{We are putting the word
\textquotedblleft inverse\textquotedblright\ in scare quotes since we are
talking about partial maps, but the two maps are as close to being mutually
inverse as partial maps can be.} can be obtained by flipping the order of the
factors on the right hand side of (\ref{def.Tv.def}). (This \textquotedblleft
inverse\textquotedblright\ appears in \cite{JosRob20} under the name
\textquotedblleft noncommutative order elggot\textquotedblright.)

The following proposition is trivially obtained by rewriting (\ref{def.Tv.def}%
); we are merely stating it for easier reference in proofs:

\begin{proposition}
\label{prop.Tv}Let $v\in P$. For every $f\in\mathbb{K}^{\widehat{P}}$
satisfying $T_{v}f\neq\undf$, we have the following:

\begin{enumerate}
\item[\textbf{(a)}] Every $w\in\widehat{P}$ such that $w\neq v$ satisfies
$\left(  T_{v}f\right)  \left(  w\right)  =f\left(  w\right)  $.

\item[\textbf{(b)}] We have%
\[
\left(  T_{v}f\right)  \left(  v\right)  =\left(  \sum\limits_{\substack{u\in
\widehat{P};\\u\lessdot v}}f\left(  u\right)  \right)  \cdot\overline{f\left(
v\right)  }\cdot\overline{\sum\limits_{\substack{u\in\widehat{P};\\u\gtrdot
v}}\overline{f\left(  u\right)  }}.
\]

\end{enumerate}
\end{proposition}

Furthermore, the following \textquotedblleft locality
principle\textquotedblright\ (part of \cite[Proposition 5.8]{JosRob20}) is
easy to check:\footnote{In the following, equalities between partial maps are
understood in the strongest possible sense: Two partial maps
$F:X\dashrightarrow Y$ and $G:X\dashrightarrow Y$ satisfy $F=G$ if and only if
each $x\in X$ satisfies $F\left(  x\right)  =G\left(  x\right)  $. This
entails, in particular, that $F\left(  x\right)  =\undf$ holds if and only if
$G\left(  x\right)  =\undf$. Thus, $F=G$ is a stronger requirement than merely
saying that \textquotedblleft$F\left(  x\right)  =G\left(  x\right)  $
whenever neither $F\left(  x\right)  $ nor $G\left(  x\right)  $ is
$\undf$\textquotedblright.}

\begin{proposition}
\label{prop.Tv.commute}Let $v\in P$ and $w\in P$. Then $T_{v}\circ T_{w}%
=T_{w}\circ T_{v}$, unless we have either $v\lessdot w$ or $w\lessdot v$.
\end{proposition}

\begin{vershort}

\begin{proof}
[Proof of Proposition \ref{prop.Tv.commute}.]In the case when $\mathbb{K}$ is
commutative, this is essentially \cite[Proposition 14]{bir-row-1}, except that
we are now more careful about well-definedness (since only invertible elements
have inverses). Yet, the proof given in \cite{bir-row-1} can easily be adapted
to the general (noncommutative) case. The details can be found in the detailed
version of this paper (but the reader should have an easy time reconstructing them).
\end{proof}
\end{vershort}

\begin{verlong}

\begin{proof}
[Proof of Proposition \ref{prop.Tv.commute}.]In the case when $\mathbb{K}$ is
commutative, this is essentially \cite[Proposition 2.10]{bir-row-arxiv},
except that we are now working with partial maps instead of rational maps. The
proof below is an adaptation of the proof given in \cite{bir-row-arxiv} to the
general (noncommutative) case, but it is structured more carefully in order to
pay the requisite attention to cases when some values are $\undf$.

Let us first forget that we fixed $v$ and $w$. We now introduce a convenient
notation. Namely, if $f\in\mathbb{K}^{\widehat{P}}$ is a $\mathbb{K}%
$-labeling, and if $v\in P$, then we define the element $X_{v}\left(
f\right)  \in\mathbb{K}\sqcup\left\{  \undf\right\}  $ as follows: We set%
\[
X_{v}\left(  f\right)  :=\left(  \sum\limits_{\substack{u\in\widehat{P}%
;\\u\lessdot v}}f\left(  u\right)  \right)  \cdot\overline{f\left(  v\right)
}\cdot\overline{\sum\limits_{\substack{u\in\widehat{P};\\u\gtrdot v}%
}\overline{f\left(  u\right)  }}%
\]
if the right hand side of this equation is well-defined; otherwise, we set
$X_{v}\left(  f\right)  :=\undf$.

The equality (\ref{def.Tv.def}) from the definition of the $v$-toggle $T_{v}$
can thus be rewritten as follows:%
\begin{equation}
\left(  T_{v}f\right)  \left(  w\right)  =%
\begin{cases}
f\left(  w\right)  , & \text{if }w\neq v;\\
X_{v}\left(  f\right)  , & \text{if }w=v
\end{cases}
\ \ \ \ \ \ \ \ \ \ \text{for all }w\in\widehat{P}.
\label{pf.prop.Tv.commute.defTv}%
\end{equation}
This holds for any $\mathbb{K}$-labeling $f\in\mathbb{K}^{\widehat{P}}$ and
any $v\in P$, provided that the expression $\left(  \sum
\limits_{\substack{u\in\widehat{P};\\u\lessdot v}}f\left(  u\right)  \right)
\cdot\overline{f\left(  v\right)  }\cdot\overline{\sum\limits_{\substack{u\in
\widehat{P};\\u\gtrdot v}}\overline{f\left(  u\right)  }}$ is well-defined
(i.e., provided that $X_{v}\left(  f\right)  \neq\undf$).

Thus, the definition of the $v$-toggle $T_{v}$ can be restated as follows: For
any $\mathbb{K}$-labeling $f\in\mathbb{K}^{\widehat{P}}$ and any element $v\in
P$, we let $T_{v}f$ be the $\mathbb{K}$-labeling defined by
(\ref{pf.prop.Tv.commute.defTv}) if $X_{v}\left(  f\right)  \neq\undf$;
otherwise, we set $T_{v}f:=\undf$. In particular, for any $\mathbb{K}%
$-labeling $f\in\mathbb{K}^{\widehat{P}}$ and any element $v\in P$, we have%
\begin{equation}
T_{v}f=\undf\text{ if and only if }X_{v}\left(  f\right)  =\undf.
\label{pf.prop.Tv.commute.Tvfnone}%
\end{equation}

Now, let $v\in P$ and $w\in P$ be two elements that satisfy neither $v\lessdot
w$ nor $w\lessdot v$. We must show that $T_{v}\circ T_{w}=T_{w}\circ T_{v}$.

If $v=w$, then this is obvious. Thus, we WLOG assume that $v\neq w$.

We fix a $\mathbb{K}$-labeling $f\in\mathbb{K}^{\widehat{P}}$. We will show
that $\left(  T_{v}\circ T_{w}\right)  f=\left(  T_{w}\circ T_{v}\right)  f$.

First we prove the following observation:

\begin{statement}
\textit{Observation 1:} If $T_{w}f\neq\undf$, then $X_{v}\left(
T_{w}f\right)  =X_{v}\left(  f\right)  $.
\end{statement}

\begin{proof}
[Proof of Observation 1.]Assume that $T_{w}f\neq\undf$.

We have $v\neq w$ and thus
\begin{equation}
\left(  T_{w}f\right)  \left(  v\right)  =f\left(  v\right)
\label{pf.prop.Tv.commute.o1.pf.0}%
\end{equation}
(by Proposition \ref{prop.Tv} \textbf{(a)}, applied to $w$ and $v$ instead of
$v$ and $w$).

For each $u\in\widehat{P}$ satisfying $u\gtrdot v$, we have $u\neq w$ (because
if we had $u=w$, then we would have $w=u\gtrdot v$ and thus $v\lessdot w$,
which would contradict the fact that we don't have $v\lessdot w$) and
therefore
\[
\left(  T_{w}f\right)  \left(  u\right)  =f\left(  u\right)
\]
(by Proposition \ref{prop.Tv} \textbf{(a)}, applied to $w$ and $u$ instead of
$v$ and $w$). Hence,%
\begin{equation}
\sum\limits_{\substack{u\in\widehat{P};\\u\gtrdot v}}\overline{\left(
T_{w}f\right)  \left(  u\right)  }=\sum\limits_{\substack{u\in\widehat{P}%
;\\u\gtrdot v}}\overline{f\left(  u\right)  }.
\label{pf.prop.Tv.commute.o1.pf.2}%
\end{equation}

For each $u\in\widehat{P}$ satisfying $u\lessdot v$, we have $u\neq w$
(because if we had $u=w$, then we would have $w=u\lessdot v$, which would
contradict the fact that we don't have $w\lessdot v$) and therefore
\[
\left(  T_{w}f\right)  \left(  u\right)  =f\left(  u\right)
\]
(by Proposition \ref{prop.Tv} \textbf{(a)}, applied to $w$ and $u$ instead of
$v$ and $w$). Hence,%
\begin{equation}
\sum\limits_{\substack{u\in\widehat{P};\\u\lessdot v}}\left(  T_{w}f\right)
\left(  u\right)  =\sum\limits_{\substack{u\in\widehat{P};\\u\lessdot
v}}f\left(  u\right)  . \label{pf.prop.Tv.commute.o1.pf.1}%
\end{equation}

The definition of $X_{v}\left(  f\right)  $ yields%
\begin{equation}
X_{v}\left(  f\right)  =\left(  \sum\limits_{\substack{u\in\widehat{P}%
;\\u\lessdot v}}f\left(  u\right)  \right)  \cdot\overline{f\left(  v\right)
}\cdot\overline{\sum\limits_{\substack{u\in\widehat{P};\\u\gtrdot v}%
}\overline{f\left(  u\right)  }}, \label{pf.prop.Tv.commute.o1.pf.5}%
\end{equation}
with the understanding that the right hand side is understood to be $\undf$ if
any of its sub-expressions is not well-defined.

The definition of $X_{v}\left(  T_{w}f\right)  $ yields%
\begin{align*}
X_{v}\left(  T_{w}f\right)   &  =\left(  \sum\limits_{\substack{u\in
\widehat{P};\\u\lessdot v}}\left(  T_{w}f\right)  \left(  u\right)  \right)
\cdot\overline{\left(  T_{w}f\right)  \left(  v\right)  }\cdot\overline
{\sum\limits_{\substack{u\in\widehat{P};\\u\gtrdot v}}\overline{\left(
T_{w}f\right)  \left(  u\right)  }}\\
&  =\left(  \sum\limits_{\substack{u\in\widehat{P};\\u\lessdot v}}f\left(
u\right)  \right)  \cdot\overline{f\left(  v\right)  }\cdot\overline
{\sum\limits_{\substack{u\in\widehat{P};\\u\gtrdot v}}\overline{f\left(
u\right)  }} \ \ \ \ \ \ \ \ \ \ \left(  \text{by
(\ref{pf.prop.Tv.commute.o1.pf.1}), (\ref{pf.prop.Tv.commute.o1.pf.2}) and
(\ref{pf.prop.Tv.commute.o1.pf.0})}\right)  .
\end{align*}
Comparing this with (\ref{pf.prop.Tv.commute.o1.pf.5}), we obtain
$X_{v}\left(  T_{w}f\right)  =X_{v}\left(  f\right)  $. This proves
Observation 1.
\end{proof}

Now, we are in one of the following four cases:

\textit{Case 1:} We have $X_{v}\left(  f\right)  \neq\undf$ and $X_{w}\left(
f\right)  \neq\undf$.

\textit{Case 2:} We have $X_{v}\left(  f\right)  =\undf$ and $X_{w}\left(
f\right)  \neq\undf$.

\textit{Case 3:} We have $X_{v}\left(  f\right)  \neq\undf$ and $X_{w}\left(
f\right)  =\undf$.

\textit{Case 4:} We have $X_{v}\left(  f\right)  =\undf$ and $X_{w}\left(
f\right)  =\undf$.

Let us first consider Case 1. In this case, we have $X_{v}\left(  f\right)
\neq\undf$ and $X_{w}\left(  f\right)  \neq\undf$. From $X_{v}\left(
f\right)  \neq\undf$, we obtain $T_{v}f\neq\undf$ (by
(\ref{pf.prop.Tv.commute.Tvfnone})). Similarly, $T_{w}f\neq\undf$. Thus,
Observation 1 yields $X_{v}\left(  T_{w}f\right)  =X_{v}\left(  f\right)
\neq\undf$. From this, we obtain $T_{v}\left(  T_{w}f\right)  \neq\undf$
(since (\ref{pf.prop.Tv.commute.Tvfnone}) (applied to $T_{w}f$ instead of $f$)
shows that we have $T_{v}\left(  T_{w}f\right)  =\undf$ if and only if
$X_{v}\left(  T_{w}f\right)  =\undf$). Similarly, $T_{w}\left(  T_{v}f\right)
\neq\undf$.

Now, let $x\in\widehat{P}$. Then%
\begin{equation}
\left(  T_{w}f\right)  \left(  x\right)  =%
\begin{cases}
f\left(  x\right)  , & \text{if }x\neq w;\\
X_{w}\left(  f\right)  , & \text{if }x=w
\end{cases}
\label{pf.prop.Tv.commute.c1.1}%
\end{equation}
(by (\ref{pf.prop.Tv.commute.defTv}), applied to $x$ and $w$ instead of $w$
and $v$). Furthermore, (\ref{pf.prop.Tv.commute.defTv}) (applied to $T_{w}f$
and $x$ instead of $f$ and $w$) yields%
\begin{align}
\left(  T_{v}\left(  T_{w}f\right)  \right)  \left(  x\right)   &  =%
\begin{cases}
\left(  T_{w}f\right)  \left(  x\right)  , & \text{if }x\neq v;\\
X_{v}\left(  T_{w}f\right)  , & \text{if }x=v
\end{cases}
\nonumber\\
&  =%
\begin{cases}
\left(  T_{w}f\right)  \left(  x\right)  , & \text{if }x\neq v;\\
X_{v}\left(  f\right)  , & \text{if }x=v
\end{cases}
\ \ \ \ \ \ \ \ \ \ \left(  \text{since }X_{v}\left(  T_{w}f\right)
=X_{v}\left(  f\right)  \right) \nonumber\\
&  =%
\begin{cases}%
\begin{cases}
f\left(  x\right)  , & \text{if }x\neq w;\\
X_{w}\left(  f\right)  , & \text{if }x=w,
\end{cases}
& \text{if }x\neq v;\\
X_{v}\left(  f\right)  , & \text{if }x=v
\end{cases}
\ \ \ \ \ \ \ \ \ \ \left(  \text{by (\ref{pf.prop.Tv.commute.c1.1})}\right)
\nonumber\\
&  =%
\begin{cases}
f\left(  x\right)  , & \text{if }x\neq v\text{ and }x\neq w;\\
X_{w}\left(  f\right)  , & \text{if }x\neq v\text{ and }x=w;\\
X_{v}\left(  f\right)  , & \text{if }x=v
\end{cases}
\nonumber\\
&  =%
\begin{cases}
f\left(  x\right)  , & \text{if }x\neq v\text{ and }x\neq w;\\
X_{w}\left(  f\right)  , & \text{if }x=w;\\
X_{v}\left(  f\right)  , & \text{if }x=v
\end{cases}
\label{pf.prop.Tv.commute.c1.2}%
\end{align}
(since the condition \textquotedblleft$x\neq v$ and $x=w$\textquotedblright%
\ is equivalent to \textquotedblleft$x=w$\textquotedblright\ (because $v\neq
w$)). The same argument (but with the roles of $v$ and $w$ interchanged) shows
that%
\begin{align*}
\left(  T_{w}\left(  T_{v}f\right)  \right)  \left(  x\right)   &  =%
\begin{cases}
f\left(  x\right)  , & \text{if }x\neq w\text{ and }x\neq v;\\
X_{v}\left(  f\right)  , & \text{if }x=v;\\
X_{w}\left(  f\right)  , & \text{if }x=w
\end{cases}
\\
&  =%
\begin{cases}
f\left(  x\right)  , & \text{if }x\neq w\text{ and }x\neq v;\\
X_{w}\left(  f\right)  , & \text{if }x=w;\\
X_{v}\left(  f\right)  , & \text{if }x=v
\end{cases}
\\
&  =%
\begin{cases}
f\left(  x\right)  , & \text{if }x\neq v\text{ and }x\neq w;\\
X_{w}\left(  f\right)  , & \text{if }x=w;\\
X_{v}\left(  f\right)  , & \text{if }x=v.
\end{cases}
\end{align*}
Comparing this with (\ref{pf.prop.Tv.commute.c1.2}), we obtain $\left(
T_{v}\left(  T_{w}f\right)  \right)  \left(  x\right)  =\left(  T_{w}\left(
T_{v}f\right)  \right)  \left(  x\right)  $.

Forget that we fixed $x$. We thus have proved that $\left(  T_{v}\left(
T_{w}f\right)  \right)  \left(  x\right)  =\left(  T_{w}\left(  T_{v}f\right)
\right)  \left(  x\right)  $ for each $x\in\widehat{P}$. In other words,
$T_{v}\left(  T_{w}f\right)  =T_{w}\left(  T_{v}f\right)  $. Thus, $\left(
T_{v}\circ T_{w}\right)  f=T_{v}\left(  T_{w}f\right)  =T_{w}\left(
T_{v}f\right)  =\left(  T_{w}\circ T_{v}\right)  f$. We have therefore proved
$\left(  T_{v}\circ T_{w}\right)  f=\left(  T_{w}\circ T_{v}\right)  f$ in
Case 1.

Let us now consider Case 2. In this case, we have $X_{v}\left(  f\right)
=\undf$ and $X_{w}\left(  f\right)  \neq\undf$. From $X_{v}\left(  f\right)
=\undf$, we obtain $T_{v}f=\undf$ (by (\ref{pf.prop.Tv.commute.Tvfnone})) and
therefore $T_{w}\left(  T_{v}f\right)  =T_{w}\left(  \undf\right)  =\undf$. On
the other hand, (\ref{pf.prop.Tv.commute.Tvfnone}) (applied to $w$ instead of
$v$) shows that $T_{w}f=\undf$ if and only if $X_{w}\left(  f\right)  =\undf$.
Hence, we have $T_{w}f\neq\undf$ (since $X_{w}\left(  f\right)  \neq\undf$).
Therefore, Observation 1 yields $X_{v}\left(  T_{w}f\right)  =X_{v}\left(
f\right)  =\undf$. However, (\ref{pf.prop.Tv.commute.Tvfnone}) (applied to
$T_{w}f$ instead of $f$) shows that $T_{v}\left(  T_{w}f\right)  =\undf$ if
and only if $X_{v}\left(  T_{w}f\right)  =\undf$. Hence, we have $T_{v}\left(
T_{w}f\right)  =\undf$ (since $X_{v}\left(  T_{w}f\right)  =\undf$).
Altogether, we now obtain%
\[
\left(  T_{v}\circ T_{w}\right)  f=T_{v}\left(  T_{w}f\right)  =\undf=T_{w}%
\left(  T_{v}f\right)  =\left(  T_{w}\circ T_{v}\right)  f.
\]
Hence, we have proved $\left(  T_{v}\circ T_{w}\right)  f=\left(  T_{w}\circ
T_{v}\right)  f$ in Case 2.

The proof of $\left(  T_{v}\circ T_{w}\right)  f=\left(  T_{w}\circ
T_{v}\right)  f$ in Case 3 is analogous to the proof we just showed in Case 2;
we only need to interchange $v$ with $w$.

Finally, let us consider Case 4. In this case, we have $X_{v}\left(  f\right)
=\undf$ and $X_{w}\left(  f\right)  =\undf$. From $X_{v}\left(  f\right)
=\undf$, we obtain $T_{v}f=\undf$ (by (\ref{pf.prop.Tv.commute.Tvfnone})).
Similarly, $T_{w}f=\undf$. Now,%
\[
\left(  T_{v}\circ T_{w}\right)  f=T_{v}\left(  \underbrace{T_{w}f}%
_{=\undf}\right)  =T_{v}\left(  \undf\right)  =\undf.
\]
Similarly, $\left(  T_{w}\circ T_{v}\right)  f=\undf$. Comparing these two
equalities, we obtain $\left(  T_{v}\circ T_{w}\right)  f=\left(  T_{w}\circ
T_{v}\right)  f$. Hence, we have proved $\left(  T_{v}\circ T_{w}\right)
f=\left(  T_{w}\circ T_{v}\right)  f$ in Case 4.

We have now proved $\left(  T_{v}\circ T_{w}\right)  f=\left(  T_{w}\circ
T_{v}\right)  f$ in all four Cases 1, 2, 3 and 4. Thus, $\left(  T_{v}\circ
T_{w}\right)  f=\left(  T_{w}\circ T_{v}\right)  f$ always holds.

Forget that we fixed $f$. We thus have shown that $\left(  T_{v}\circ
T_{w}\right)  f=\left(  T_{w}\circ T_{v}\right)  f$ for each $f\in
\mathbb{K}^{\widehat{P}}$. In other words, $T_{v}\circ T_{w}=T_{w}\circ T_{v}%
$. This proves Proposition \ref{prop.Tv.commute}.
\end{proof}
\end{verlong}

As a particular case of Proposition \ref{prop.Tv.commute}, we have the following:

\begin{corollary}
\label{cor.Tv.commute}Let $v$ and $w$ be two elements of $P$ which are
incomparable. Then $T_{v}\circ T_{w}=T_{w}\circ T_{v}$.
\end{corollary}

\begin{verlong}

\begin{proof}
Since $v$ and $w$ are incomparable, we have neither $v\lessdot w$ nor
$w\lessdot v$. Thus, Proposition \ref{prop.Tv.commute} yields $T_{v}\circ
T_{w}=T_{w}\circ T_{v}$. This proves Corollary \ref{cor.Tv.commute}.
\end{proof}
\end{verlong}

\begin{corollary}
\label{cor.R.welldef}Let $\left(  v_{1},v_{2},\ldots,v_{m}\right)  $ be a
linear extension of $P$. Then the partial map $T_{v_{1}}\circ T_{v_{2}}%
\circ\cdots\circ T_{v_{m}}:\mathbb{K}^{\widehat{P}}\dashrightarrow
\mathbb{K}^{\widehat{P}}$ is independent of the choice of the linear extension
$\left(  v_{1},v_{2},\ldots,v_{m}\right)  $.
\end{corollary}

\begin{vershort}

\begin{proof}
Combine Corollary \ref{cor.Tv.commute} with Proposition
\ref{prop.linext.transitive}.
\end{proof}
\end{vershort}

\begin{verlong}

\begin{proof}
[Proof of Corollary \ref{cor.R.welldef}.]Forget that we fixed the linear
extension $\left(  v_{1},v_{2},\ldots,v_{m}\right)  $.

If $\mathbf{v}=\left(  v_{1},v_{2},\ldots,v_{m}\right)  $ is a linear
extension of $P$, then we denote the partial map $T_{v_{1}}\circ T_{v_{2}%
}\circ\cdots\circ T_{v_{m}}$ by $R_{\mathbf{v}}$. We must prove that this
partial map $R_{\mathbf{v}}$ is independent of the choice of the linear
extension $\mathbf{v}$. In other words, we must prove the following claim:

\begin{statement}
\textit{Claim 1:} If $\mathbf{v}$ and $\mathbf{w}$ are any two linear
extensions of $P$, then $R_{\mathbf{v}}=R_{\mathbf{w}}$.
\end{statement}

Our proof of Claim 1 will rely on Proposition \ref{prop.linext.transitive}.

Consider the equivalence relation $\sim$ on $\mathcal{L}\left(  P\right)  $
introduced in Proposition \ref{prop.linext.transitive}. According to
Proposition \ref{prop.linext.transitive}, any two elements of $\mathcal{L}%
\left(  P\right)  $ are equivalent under the relation $\sim$.

We say that two linear extensions $\mathbf{v}$ and $\mathbf{w}$ of $P$ are
\emph{adjacent} if and only if they can be written in the forms%
\begin{align*}
\mathbf{v}  &  =\left(  v_{1},v_{2},\ldots,v_{m}\right)
\ \ \ \ \ \ \ \ \ \ \text{and}\\
\mathbf{w}  &  =\left(  v_{1},v_{2},\ldots,v_{i-1},v_{i+1},v_{i}%
,v_{i+2},v_{i+3},\ldots,v_{m}\right)  ,
\end{align*}
where $i\in\left\{  1,2,\ldots,m-1\right\}  $ is such that the elements
$v_{i}$ and $v_{i+1}$ of $P$ are incomparable. In other words, we say that two
linear extensions $\mathbf{v}$ and $\mathbf{w}$ of $P$ are \emph{adjacent} if
and only if $\mathbf{w}$ can be obtained from $\mathbf{v}$ by swapping two
consecutive entries, provided that these two entries are incomparable. It is
clear that the relation \textquotedblleft adjacent\textquotedblright\ is
symmetric: i.e., if two linear extensions $\mathbf{v}$ and $\mathbf{w}$ are
adjacent, then $\mathbf{w}$ and $\mathbf{v}$ are adjacent as well (because if
we swap two entries of $\mathbf{v}$ and then swap them again, then they end up
back in their original positions).

Now, we notice the following fact:

\begin{statement}
\textit{Claim 2:} If $\mathbf{v}$ and $\mathbf{w}$ are two adjacent linear
extensions of $P$, then $R_{\mathbf{v}}=R_{\mathbf{w}}$.
\end{statement}

\begin{proof}
[Proof of Claim 2.]Let $\mathbf{v}$ and $\mathbf{w}$ be two adjacent linear
extensions of $P$. According to the definition of \textquotedblleft
adjacent\textquotedblright, we can thus write $\mathbf{v}$ and $\mathbf{w}$ in
the forms
\begin{align*}
\mathbf{v}  &  =\left(  v_{1},v_{2},\ldots,v_{m}\right)
\ \ \ \ \ \ \ \ \ \ \text{and}\\
\mathbf{w}  &  =\left(  v_{1},v_{2},\ldots,v_{i-1},v_{i+1},v_{i}%
,v_{i+2},v_{i+3},\ldots,v_{m}\right)  ,
\end{align*}
where $i\in\left\{  1,2,\ldots,m-1\right\}  $ is such that the elements
$v_{i}$ and $v_{i+1}$ of $P$ are incomparable. Write them in this form, and
consider this $i$.

Since $v_{i}$ and $v_{i+1}$ are incomparable, we have $T_{v_{i}}\circ
T_{v_{i+1}}=T_{v_{i+1}}\circ T_{v_{i}}$ (by Corollary \ref{cor.Tv.commute}).

The definition of $R_{\mathbf{v}}$ yields
\begin{align*}
R_{\mathbf{v}}  &  =T_{v_{1}}\circ T_{v_{2}}\circ\cdots\circ T_{v_{m}%
}\ \ \ \ \ \ \ \ \ \ \left(  \text{since }\mathbf{v}=\left(  v_{1}%
,v_{2},\ldots,v_{m}\right)  \right) \\
&  =T_{v_{1}}\circ T_{v_{2}}\circ\cdots\circ T_{v_{i-1}}\circ
\underbrace{T_{v_{i}}\circ T_{v_{i+1}}}_{=T_{v_{i+1}}\circ T_{v_{i}}}\circ
T_{v_{i+2}}\circ T_{v_{i+3}}\circ\cdots\circ T_{v_{m}}\\
&  =T_{v_{1}}\circ T_{v_{2}}\circ\cdots\circ T_{v_{i-1}}\circ T_{v_{i+1}}\circ
T_{v_{i}}\circ T_{v_{i+2}}\circ T_{v_{i+3}}\circ\cdots\circ T_{v_{m}}.
\end{align*}
On the other hand, the definition of $R_{\mathbf{w}}$ yields
\begin{align*}
R_{\mathbf{w}}  &  =T_{v_{1}}\circ T_{v_{2}}\circ\cdots\circ T_{v_{i-1}}\circ
T_{v_{i+1}}\circ T_{v_{i}}\circ T_{v_{i+2}}\circ T_{v_{i+3}}\circ\cdots\circ
T_{v_{m}}\\
&  \ \ \ \ \ \ \ \ \ \ \ \ \ \ \ \ \ \ \ \ \left(  \text{since }%
\mathbf{w}=\left(  v_{1},v_{2},\ldots,v_{i-1},v_{i+1},v_{i},v_{i+2}%
,v_{i+3},\ldots,v_{m}\right)  \right)  .
\end{align*}
Comparing these two equalities, we obtain $R_{\mathbf{v}}=R_{\mathbf{w}}$.
This proves Claim 2.
\end{proof}

Now, recall that the equivalence relation $\sim$ is generated by the
elementary relations
\[
\left(  v_{1},v_{2},\ldots,v_{m}\right)  \sim\left(  v_{1},v_{2}%
,\ldots,v_{i-1},v_{i+1},v_{i},v_{i+2},v_{i+3},\ldots,v_{m}\right)  ,
\]
where $\left(  v_{1},v_{2},\ldots,v_{m}\right)  $ is a linear extension of $P$
and where $i\in\left\{  1,2,\ldots,m-1\right\}  $ is chosen such that the
elements $v_{i}$ and $v_{i+1}$ of $P$ are incomparable. In other words, the
equivalence relation $\sim$ is generated by the elementary relations%
\[
\mathbf{v}\sim\mathbf{w},\ \ \ \ \ \ \ \ \ \ \text{where }\mathbf{v}\text{ and
}\mathbf{w}\text{ are adjacent linear extensions}%
\]
(by the definition of \textquotedblleft adjacent\textquotedblright). In other
words, the equivalence relation $\sim$ is the reflexive, transitive and
symmetric closure of the relation \textquotedblleft adjacent\textquotedblright%
. In other words, the following holds:

\begin{statement}
\textit{Claim 3:} Let $\mathbf{v}$ and $\mathbf{w}$ be two linear extensions
of $P$. Then we have $\mathbf{v}\sim\mathbf{w}$ if and only if there exists a
tuple $\left(  \mathbf{u}_{0},\mathbf{u}_{1},\ldots,\mathbf{u}_{k}\right)  $
of linear extensions of $P$ such that $\mathbf{u}_{0}=\mathbf{v}$ and
$\mathbf{u}_{k}=\mathbf{w}$ and such that for each $i\in\left\{
1,2,\ldots,k\right\}  $, we have%
\[
\left(  \mathbf{u}_{i-1}\text{ and }\mathbf{u}_{i}\text{ are adjacent}\right)
\text{ or }\left(  \mathbf{u}_{i}\text{ and }\mathbf{u}_{i-1}\text{ are
adjacent}\right)  .
\]

\end{statement}

We are now ready to prove Claim 1:

\begin{proof}
[Proof of Claim 1.]Let $\mathbf{v}$ and $\mathbf{w}$ be any two linear
extensions of $P$. Then $\mathbf{v}$ and $\mathbf{w}$ are two elements of
$\mathcal{L}\left(  P\right)  $. Hence, $\mathbf{v}\sim\mathbf{w}$ (since any
two elements of $\mathcal{L}\left(  P\right)  $ are equivalent under the
relation $\sim$). Thus, Claim 3 shows that there exists a tuple $\left(
\mathbf{u}_{0},\mathbf{u}_{1},\ldots,\mathbf{u}_{k}\right)  $ of linear
extensions of $P$ such that $\mathbf{u}_{0}=\mathbf{v}$ and $\mathbf{u}%
_{k}=\mathbf{w}$ and such that for each $i\in\left\{  1,2,\ldots,k\right\}  $,
we have%
\begin{equation}
\left(  \mathbf{u}_{i-1}\text{ and }\mathbf{u}_{i}\text{ are adjacent}\right)
\text{ or }\left(  \mathbf{u}_{i}\text{ and }\mathbf{u}_{i-1}\text{ are
adjacent}\right)  . \label{pf.cor.R.welldef.c1.pf.1}%
\end{equation}
Consider this tuple $\left(  \mathbf{u}_{0},\mathbf{u}_{1},\ldots
,\mathbf{u}_{k}\right)  $.

Let $i\in\left\{  1,2,\ldots,k\right\}  $. Then $\left(  \mathbf{u}%
_{i-1}\text{ and }\mathbf{u}_{i}\text{ are adjacent}\right)  $ or $\left(
\mathbf{u}_{i}\text{ and }\mathbf{u}_{i-1}\text{ are adjacent}\right)  $ (by
(\ref{pf.cor.R.welldef.c1.pf.1})). In either of these two cases, we conclude
that $\mathbf{u}_{i-1}$ and $\mathbf{u}_{i}$ are adjacent (since the relation
\textquotedblleft adjacent\textquotedblright\ is symmetric). Hence, Claim 2
(applied to $\mathbf{u}_{i-1}$ and $\mathbf{u}_{i}$ instead of $\mathbf{v}$
and $\mathbf{w}$) yields $R_{\mathbf{u}_{i-1}}=R_{\mathbf{u}_{i}}$.

Forget that we fixed $i$. We thus have proved the equality $R_{\mathbf{u}%
_{i-1}}=R_{\mathbf{u}_{i}}$ for each $i\in\left\{  1,2,\ldots,k\right\}  $.
Combining all these equalities, we obtain%
\[
R_{\mathbf{u}_{0}}=R_{\mathbf{u}_{1}}=R_{\mathbf{u}_{2}}=\cdots=R_{\mathbf{u}%
_{k}}.
\]
Hence, $R_{\mathbf{u}_{0}}=R_{\mathbf{u}_{k}}$. In view of $\mathbf{u}%
_{0}=\mathbf{v}$ and $\mathbf{u}_{k}=\mathbf{w}$, we can rewrite this as
$R_{\mathbf{v}}=R_{\mathbf{w}}$. This proves Claim 1.
\end{proof}

Thus, the proof of Corollary \ref{cor.R.welldef} is complete (since Claim 1 is proved).
\end{proof}
\end{verlong}

\subsection{Birational rowmotion}

Recall that $P$ is a finite poset. Corollary \ref{cor.R.welldef} lets us make
the following definition.

\begin{definition}
\label{def.rm}\emph{Birational rowmotion} (or, more precisely, the
\emph{birational rowmotion of }$P$) is defined as the partial map $T_{v_{1}%
}\circ T_{v_{2}}\circ\cdots\circ T_{v_{m}}:\mathbb{K}^{\widehat{P}%
}\dashrightarrow\mathbb{K}^{\widehat{P}}$, where $\left(  v_{1},v_{2}%
,\ldots,v_{m}\right)  $ is a linear extension of $P$. This partial map is
well-defined, because

\begin{itemize}
\item Theorem \ref{thm.linext.ex} shows that a linear extension of $P$ exists, and

\item Corollary \ref{cor.R.welldef} shows that the partial map $T_{v_{1}}\circ
T_{v_{2}}\circ\cdots\circ T_{v_{m}}$ is independent of the choice of the
linear extension $\left(  v_{1},v_{2},\ldots,v_{m}\right)  $.
\end{itemize}

This partial map will be denoted by $R$.
\end{definition}

Birational rowmotion is called \textquotedblleft birational
NOR-motion\textquotedblright\ (and denoted $\operatorname*{NOR}$) in the paper
\cite[Definition 5.9]{JosRob20}\footnote{To be more precise, \cite[Definition
5.9]{JosRob20} works in a slightly less general context, requiring
$\mathbb{K}$ to be a skew field and that $f\left(  0\right)  =1$ and $f\left(
1\right)  =C$ for some $C$ in the center of $\mathbb{K}$.}. When $\mathbb{K}$
is commutative, it agrees with the standard concept of birational rowmotion as
studied in \cite{einstein-propp} and \cite{bir-row-arxiv}.

\newpage

\begin{example}
\label{ex.rowmotion.2x2}Let us demonstrate the effect of birational toggles
and birational rowmotion. Namely, for this example, we let $P$ be the poset
$P=\left\{  1,2\right\}  \times\left\{  1,2\right\}  $ introduced in Example
\ref{exa.2x2-rect}.

In order to disencumber our formulas, we agree to write $g\left(  i,j\right)
$ for $g\left(  \left(  i,j\right)  \right)  $ when $g$ is a $\mathbb{K}%
$-labeling of $P$ and $\left(  i,j\right)  $ is an element of $P$.

As in Example \ref{exa.2x2-rect}, we visualize a $\mathbb{K}$-labeling $f$ of
$P$ by replacing, in the Hasse diagram of $\widehat{P}$, each element
$v\in\widehat{P}$ by the label $f\left(  v\right)  $. Let $f$ be a
$\mathbb{K}$-labeling sending $0$, $\left(  1,1\right)  $, $\left(
1,2\right)  $, $\left(  2,1\right)  $, $\left(  2,2\right)  $, and $1$ to $a$,
$w$, $y$, $x$, $z$, and $b$, respectively (for some elements $a$, $b$, $x$,
$y$, $z$, $w$ of $\mathbb{K}$); this $f$ is then visualized as follows:
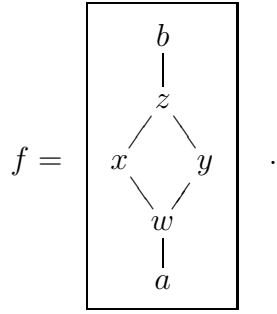
\[%
\begin{tabular}
[c]{c|c|}\cline{2-2}%
$\xymatrixrowsep{0.9pc}\xymatrixcolsep{0.20pc}%
\xymatrix{
\vphantom{a}\\
\vphantom{x}\\
f = \\
\\
}%
$ & $\xymatrixrowsep{0.9pc}\xymatrixcolsep{0.20pc}%
\xymatrix{
& b \ar@{-}[d] & \\
& z \ar@{-}[rd] \ar@{-}[ld] & \\
x \ar@{-}[rd] & & y \ar@{-}[ld] \\
& w \ar@{-}[d] & \\
& a &
}%
$\\\cline{2-2}%
\end{tabular}
\ \ \ .
\]
(As before, we draw $\left(  2,1\right)  $ on the western corner and $\left(
1,2\right)  $ on the eastern corner.)

Now, recall the definition of birational rowmotion $R$ on our poset $P$. Since
the list $\left(  \left(  1,1\right)  ,\left(  1,2\right)  ,\left(
2,1\right)  ,\left(  2,2\right)  \right)  $ is a linear extension of $P$, we
have $R=T_{\left(  1,1\right)  }\circ T_{\left(  1,2\right)  }\circ T_{\left(
2,1\right)  }\circ T_{\left(  2,2\right)  }$. Let us track how this transforms
our labeling $f$:

We first apply $T_{\left(  2,2\right)  }$, obtaining the $\mathbb{K}$-labeling%
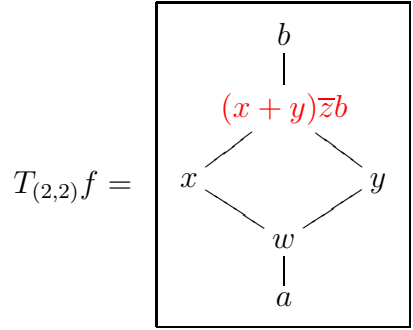
\[%
\begin{tabular}
[c]{c|c|}\cline{2-2}%
$\xymatrixrowsep{0.9pc}\xymatrixcolsep{0.20pc}%
\xymatrix{
\vphantom{b} \\
\vphantom{(x+y)\overline{z}b} \\
T_{\left(  2,2\right)  } f = \\
\\
}%
$ & $\xymatrixrowsep{0.9pc}\xymatrixcolsep{0.20pc}%
\xymatrix{
& b \ar@{-}[d] & \\
& \red{(x+y)\overline{z}b} \ar@{-}[rd] \ar@{-}[ld] & \\
x \ar@{-}[rd] & & y \ar@{-}[ld] \\
& w \ar@{-}[d] & \\
& a &
}%
$\\\cline{2-2}%
\end{tabular}
\
\]
(where we colored the label at $\left(  2,2\right)  $ red to signify that it
is the label at the element which got toggled). Indeed, the only label that
changes under $T_{\left(  2,2\right)  }$ is the one at $\left(  2,2\right)  $,
and this label becomes%
\begin{align*}
\left(  T_{\left(  2,2\right)  }f\right)  \left(  2,2\right)   &  =\left(
\sum\limits_{\substack{u\in\widehat{P};\\u\lessdot\left(  2,2\right)
}}f\left(  u\right)  \right)  \cdot\overline{f\left(  2,2\right)  }%
\cdot\overline{\sum\limits_{\substack{u\in\widehat{P};\\u\gtrdot\left(
2,2\right)  }}\overline{f\left(  u\right)  }}\\
&  =\left(  f\left(  1,2\right)  +f\left(  2,1\right)  \right)  \cdot
\overline{f\left(  2,2\right)  }\cdot\overline{\overline{f\left(  1\right)  }%
}\\
&  =\left(  y+x\right)  \cdot\overline{z}\cdot\overline{\overline{b}}=\left(
x+y\right)  \cdot\overline{z}\cdot b.
\end{align*}
(We assume that $z$ and $b$ are indeed invertible; otherwise, $T_{\left(
2,2\right)  }f$ would be $\undf$ and would remain $\undf$ after any further
toggles. Likewise, as we apply further toggles, we assume that everything else
we need to invert is invertible.)

Having applied $T_{\left(  2,2\right)  }$, we next apply $T_{\left(
2,1\right)  }$, obtaining%
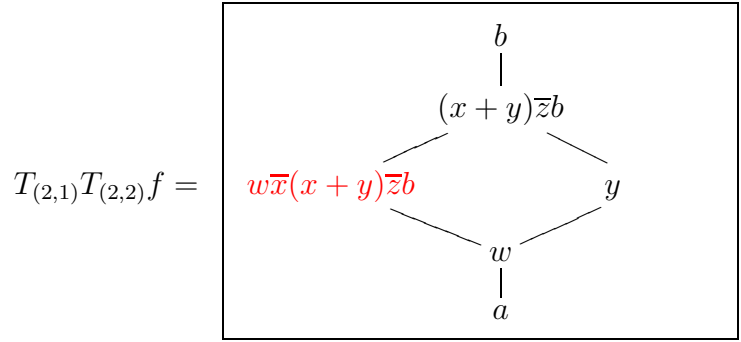
\[%
\begin{tabular}
[c]{c|c|}\cline{2-2}%
$\xymatrixrowsep{0.9pc}\xymatrixcolsep{0.20pc}%
\xymatrix{
\vphantom{b} \\
\vphantom{(x+y)\overline{z}b} \\
T_{\left(2,1\right)} T_{\left(  2,2\right)  } f = \\
\\
}%
$ & $\xymatrixrowsep{0.9pc}\xymatrixcolsep{0.20pc}%
\xymatrix{
& b \ar@{-}[d] & \\
& (x+y)\overline{z}b \ar@{-}[rd] \ar@{-}[ld] & \\
\red{w\overline{x}(x+y)\overline{z}b} \ar@{-}[rd] & & \phantom{x}%
y\phantom{abcdef} \ar@{-}[ld] \\
& w \ar@{-}[d] & \\
& a &
}%
$\\\cline{2-2}%
\end{tabular}
\ \ \ .
\]
Next, we apply $T_{\left(  1,2\right)  }$, obtaining%
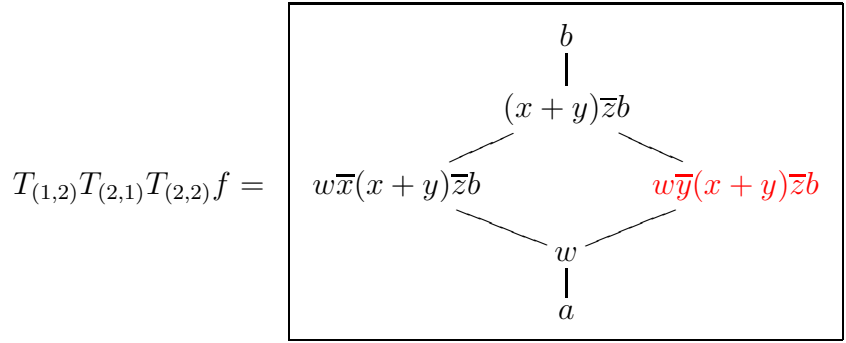
\[%
\begin{tabular}
[c]{c|c|}\cline{2-2}%
$\xymatrixrowsep{0.9pc}\xymatrixcolsep{0.20pc}%
\xymatrix{
\phantom{b} \\
\phantom{(x+y)\overline{z}b} \\
T_{\left(1,2\right)} T_{\left(2,1\right)} T_{\left(  2,2\right)  }
f = \\
\\
}%
$ & $\xymatrixrowsep{0.9pc}\xymatrixcolsep{0.20pc}%
\xymatrix{
& b \ar@{-}[d] & \\
& (x+y)\overline{z}b \ar@{-}[rd] \ar@{-}[ld] & \\
w\overline{x}(x+y)\overline{z}b \ar@{-}[rd] & & \red{w\overline{y}%
(x+y)\overline{z}b} \ar@{-}[ld] \\
& w \ar@{-}[d] & \\
& a &
}%
$\\\cline{2-2}%
\end{tabular}
\ \ \ .
\]
Finally, we apply $T_{\left(  1,1\right)  }$, resulting in
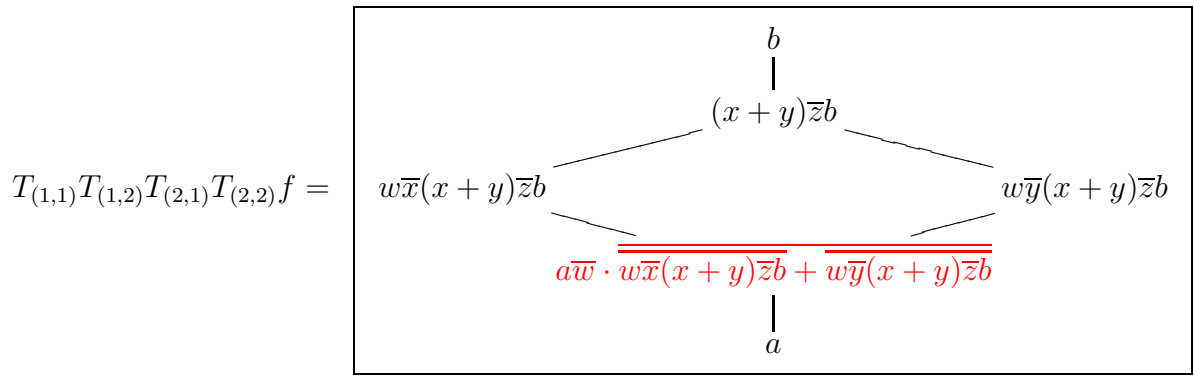
\[%
\begin{tabular}
[c]{c|c|}\cline{2-2}%
$\xymatrixrowsep{0.9pc}\xymatrixcolsep{0.20pc}%
\xymatrix{
\vphantom{b} \\
\vphantom{(x+y)\overline{z}b} \\
T_{\left(1,1\right)} T_{\left(1,2\right)} T_{\left(2,1\right)} T_{\left
(  2,2\right)  } f = \\
\\
}%
$ & $\xymatrixrowsep{0.9pc}\xymatrixcolsep{-0.20pc}%
\xymatrix{
& b \ar@{-}[d] & \\
& (x+y)\overline{z}b \ar@{-}[rd] \ar@{-}[ld] & \\
w\overline{x}(x+y)\overline{z}b \ar@{-}[rd] & & w\overline{y}(x+y)\overline
{z}b \ar@{-}[ld] \\
& \red{a \overline{w} \cdot\overline{\overline{w \overline{x}
(x+y) \overline{z} b} + \overline{w \overline{y} (x+y) \overline{z} b}}}
\ar@{-}[d] & \\
& a &
}%
$\\\cline{2-2}%
\end{tabular}
\ \ \ .
\]
The unwieldy expression $\overline{w}\cdot\overline{\overline{w\overline
{x}(x+y)\overline{z}b}+\overline{w\overline{y}(x+y)\overline{z}b}}$ in the
label at $\left(  1,1\right)  $ can be simplified to $\overline{z}b$ (using
standard laws such as $\overline{p}\cdot\overline{q}=\overline{qp}$ and
distributivity), so this rewrites as%
\[%
\begin{tabular}
[c]{c|c|}\cline{2-2}%
$\xymatrixrowsep{0.9pc}\xymatrixcolsep{0.20pc}%
\xymatrix{
\vphantom{b} \\
\vphantom{(x+y)\overline{z}b} \\
T_{\left(1,1\right)} T_{\left(1,2\right)} T_{\left(2,1\right)} T_{\left
(  2,2\right)  } f = \\
\\
}%
$ & $\xymatrixrowsep{0.9pc}\xymatrixcolsep{0.20pc}%
\xymatrix{
& b \ar@{-}[d] & \\
& (x+y)\overline{z}b \ar@{-}[rd] \ar@{-}[ld] & \\
w\overline{x}(x+y)\overline{z}b \ar@{-}[rd] & & w\overline{y}(x+y)\overline
{z}b \ar@{-}[ld] \\
& a \overline{z} b \ar@{-}[d] & \\
& a &
}%
$\\\cline{2-2}%
\end{tabular}
\ \ \ .
\]
We thus have computed $Rf$ (since $Rf=T_{\left(  1,1\right)  }T_{\left(
1,2\right)  }T_{\left(  2,1\right)  }T_{\left(  2,2\right)  }f$).

By repeating this procedure (or just substituting the labels of $Rf$ obtained
as variables), we can compute $R^{2}f$, $R^{3}f$ etc., obtaining \begingroup
\allowdisplaybreaks
\begin{align*}
&
\begin{tabular}
[c]{c|c|}\cline{2-2}%
$\xymatrixrowsep{0.9pc}\xymatrixcolsep{0.20pc}%
\xymatrix{
\vphantom{b} \\
\vphantom{(x+y)\overline{z}b} \\
R f = \\
\\
}%
$ & $\xymatrixrowsep{0.9pc}\xymatrixcolsep{0.20pc}%
\xymatrix{
& b \ar@{-}[d] & \\
& (x+y)\overline{z}b \ar@{-}[rd] \ar@{-}[ld] & \\
w\overline{x}(x+y)\overline{z}b \ar@{-}[rd] & & w\overline{y}(x+y)\overline
{z}b \ar@{-}[ld] \\
& a \overline{z} b \ar@{-}[d] & \\
& a &
}%
$\\\cline{2-2}%
\end{tabular}
\ \ ,\\
& \\
&
\begin{tabular}
[c]{c|c|}\cline{2-2}%
$\xymatrixrowsep{0.9pc}\xymatrixcolsep{0.20pc}%
\xymatrix{
\vphantom{b} \\
\vphantom{(x+y)\overline{z}b} \\
R^2 f = \\
\\
}%
$ & $\xymatrixrowsep{0.9pc}\xymatrixcolsep{0.20pc}%
\xymatrix{
& b \ar@{-}[d] & \\
& w\left(\overline{x}+\overline{y}\right)b \ar@{-}[rd] \ar@{-}[ld] & \\
a\cdot\overline{x+y}\cdot x\left(\overline{x}+\overline{y}\right)b \ar@
{-}[rd] & & a\cdot\overline{x+y}\cdot y\left(\overline{x}+\overline{y}%
\right)b \ar@{-}[ld] \\
& a\overline{b}z\cdot\overline{x+y}\cdot b \ar@{-}[d] & \\
& a &
}%
$\\\cline{2-2}%
\end{tabular}
\ \ ,\\
& \\
&
\begin{tabular}
[c]{c|c|}\cline{2-2}%
$\xymatrixrowsep{0.9pc}\xymatrixcolsep{0.20pc}%
\xymatrix{
\vphantom{b} \\
\vphantom{(x+y)\overline{z}b} \\
R^3 f = \\
\\
}%
$ & $\xymatrixrowsep{0.9pc}\xymatrixcolsep{0.20pc}%
\xymatrix{
& b \ar@{-}[d] & \\
& a\overline{w}b \ar@{-}[rd] \ar@{-}[ld] & \\
\cdots\ar@{-}[rd] & & a \overline{b}z\cdot\overline{x+y}\cdot\overline
{\overline{x}+\overline{y}}\cdot\overline{y}\cdot\left(x+y\right)\overline
{w}b \ar@{-}[ld] \\
& a\overline{b}\cdot\overline{\overline{x}+\overline{y}}\cdot\overline
{w}b \ar@{-}[d] & \\
& a &
}%
$\\\cline{2-2}%
\end{tabular}
\ \ ,\\
& \\
&
\begin{tabular}
[c]{c|c|}\cline{2-2}%
$\xymatrixrowsep{0.9pc}\xymatrixcolsep{0.20pc}%
\xymatrix{
\vphantom{b} \\
\vphantom{(x+y)\overline{z}b} \\
R^4 f = \\
\\
}%
$ & $\xymatrixrowsep{0.9pc}\xymatrixcolsep{0.20pc}%
\xymatrix{
& b \ar@{-}[d] & \\
& a\overline{b}z\overline{a}b \ar@{-}[rd] \ar@{-}[ld] & \\
\cdots\ar@{-}[rd] & & a\overline{b}\cdot\overline{\overline{x}+\overline{y}%
}\cdot\overline{x+y}\cdot y\left(\overline{x}+\overline{y}\right
)\left(x+y\right)\overline{a}b \ar@{-}[ld] \\
& a\overline{b}w\overline{a}b \ar@{-}[d] & \\
& a &
}%
$\\\cline{2-2}%
\end{tabular}
\ \ .
\end{align*}
\endgroup
Here, we have omitted the label at $\left(  2,1\right)  $ for both $R^{3}f$
and $R^{4}f$, since it can be obtained from the respective label at $\left(
1,2\right)  $ by interchanging $x$ with $y$ (thanks to an obvious symmetry
between $\left(  1,2\right)  $ and $\left(  2,1\right)  $).

The above might suggest that the labels get progressively more complicated as
we apply $R$ over and over. For a general poset $P$, this is indeed the case.
However, for our poset $P=\left\{  1,2\right\}  \times\left\{  1,2\right\}  $,
a surprising periodicity-like pattern emerges. Indeed, our above expressions
for $R^{2}f,\ R^{3}f,\ R^{4}f$ can be simplified as
follows:\begingroup\allowdisplaybreaks%
\begin{align*}
&
\begin{tabular}
[c]{c|c|}\cline{2-2}%
$\xymatrixrowsep{0.9pc}\xymatrixcolsep{0.20pc}%
\xymatrix{
\vphantom{b} \\
\vphantom{(x+y)\overline{z}b} \\
R^2 f = \\
\\
}%
$ & $\xymatrixrowsep{0.9pc}\xymatrixcolsep{0.20pc}%
\xymatrix{
& b \ar@{-}[d] & \\
& w\left(\overline{x}+\overline{y}\right)b \ar@{-}[rd] \ar@{-}[ld] & \\
a\overline{y}b \ar@{-}[rd] & & a\overline{x}b \ar@{-}[ld] \\
& a\overline{b}z\cdot\overline{x+y}\cdot b \ar@{-}[d] & \\
& a &
}%
$\\\cline{2-2}%
\end{tabular}
\ \ \ ,\\
& \\
&
\begin{tabular}
[c]{c|c|}\cline{2-2}%
$\xymatrixrowsep{0.9pc}\xymatrixcolsep{0.20pc}%
\xymatrix{
\vphantom{b} \\
\vphantom{(x+y)\overline{z}b} \\
R^3 f = \\
\\
}%
$ & $\xymatrixrowsep{0.9pc}\xymatrixcolsep{0.20pc}%
\xymatrix{
& b \ar@{-}[d] & \\
& a\overline{w}b \ar@{-}[rd] \ar@{-}[ld] & \\
a \overline{b}z\cdot\overline{x+y}\cdot y\overline{w}b \ar@{-}%
[rd] & & a \overline{b}z\cdot\overline{x+y}\cdot x\overline{w}b \ar@{-}[ld] \\
& a\overline{b}\cdot\overline{\overline{x}+\overline{y}}\cdot\overline
{w}b \ar@{-}[d] & \\
& a &
}%
$\\\cline{2-2}%
\end{tabular}
\ \ \ ,\\
& \\
&
\begin{tabular}
[c]{c|c|}\cline{2-2}%
$\xymatrixrowsep{0.9pc}\xymatrixcolsep{0.20pc}%
\xymatrix{
\vphantom{b} \\
\vphantom{(x+y)\overline{z}b} \\
R^4 f = \\
\\
}%
$ & $\xymatrixrowsep{0.9pc}\xymatrixcolsep{0.20pc}%
\xymatrix{
& b \ar@{-}[d] & \\
& a\overline{b}z\overline{a}b \ar@{-}[rd] \ar@{-}[ld] & \\
a\overline{b}x\overline{a}b \ar@{-}[rd] & & a\overline{b}y\overline{a}%
b \ar@{-}[ld] \\
& a\overline{b}w\overline{a}b \ar@{-}[d] & \\
& a &
}%
$\\\cline{2-2}%
\end{tabular}
\ \ \ .
\end{align*}
\endgroup Thus, the labels of $R^{4}f$ are closely related to those of $f$:
For each $v\in P$, we have
\[
\left(  R^{4}f\right)  \left(  v\right)  =a\overline{b}\cdot f\left(
v\right)  \cdot\overline{a}b.
\]
(This holds for $v=0$ and $v=1$ as well, as one can easily check.) Note that
if $ab=ba$, then this entails that $\left(  R^{4}f\right)  \left(  v\right)  $
is conjugate to $v$ in $\mathbb{K}$.

In Theorem \ref{thm.rect.ord}, we will generalize this phenomenon to arbitrary
\textquotedblleft rectangular\textquotedblright\ posets -- i.e., posets of the
form $\left\{  1,2,\ldots,p\right\}  \times\left\{  1,2,\ldots,q\right\}  $
with entrywise order. The \textquotedblleft period\textquotedblright\ in this
situation will be $p+q$.

Our $P=\left\{  1,2\right\}  \times\left\{  1,2\right\}  $ example also
exhibits a reciprocity-like phenomenon. Indeed, our above expressions for
$Rf,\ R^{2}f,\ R^{3}f$ reveal that%
\begin{align*}
\left(  Rf\right)  \left(  1,1\right)   &  =a\overline{z}b =a\cdot
\overline{f\left(  2,2\right)  }\cdot b;\\
\left(  R^{2}f\right)  \left(  1,2\right)   &  =a\overline{x}b =a\cdot
\overline{f\left(  2,1\right)  }\cdot b;\\
\left(  R^{2}f\right)  \left(  2,1\right)   &  =a\overline{y}b =a\cdot
\overline{f\left(  1,2\right)  }\cdot b;\\
\left(  R^{3}f\right)  \left(  2,2\right)   &  =a\overline{w}b =a\cdot
\overline{f\left(  1,1\right)  }\cdot b.
\end{align*}
These equalities relate the label of $R^{i+j-1}f$ at an element $\left(
i,j\right)  $ with the label of $f$ at the element $\left(  3-i,\ 3-j\right)
$ (which is, visually speaking, the \textquotedblleft
antipode\textquotedblright\ of the former element $\left(  i,j\right)  $ on
the Hasse diagram of $P$). To be specific, they say that%
\[
\left(  R^{i+j-1}f\right)  \left(  i,j\right)  =a\cdot\overline{f\left(
3-i,\ 3-j\right)  }\cdot b
\]
for any $\left(  i,j\right)  \in P$. This too can be generalized to arbitrary
rectangles (Theorem \ref{thm.rect.antip}).

In the above calculation, we used the linear extension $\left(  \left(
1,1\right)  ,\left(  1,2\right)  ,\left(  2,1\right)  ,\left(  2,2\right)
\right)  $ of $P$ to compute $R$ as $T_{\left(  1,1\right)  }\circ T_{\left(
1,2\right)  }\circ T_{\left(  2,1\right)  }\circ T_{\left(  2,2\right)  }$. We
could have just as well used the linear extension $\left(  \left(  1,1\right)
,\left(  2,1\right)  ,\left(  1,2\right)  ,\left(  2,2\right)  \right)  $,
obtaining the same result. But we could not have used the list $\left(
\left(  1,1\right)  ,\left(  1,2\right)  ,\left(  2,2\right)  ,\left(
2,1\right)  \right)  $ (for example), since it is not a linear extension (and
indeed, $T_{\left(  1,1\right)  }\circ T_{\left(  1,2\right)  }\circ
T_{\left(  2,2\right)  }\circ T_{\left(  2,1\right)  }$ would not give rise to
any similar phenomenon).

\end{example}

This example shows that birational rowmotion behaves unexpectedly well for
some posets. There are also some more serious motivations to study it:
Birational rowmotion for commutative $\mathbb{K}$ generalizes
Sch\"{u}tzenberger's classical \textquotedblleft promotion\textquotedblright%
\ map on semistandard tableaux (see \cite[Remark 11.6]{bir-row-arxiv}), and is
closely related to the Zamolodchikov periodicity conjecture in type AA (see
\cite[\S 4.4]{Roby15}). The case of a noncommutative ring $\mathbb{K}$ appears
more baroque, but we expect it to find a combinatorial meaning sooner or later.

Before we formalize and prove the above phenomena, we first consider some
general properties of $R$. We begin with an implicit description of birational
rowmotion that does not involve toggles (but is essentially a restatement of
Definition \ref{def.rm}):

\begin{proposition}
\label{prop.R.implicit}Let $v\in P$. Let $f\in\mathbb{K}^{\widehat{P}}$.
Assume that $Rf\neq\undf$. Then%
\begin{align*}
\left(  Rf\right)  \left(  v\right)  =\left(  \sum\limits_{\substack{u\in
\widehat{P};\\u\lessdot v}}f\left(  u\right)  \right)  \cdot\overline{f\left(
v\right)  }\cdot\overline{\sum\limits_{\substack{u\in\widehat{P};\\u\gtrdot
v}}\overline{\left(  Rf\right)  \left(  u\right)  }}.
\end{align*}

\end{proposition}

\begin{vershort}
\begin{proof}
This is merely the noncommutative analogue of \cite[Proposition 19]%
{bir-row-1}, and the proof in \cite{bir-row-1} can be used with
straightforward modifications.
\end{proof}
\end{vershort}

\begin{verlong}
\begin{proof}
In the case when $\mathbb{K}$ is commutative, this is \cite[Proposition
2.16]{bir-row-arxiv}. The proof given in \cite{bir-row-arxiv} can be easily
modified to apply to the general case as well. Here are the details:

Let $\left(  v_{1},v_{2},\ldots,v_{m}\right)  $ be a linear extension of $P$.
(Such a linear extension exists, because of Theorem \ref{thm.linext.ex}.)

Let $i\in\left\{  1,2,\ldots,m\right\}  $ be the index satisfying $v_{i}=v$.
Thus, $T_{v_{i}}=T_{v}$.

By the definition of birational rowmotion $R$, we have $R=T_{v_{1}}\circ
T_{v_{2}}\circ\cdots\circ T_{v_{m}}$.

Define two partial maps%
\[
A:=T_{v_{i+1}}\circ T_{v_{i+2}}\circ\cdots\circ T_{v_{m}}%
\ \ \ \ \ \ \ \ \ \ \text{and}\ \ \ \ \ \ \ \ \ \ B:=T_{v_{1}}\circ T_{v_{2}%
}\circ\cdots\circ T_{v_{i-1}}%
\]
from $\mathbb{K}^{\widehat{P}}$ to $\mathbb{K}^{\widehat{P}}$. Then%
\begin{align}
R  &  =T_{v_{1}}\circ T_{v_{2}}\circ\cdots\circ T_{v_{m}}=\underbrace{T_{v_{1}%
}\circ T_{v_{2}}\circ\cdots\circ T_{v_{i-1}}}_{=B}\circ\underbrace{T_{v_{i}}%
}_{=T_{v}}\circ\underbrace{T_{v_{i+1}}\circ T_{v_{i+2}}\circ\cdots\circ
T_{v_{m}}}_{=A}\nonumber\\
&  =B\circ T_{v}\circ A. \label{pf.prop.R.implicit.BTA}%
\end{align}

Define the $\mathbb{K}$-labeling $g:=T_{v}\left(  Af\right)  $. Thus,
\[
\underbrace{R}_{\substack{=B\circ T_{v}\circ A\\\text{(by
(\ref{pf.prop.R.implicit.BTA}))}}}f=\left(  B\circ T_{v}\circ A\right)
f=B\left(  \underbrace{T_{v}\left(  Af\right)  }_{=g}\right)  =Bg.
\]
Hence, $Bg=Rf\neq\undf=B\left(  \undf\right)  $, and thus $g\neq\undf$. Hence,
$T_{v}\left(  Af\right)  =g\neq\undf=T_{v}\left(  \undf\right)  $, so that
$Af\neq\undf$. Now:

\begin{itemize}
\item Each of the maps $T_{v_{j}}$ with $j\neq i$ leaves the label at $v$
unchanged when acting on a $\mathbb{K}$-labeling (since $j\neq i$ entails
$v_{j}\neq v_{i}=v$). Hence, each of the maps $B$ and $A$ leaves the label at
$v$ unchanged (since $B$ and $A$ are compositions of maps $T_{v_{j}}$ with
$j\neq i$). Thus, $\left(  Bg\right)  \left(  v\right)  =g\left(  v\right)  $
and $\left(  Af\right)  \left(  v\right)  =f\left(  v\right)  $. Now,%
\begin{align}
\underbrace{\left(  Rf\right)  }_{=Bg}\left(  v\right)   &  =\left(
Bg\right)  \left(  v\right)  =\underbrace{g}_{=T_{v}\left(  Af\right)
}\left(  v\right) \nonumber\\
&  =\left(  T_{v}\left(  Af\right)  \right)  \left(  v\right) \nonumber\\
&  =\left(  \sum\limits_{\substack{u\in\widehat{P};\\u\lessdot v}}\left(
Af\right)  \left(  u\right)  \right)  \cdot\overline{\left(  Af\right)
\left(  v\right)  }\cdot\overline{\sum\limits_{\substack{u\in\widehat{P}%
;\\u\gtrdot v}}\overline{\left(  Af\right)  \left(  u\right)  }}\nonumber\\
&  \ \ \ \ \ \ \ \ \ \ \ \ \ \ \ \ \ \ \ \ \left(
\begin{array}
[c]{c}%
\text{by Proposition \ref{prop.Tv} \textbf{(b)},}\\
\text{applied to }Af\text{ instead of }f
\end{array}
\right) \nonumber\\
&  =\left(  \sum\limits_{\substack{u\in\widehat{P};\\u\lessdot v}}\left(
Af\right)  \left(  u\right)  \right)  \cdot\overline{f\left(  v\right)  }%
\cdot\overline{\sum\limits_{\substack{u\in\widehat{P};\\u\gtrdot v}%
}\overline{\left(  Af\right)  \left(  u\right)  }} \label{pf.R.implicit.1}%
\end{align}
(since $\left(  Af\right)  \left(  v\right)  =f\left(  v\right)  $).

\item Let $u\in\widehat{P}$ be such that $u\lessdot v$. Then $u<v=v_{i}$ in
$\widehat{P}$. Hence, $u$ is none of the elements $v_{i+1}$, $v_{i+2}$,
$\ldots$, $v_{m}$ (because $\left(  v_{1},v_{2},\ldots,v_{m}\right)  $ is a
linear extension of $P$). Thus, each of the maps $T_{v_{i+1}}$, $T_{v_{i+2}}$,
$\ldots$, $T_{v_{m}}$ leaves the label at $u$ invariant when acting on a
$\mathbb{K}$-labeling. Therefore, the map $A$ also leaves the label at $u$
invariant (since $A$ is a composition of these maps $T_{v_{i+1}}$,
$T_{v_{i+2}}$, $\ldots$, $T_{v_{m}}$). Hence, $\left(  Af\right)  \left(
u\right)  =f\left(  u\right)  $.

Forget that we fixed $u$. We have thus shown that%
\begin{equation}
\left(  Af\right)  \left(  u\right)  =f\left(  u\right)
\ \ \ \ \ \ \ \ \ \ \text{for every }u\in\widehat{P}\text{ such that
}u\lessdot v. \label{pf.R.implicit.2}%
\end{equation}

\item Let $u\in\widehat{P}$ be such that $u\gtrdot v$. Then $u>v=v_{i}$ in
$\widehat{P}$. Hence, $u$ is none of the elements $v_{1}$, $v_{2}$, $\ldots$,
$v_{i-1}$ (because $\left(  v_{1},v_{2},\ldots,v_{m}\right)  $ is a linear
extension of $P$). Thus, each of the maps $T_{v_{1}}$, $T_{v_{2}}$, $\ldots$,
$T_{v_{i-1}}$ leaves the label at $u$ invariant when acting on a $\mathbb{K}%
$-labeling. Therefore, $B$ also leaves the label at $u$ invariant (since $B$
is a composition of these maps $T_{v_{1}}$, $T_{v_{2}}$, $\ldots$,
$T_{v_{i-1}}$). Since $T_{v}$ also leaves the label at $u$ invariant (because
$u\neq v$ (since $u > v$)), this yields that the composition $B\circ T_{v}$
also leaves the label at $u$ invariant. Hence, $\left(  \left(  B\circ
T_{v}\right)  \left(  Af\right)  \right)  \left(  u\right)  =\left(
Af\right)  \left(  u\right)  $, so that%
\[
\left(  Af\right)  \left(  u\right)  =\left(  \left(  B\circ T_{v}\right)
\left(  Af\right)  \right)  \left(  u\right)  =\left(  \underbrace{\left(
B\circ T_{v}\circ A\right)  }_{\substack{=R\\\text{(by
(\ref{pf.prop.R.implicit.BTA}))}}}f\right)  \left(  u\right)  =\left(
Rf\right)  \left(  u\right)  .
\]

Forget that we fixed $u$. We thus have proven that%
\begin{equation}
\left(  Af\right)  \left(  u\right)  =\left(  Rf\right)  \left(  u\right)
\ \ \ \ \ \ \ \ \ \ \text{for every }u\in\widehat{P}\text{ such that }u\gtrdot
v. \label{pf.R.implicit.3}%
\end{equation}

\end{itemize}

Now, substituting (\ref{pf.R.implicit.2}) and (\ref{pf.R.implicit.3}) into
(\ref{pf.R.implicit.1}), we obtain%
\[
\left(  Rf\right)  \left(  v\right)  =\left(  \sum\limits_{\substack{u\in
\widehat{P};\\u\lessdot v}}f\left(  u\right)  \right)  \cdot\overline{f\left(
v\right)  }\cdot\overline{\sum\limits_{\substack{u\in\widehat{P};\\u\gtrdot
v}}\overline{\left(  Rf\right)  \left(  u\right)  }}.
\]
This proves Proposition \ref{prop.R.implicit}.
\end{proof}
\end{verlong}

The following near-trivial fact completes the picture:

\begin{proposition}
\label{prop.R.implicit.01}Let $f\in\mathbb{K}^{\widehat{P}}$. Assume that
$Rf\neq\undf$. Then $\left(  Rf\right)  \left(  0\right)  =f\left(  0\right)
$ and $\left(  Rf\right)  \left(  1\right)  =f\left(  1\right)  $.
\end{proposition}

\begin{vershort}
\begin{proof}
None of the toggles $T_{v}$, when applied to a $\mathbb{K}$-labeling, changes
the label of $0$ or the label of $1$. Hence, the same is true for the partial
map $R$ (since $R$ is a composition of such toggles $T_{v}$).
\end{proof}
\end{vershort}

\begin{verlong}
\begin{proof}
Let $\left(  v_{1},v_{2},\ldots,v_{m}\right)  $ be a linear extension of $P$.
(Such a linear extension exists, because of Theorem \ref{thm.linext.ex}.)

By the definition of birational rowmotion $R$, we have $R=T_{v_{1}}\circ
T_{v_{2}}\circ\cdots\circ T_{v_{m}}$.

Each of the maps $T_{v_{j}}$ with $j\in\left\{  1,2,\ldots,m\right\}  $ leaves
the label at $0$ unchanged when acting on a $\mathbb{K}$-labeling (since
$0\neq v_{j}$ (because $0\notin P$ whereas $v_{j}\in P$)). Therefore, the map
$R$ also leaves the label at $0$ unchanged (since $R$ is the composition
$T_{v_{1}}\circ T_{v_{2}}\circ\cdots\circ T_{v_{m}}$ of these maps $T_{v_{j}}%
$). Hence, $\left(  Rf\right)  \left(  0\right)  =f\left(  0\right)  $.
Similarly, $\left(  Rf\right)  \left(  1\right)  =f\left(  1\right)  $. This
proves Proposition \ref{prop.R.implicit.01}.
\end{proof}
\end{verlong}

A trivial corollary of Proposition \ref{prop.R.implicit.01} is:

\begin{corollary}
\label{cor.R.implicit.01}Let $f\in\mathbb{K}^{\widehat{P}}$ and $\ell
\in\mathbb{N}$. Assume that $R^{\ell}f\neq\undf$. Then $\left(  R^{\ell
}f\right)  \left(  0\right)  =f\left(  0\right)  $ and $\left(  R^{\ell
}f\right)  \left(  1\right)  =f\left(  1\right)  $.
\end{corollary}

(Recall that $\mathbb{N}$ denotes the set $\left\{  0,1,2,\ldots\right\}  $ in
this paper.)

\begin{verlong}
\begin{proof}
[Proof of Corollary \ref{cor.R.implicit.01}.]Let $i\in\left\{  1,2,\ldots
,\ell\right\}  $. Then $R^{\ell-i}\left(  R^{i}f\right)  =R^{\ell}%
f\neq\undf=R^{\ell-i}\left(  \undf\right)  $, so that $R^{i}f\neq\undf$. In
other words, $R\left(  R^{i-1}f\right)  \neq\undf$ (since $R\left(
R^{i-1}f\right)  =R^{i}f$). Hence, we can apply Proposition
\ref{prop.R.implicit.01} to $R^{i-1}f$ instead of $f$. As a result, we obtain%
\[
\left(  R\left(  R^{i-1}f\right)  \right)  \left(  0\right)  =\left(
R^{i-1}f\right)  \left(  0 \right)  \ \ \ \ \ \ \ \ \ \ \text{and}%
\ \ \ \ \ \ \ \ \ \ \left(  R\left(  R^{i-1}f\right)  \right)  \left(
1\right)  =\left(  R^{i-1}f\right)  \left(  1\right)  .
\]
From $R^{i}f = R\left(  R^{i-1}f\right)  $, we now obtain $\left(
R^{i}f\right)  \left(  0\right)  =\left(  R\left(  R^{i-1}f\right)  \right)
\left(  0\right)  =\left(  R^{i-1}f\right)  \left(  0\right)  $. Hence,
$\left(  R^{i-1}f\right)  \left(  0\right)  =\left(  R^{i}f\right)  \left(
0\right)  $.

Forget that we fixed $i$. We thus have proved the equality $\left(
R^{i-1}f\right)  \left(  0\right)  =\left(  R^{i}f\right)  \left(  0\right)  $
for each $i\in\left\{  1,2,\ldots,\ell\right\}  $. Combining all of these
equalities, we obtain%
\[
\left(  R^{0}f\right)  \left(  0\right)  =\left(  R^{1}f\right)  \left(
0\right)  =\left(  R^{2}f\right)  \left(  0\right)  =\cdots=\left(  R^{\ell
}f\right)  \left(  0\right)  .
\]
Hence, $\left(  R^{\ell}f\right)  \left(  0\right)  =\underbrace{\left(
R^{0}f\right)  }_{=f}\left(  0\right)  =f\left(  0\right)  $. A similar
argument shows that $\left(  R^{\ell}f\right)  \left(  1\right)  =f\left(
1\right)  $. Thus, Corollary \ref{cor.R.implicit.01} is proven.
\end{proof}
\end{verlong}

\subsection{Well-definedness lemmas}

We next show some simple lemmas which say that certain inverses exist under
the assumption that $R^{\ell}f$ is well-defined for some values of $\ell$.
These lemmas are easy and unexciting, but are necessary in order to rigorously
prove the more substantial results that will follow. We recommend the reader
skip the proofs, at least on a first reading.

\begin{lemma}
\label{lem.R.wd-triv}Let $f\in\mathbb{K}^{\widehat{P}}$ and $k,\ell
\in\mathbb{N}$ satisfy $k\leq\ell$ and $R^{\ell}f\neq\undf$. Then, $R^{k}%
f\neq\undf$.
\end{lemma}

\begin{vershort}
\begin{proof}
We have $R^{\ell-k}\left(  R^{k}f\right)  =R^{\ell}f\neq\undf=R^{\ell
-k}\left(  \undf\right)  $, so that $R^{k}f\neq\undf$.
\end{proof}
\end{vershort}

\begin{verlong}
\begin{proof}
From $k\leq\ell$, we obtain $R^{\ell}f=R^{\ell-k}\left(  R^{k}f\right)  $.
Thus, if we had $R^{k}f=\undf$, then we would obtain%
\[
R^{\ell}f=R^{\ell-k}\left(  \underbrace{R^{k}f}_{=\undf}\right)  =R^{\ell
-k}\left(  \undf\right)  =\undf,
\]
which would contradict $R^{\ell}f\neq\undf$. Hence, we must have $R^{k}%
f\neq\undf$. This proves Lemma \ref{lem.R.wd-triv}.
\end{proof}
\end{verlong}

\begin{lemma}
\label{lem.R.inv}Let $f\in\mathbb{K}^{\widehat{P}}$ satisfy $Rf\neq\undf$. Let
$v\in P$. Then, $f\left(  v\right)  $ is invertible.
\end{lemma}

\begin{proof}
Proposition \ref{prop.R.implicit} yields
\[
\left(  Rf\right)  \left(  v\right)  =\left(  \sum\limits_{\substack{u\in
\widehat{P};\\u\lessdot v}}f\left(  u\right)  \right)  \cdot\overline{f\left(
v\right)  }\cdot\overline{\sum\limits_{\substack{u\in\widehat{P};\\u\gtrdot
v}}\overline{\left(  Rf\right)  \left(  u\right)  }}.
\]
Thus, in particular, $\overline{f\left(  v\right)  }$ is well-defined. In
other words, $f\left(  v\right)  $ is invertible. This proves Lemma
\ref{lem.R.inv}.
\end{proof}

\begin{lemma}
\label{lem.R.1inv}Assume that $P\neq\varnothing$. Let $f\in\mathbb{K}%
^{\widehat{P}}$ satisfy $Rf\neq\undf$. Then, $f\left(  1\right)  $ is invertible.
\end{lemma}

\begin{vershort}
\begin{proof}
We have $P\neq\varnothing$. Thus, the poset $P$ has a maximal element $y$ (by
Proposition \ref{prop.poset-minmax} \textbf{(b)}). This $y$ then satisfies
$1\gtrdot y$ in $\widehat{P}$.

We have $Rf\neq\undf$. Therefore, Proposition \ref{prop.R.implicit} (applied
to $v=y$) yields%
\[
\left(  Rf\right)  \left(  y\right)  =\left(  \sum\limits_{\substack{u\in
\widehat{P};\\u\lessdot y}}f\left(  u\right)  \right)  \cdot\overline{f\left(
y\right)  }\cdot\overline{\sum\limits_{\substack{u\in\widehat{P};\\u\gtrdot
y}}\overline{\left(  Rf\right)  \left(  u\right)  }}.
\]
Hence, in particular, $\overline{\left(  Rf\right)  \left(  u\right)  }$ is
well-defined for each $u\in\widehat{P}$ satisfying $u\gtrdot y$. We can apply
this to $u=1$ (since $1 \gtrdot y$), and thus conclude that $\overline{\left(
Rf\right)  \left(  1\right)  }$ is well-defined. In other words, $\left(
Rf\right)  \left(  1\right)  $ is invertible. However, Proposition
\ref{prop.R.implicit.01} yields $\left(  Rf\right)  \left(  1\right)
=f\left(  1\right)  $. Thus, $f\left(  1\right)  $ is invertible.
\end{proof}
\end{vershort}

\begin{verlong}
\begin{proof}
We have $P\neq\varnothing$. Thus, the poset $P$ has a maximal element $y$ (by
Proposition \ref{prop.poset-minmax} \textbf{(b)}). Consider this $y$. The
element $y$ of $P$ is maximal. Thus, in $\widehat{P}$, we have $1\gtrdot y$
(by Remark \ref{rmk.Phat.covers} \textbf{(b)}). In other words, $1$ is a
$u\in\widehat{P}$ satisfying $u\gtrdot y$.

We have $Rf\neq\undf$. Therefore, Proposition \ref{prop.R.implicit} (applied
to $v=y$) yields%
\[
\left(  Rf\right)  \left(  y\right)  =\left(  \sum\limits_{\substack{u\in
\widehat{P};\\u\lessdot y}}f\left(  u\right)  \right)  \cdot\overline{f\left(
y\right)  }\cdot\overline{\sum\limits_{\substack{u\in\widehat{P};\\u\gtrdot
y}}\overline{\left(  Rf\right)  \left(  u\right)  }}.
\]
Hence, in particular, $\overline{\left(  Rf\right)  \left(  u\right)  }$ is
well-defined for each $u\in\widehat{P}$ satisfying $u\gtrdot y$. We can apply
this to $u=1$ (since $1$ is a $u\in\widehat{P}$ satisfying $u\gtrdot y$), and
thus conclude that $\overline{\left(  Rf\right)  \left(  1\right)  }$ is
well-defined. In other words, $\left(  Rf\right)  \left(  1\right)  $ is
invertible. However, Proposition \ref{prop.R.implicit.01} yields $\left(
Rf\right)  \left(  1\right)  =f\left(  1\right)  $. Thus, $f\left(  1\right)
$ is invertible (since $\left(  Rf\right)  \left(  1\right)  $ is invertible).
This proves Lemma \ref{lem.R.1inv}.
\end{proof}
\end{verlong}

\begin{lemma}
\label{lem.R.01inv}Assume that $P\neq\varnothing$. Let $f\in\mathbb{K}%
^{\widehat{P}}$ satisfy $R^{2}f\neq\undf$. Then, $f\left(  0\right)  $ and
$f\left(  1\right)  $ are invertible.
\end{lemma}

\begin{vershort}
\begin{proof}
The poset $P$ has a minimal element $x$ (by Proposition
\ref{prop.poset-minmax} \textbf{(a)}).

From $R^{2}f \neq\undf$, we obtain $Rf\neq\undf$ (by Lemma \ref{lem.R.wd-triv}%
); thus, $Rf\in\mathbb{K}^{\widehat{P}}$. Hence, Lemma \ref{lem.R.1inv} yields
that $f\left(  1\right)  $ is invertible. Furthermore, Lemma \ref{lem.R.inv}
(applied to $Rf$ and $x$ instead of $f$ and $v$) yields that $\left(
Rf\right)  \left(  x\right)  $ is invertible (since $R\left(  Rf\right)
\neq\undf$).

Recall again that $Rf\neq\undf$. Hence, Proposition \ref{prop.R.implicit}
(applied to $v=x$) yields%
\begin{equation}
\left(  Rf\right)  \left(  x\right)  =\left(  \sum\limits_{\substack{u\in
\widehat{P};\\u\lessdot x}}f\left(  u\right)  \right)  \cdot\overline{f\left(
x\right)  }\cdot\overline{\sum\limits_{\substack{u\in\widehat{P};\\u\gtrdot
x}}\overline{\left(  Rf\right)  \left(  u\right)  }}.
\label{pf.lem.R.01inv.short.1}%
\end{equation}

The only $u\in\widehat{P}$ satisfying $u\lessdot x$ is the element $0$ of
$\widehat{P}$ (since $x$ is a minimal element of $P$). Thus, $\sum
\limits_{\substack{u\in\widehat{P};\\u\lessdot x}}f\left(  u\right)  =f\left(
0\right)  $. Hence, (\ref{pf.lem.R.01inv.short.1}) rewrites as%
\[
\left(  Rf\right)  \left(  x\right)  =f\left(  0\right)  \cdot\overline
{f\left(  x\right)  }\cdot\overline{\sum\limits_{\substack{u\in\widehat{P}%
;\\u\gtrdot x}}\overline{\left(  Rf\right)  \left(  u\right)  }}.
\]
Solving this equality for $f\left(  0\right)  $, we obtain
\[
f\left(  0\right)  = \left(  Rf\right)  \left(  x\right)  \cdot\left(
\sum\limits_{\substack{u\in\widehat{P};\\u\gtrdot x}} \left(  Rf\right)
\left(  u\right)  \right)  \cdot f\left(  x\right)  .
\]
The right hand side of this equality is a product of three invertible elements
(indeed, the two factors $\sum\limits_{\substack{u\in\widehat{P};\\u\gtrdot
x}} \left(  Rf\right)  \left(  u\right)  $ and $f\left(  x\right)  $ are
invertible because their inverses appear in (\ref{pf.lem.R.01inv.short.1}),
and we already know that the factor $\left(  Rf\right)  \left(  x\right)  $ is
invertible), and thus itself invertible. Hence, the left hand side is
invertible. In other words, $f\left(  0\right)  $ is invertible.
\end{proof}
\end{vershort}

\begin{verlong}
\begin{proof}
We have $P\neq\varnothing$. Thus, the poset $P$ has a minimal element $x$ (by
Proposition \ref{prop.poset-minmax} \textbf{(a)}). Consider this $x$.

From $R\left(  Rf\right)  =R^{2}f\neq\undf=R\left(  \undf\right)  $, we obtain
$Rf\neq\undf$; thus, $Rf\in\mathbb{K}^{\widehat{P}}$. Hence, Lemma
\ref{lem.R.1inv} yields that $f\left(  1\right)  $ is invertible. Furthermore,
Lemma \ref{lem.R.inv} (applied to $Rf$ and $x$ instead of $f$ and $v$) yields
that $\left(  Rf\right)  \left(  x\right)  $ is invertible (since $R\left(
Rf\right)  \neq\undf$).

Recall again that $Rf\neq\undf$. Hence, Proposition \ref{prop.R.implicit}
(applied to $v=x$) yields%
\begin{equation}
\left(  Rf\right)  \left(  x\right)  =\left(  \sum\limits_{\substack{u\in
\widehat{P};\\u\lessdot x}}f\left(  u\right)  \right)  \cdot\overline{f\left(
x\right)  }\cdot\overline{\sum\limits_{\substack{u\in\widehat{P};\\u\gtrdot
x}}\overline{\left(  Rf\right)  \left(  u\right)  }}. \label{pf.lem.R.01inv.1}%
\end{equation}

The only $u\in\widehat{P}$ satisfying $u\lessdot x$ is the element $0$ of
$\widehat{P}$ (since $x$ is a minimal element of $P$). Thus, $\sum
\limits_{\substack{u\in\widehat{P};\\u\lessdot x}}f\left(  u\right)  =f\left(
0\right)  $. Hence, (\ref{pf.lem.R.01inv.1}) rewrites as%
\begin{equation}
\left(  Rf\right)  \left(  x\right)  =f\left(  0\right)  \cdot\overline
{f\left(  x\right)  }\cdot\overline{\sum\limits_{\substack{u\in\widehat{P}%
;\\u\gtrdot x}}\overline{\left(  Rf\right)  \left(  u\right)  }}.
\label{pf.lem.R.01inv.2}%
\end{equation}
This equality shows that $\overline{f\left(  x\right)  }$ and $\overline
{\sum\limits_{\substack{u\in\widehat{P};\\u\gtrdot x}}\overline{\left(
Rf\right)  \left(  u\right)  }}$ are well-defined. Thus, the elements
$f\left(  x\right)  $ and $\sum\limits_{\substack{u\in\widehat{P};\\u\gtrdot
x}}\overline{\left(  Rf\right)  \left(  u\right)  }$ are invertible.
Multiplying both sides of the equality (\ref{pf.lem.R.01inv.2}) by $\left(
\sum\limits_{\substack{u\in\widehat{P};\\u\gtrdot x}}\overline{\left(
Rf\right)  \left(  u\right)  }\right)  \cdot f\left(  x\right)  $ on the
right, we obtain%
\begin{align*}
\left(  Rf\right)  \left(  x\right)  \cdot\left(  \sum\limits_{\substack{u\in
\widehat{P};\\u\gtrdot x}}\overline{\left(  Rf\right)  \left(  u\right)
}\right)  \cdot f\left(  x\right)   &  =f\left(  0\right)  \cdot
\overline{f\left(  x\right)  }\cdot\underbrace{\overline{\sum
\limits_{\substack{u\in\widehat{P};\\u\gtrdot x}}\overline{\left(  Rf\right)
\left(  u\right)  }}\cdot\left(  \sum\limits_{\substack{u\in\widehat{P}%
;\\u\gtrdot x}}\overline{\left(  Rf\right)  \left(  u\right)  }\right)  }%
_{=1}\cdot\,f\left(  x\right) \\
&  =f\left(  0\right)  \cdot\underbrace{\overline{f\left(  x\right)  }\cdot
f\left(  x\right)  }_{=1}=f\left(  0\right)  .
\end{align*}
The left hand side of this equality is invertible (since it is the product of
the three invertible elements $\left(  Rf\right)  \left(  x\right)  $,
$\sum\limits_{\substack{u\in\widehat{P};\\u\gtrdot x}}\overline{\left(
Rf\right)  \left(  u\right)  }$ and $f\left(  x\right)  $). Thus, its right
hand side is invertible as well. In other words, $f\left(  0\right)  $ is
invertible. This completes the proof of Lemma \ref{lem.R.01inv}.
\end{proof}
\end{verlong}

\begin{lemma}
\label{lem.not-minimal.then-covered}Let $v\in P$. Assume that $v$ is not a
minimal element of $P$. Then, there exists at least one element $w\in P$
satisfying $v\gtrdot w$.
\end{lemma}

\begin{vershort}
\begin{proof}
Apply Proposition \ref{prop.poset-minmax} \textbf{(b)} to the subposet
$P_{<v}:=\left\{  u\in P\ \mid\ u<v\right\}  $ of $P$. Details are left as an exercise.
\end{proof}
\end{vershort}

\begin{verlong}
\begin{proof}
The element $v$ of $P$ is not minimal. Thus, there exists some $u\in P$
satisfying $u<v$. In other words, the set $P_{<v}:=\left\{  u\in
P\ \mid\ u<v\right\}  $ is nonempty. Consider this set $P_{<v}$ as a subposet
of $P$ (with its partial order inherited from $P$). Then, $P_{<v}$ is a finite
poset and thus has a maximal element (by Proposition \ref{prop.poset-minmax}
\textbf{(b)}, applied to $P_{<v}$ instead of $P$). Let $m$ be this maximal
element. Then, $m\in P_{<v}=\left\{  u\in P\ \mid\ u<v\right\}  $; in other
words, $m\in P$ and $m<v$. There exists no element of $P_{<v}$ that is larger
than $m$ (since $m$ is a \textbf{maximal} element of $P_{<v}$).

If there was some $w\in P$ satisfying $m<w<v$, then this $w$ would belong to
$P_{<v}$ (since $w<v$) but would be larger than $m$ in the poset $P_{<v}$
(since $m<w$); this would contradict the fact that there exists no element of
$P_{<v}$ that is larger than $m$. Hence, there is no $w\in P$ satisfying
$m<w<v$. In other words, we have $m\lessdot v$ in $P$ (since $m<v$). In other
words, we have $v\gtrdot m$ in $P$. Hence, there exists at least one element
$w\in P$ satisfying $v\gtrdot w$ (namely, $w=m$). This proves Lemma
\ref{lem.not-minimal.then-covered}.
\end{proof}
\end{verlong}

\begin{lemma}
\label{lem.R.inv-not-min}Let $f\in\mathbb{K}^{\widehat{P}}$ satisfy
$Rf\neq\undf$. Let $v\in P$. Assume that $v$ is not a minimal element of $P$.
Then, $\left(  Rf\right)  \left(  v\right)  $ is invertible.
\end{lemma}

\begin{proof}
Lemma \ref{lem.not-minimal.then-covered} shows that there exists at least one
element $w\in P$ satisfying $v\gtrdot w$. Consider this $w$. Proposition
\ref{prop.R.implicit} (applied to $w$ instead of $v$) yields%
\[
\left(  Rf\right)  \left(  w\right)  =\left(  \sum\limits_{\substack{u\in
\widehat{P};\\u\lessdot w}}f\left(  u\right)  \right)  \cdot\overline{f\left(
w\right)  }\cdot\overline{\sum\limits_{\substack{u\in\widehat{P};\\u\gtrdot
w}}\overline{\left(  Rf\right)  \left(  u\right)  }}.
\]
In particular, $\overline{\left(  Rf\right)  \left(  u\right)  }$ is
well-defined for each $u\in\widehat{P}$ satisfying $u\gtrdot w$. Applying this
to $u=v$, we conclude that $\overline{\left(  Rf\right)  \left(  v\right)  }$
is well-defined (since $v\in P\subseteq\widehat{P}$ and $v\gtrdot w$). In
other words, $\left(  Rf\right)  \left(  v\right)  $ is invertible.
\end{proof}

\begin{lemma}
\label{lem.R.inv-2}Assume that $P\neq\varnothing$. Let $f\in\mathbb{K}%
^{\widehat{P}}$ satisfy $Rf\neq\undf$. Let $v\in\widehat{P}$. Assume that
$f\left(  0\right)  $ is invertible. Then, $\left(  Rf\right)  \left(
v\right)  $ is invertible.
\end{lemma}

\begin{vershort}
\begin{proof}
If $v=0$, then the claim follows from our assumption about $f\left(  0\right)
$ (since Proposition \ref{prop.R.implicit.01} yields $\left(  Rf\right)
\left(  0\right)  =f\left(  0\right)  $). If $v=1$, then it instead follows
from Lemma \ref{lem.R.1inv} (since Proposition \ref{prop.R.implicit.01} yields
$\left(  Rf\right)  \left(  1\right)  =f\left(  1\right)  $). Thus, we assume
from now on that $v$ is neither $0$ nor $1$. Hence, $v\in P$.

If $v$ is not a minimal element of $P$, then the claim follows from Lemma
\ref{lem.R.inv-not-min}. Hence, we assume from now on that $v$ is a minimal
element of $P$. Therefore, the only $u\in\widehat{P}$ satisfying $u\lessdot v$
is the element $0$. Thus, $\sum\limits_{\substack{u\in\widehat{P};\\u\lessdot
v}}f\left(  u\right)  =f\left(  0\right)  $. Now, Proposition
\ref{prop.R.implicit} yields%
\[
\left(  Rf\right)  \left(  v\right)  =\underbrace{\left(  \sum
\limits_{\substack{u\in\widehat{P};\\u\lessdot v}}f\left(  u\right)  \right)
}_{=f\left(  0\right)  }\cdot\overline{f\left(  v\right)  }\cdot\overline
{\sum\limits_{\substack{u\in\widehat{P};\\u\gtrdot v}}\overline{\left(
Rf\right)  \left(  u\right)  }}=f\left(  0\right)  \cdot\overline{f\left(
v\right)  }\cdot\overline{\sum\limits_{\substack{u\in\widehat{P};\\u\gtrdot
v}}\overline{\left(  Rf\right)  \left(  u\right)  }}.
\]
The right hand side of this equality is a product of three invertible elements
(since $f\left(  0\right)  $ is invertible, and since $\overline{f\left(
v\right)  }$ and $\overline{\sum\limits_{\substack{u\in\widehat{P};\\u\gtrdot
v}}\overline{\left(  Rf\right)  \left(  u\right)  }}$ are
invertible\footnote{because an inverse is always invertible}), and thus itself
is invertible. Thus, the left hand side is invertible as well. In other words,
$\left(  Rf\right)  \left(  v\right)  $ is invertible.
\end{proof}
\end{vershort}

\begin{verlong}
\begin{proof}
We are in one of the following four cases:

\textit{Case 1:} We have $v=0$.

\textit{Case 2:} We have $v=1$.

\textit{Case 3:} The element $v$ is a minimal element of $P$.

\textit{Case 4:} The element $v$ is neither $0$ nor $1$ nor a minimal element
of $P$.

Let us first consider Case 1. In this case, we have $v=0$. Hence, $\left(
Rf\right)  \left(  v\right)  =\left(  Rf\right)  \left(  0\right)  =f\left(
0\right)  $ (by Proposition \ref{prop.R.implicit.01}). Hence, $\left(
Rf\right)  \left(  v\right)  $ is invertible (since $f\left(  0\right)  $ is
invertible). This proves Lemma \ref{lem.R.inv-2} in Case 1.

Let us now consider Case 2. In this case, we have $v=1$. Hence, $\left(
Rf\right)  \left(  v\right)  =\left(  Rf\right)  \left(  1\right)  =f\left(
1\right)  $ (by Proposition \ref{prop.R.implicit.01}). However, $f\left(
1\right)  $ is invertible (by Lemma \ref{lem.R.1inv}). In other words,
$\left(  Rf\right)  \left(  v\right)  $ is invertible (since $\left(
Rf\right)  \left(  v\right)  =f\left(  1\right)  $). This proves Lemma
\ref{lem.R.inv-2} in Case 2.

Next, let us consider Case 3. In this case, the element $v$ is a minimal
element of $P$. Hence, the only $u\in\widehat{P}$ satisfying $u\lessdot v$ is
the element $0$. Therefore, $\sum\limits_{\substack{u\in\widehat{P}%
;\\u\lessdot v}}f\left(  u\right)  =f\left(  0\right)  $. Now, Proposition
\ref{prop.R.implicit} yields%
\[
\left(  Rf\right)  \left(  v\right)  =\underbrace{\left(  \sum
\limits_{\substack{u\in\widehat{P};\\u\lessdot v}}f\left(  u\right)  \right)
}_{=f\left(  0\right)  }\cdot\overline{f\left(  v\right)  }\cdot\overline
{\sum\limits_{\substack{u\in\widehat{P};\\u\gtrdot v}}\overline{\left(
Rf\right)  \left(  u\right)  }}=f\left(  0\right)  \cdot\overline{f\left(
v\right)  }\cdot\overline{\sum\limits_{\substack{u\in\widehat{P};\\u\gtrdot
v}}\overline{\left(  Rf\right)  \left(  u\right)  }}.
\]
The right hand side of this equality is a product of three invertible elements
(since $f\left(  0\right)  $ is invertible, and since $\overline{f\left(
v\right)  }$ and $\overline{\sum\limits_{\substack{u\in\widehat{P};\\u\gtrdot
v}}\overline{\left(  Rf\right)  \left(  u\right)  }}$ are
invertible\footnote{because an inverse is always invertible}), and thus itself
is invertible. Thus, the left hand side is invertible as well. In other words,
$\left(  Rf\right)  \left(  v\right)  $ is invertible. This proves Lemma
\ref{lem.R.inv-2} in Case 3.

Finally, let us consider Case 4. In this case, the element $v$ is neither $0$
nor $1$ nor a minimal element of $P$. However, $v\in P$ (since $v$ is neither
$0$ nor $1$). Thus, Lemma \ref{lem.R.inv-not-min} yields that $\left(
Rf\right)  \left(  v\right)  $ is invertible. This proves Lemma
\ref{lem.R.inv-2} in Case 4.

We have now proved Lemma \ref{lem.R.inv-2} in all four Cases 1, 2, 3 and 4.
\end{proof}
\end{verlong}

\section{\label{sec.rect}The rectangle: statements of the results}

\subsection{The $p\times q$-rectangle}

As promised, we now state the phenomena observed in Example
\ref{ex.rowmotion.2x2} in greater generality (and afterwards prove them).
First we define the posets on which these phenomena manifest:

\begin{definition}
For $p\in\mathbb{Z}$, we let $\left[  p\right]  $ denote the totally ordered
set $\left\{  1,2,\ldots,p\right\}  $ (with its usual total order:
$1<2<\cdots<p$). This set is empty if $p\leq0$.
\end{definition}

\begin{verlong}
We recall that the \emph{Cartesian product} $P\times Q$ of two posets $P$ and
$Q$ is defined to be the set $P\times Q$, equipped with the entrywise partial
order (i.e., the partial order in which two pairs $\left(  p_{1},q_{1}\right)
\in P\times Q$ and $\left(  p_{2},q_{2}\right)  \in P\times Q$ satisfy
$\left(  p_{1},q_{1}\right)  \leq\left(  p_{2},q_{2}\right)  $ if and only if
$p_{1}\leq p_{2}$ and $q_{1}\leq q_{2}$).
\end{verlong}

\begin{definition}
\label{def.rect}Let $p$ and $q$ be two positive integers. The $p\times
q$\emph{-rectangle} will mean the Cartesian product $\left[  p\right]
\times\left[  q\right]  $ of the two posets $\left[  p\right]  $ and $\left[
q\right]  $. Explicitly, this is the set $\left[  p\right]  \times\left[
q\right]  =\left\{  1,2,\ldots,p\right\}  \times\left\{  1,2,\ldots,q\right\}
$, equipped with the partial order defined as follows: For two elements
$\left(  i,j\right)  $ and $\left(  i^{\prime},j^{\prime}\right)  $ of
$\left[  p\right]  \times\left[  q\right]  $, we set $\left(  i,j\right)
\leq\left(  i^{\prime},j^{\prime}\right)  $ if and only if $\left(  i\leq
i^{\prime}\text{ and }j\leq j^{\prime}\right)  $.

Henceforth, if we speak of $\left[  p\right]  \times\left[  q\right]  $, we
implicitly assume that $p$ and $q$ are two positive integers.
\end{definition}

The $p\times q$-rectangle has been denoted by $\operatorname*{Rect}\left(
p,q\right)  $ in \cite{bir-row-arxiv}.

\Needspace{9\baselineskip}

\begin{example}
Here is the Hasse diagram of the $2\times3$-rectangle $\left[  2\right]
\times\left[  3\right]  $:%
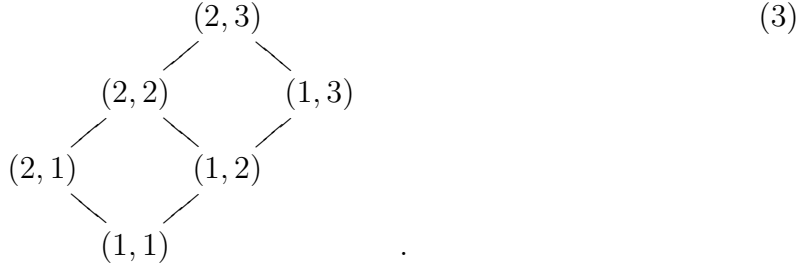
\begin{equation}
\xymatrixrowsep{0.9pc}\xymatrixcolsep{0.20pc}\xymatrix{ & & \left(2,3\right) \ar@{-}[rd] \ar@{-}[ld] & \\ & \left(2,2\right) \ar@{-}[rd] \ar@{-}[ld] & & \left(1,3\right) \ar@{-}[ld]\\ \left(2,1\right) \ar@{-}[rd] & & \left(1,2\right) \ar@{-}[ld] & \\ & \left(1,1\right) & & & & . }
\label{exa.rect.2x3.eq}%
\end{equation}

\end{example}

\begin{convention}
\label{conv.rect.draw}In the following, the Hasse diagram of a $p\times
q$-rectangle will always be drawn as in (\ref{exa.rect.2x3.eq}). That is, the
elements $\left(  i,j\right)  $ of $\left[  p\right]  \times\left[  q\right]
$ will be aligned in a rectangular grid, with the $x$-axis going southeast to
northwest and the $y$-axis going southwest to northeast. Thus, for instance,
the northwestern neighbor of an element $\left(  i,j\right)  $ is always
$\left(  i+1,j\right)  $.

Two elements $s$ and $t$ of $\widehat{P}$ will be called \emph{adjacent} if
they satisfy $s\gtrdot t$ or $t\gtrdot s$.
\end{convention}

The poset $\left[  p\right]  \times\left[  q\right]  $ has a unique minimal
element, $\left(  1,1\right)  $, and a unique maximal element, $\left(
p,q\right)  $. Its covering relation can be characterized by the following
easy remark (which will be used without explicit mention):

\begin{remark}
\label{rmk.rect.cover}Let $\left(  i,j\right)  $ and $\left(  i^{\prime
},j^{\prime}\right)  $ be two elements of $\left[  p\right]  \times\left[
q\right]  $. Then, $\left(  i,j\right)  \lessdot\left(  i^{\prime},j^{\prime
}\right)  $ if and only if $\left(  i^{\prime},j^{\prime}\right)  $ is either
$\left(  i+1,j\right)  $ or $\left(  i,j+1\right)  $.
\end{remark}

\begin{convention}
Let $P=\left[  p\right]  \times\left[  q\right]  $. If $f$ is a function
defined on $P$ or on $\widehat{P}$, and if $\left(  i,j\right)  $ is any
element of $P$, then we will write $f\left(  i,j\right)  $ for $f\left(
\left(  i,j\right)  \right)  $.
\end{convention}

\subsection{Periodicity}

The following theorem (conjectured by the first author in 2014) generalizes
the periodicity-like phenomenon seen in Example \ref{ex.rowmotion.2x2}:

\begin{theorem}
[Periodicity theorem for the $p\times q$-rectangle]\label{thm.rect.ord} Let
$P=\left[  p\right]  \times\left[  q\right]  $, and let $f\in\mathbb{K}%
^{\widehat{P}}$ be a $\mathbb{K}$-labeling such that $R^{p+q}f\neq\undf$. Set
$a=f\left(  0\right)  $ and $b=f\left(  1\right)  $. Then, $a$ and $b$ are
invertible, and for any $x\in\widehat{P}$ we have
\begin{equation}
\left(  R^{p+q}f\right)  \left(  x\right)  =a\overline{b}\cdot f\left(
x\right)  \cdot\overline{a}b. \label{eq.thm.rect.ord.claim}%
\end{equation}

\end{theorem}

If the ring $\mathbb{K}$ is commutative\footnote{or, more generally, if the
elements $a$ and $b$ commute with each other and with $f\left(  x\right)  $},
then (\ref{eq.thm.rect.ord.claim}) simplifies to $\left(  R^{p+q}f\right)
\left(  x\right)  =f\left(  x\right)  $ (since commutativity of $\mathbb{K}$
yields $a\overline{b}\cdot f\left(  x\right)  \cdot\overline{a}%
b=\underbrace{a\overline{a}}_{=1}\cdot\,f\left(  x\right)  \cdot
\underbrace{\overline{b}b}_{=1}=f\left(  x\right)  $). Thus, if $\mathbb{K}$
is commutative, then the claim of Theorem \ref{thm.rect.ord} can be rewritten
as $R^{p+q}f=f$, generalizing the main part of \cite[Theorem 11.5]%
{bir-row-arxiv} (which itself generalizes similar properties of rowmotion
operators on other levels). Unlike in \cite[Theorem 11.5]{bir-row-arxiv}, we
cannot honestly claim that $R^{p+q}=\operatorname*{id}$ even when $\mathbb{K}$
is commutative, since the partial map $R^{p+q}$ takes the value $\undf$ on
some $\mathbb{K}$-labelings $f$ (while $\operatorname*{id}$ does not).

The parallel result for birational antichain rowmotion~\cite[Conjecture
5.10]{JosRob20} follows from Theorem \ref{thm.rect.ord}.

\subsection{Reciprocity}

Theorem \ref{thm.rect.ord} shows that the \textquotedblleft periodicity
phenomenon\textquotedblright\ we have observed on $\left[  2\right]
\times\left[  2\right]  $ in Example \ref{ex.rowmotion.2x2} was not a
coincidence. The \textquotedblleft reciprocity phenomenon\textquotedblright%
\ is similarly the $p=q=2$ case of a general fact:

\begin{theorem}
[Reciprocity theorem for $p\times q$-rectangle]\label{thm.rect.antip} Let
$P=\left[  p\right]  \times\left[  q\right]  $. Fix $\ell\in\mathbb{N}$, and
let $f\in\mathbb{K}^{\widehat{P}}$ be a $\mathbb{K}$-labeling such that
$R^{\ell}f\neq\undf$. Set $a=f\left(  0\right)  $ and $b=f\left(  1\right)  $.
Let $\left(  i,j\right)  \in P$ satisfy $\ell-i-j+1\geq0$. Then,%
\begin{equation}
\left(  R^{\ell}f\right)  \left(  i,j\right)  =a\cdot\overline{\left(
R^{\ell-i-j+1}f\right)  \left(  p+1-i,\ q+1-j\right)  }\cdot b.
\label{eq.thm.rect.antip.claim}%
\end{equation}

\end{theorem}

\begin{vershort}
Theorem \ref{thm.rect.antip} directly generalizes the analogous
theorem~\cite[Theorem 11.7]{bir-row-arxiv} in the commutative setting.
\end{vershort}

\begin{verlong}
Theorem \ref{thm.rect.antip} directly generalizes the analogous
theorem~\cite[Theorem 11.7]{bir-row-arxiv} in the commutative
setting\footnote{In fact, let us assume that $\mathbb{K}$ is a (commutative)
field. Keep using the notation of Theorem \ref{thm.rect.antip}. Then,
Theorem~\ref{thm.rect.antip} (applied to $\ell=i+j-1$) yields%
\begin{align*}
\left(  R^{i+j-1}f\right)  \left(  i,j\right)   &  =a\cdot\overline{\left(
R^{\left(  i+j-1\right)  -i-j+1}f\right)  \left(  p+1-i,\ q+1-j\right)  }\cdot
b\\
&  =a\cdot\overline{f\left(  p+1-i,\ q+1-j\right)  }\cdot
b\ \ \ \ \ \ \ \ \ \ \left(  \text{since }\underbrace{R^{\left(  i+j-1\right)
-i-j+1}}_{=R^{0}=\id}f=f\right) \\
&  =\dfrac{ab}{f\left(  p+1-i,\ q+1-j\right)  }.
\end{align*}
Solving this equation for $f\left(  p+1-i,\ q+1-j\right)  $, we find%
\[
f\left(  p+1-i,\ q+1-j\right)  =\dfrac{ab}{\left(  R^{i+j-1}f\right)  \left(
i,j\right)  }=\dfrac{f\left(  0\right)  \cdot f\left(  1\right)  }{\left(
R^{i+j-1}f\right)  \left(  i,j\right)  }\ \ \ \ \ \ \ \ \ \ \left(
\text{since $a=f\left(  0\right)  $ and $b=f\left(  1\right)  $}\right)  ,
\]
which is precisely the claim of \cite[Theorem 11.7]{bir-row-arxiv} (with $k$
renamed as $j$).}.
\end{verlong}

\subsection{The structure of the proofs}

Theorems \ref{thm.rect.antip} and \ref{thm.rect.ord} are the main results of
this paper, and most of it will be devoted to their proofs. We first summarize
the large-scale structure of these proofs:

\begin{enumerate}
\item In Section \ref{sec.antip-to-ord}, we show that twisted periodicity
(Theorem \ref{thm.rect.ord}) follows from reciprocity (Theorem
\ref{thm.rect.antip}). Thus, proving the latter will suffice.

\item In Section \ref{sec.proof-nots}, we introduce some notations. Some of
these notations ($a$, $b$ and $x_{\ell}$) are mere abbreviations for the
labels of $R^{\ell}f$, while others ($\downslack_{\ell}^{v}$, $\upslack_{\ell
}^{v}$, $\downslack_{\ell}^{\mathbf{p}}$, $\upslack_{\ell}^{\mathbf{p}}$,
$\downslack_{\ell}^{u\rightarrow v}$ and $\upslack_{\ell}^{u\rightarrow v}$)
stand for certain derived quantities and will play a more active role. We also
define \textquotedblleft paths\textquotedblright\ on the poset $P$, and
introduce a few of their basic features.

\item In Section \ref{sec.proof-lems}, we prove a few simple results. The most
important of these results are Proposition \ref{prop.rect.trans} (which
reveals how birational rowmotion transforms $\downslack_{\ell-1}^{v}$ into
$\upslack_{\ell}^{v}$) and Theorem \ref{thm.rect.path} (which allows us to
recover the original labels $x_{\ell}$ from either $\downslack_{\ell}^{v}$ or
$\upslack_{\ell}^{v}$).

\item In Section \ref{sec.ij=11}, we prove Theorem \ref{thm.rect.antip} in the
case when $\left(  i,j\right)  =\left(  1,1\right)  $. This proof warrants its
own section both because it is conceptually easier than the general case, and
because it requires some \textquotedblleft well-definedness\textquotedblright%
\ technicalities that are (surprisingly) not needed in any other cases.

\item In Section \ref{sec.conversion}, we saddle the main workhorse of our
proof: a lemma (Lemma \ref{lem.rect.conv}) that connects certain
$\downslack_{\ell}^{u\rightarrow v}$ quantities with certain $\upslack_{\ell
}^{u\rightarrow v}$ quantities with the same $\ell$. We prove this using a
variant of paths, which we call \textquotedblleft
path-jump-paths\textquotedblright\ and which allow us to interpolate between
$\downslack_{\ell}^{u\rightarrow v}$ and $\upslack_{\ell}^{u\rightarrow v}$.

\item In Section \ref{sec.Rfi1}, we combine the previous results with this
lemma to prove Theorem \ref{thm.rect.antip} in the case when $j=1$.

\item In Section \ref{sec.gencase}, we finally complete the proof of Theorem
\ref{thm.rect.antip} in the general case. This requires almost no new ideas,
just an induction that extends Theorem \ref{thm.rect.antip} from four
\textquotedblleft adjacent\textquotedblright\ elements of $P$ (labeled
$u,m,s,t$ in diagram (\ref{pf.thm.rect.antip.short.diagram1})) to the fifth
element $v$.
\end{enumerate}

\section{\label{sec.antip-to-ord} Twisted periodicity follows from
reciprocity}

Our first step towards the proofs of twisted periodicity (Theorem
\ref{thm.rect.ord}) and reciprocity (Theorem~\ref{thm.rect.antip}) is to show
that the latter implies the former.\footnote{This reduction is not new; it
appears already in \cite[proof of Corollary 2.12]{MusRob17} in the commutative
case.}

\begin{vershort}
\begin{proof}
[Proof of Theorem \ref{thm.rect.ord} using Theorem \ref{thm.rect.antip}%
.]Assume that Theorem \ref{thm.rect.antip} has been proved. Let $p$, $q$, $P$,
$f$, $a$ and $b$ be as in Theorem \ref{thm.rect.ord}. Let $x\in\widehat{P}$.
From $p\geq1$ and $q\geq1$, we obtain $p+q\geq2$. Hence, from $R^{p+q}%
f\neq\undf$, we obtain $R^{2}f\neq\undf$ (by Lemma \ref{lem.R.wd-triv}).
Therefore, Lemma \ref{lem.R.01inv} yields that $a$ and $b$ are invertible
(since $a=f\left(  0\right)  $ and $b=f\left(  1\right)  $).

It remains to prove (\ref{eq.thm.rect.ord.claim}). First, we note that%
\[
a\overline{b}\cdot\underbrace{f\left(  0\right)  }_{=a}\cdot\overline
{a}b=a\overline{b}\cdot\underbrace{a\cdot\overline{a}}_{=1}%
b=a\underbrace{\overline{b}\cdot b}_{=1}=a.
\]
However, Corollary \ref{cor.R.implicit.01} yields $\left(  R^{p+q}f\right)
\left(  0\right)  =f\left(  0\right)  =a$. Comparing these, we find that
$\left(  R^{p+q}f\right)  \left(  0\right)  =a\overline{b}\cdot f\left(
0\right)  \cdot\overline{a}b$. Thus, the equality (\ref{eq.thm.rect.ord.claim}%
) holds for $x=0$. Similarly, this equality also holds for $x=1$. So from now
on, we WLOG assume that $x$ is neither $0$ nor $1$. Hence, $x = (i,j) \in
P=\left[  p\right]  \times\left[  q\right]  $.

From $R^{p+q}f\neq\undf$, we obtain $Rf\neq\undf$ (by Lemma
\ref{lem.R.wd-triv}, since $1\leq2\leq p+q$). Thus, Lemma \ref{lem.R.inv}
(applied to $v=\left(  i,j\right)  $) yields that $f\left(  i,j\right)  $ is
invertible. Hence, $\overline{f\left(  i,j\right)  }$ is well-defined. The
element $\overline{f\left(  i,j\right)  }$ of $\mathbb{K}$ is invertible
(since it has inverse $f\left(  i,j\right)  $).

Set $i^{\prime}:=p+1-i\in\left[  p\right]  $ and $j^{\prime}:=q+1-j \in\left[
q\right]  $, so that $\left(  i^{\prime-},j^{\prime}\right)  \in\left[
p\right]  \times\left[  q\right]  =P$ and $i^{\prime}+j^{\prime}\geq1+1=2$.
Also, the definitions of $i^{\prime}$ and $j^{\prime}$ readily yield
$p+1-i^{\prime}=i$ and $q+1-j^{\prime}=j$ and $i^{\prime}+j^{\prime
}-1=p+q-i-j+1$.

Now $i^{\prime}+j^{\prime}\leq p+q$, so $i^{\prime}+j^{\prime}-1\leq
i^{\prime}+j^{\prime}\leq p+q$ and thus $R^{i^{\prime}+j^{\prime}-1}%
f\neq\undf$ (by Lemma \ref{lem.R.wd-triv}, since $R^{p+q}f\neq\undf$). Thus,
Theorem \ref{thm.rect.antip} (applied to $i^{\prime}+j^{\prime}-1$,
$i^{\prime}$ and $j^{\prime}$ instead of $\ell$, $i$ and $j$) yields%
\begin{align}
\left(  R^{i^{\prime}+j^{\prime}-1}f\right)  \left(  i^{\prime},j^{\prime
}\right)   &  =a\cdot\overline{\underbrace{\left(  R^{\left(  i^{\prime
}+j^{\prime}-1\right)  -i^{\prime}-j^{\prime}+1}f\right)  }%
_{\substack{=f\\\text{(since }\left(  i^{\prime}+j^{\prime}-1\right)
-i^{\prime}-j^{\prime}+1=0\text{)}}}\left(  \underbrace{p+1-i^{\prime}}%
_{=i},\ \underbrace{q+1-j^{\prime}}_{=j}\right)  }\cdot b\nonumber\\
&  =a\cdot\overline{f\left(  i,j\right)  }\cdot b.
\label{pf.thm.rect.ord.short.2}%
\end{align}

However, we also have $p+q-i-j+1=i^{\prime}+j^{\prime}-1\geq0$ (since
$i^{\prime}+j^{\prime}\geq2\geq1$). Thus, Theorem \ref{thm.rect.antip}
(applied to $\ell=p+q$) yields%
\begin{align*}
\left(  R^{p+q}f\right)  \left(  i,j\right)   &  =a\cdot\overline{\left(
R^{p+q-i-j+1}f\right)  \left(  p+1-i,\ q+1-j\right)  }\cdot b\\
&  =a\cdot\overline{\left(  R^{i^{\prime}+j^{\prime}-1}f\right)  \left(
i^{\prime},j^{\prime}\right)  }\cdot b\ \ \ \ \ \ \ \ \ \ \left(
\begin{array}
[c]{c}%
\text{since }p+q-i-j+1=i^{\prime}+j^{\prime}-1\\
\text{and }p+1-i=i^{\prime}\text{ and }q+1-j=j^{\prime}%
\end{array}
\right) \\
&  =a\cdot\underbrace{\overline{a\cdot\overline{f\left(  i,j\right)  }\cdot
b}}_{\substack{=\overline{b}\cdot f\left(  i,j\right)  \cdot\overline
{a}\\\text{(by Proposition \ref{prop.inverses.ab} \textbf{(c)},}\\\text{since
}a\text{ and }\overline{f\left(  i,j\right)  }\text{ and }b\text{ are
invertible)}}}\cdot\,b\ \ \ \ \ \ \ \ \ \ \left(  \text{by
(\ref{pf.thm.rect.ord.short.2})}\right) \\
&  =a\overline{b}\cdot f\left(  i,j\right)  \cdot\overline{a}b.
\end{align*}
Since $x=\left(  i,j\right)  $, we can rewrite this as%
\[
\left(  R^{p+q}f\right)  \left(  x\right)  =a\overline{b}\cdot f\left(
x\right)  \cdot\overline{a}b.
\]
Thus twisted periodicity (Theorem~\ref{thm.rect.ord}) is proved, assuming
reciprocity (Theorem~\ref{thm.rect.antip}) holds.
\end{proof}
\end{vershort}

\begin{verlong}
\begin{proof}
[Proof of Theorem \ref{thm.rect.ord} using Theorem \ref{thm.rect.antip}%
.]Assume that Theorem \ref{thm.rect.antip} has been proved. Now, let $p$, $q$,
$P$, $f$, $a$ and $b$ be as in Theorem \ref{thm.rect.ord}. From $p\geq1$ and
$q\geq1$, we obtain $p+q\geq1+1=2$, so that $2\leq p+q$. Hence, from
$R^{p+q}f\neq\undf$, we obtain $R^{2}f\neq\undf$ (by Lemma \ref{lem.R.wd-triv}%
). Therefore, Lemma \ref{lem.R.01inv} yields that $f\left(  0\right)  $ and
$f\left(  1\right)  $ are invertible. In other words, $a$ and $b$ are
invertible (since $a=f\left(  0\right)  $ and $b=f\left(  1\right)  $).

Let $x \in\widehat{P}$. It thus remains to prove the equality
(\ref{eq.thm.rect.ord.claim}). Let us first verify that this equality holds
for $x=0$. Indeed,%
\[
a\overline{b}\cdot\underbrace{f\left(  0\right)  }_{=a}\cdot\overline
{a}b=a\overline{b}\cdot\underbrace{a\cdot\overline{a}}_{=1}%
b=a\underbrace{\overline{b}\cdot b}_{=1}=a=\left(  R^{p+q}f\right)  \left(
0\right)
\]
(because Corollary \ref{cor.R.implicit.01} (applied to $\ell=p+q$) yields
$\left(  R^{p+q}f\right)  \left(  0\right)  =f\left(  0\right)  =a$). In other
words, $\left(  R^{p+q}f\right)  \left(  0\right)  =a\overline{b}\cdot
f\left(  0\right)  \cdot\overline{a}b$. In other words, the equality
(\ref{eq.thm.rect.ord.claim}) holds for $x=0$.

Similarly, this equality also holds for $x=1$. Thus, for the rest of this
proof, we WLOG assume that $x$ is neither $0$ nor $1$. Hence, $x\in P=\left[
p\right]  \times\left[  q\right]  $. In other words, $x=\left(  i,j\right)  $
for some $i\in\left[  p\right]  $ and $j\in\left[  q\right]  $. Consider these
$i$ and $j$.

From $R^{p+q}f\neq\undf$, we obtain $R^{1}f\neq\undf$ (by Lemma
\ref{lem.R.wd-triv}, since $1\leq2\leq p+q$). In other words, $Rf\neq\undf$.
Thus, Lemma \ref{lem.R.inv} (applied to $v=\left(  i,j\right)  $) yields that
$f\left(  i,j\right)  $ is invertible. Hence, $\overline{f\left(  i,j\right)
}$ is well-defined.

The element $\overline{f\left(  i,j\right)  }$ of $\mathbb{K}$ is invertible
(since it has inverse $f\left(  i,j\right)  $), and so are the elements $a$
and $b$ (as we have proved above). Thus, Proposition \ref{prop.inverses.ab}
\textbf{(c)} yields%
\begin{equation}
\overline{a\cdot\overline{f\left(  i,j\right)  }\cdot b}=\overline{b}%
\cdot\underbrace{\overline{\overline{f\left(  i,j\right)  }}}_{=f\left(
i,j\right)  }\cdot\overline{a}=\overline{b}\cdot f\left(  i,j\right)
\cdot\overline{a}. \label{pf.thm.rect.ord.inv-prod}%
\end{equation}

Set $i^{\prime}:=p+1-i$ and $j^{\prime}:=q+1-j$. Thus,
\begin{align*}
p+1-\underbrace{i^{\prime}}_{=p+1-i}  &  = p+1-\left(  p+1-i\right)  =
i\qquad\qquad\text{and}\\
q+1-\underbrace{j^{\prime}}_{=q+1-j}  &  = q+1-\left(  q+1-j\right)  = j
\end{align*}
and
\[
p+q-i-j+1=\underbrace{p+1-i}_{=i^{\prime}}+\underbrace{q+1-j}_{=j^{\prime}%
}-1=i^{\prime}+j^{\prime}-1.
\]
Therefore, $i^{\prime}+j^{\prime}-1=p+q-\underbrace{i}_{\leq p}-\underbrace{j}%
_{\leq q}+1\geq p+q-p-q+1=1$, so that $i^{\prime}+j^{\prime}-1\in\mathbb{N}$.
Moreover, $i^{\prime}=p+1-i\in\left[  p\right]  $ (since $i\in\left[
p\right]  $) and $j^{\prime}\in\left[  q\right]  $ (similarly). Hence,
$\left(  i^{\prime},j^{\prime}\right)  \in\left[  p\right]  \times\left[
q\right]  =P$.

We have $i^{\prime}\leq p$ (since $i^{\prime}\in\left[  p\right]  $) and
similarly $j^{\prime}\leq q$. Adding these two inequalities together, we find
$i^{\prime}+j^{\prime}\leq p+q$. Hence, $i^{\prime}+j^{\prime}-1\leq
i^{\prime}+j^{\prime}\leq p+q$ and thus $R^{i^{\prime}+j^{\prime}-1}%
f\neq\undf$ (by Lemma \ref{lem.R.wd-triv}, since $R^{p+q}f\neq\undf$). Also,
$\left(  i^{\prime}+j^{\prime}-1\right)  -i^{\prime}-j^{\prime}+1=0\geq0$.

Thus, Theorem \ref{thm.rect.antip} (applied to $i^{\prime}+j^{\prime}-1$,
$i^{\prime}$ and $j^{\prime}$ instead of $\ell$, $i$ and $j$) yields%
\begin{align}
\left(  R^{i^{\prime}+j^{\prime}-1}f\right)  \left(  i^{\prime},j^{\prime
}\right)   &  =a\cdot\overline{\left(  \underbrace{R^{\left(  i^{\prime
}+j^{\prime}-1\right)  -i^{\prime}-j^{\prime}+1}}_{\substack{=R^{0}%
\\\text{(since }\left(  i^{\prime}+j^{\prime}-1\right)  -i^{\prime}-j^{\prime
}+1=0\text{)}}}f\right)  \left(  \underbrace{p+1-i^{\prime}}_{=i}%
,\ \underbrace{q+1-j^{\prime}}_{=j}\right)  }\cdot b\nonumber\\
&  =a\cdot\overline{\underbrace{\left(  R^{0}f\right)  }%
_{\substack{=f\\\text{(since }R^{0}=\operatorname*{id}\text{)}}}\left(
i,j\right)  }\cdot b\nonumber\\
&  =a\cdot\overline{f\left(  i,j\right)  }\cdot b. \label{pf.thm.rect.ord.2}%
\end{align}

However, we also have $p+q-i-j+1=i^{\prime}+j^{\prime}-1\geq1\geq0$. Thus,
Theorem \ref{thm.rect.antip} (applied to $\ell=p+q$) yields%
\begin{align*}
\left(  R^{p+q}f\right)  \left(  i,j\right)   &  =a\cdot\overline{\left(
R^{p+q-i-j+1}f\right)  \left(  p+1-i,\ q+1-j\right)  }\cdot b\\
&  =a\cdot\overline{\left(  R^{i^{\prime}+j^{\prime}-1}f\right)  \left(
i^{\prime},j^{\prime}\right)  }\cdot b\\
&  \ \ \ \ \ \ \ \ \ \ \ \ \ \ \ \ \ \ \ \ \left(
\begin{array}
[c]{c}%
\text{since }p+q-i-j+1=i^{\prime}+j^{\prime}-1\\
\text{and }p+1-i=i^{\prime}\text{ and }q+1-j=j^{\prime}%
\end{array}
\right) \\
&  =a\cdot\underbrace{\overline{a\cdot\overline{f\left(  i,j\right)  }\cdot
b}}_{\substack{=\overline{b}\cdot f\left(  i,j\right)  \cdot\overline
{a}\\\text{(by (\ref{pf.thm.rect.ord.inv-prod}))}}}\cdot
b\ \ \ \ \ \ \ \ \ \ \left(  \text{by (\ref{pf.thm.rect.ord.2})}\right) \\
&  =a\overline{b}\cdot f\left(  i,j\right)  \cdot\overline{a}b.
\end{align*}
Since $x=\left(  i,j\right)  $, we can rewrite this as%
\[
\left(  R^{p+q}f\right)  \left(  x\right)  =a\overline{b}\cdot f\left(
x\right)  \cdot\overline{a}b.
\]
This proves the equality (\ref{eq.thm.rect.ord.claim}). Theorem
\ref{thm.rect.ord} is thus proved, assuming that Theorem \ref{thm.rect.antip} holds.
\end{proof}
\end{verlong}


\section{\label{sec.proof-nots}Proof of reciprocity: notations}

It now suffices to prove Theorem \ref{thm.rect.antip}, which will be the
ultimate goal of the next few sections. First we introduce some notations that
will be used throughout these sections.



Fix two positive integers $p$ and $q$. Assume that $P = \left[  p\right]
\times\left[  q\right]  $. Let $f\in\mathbb{K}^{\widehat{P}}$ be a
$\mathbb{K}$-labeling of $P$. Set%
\[
a:=f\left(  0\right)  \ \ \ \ \ \ \ \ \ \ \text{and}%
\ \ \ \ \ \ \ \ \ \ b:=f\left(  1\right)  .
\]


For any $x=\left(  i,j\right)  \in P$, we define an element $x^{\sim}\in P$
by
\[
x^{\sim} :=\left(  p+1-i,\ q+1-j\right)  .
\]
We call this element $x^{\sim}$ the \emph{antipode} of $x$. Thus, the desired
equality (\ref{eq.thm.rect.antip.claim}) can be rewritten as%
\begin{equation}
\left(  R^{\ell}f\right)  \left(  x\right)  =a\cdot\overline{\left(
R^{\ell-i-j+1}f\right)  \left(  x^{\sim}\right)  }\cdot b
\label{eq.thm.rect.antip.claim.rewr}%
\end{equation}
for $x=\left(  i,j\right)  $.

For any $x\in\widehat{P}$ and $\ell\in\mathbb{N}$, we write%
\begin{equation}
x_{\ell}:=\left(  R^{\ell}f\right)  \left(  x\right)  , \label{eq.xl=}%
\end{equation}
which is well-defined whenever $R^{\ell}f\neq\undf$. This compact notation
will make upcoming formulas more readable.

\begin{vershort}
In particular, for each $x\in\widehat{P}$, we have $x_{0}=\left(
R^{0}f\right)  \left(  x\right)  =f\left(  x\right)  $.
Moreover, for each $\ell\in\mathbb{N}$ satisfying $R^{\ell}f\neq\undf$, we
have
\begin{align}
0_{\ell}  &  =\left(  R^{\ell}f\right)  \left(  0\right)  =a
\ \ \ \ \ \ \ \ \ \ \left(  \text{via Corollary \ref{cor.R.implicit.01}%
}\right)  \label{eq.0l.a.short}%
\end{align}
and similarly $1_{\ell}=b$.
\end{vershort}

\begin{verlong}
In particular, for each $x\in\widehat{P}$, we have
\[
x_{0}=\underbrace{\left(  R^{0}f\right)  }_{=f}\left(  x\right)  =f\left(
x\right)  .
\]
Moreover, for each $\ell\in\mathbb{N}$ satisfying $R^{\ell}f\neq\undf$, we
have
\begin{align}
0_{\ell}  &  =\left(  R^{\ell}f\right)  \left(  0\right)  =f\left(  0\right)
\ \ \ \ \ \ \ \ \ \ \left(  \text{by Corollary \ref{cor.R.implicit.01}}\right)
\nonumber\\
&  =a \label{eq.0l.a}%
\end{align}
and
\begin{align}
1_{\ell}  &  =\left(  R^{\ell}f\right)  \left(  1\right)  =f\left(  1\right)
\ \ \ \ \ \ \ \ \ \ \left(  \text{by Corollary \ref{cor.R.implicit.01}}\right)
\nonumber\\
&  =b. \label{eq.1l.b}%
\end{align}

\end{verlong}

We can further rewrite the equality (\ref{eq.thm.rect.antip.claim.rewr}) as
$x_{\ell}=a\cdot\overline{x_{\ell-i-j+1}^{\sim}}\cdot b$ (since $x_{\ell
}=\left(  R^{\ell}f\right)  \left(  x\right)  $ and $x_{\ell-i-j+1}^{\sim
}=\left(  R^{\ell-i-j+1}f\right)  \left(  x^{\sim}\right)  $). Hence, our
desired Theorem \ref{thm.rect.antip} takes the following form:

\begin{statement}
\textbf{Theorem \ref{thm.rect.antip}, restated.} If $x=\left(  i,j\right)  \in
P$ and $\ell\in\mathbb{N}$ satisfy $\ell-i-j+1\geq0$ and $R^{\ell}f\neq\undf$,
then%
\begin{equation}
x_{\ell}=a\cdot\overline{x_{\ell-i-j+1}^{\sim}}\cdot b.
\label{pf.thm.rect.antip.restate1}%
\end{equation}

\end{statement}

\begin{vershort}
Proposition \ref{prop.R.implicit} yields that for each $v\in P$, we
have\footnote{assuming that $Rf\neq\undf$}%
\begin{equation}
\left(  Rf\right)  \left(  v\right)  =\left(  \sum\limits_{u\lessdot
v}f\left(  u\right)  \right)  \cdot\overline{f\left(  v\right)  }%
\cdot\overline{\sum\limits_{u\gtrdot v}\overline{\left(  Rf\right)  \left(
u\right)  }}. \label{pf.thm.rect.antip.Rfv=}%
\end{equation}
(In both sums, $u$ ranges over $\widehat{P}$; from now on, this will always be
understood if not otherwise specified.) Applying this equality
(\ref{pf.thm.rect.antip.Rfv=}) to $R^{\ell}f$ instead of $f$, we obtain%
\[
\left(  R^{\ell+1}f\right)  \left(  v\right)  =\left(  \sum\limits_{u\lessdot
v}\left(  R^{\ell}f\right)  \left(  u\right)  \right)  \cdot\overline{\left(
R^{\ell}f\right)  \left(  v\right)  }\cdot\overline{\sum\limits_{u\gtrdot
v}\overline{\left(  R^{\ell+1}f\right)  \left(  u\right)  }}%
\]
for each $v\in P$ and $\ell\in\mathbb{N}$ satisfying $R^{\ell+1}f\neq\undf$
(since $R\left(  R^{\ell}f\right)  =R^{\ell+1}f$).
\end{vershort}

\begin{verlong}
Proposition \ref{prop.R.implicit} yields that for each $v\in P$, we
have\footnote{assuming that $Rf\neq\undf$}%
\begin{equation}
\left(  Rf\right)  \left(  v\right)  =\left(  \sum\limits_{u\lessdot
v}f\left(  u\right)  \right)  \cdot\overline{f\left(  v\right)  }%
\cdot\overline{\sum\limits_{u\gtrdot v}\overline{\left(  Rf\right)  \left(
u\right)  }}.
\end{equation}
Here, the summation index $u$ under both sums is understood to range over
$\widehat{P}$; from now on, this will always be understood if not otherwise specified.

For each $v\in P$ and $\ell\in\mathbb{N}$ satisfying $R^{\ell+1}f\neq\undf$,
we have%
\begin{align*}
\left(  R^{\ell+1}f\right)  \left(  v\right)   &  =\left(  R\left(  R^{\ell
}f\right)  \right)  \left(  v\right)  \ \ \ \ \ \ \ \ \ \ \left(  \text{since
}R^{\ell+1}f=R\left(  R^{\ell}f\right)  \right) \\
&  =\left(  \sum\limits_{\substack{u\in\widehat{P};\\u\lessdot v}}\left(
R^{\ell}f\right)  \left(  u\right)  \right)  \cdot\overline{\left(  R^{\ell
}f\right)  \left(  v\right)  }\cdot\overline{\sum\limits_{\substack{u\in
\widehat{P};\\u\gtrdot v}}\overline{\left(  R\left(  R^{\ell}f\right)
\right)  \left(  u\right)  }}\\
&  \ \ \ \ \ \ \ \ \ \ \ \ \ \ \ \ \ \ \ \ \left(
\begin{array}
[c]{c}%
\text{by Proposition \ref{prop.R.implicit}, applied to }R^{\ell}f\text{
instead of }f\\
\text{(since }R\left(  R^{\ell}f\right)  =R^{\ell+1}f\neq\undf\text{)}%
\end{array}
\right) \\
&  =\left(  \sum\limits_{\substack{u\in\widehat{P};\\u\lessdot v}}\left(
R^{\ell}f\right)  \left(  u\right)  \right)  \cdot\overline{\left(  R^{\ell
}f\right)  \left(  v\right)  }\cdot\overline{\sum\limits_{\substack{u\in
\widehat{P};\\u\gtrdot v}}\overline{\left(  R^{\ell+1}f\right)  \left(
u\right)  }}\\
&  \ \ \ \ \ \ \ \ \ \ \ \ \ \ \ \ \ \ \ \ \left(  \text{since }R\left(
R^{\ell}f\right)  =R^{\ell+1}f\right) \\
&  =\left(  \sum\limits_{u\lessdot v}\left(  R^{\ell}f\right)  \left(
u\right)  \right)  \cdot\overline{\left(  R^{\ell}f\right)  \left(  v\right)
}\cdot\overline{\sum\limits_{u\gtrdot v}\overline{\left(  R^{\ell+1}f\right)
\left(  u\right)  }}%
\end{align*}
(here, we have rewritten the summation signs $\sum\limits_{\substack{u\in
\widehat{P};\\u\lessdot v}}$ and $\sum\limits_{\substack{u\in\widehat{P}%
;\\u\gtrdot v}}$ as $\sum\limits_{u\lessdot v}$ and $\sum\limits_{u\gtrdot v}%
$, since we understand the summation index $u$ to range over $\widehat{P}$ by default).
\end{verlong}

Using (\ref{eq.xl=}), we can rewrite this as follows:%
\begin{equation}
v_{\ell+1}=\left(  \sum\limits_{u\lessdot v}u_{\ell}\right)  \cdot
\overline{v_{\ell}}\cdot\overline{\sum\limits_{u\gtrdot v}\overline{u_{\ell
+1}}} \label{eq.vl+1=}%
\end{equation}
for each $v\in P$ and $\ell\in\mathbb{N}$ satisfying $R^{\ell+1}f\neq\undf$.

Next, we formally define the paths that will play a key role in the proof. A
\emph{path} means a sequence $\left(  v_{0},v_{1},\ldots,v_{k}\right)  $ of
elements of $\widehat{P}$ satisfying $v_{0}\gtrdot v_{1}\gtrdot\cdots\gtrdot
v_{k}$. We denote this path by $\left(  v_{0}\gtrdot v_{1}\gtrdot\cdots\gtrdot
v_{k}\right)  $, and we will call it a \emph{path from }$v_{0}$\emph{ to
}$v_{k}$ (or, for short, a \emph{path }$v_{0}\rightarrow v_{k}$). The
\emph{vertices} of this path are defined to be the elements $v_{0}%
,v_{1},\ldots,v_{k}$. We say that this path \emph{starts} at $v_{0}$ and
\emph{ends} at $v_{k}$.

For each $v\in P$ and $\ell\in\mathbb{N}$, we set\footnote{We recall that the
summation signs \textquotedblleft$\sum\limits_{u\lessdot v}$\textquotedblright%
\ and \textquotedblleft$\sum\limits_{u\gtrdot v}$\textquotedblright\ mean
\textquotedblleft$\sum\limits_{\substack{u\in\widehat{P};\\u\lessdot v}%
}$\textquotedblright\ and \textquotedblleft$\sum\limits_{\substack{u\in
\widehat{P};\\u\gtrdot v}}$\textquotedblright, respectively.}%
\[
\downslack_{\ell}^{v}:=v_{\ell}\cdot\overline{\sum\limits_{u\lessdot v}%
u_{\ell}}\ \ \ \ \ \ \ \ \ \ \text{and}\ \ \ \ \ \ \ \ \ \ \upslack_{\ell}%
^{v}:=\overline{\sum\limits_{u\gtrdot v}\overline{u_{\ell}}}\cdot
\overline{v_{\ell}}.
\]
\footnote{These elements $\downslack_{\ell}^{v}$ and $\upslack_{\ell}^{v}$ are
not always well-defined. For $\downslack_{\ell}^{v}$ to be well-defined, we
need to have $R^{\ell}f\neq\undf$, and we need the element $\sum
\limits_{u\lessdot v}u_{\ell}$ to be invertible. For $\upslack_{\ell}^{v}$ to
be well-defined, we need to have $R^{\ell}f\neq\undf$, and we need the
elements $\overline{u_{\ell}}$ (for $u\gtrdot v$) and $\sum\limits_{u\gtrdot
v}\overline{u_{\ell}}$ and $v_{\ell}$ to be invertible.} Furthermore, when
$v\in\left\{  0,1\right\}  $, we set
\begin{equation}
\downslack_{\ell}^{v}:=1\ \ \ \ \ \ \ \ \ \ \text{and}%
\ \ \ \ \ \ \ \ \ \ \upslack_{\ell}^{v}:=1 \label{eq.slack.1}%
\end{equation}
for all $\ell\in\mathbb{N}$.

For any path $\mathbf{p}=\left(  v_{0}\gtrdot v_{1}\gtrdot\cdots\gtrdot
v_{k}\right)  $ and any $\ell\in\mathbb{N}$, we set
\begin{align*}
\downslack_{\ell}^{\mathbf{p}}  &  :=\downslack_{\ell}^{v_{0}}\downslack_{\ell
}^{v_{1}}\cdots\downslack_{\ell}^{v_{k}}\ \ \ \ \ \ \ \ \ \ \text{and}\\
\upslack_{\ell}^{\mathbf{p}}  &  :=\upslack_{\ell}^{v_{0}}\upslack_{\ell
}^{v_{1}}\cdots\upslack_{\ell}^{v_{k}}%
\end{align*}
(assuming that the factors on the right hand sides are well-defined).

If $u$ and $v$ are elements of $\widehat{P}$, and if $\ell\in\mathbb{N}$, then
we set%
\begin{align}
\downslack_{\ell}^{u\rightarrow v}  &  :=\sum_{\mathbf{p}\text{ is a path from
}u\text{ to }v}\downslack_{\ell}^{\mathbf{p}}\ \ \ \ \ \ \ \ \ \ \text{and}%
\label{eq.slack.down-uv}\\
\upslack_{\ell}^{u\rightarrow v}  &  :=\sum_{\mathbf{p}\text{ is a path from
}u\text{ to }v}\upslack_{\ell}^{\mathbf{p}} \label{eq.slack.up-uv}%
\end{align}
(assuming that all addends on the right hand sides are well-defined).

\begin{example}
Let $P=\left[  2\right]  \times\left[  2\right]  $ and $f\in\mathbb{K}%
^{\widehat{P}}$ be as in Example \ref{ex.rowmotion.2x2}. Then,%
\begin{align*}
\left(  1,1\right)  ^{\sim}  &  =\left(  2,2\right)
,\ \ \ \ \ \ \ \ \ \ \left(  1,2\right)  ^{\sim}=\left(  2,1\right)
,\ \ \ \ \ \ \ \ \ \ \left(  2,1\right)  ^{\sim}=\left(  1,2\right)
,\ \ \ \ \ \ \ \ \ \ \left(  2,2\right)  ^{\sim}=\left(  1,1\right)  ,\\
\left(  1,1\right)  _{0}  &  =f\left(  1,1\right)
=w,\ \ \ \ \ \ \ \ \ \ \left(  1,1\right)  _{1}=\left(  Rf\right)  \left(
1,1\right)  =a\overline{z}b,\\
\left(  1,1\right)  _{2}  &  =\left(  R^{2}f\right)  \left(  1,1\right)
=a\overline{b}z\cdot\overline{x+y}\cdot b,\\
\left(  1,2\right)  _{2}  &  =\left(  R^{2}f\right)  \left(  1,2\right)
=a\cdot\overline{x+y}\cdot y\left(  \overline{x}+\overline{y}\right)  b.
\end{align*}

There are only two paths from $\left(  2,2\right)  $ to $\left(  1,1\right)
$: namely, the path $\left(  \left(  2,2\right)  \gtrdot\left(  1,2\right)
\gtrdot\left(  1,1\right)  \right)  $ and the path $\left(  \left(
2,2\right)  \gtrdot\left(  2,1\right)  \gtrdot\left(  1,1\right)  \right)  $.
Each of these two paths has three vertices. There are no paths from $\left(
1,1\right)  $ to $\left(  2,2\right)  $, since we don't have $\left(
1,1\right)  \geq\left(  2,2\right)  $. The only path from $0$ to $0$ is the
trivial path $\left(  0\right)  $.

We have
\begin{align*}
\downslack_{0}^{\left(  1,1\right)  }  &  =\left(  1,1\right)  _{0}%
\cdot\overline{\sum\limits_{u\lessdot\left(  1,1\right)  }u_{0}}=\left(
1,1\right)  _{0}\cdot\overline{0_{0}}=w\cdot\overline{a},\\
\downslack_{0}^{\left(  2,2\right)  }  &  =\left(  2,2\right)  _{0}%
\cdot\overline{\sum\limits_{u\lessdot\left(  2,2\right)  }u_{0}}=\left(
2,2\right)  _{0}\cdot\overline{\left(  1,2\right)  _{0}+\left(  2,1\right)
_{0}}=z\cdot\overline{y+x},\\
\downslack_{1}^{\left(  1,1\right)  }  &  =\left(  1,1\right)  _{1}%
\cdot\overline{\sum\limits_{u\lessdot\left(  1,1\right)  }u_{1}}=\left(
1,1\right)  _{1}\cdot\overline{0_{1}}=a\overline{z}b\cdot\overline{a},\\
\upslack_{0}^{\left(  1,1\right)  }  &  =\overline{\sum\limits_{u\gtrdot
\left(  1,1\right)  }\overline{u_{0}}}\cdot\overline{\left(  1,1\right)  _{0}%
}=\overline{\overline{\left(  1,2\right)  _{0}}+\overline{\left(  2,1\right)
_{0}}}\cdot\overline{\left(  1,1\right)  _{0}}=\overline{\overline
{y}+\overline{x}}\cdot\overline{w},\\
\upslack_{1}^{\left(  1,1\right)  }  &  =\overline{\sum\limits_{u\gtrdot
\left(  1,1\right)  }\overline{u_{1}}}\cdot\overline{\left(  1,1\right)  _{1}%
}=\overline{\overline{\left(  1,2\right)  _{1}}+\overline{\left(  2,1\right)
_{1}}}\cdot\overline{\left(  1,1\right)  _{1}}\\
&  =\overline{\overline{w\overline{y}\left(  x+y\right)  \overline{z}%
b}+\overline{w\overline{x}\left(  x+y\right)  \overline{z}b}}\cdot
\overline{a\overline{z}b}=w\cdot\overline{a}\ \ \ \ \ \ \ \ \ \ \left(
\text{after simplifications}\right)  .
\end{align*}

Furthermore, for any $\ell\in\mathbb{N}$, we have%
\begin{align*}
\downslack_{\ell}^{\left(  \left(  2,2\right)  \gtrdot\left(  1,2\right)
\gtrdot\left(  1,1\right)  \right)  }  &  =\downslack_{\ell}^{\left(
2,2\right)  }\downslack_{\ell}^{\left(  1,2\right)  }\downslack_{\ell
}^{\left(  1,1\right)  };\\
\downslack_{\ell}^{\left(  2,2\right)  \rightarrow\left(  1,1\right)  }  &
=\downslack_{\ell}^{\left(  \left(  2,2\right)  \gtrdot\left(  1,2\right)
\gtrdot\left(  1,1\right)  \right)  }+\downslack_{\ell}^{\left(  \left(
2,2\right)  \gtrdot\left(  2,1\right)  \gtrdot\left(  1,1\right)  \right)  }\\
&  =\downslack_{\ell}^{\left(  2,2\right)  }\downslack_{\ell}^{\left(
1,2\right)  }\downslack_{\ell}^{\left(  1,1\right)  }+\downslack_{\ell
}^{\left(  2,2\right)  }\downslack_{\ell}^{\left(  2,1\right)  }%
\downslack_{\ell}^{\left(  1,1\right)  }%
\end{align*}
(and similarly for $\upslack$ instead of $\downslack$).
\end{example}

The letter $\ell$ will always stand for a nonnegative integer (but will not be fixed).

\begin{remark}
The elements $\downslack_{\ell}^{v}$ and $\upslack_{\ell}^{v}$ (for $v\in P$
and $\ell\in\mathbb{N}$) are not entirely new. They are closely connected with
the down-transfer operator $\nabla$ and the up-transfer operator $\Delta$
studied in \cite[Definition 5.11]{JosRob20}; to be specific, we have
$\downslack_{\ell}^{v}=\left(  \nabla R^{\ell}f\right)  \left(  v\right)  $
and $\upslack_{\ell}^{v}=\left(  \Delta\Theta R^{\ell}f\right)  \left(
v\right)  $ using the notations of \cite[Definition 5.11]{JosRob20}. These
operators $\nabla$ and $\Delta$ have a long history, going back to Stanley's
\textquotedblleft transfer map\textquotedblright\ $\phi$ between the order
polytope and the chain polytope of a poset (see \cite[Definition
3.1]{stanley-polytopes}). The down-transfer operator $\nabla$ does indeed
restrict to $\phi$ when $\mathbb{K}$ is an appropriate tropical semiring. For
this reason, we have been informally referring to $\downslack_{\ell}^{v}$ and
$\upslack_{\ell}^{v}$ as the \emph{down-slack} and the \emph{up-slack} of $v$
at time $\ell$ (harkening back to the notion of slack from linear
optimization). Arguably, the behavior of these operators when $\mathbb{K}$ is
the tropical semiring is not very indicative of the general case.

When $\mathbb{K}$ is commutative, our $\downslack_{0}^{v}$ have also
implicitly appeared in \cite{MusRob17}: If $v=\left(  i,j\right)  \in P$, then
$\downslack_{0}^{v}=\overline{A_{ij}}$, where $A_{ij}$ is defined as in
\cite[(1)]{MusRob17}.
\end{remark}

\section{\label{sec.proof-lems}Proof of reciprocity: simple lemmas}

Throughout this section, we use the notations introduced in Section
\ref{sec.proof-nots}.


Let us prove some relations between the elements we have introduced. We begin
with a well-definedness result:

\begin{lemma}
\label{lem.slacks.wd}Let $\ell\in\mathbb{N}$ be such that $\ell\geq1$ and
$R^{\ell}f\neq\undf$. Assume furthermore that $a$ is invertible. Let
$v\in\widehat{P}$. Then:

\begin{enumerate}
\item[\textbf{(a)}] The element $v_{\ell}$ is well-defined and invertible.

\item[\textbf{(b)}] The element $v_{\ell-1}$ is well-defined and invertible.

\item[\textbf{(c)}] The element $\downslack_{\ell-1}^{v}$ is well-defined and invertible.

\item[\textbf{(d)}] The element $\upslack_{\ell}^{v}$ is well-defined and invertible.
\end{enumerate}
\end{lemma}


\begin{vershort}

\begin{proof}
From $R^{\ell}f \neq\undf$, we obtain $R^{\ell- 1} f \neq\undf$. Hence,
Corollary \ref{cor.R.implicit.01} yields that $\left(  R^{\ell-1}f\right)
\left(  0\right)  =f\left(  0\right)  =a$, which is invertible by assumption.
\medskip

\textbf{(a)} Clearly, $v_{\ell}=\left(  R^{\ell}f\right)  \left(  v\right)  $
is well-defined, and we have $v_{\ell} = \left(  R\left(  R^{\ell-1}f\right)
\right)  \left(  v\right)  $. Hence, Lemma~\ref{lem.R.inv-2} (applied to
$R^{\ell-1}f$ instead of $f$) yields that $v_{\ell}$ is invertible. \medskip

\textbf{(b)} If $v=0$, then this follows from part \textbf{(a)}, because
(\ref{eq.0l.a.short}) yields that $v_{\ell-1}=a=v_{\ell}$ in this case. An
analogous argument works if $v=1$. Thus, we WLOG assume that $v\notin\left\{
0,1\right\}  $, so that $v\in P$. The element $v_{\ell-1}=\left(  R^{\ell
-1}f\right)  \left(  v\right)  $ is clearly well-defined, and is invertible by
Lemma \ref{lem.R.inv} (applied to $R^{\ell-1}f$ instead of $f$). \medskip

\textbf{(c)} If $v\in\left\{  0,1\right\}  $, then this follows from
(\ref{eq.slack.1}). Otherwise, $v\in P$. Applying (\ref{eq.vl+1=}) to $\ell-1$
instead of $\ell$, we obtain%
\begin{equation}
v_{\ell}=\left(  \sum\limits_{u\lessdot v}u_{\ell-1}\right)  \cdot
\overline{v_{\ell-1}}\cdot\overline{\sum\limits_{u\gtrdot v}\overline{u_{\ell
}}}. \label{pf.lem.slacks.wd.short.c.1}%
\end{equation}
This equality shows that $\overline{v_{\ell-1}}$ and $\overline{\sum
\limits_{u\gtrdot v}\overline{u_{\ell}}}$ are well-defined, i.e., the elements
$v_{\ell-1}$ and $\sum\limits_{u\gtrdot v}\overline{u_{\ell}}$ are invertible.
Also, $v_{\ell}$ is invertible (by Lemma \ref{lem.slacks.wd} \textbf{(a)}).

Solving the equality (\ref{pf.lem.slacks.wd.short.c.1}) for the first factor
on its right hand side, we obtain
\begin{align*}
\sum\limits_{u\lessdot v}u_{\ell-1} = v_{\ell}\cdot\left(  \sum
\limits_{u\gtrdot v}\overline{u_{\ell}}\right)  \cdot v_{\ell-1} .
\end{align*}
The right hand side of this equality is a product of three invertible
elements; thus, both sides are invertible. Therefore, the element
$\overline{\sum\limits_{u\lessdot v}u_{\ell-1}}$ is well-defined, hence
invertible (since an inverse is always invertible).

Finally, $\downslack_{\ell-1}^{v}$ is defined to be the product $v_{\ell
-1}\cdot\overline{\sum\limits_{u\lessdot v}u_{\ell-1}}$, and thus is
well-defined and invertible because both of its factors are. \medskip

\textbf{(d)} If $v\in\left\{  0,1\right\}  $, then this follows from
(\ref{eq.slack.1}). Otherwise, $v\in P$.

Lemma \ref{lem.slacks.wd} \textbf{(a)} shows that $v_{\ell}$ is invertible;
hence, $\overline{v_{\ell}}$ is well-defined, and invertible. Also, in the
proof of Lemma \ref{lem.slacks.wd} \textbf{(c)}, we have shown that
$\overline{\sum\limits_{u\gtrdot v}\overline{u_{\ell}}}$ is well-defined, so
it too is invertible. Finally, $\upslack_{\ell}^{v}$ is defined to be the
product $\overline{\sum\limits_{u\gtrdot v}\overline{u_{\ell}}}\cdot
\overline{v_{\ell}}$, and thus is also well-defined and invertible because
both of its factors are.
\end{proof}
\end{vershort}

\begin{verlong}

\begin{proof}
We know that $a$ is invertible. In other words, $f\left(  0\right)  $ is
invertible (since $a=f\left(  0\right)  $). Also, $\ell-1\in\mathbb{N}$ (since
$\ell\geq1$) and $R^{\ell-1}f\neq\undf$ (since $R\left(  R^{\ell-1}f\right)
=R^{\ell}f\neq\undf=R\left(  \undf\right)  $). Hence, Corollary
\ref{cor.R.implicit.01} (applied to $\ell-1$ instead of $\ell$) yields
$\left(  R^{\ell-1}f\right)  \left(  0\right)  =f\left(  0\right)  $ and
$\left(  R^{\ell-1}f\right)  \left(  1\right)  =f\left(  1\right)  $. Thus,
$\left(  R^{\ell-1}f\right)  \left(  0\right)  =f\left(  0\right)  =a$, so
that $\left(  R^{\ell-1}f\right)  \left(  0\right)  $ is invertible (since $a$
is invertible). \medskip

\textbf{(a)} Clearly, $v_{\ell}=\left(  R^{\ell}f\right)  \left(  v\right)  $
is well-defined (since $R^{\ell}f\neq\undf$). We have $R\left(  R^{\ell
-1}f\right)  =R^{\ell}f\neq\undf$. Hence, Lemma \ref{lem.R.inv-2} (applied to
$R^{\ell-1}f$ instead of $f$) yields that $\left(  R\left(  R^{\ell
-1}f\right)  \right)  \left(  v\right)  $ is invertible (since $\left(
R^{\ell-1}f\right)  \left(  0\right)  $ is invertible). In other words,
$v_{\ell}$ is invertible (since $v_{\ell}=\underbrace{\left(  R^{\ell
}f\right)  }_{=R\left(  R^{\ell-1}f\right)  }\left(  v\right)  =\left(
R\left(  R^{\ell-1}f\right)  \right)  \left(  v\right)  $). This proves Lemma
\ref{lem.slacks.wd} \textbf{(a)}. \medskip

\textbf{(b)} Clearly, $v_{\ell-1}=\left(  R^{\ell-1}f\right)  \left(
v\right)  $ is well-defined (since $R^{\ell-1}f\neq\undf$). It remains to
prove that $v_{\ell-1}$ is invertible. If $v$ is $0$ or $1$, then this follows
easily from part \textbf{(a)}\footnote{\textit{Proof.} Assume that $v$ is $0$
or $1$. We WLOG assume that $v=0$ (since the case $v=1$ is analogous).
\par
We must show that $v_{\ell-1}$ is invertible. However, (\ref{eq.0l.a}) yields
$0_{\ell}=a$. Similarly, $0_{\ell-1}=a$. Comparing these two equalities, we
obtain $0_{\ell}=0_{\ell-1}$. In other words, $v_{\ell}=v_{\ell-1}$ (since
$v=0$). But Lemma \ref{lem.slacks.wd} \textbf{(a)} shows that $v_{\ell}$ is
invertible. Thus, $v_{\ell-1}$ is invertible (since $v_{\ell}=v_{\ell-1}$).
Qed.}. Thus, for the rest of this proof, we WLOG assume that $v$ is neither
$0$ nor $1$. Hence, $v\in P$.

We have $R\left(  R^{\ell-1}f\right)  =R^{\ell}f\neq\undf$. Hence, Lemma
\ref{lem.R.inv} (applied to $R^{\ell-1}f$ instead of $f$) yields that $\left(
R^{\ell-1}f\right)  \left(  v\right)  $ is invertible. In other words,
$v_{\ell-1}$ is invertible (since $v_{\ell-1}=\left(  R^{\ell-1}f\right)
\left(  v\right)  $). This proves Lemma \ref{lem.slacks.wd} \textbf{(b)}.
\medskip

\textbf{(c)} If $v\in\left\{  0,1\right\}  $, then this follows from
(\ref{eq.slack.1}). Thus, we WLOG assume that $v\notin\left\{  0,1\right\}  $.
Hence, $v\in P$.

Thus, applying (\ref{eq.vl+1=}) to $\ell-1$ instead of $\ell$, we obtain%
\begin{equation}
v_{\ell}=\left(  \sum\limits_{u\lessdot v}u_{\ell-1}\right)  \cdot
\overline{v_{\ell-1}}\cdot\overline{\sum\limits_{u\gtrdot v}\overline{u_{\ell
}}} \label{pf.lem.slacks.wd.c.1}%
\end{equation}
(since $R^{\ell}f \neq\undf$). This equality shows that $\overline{v_{\ell-1}%
}$ and $\overline{\sum\limits_{u\gtrdot v}\overline{u_{\ell}}}$ are
well-defined. In other words, the elements $v_{\ell-1}$ and $\sum
\limits_{u\gtrdot v}\overline{u_{\ell}}$ are invertible. Also, $v_{\ell}$ is
invertible (by Lemma \ref{lem.slacks.wd} \textbf{(a)}).

Multiplying the equality (\ref{pf.lem.slacks.wd.c.1}) by $\left(
\sum\limits_{u\gtrdot v}\overline{u_{\ell}}\right)  \cdot v_{\ell-1}$ on the
right, we obtain%
\begin{align*}
v_{\ell}\cdot\left(  \sum\limits_{u\gtrdot v}\overline{u_{\ell}}\right)  \cdot
v_{\ell-1}  &  =\left(  \sum\limits_{u\lessdot v}u_{\ell-1}\right)
\cdot\overline{v_{\ell-1}}\cdot\underbrace{\overline{\sum\limits_{u\gtrdot
v}\overline{u_{\ell}}}\cdot\left(  \sum\limits_{u\gtrdot v}\overline{u_{\ell}%
}\right)  }_{=1}\cdot v_{\ell-1}\\
&  =\left(  \sum\limits_{u\lessdot v}u_{\ell-1}\right)  \cdot
\underbrace{\overline{v_{\ell-1}}\cdot v_{\ell-1}}_{=1}=\sum\limits_{u\lessdot
v}u_{\ell-1}.
\end{align*}
The left hand side of this equality is a product of three invertible elements
(since $v_{\ell}$ and $\sum\limits_{u\gtrdot v}\overline{u_{\ell}}$ and
$v_{\ell-1}$ are invertible), and thus must itself be invertible. Hence, the
right hand side is invertible. In other words, $\sum\limits_{u\lessdot
v}u_{\ell-1}$ is invertible. Thus, the element $\overline{\sum
\limits_{u\lessdot v}u_{\ell-1}}$ is well-defined. This element is furthermore
invertible (since an inverse is always invertible).

Now, the definition of $\downslack_{\ell-1}^{v}$ yields $\downslack_{\ell
-1}^{v}=v_{\ell-1}\cdot\overline{\sum\limits_{u\lessdot v}u_{\ell-1}}$. This
shows that $\downslack_{\ell-1}^{v}$ is well-defined (since $v_{\ell-1}$ and
$\overline{\sum\limits_{u\lessdot v}u_{\ell-1}}$ are well-defined) and
invertible (since $v_{\ell-1}$ and $\overline{\sum\limits_{u\lessdot v}%
u_{\ell-1}}$ are invertible). This proves Lemma \ref{lem.slacks.wd}
\textbf{(c)}. \medskip

\textbf{(d)} If $v\in\left\{  0,1\right\}  $, then this follows from
(\ref{eq.slack.1}). Thus, we WLOG assume that $v\notin\left\{  0,1\right\}  $.
Hence, $v\in P$.

Lemma \ref{lem.slacks.wd} \textbf{(a)} shows that $v_{\ell}$ is invertible.
Hence, $\overline{v_{\ell}}$ is well-defined, and invertible as well (since an
inverse is always invertible). Also, in the proof of Lemma \ref{lem.slacks.wd}
\textbf{(c)}, we have shown that $\overline{\sum\limits_{u\gtrdot v}%
\overline{u_{\ell}}}$ is well-defined. Thus, $\overline{\sum\limits_{u\gtrdot
v}\overline{u_{\ell}}}$ is invertible (since an inverse is always invertible).
Now, the definition of $\upslack_{\ell}^{v}$ yields $\upslack_{\ell}%
^{v}=\overline{\sum\limits_{u\gtrdot v}\overline{u_{\ell}}}\cdot
\overline{v_{\ell}}$. Hence, $\upslack_{\ell}^{v}$ is well-defined (since
$\overline{\sum\limits_{u\gtrdot v}\overline{u_{\ell}}}$ and $\overline
{v_{\ell}}$ are well-defined) and invertible (since $\overline{\sum
\limits_{u\gtrdot v}\overline{u_{\ell}}}$ and $\overline{v_{\ell}}$ are
invertible). This proves Lemma \ref{lem.slacks.wd} \textbf{(d)}.
\end{proof}
\end{verlong}

Next we show some simple recursions for $\downslack_{\ell}^{s\rightarrow t}$
and $\upslack_{\ell}^{s\rightarrow t}$:

\begin{proposition}
\label{prop.slacks.rec}Let $s$ and $t$ be two distinct elements of
$\widehat{P}$, and fix $\ell\in\mathbb{N}$. Then
\begin{align}
\downslack_{\ell}^{s\rightarrow t}  &  =\downslack_{\ell}^{s}\sum
_{\substack{u\in\widehat{P};\\s\gtrdot u}}\downslack_{\ell}^{u\rightarrow
t}\label{eq.prop.slacks.rec.down1}\\
&  =\sum_{\substack{u\in\widehat{P};\\u\gtrdot t}}\downslack_{\ell
}^{s\rightarrow u}\downslack_{\ell}^{t} \label{eq.prop.slacks.rec.down2}%
\end{align}
and
\begin{align}
\upslack_{\ell}^{s\rightarrow t}  &  =\upslack_{\ell}^{s}\sum_{\substack{u\in
\widehat{P};\\s\gtrdot u}}\upslack_{\ell}^{u\rightarrow t}%
\label{eq.prop.slacks.rec.up1}\\
&  =\sum_{\substack{u\in\widehat{P};\\u\gtrdot t}}\upslack_{\ell
}^{s\rightarrow u}\upslack_{\ell}^{t}. \label{eq.prop.slacks.rec.up2}%
\end{align}
Here, we assume that all the terms in the respective equalities are well-defined.
\end{proposition}

\begin{vershort}

\begin{proof}
Since $s\neq t$, every path from $s$ to $t$ must contain an element covered by
$s$ as its second vertex.

Fix an element $u\in\widehat{P}$ satisfying $s\gtrdot u$. If $\left(
v_{0}\gtrdot v_{1}\gtrdot\cdots\gtrdot v_{k}\right)  $ is a path from $s$ to
$t$ satisfying $v_{1}=u$, then $\left(  v_{1}\gtrdot v_{2}\gtrdot\cdots\gtrdot
v_{k}\right)  $ is a path from $u$ to $t$. Hence, we have found a map
\begin{align*}
&  \text{from }\left\{  \text{paths }\left(  v_{0}\gtrdot v_{1}\gtrdot
\cdots\gtrdot v_{k}\right)  \text{ from }s\text{ to }t\text{ satisfying }%
v_{1}=u\right\} \\
&  \text{to }\left\{  \text{paths from }u\text{ to }t\right\}
\end{align*}
that sends each path $\left(  v_{0}\gtrdot v_{1}\gtrdot\cdots\gtrdot
v_{k}\right)  $ to $\left(  v_{1}\gtrdot v_{2}\gtrdot\cdots\gtrdot
v_{k}\right)  $. This map is a bijection (since any path from $u$ to $t$ can
be uniquely extended to a path from $s$ to $t$ by inserting the vertex $s$ at
the front). We can use this bijection to substitute $\left(  v_{1}\gtrdot
v_{2}\gtrdot\cdots\gtrdot v_{k}\right)  $ for $\mathbf{p}$ in a sum that
ranges over all paths $\mathbf{p}$ from $u$ to $t$. In particular,%
\begin{align}
&  \sum_{\mathbf{p}\text{ is a path from }u\text{ to }t}\downslack_{\ell}%
^{s}\downslack_{\ell}^{\mathbf{p}}\nonumber\\
&  =\sum_{\substack{\left(  v_{0}\gtrdot v_{1}\gtrdot\cdots\gtrdot
v_{k}\right)  \text{ is a path from }s\text{ to }t;\\v_{1}=u}%
}\underbrace{\downslack_{\ell}^{s}}_{\substack{=\downslack_{\ell}^{v_{0}%
}\\\text{(since }s=v_{0}\text{)}}}\underbrace{\downslack_{\ell}^{\left(
v_{1}\gtrdot v_{2}\gtrdot\cdots\gtrdot v_{k}\right)  }}%
_{\substack{=\downslack_{\ell}^{v_{1}}\downslack_{\ell}^{v_{2}}\cdots
\downslack_{\ell}^{v_{k}}\\\text{(by the definition of }\downslack_{\ell
}^{\left(  v_{1}\gtrdot v_{2}\gtrdot\cdots\gtrdot v_{k}\right)  }\text{)}%
}}\nonumber\\
&  =\sum_{\substack{\left(  v_{0}\gtrdot v_{1}\gtrdot\cdots\gtrdot
v_{k}\right)  \text{ is a path from }s\text{ to }t;\\v_{1}=u}%
}\underbrace{\downslack_{\ell}^{v_{0}}\downslack_{\ell}^{v_{1}}%
\downslack_{\ell}^{v_{2}}\cdots\downslack_{\ell}^{v_{k}}}%
_{\substack{=\downslack_{\ell}^{\left(  v_{0}\gtrdot v_{1}\gtrdot\cdots\gtrdot
v_{k}\right)  }\\\text{(by the definition of }\downslack_{\ell}^{\left(
v_{0}\gtrdot v_{1}\gtrdot\cdots\gtrdot v_{k}\right)  }\text{)}}}\nonumber\\
&  =\sum_{\substack{\left(  v_{0}\gtrdot v_{1}\gtrdot\cdots\gtrdot
v_{k}\right)  \text{ is a path from }s\text{ to }t;\\v_{1}=u}}\downslack_{\ell
}^{\left(  v_{0}\gtrdot v_{1}\gtrdot\cdots\gtrdot v_{k}\right)  }.
\label{pf.prop.slacks.rec.short.1}%
\end{align}

Now, forget that we fixed $u$. We thus have proved
(\ref{pf.prop.slacks.rec.short.1}) for each $u\in\widehat{P}$ satisfying
$s\gtrdot u$.

The definition of $\downslack_{\ell}^{s\rightarrow t}$ yields%
\begin{align*}
\downslack_{\ell}^{s\rightarrow t}  &  =\underbrace{\sum_{\left(  v_{0}\gtrdot
v_{1}\gtrdot\cdots\gtrdot v_{k}\right)  \text{ is a path from }s\text{ to }t}%
}_{\substack{=\sum_{\substack{u\in\widehat{P};\\s\gtrdot u}}\ \ \sum
_{\substack{\left(  v_{0}\gtrdot v_{1}\gtrdot\cdots\gtrdot v_{k}\right)
\text{ is a path from }s\text{ to }t;\\v_{1}=u}}\\\text{(because any path
}\left(  v_{0}\gtrdot v_{1}\gtrdot\cdots\gtrdot v_{k}\right)  \text{ from
}s\text{ to }t\text{ has a well-defined}\\\text{second vertex }v_{1}\text{,
and this second vertex }v_{1}\text{ satisfies }s\gtrdot v_{1}\text{)}%
}}\downslack_{\ell}^{\left(  v_{0}\gtrdot v_{1}\gtrdot\cdots\gtrdot
v_{k}\right)  }\\
&  =\sum_{\substack{u\in\widehat{P};\\s\gtrdot u}}\ \ \underbrace{\sum
_{\substack{\left(  v_{0}\gtrdot v_{1}\gtrdot\cdots\gtrdot v_{k}\right)
\text{ is a path from }s\text{ to }t;\\v_{1}=u}}\downslack_{\ell}^{\left(
v_{0}\gtrdot v_{1}\gtrdot\cdots\gtrdot v_{k}\right)  }}_{\substack{=\sum
_{\mathbf{p}\text{ is a path from }u\text{ to }t}\downslack_{\ell}%
^{s}\downslack_{\ell}^{\mathbf{p}}\\\text{(by
(\ref{pf.prop.slacks.rec.short.1}))}}}\\
&  =\sum_{\substack{u\in\widehat{P};\\s\gtrdot u}}\ \ \sum_{\mathbf{p}\text{
is a path from }u\text{ to }t}\downslack_{\ell}^{s}\downslack_{\ell
}^{\mathbf{p}}=\downslack_{\ell}^{s}\sum_{\substack{u\in\widehat{P};\\s\gtrdot
u}}\ \ \underbrace{\sum_{\mathbf{p}\text{ is a path from }u\text{ to }%
t}\downslack_{\ell}^{\mathbf{p}}}_{\substack{=\downslack_{\ell}^{u\rightarrow
t}\\\text{(by the definition of }\downslack_{\ell}^{u\rightarrow t}\text{)}%
}}\\
&  =\downslack_{\ell}^{s}\sum_{\substack{u\in\widehat{P};\\s\gtrdot
u}}\downslack_{\ell}^{u\rightarrow t}.
\end{align*}
This proves (\ref{eq.prop.slacks.rec.down1}). The same argument (but with each
$\downslack$ symbol replaced by an $\upslack$ symbol) proves
(\ref{eq.prop.slacks.rec.up1}). Moreover, a similar argument (but now
classifying paths from $s$ to $t$ according to their second-to-last vertex
instead of their second vertex) establishes (\ref{eq.prop.slacks.rec.down2})
and (\ref{eq.prop.slacks.rec.up2}). Thus, Proposition \ref{prop.slacks.rec} is proven.
\end{proof}
\end{vershort}

\begin{verlong}

\begin{proof}
Any path $\left(  v_{0}\gtrdot v_{1}\gtrdot\cdots\gtrdot v_{k}\right)  $ from
$s$ to $t$ has at least two vertices (since $s\neq t$), and thus has a
well-defined second vertex $v_{1}$. This second vertex $v_{1}$ satisfies
$s\gtrdot v_{1}$ (since $s=v_{0}\gtrdot v_{1}$). In other words, this second
vertex $v_{1}$ is an element $u\in\widehat{P}$ satisfying $s\gtrdot u$.

Fix an element $u\in\widehat{P}$ satisfying $s\gtrdot u$. If $\left(
v_{0}\gtrdot v_{1}\gtrdot\cdots\gtrdot v_{k}\right)  $ is a path from $s$ to
$t$ satisfying $v_{1}=u$, then $\left(  v_{1}\gtrdot v_{2}\gtrdot\cdots\gtrdot
v_{k}\right)  $ is a path from $u$ to $t$ (since $v_{1}=u$ and $v_{k}=t$).
Hence, we have found a map
\begin{align*}
&  \text{from }\left\{  \text{paths }\left(  v_{0}\gtrdot v_{1}\gtrdot
\cdots\gtrdot v_{k}\right)  \text{ from }s\text{ to }t\text{ satisfying }%
v_{1}=u\right\} \\
&  \text{to }\left\{  \text{paths from }u\text{ to }t\right\}
\end{align*}
that sends each path $\left(  v_{0}\gtrdot v_{1}\gtrdot\cdots\gtrdot
v_{k}\right)  $ to $\left(  v_{1}\gtrdot v_{2}\gtrdot\cdots\gtrdot
v_{k}\right)  $. This map is injective\footnote{because a path $\left(
v_{0}\gtrdot v_{1}\gtrdot\cdots\gtrdot v_{k}\right)  $ from $s$ to $t$ can be
reconstructed from its image $\left(  v_{1}\gtrdot v_{2}\gtrdot\cdots\gtrdot
v_{k}\right)  $ under this map (since its first vertex $v_{0}$ is forced to be
$s$)} and surjective\footnote{Indeed, if $\mathbf{p}=\left(  v_{1}\gtrdot
v_{2}\gtrdot\cdots\gtrdot v_{k}\right)  $ is a path from $u$ to $t$, then
$\left(  s\gtrdot v_{1}\gtrdot v_{2}\gtrdot\cdots\gtrdot v_{k}\right)  $ is a
path from $s$ to $t$ satisfying $v_{1}=u$ (since $s\gtrdot u=v_{1}$), and it
is clear that our map sends the latter path to $\mathbf{p}$.}; hence, it is a
bijection. We can use this bijection to substitute $\left(  v_{1}\gtrdot
v_{2}\gtrdot\cdots\gtrdot v_{k}\right)  $ for $\mathbf{p}$ in the sum
$\sum_{\mathbf{p}\text{ is a path from }u\text{ to }t}\downslack_{\ell}%
^{s}\downslack_{\ell}^{\mathbf{p}}$. We thus obtain
\begin{align}
&  \sum_{\mathbf{p}\text{ is a path from }u\text{ to }t}\downslack_{\ell}%
^{s}\downslack_{\ell}^{\mathbf{p}}\nonumber\\
&  =\sum_{\substack{\left(  v_{0}\gtrdot v_{1}\gtrdot\cdots\gtrdot
v_{k}\right)  \text{ is a path from }s\text{ to }t;\\v_{1}=u}%
}\underbrace{\downslack_{\ell}^{s}}_{\substack{=\downslack_{\ell}^{v_{0}%
}\\\text{(since }s=v_{0}\text{)}}}\underbrace{\downslack_{\ell}^{\left(
v_{1}\gtrdot v_{2}\gtrdot\cdots\gtrdot v_{k}\right)  }}%
_{\substack{=\downslack_{\ell}^{v_{1}}\downslack_{\ell}^{v_{2}}\cdots
\downslack_{\ell}^{v_{k}}\\\text{(by the definition of }\downslack_{\ell
}^{\left(  v_{1}\gtrdot v_{2}\gtrdot\cdots\gtrdot v_{k}\right)  }\text{)}%
}}\nonumber\\
&  =\sum_{\substack{\left(  v_{0}\gtrdot v_{1}\gtrdot\cdots\gtrdot
v_{k}\right)  \text{ is a path from }s\text{ to }t;\\v_{1}=u}%
}\underbrace{\downslack_{\ell}^{v_{0}}\downslack_{\ell}^{v_{1}}%
\downslack_{\ell}^{v_{2}}\cdots\downslack_{\ell}^{v_{k}}}%
_{\substack{=\downslack_{\ell}^{v_{0}}\downslack_{\ell}^{v_{1}}\cdots
\downslack_{\ell}^{v_{k}}\\=\downslack_{\ell}^{\left(  v_{0}\gtrdot
v_{1}\gtrdot\cdots\gtrdot v_{k}\right)  }\\\text{(by the definition of
}\downslack_{\ell}^{\left(  v_{0}\gtrdot v_{1}\gtrdot\cdots\gtrdot
v_{k}\right)  }\text{)}}}\nonumber\\
&  =\sum_{\substack{\left(  v_{0}\gtrdot v_{1}\gtrdot\cdots\gtrdot
v_{k}\right)  \text{ is a path from }s\text{ to }t;\\v_{1}=u}}\downslack_{\ell
}^{\left(  v_{0}\gtrdot v_{1}\gtrdot\cdots\gtrdot v_{k}\right)  }.
\label{pf.prop.slacks.rec.1}%
\end{align}

Now, forget that we fixed $u$. We thus have proved (\ref{pf.prop.slacks.rec.1}%
) for each $u\in\widehat{P}$ satisfying $s\gtrdot u$.

The definition of $\downslack_{\ell}^{s\rightarrow t}$ yields%
\begin{align*}
\downslack_{\ell}^{s\rightarrow t}  &  =\sum_{\mathbf{p}\text{ is a path from
}s\text{ to }t}\downslack_{\ell}^{\mathbf{p}}\\
&  =\underbrace{\sum_{\left(  v_{0}\gtrdot v_{1}\gtrdot\cdots\gtrdot
v_{k}\right)  \text{ is a path from }s\text{ to }t}}_{\substack{=\sum
_{\substack{u\in\widehat{P};\\s\gtrdot u}}\ \ \sum_{\substack{\left(
v_{0}\gtrdot v_{1}\gtrdot\cdots\gtrdot v_{k}\right)  \text{ is a path from
}s\text{ to }t;\\v_{1}=u}}\\\text{(because any path }\left(  v_{0}\gtrdot
v_{1}\gtrdot\cdots\gtrdot v_{k}\right)  \text{ from }s\text{ to }t\text{ has a
well-defined}\\\text{second vertex }v_{1}\text{, and this second vertex }%
v_{1}\text{ satisfies }s\gtrdot v_{1}\text{)}}}\downslack_{\ell}^{\left(
v_{0}\gtrdot v_{1}\gtrdot\cdots\gtrdot v_{k}\right)  }\\
&  \ \ \ \ \ \ \ \ \ \ \ \ \ \ \ \ \ \ \ \ \left(  \text{here we have renamed
the index }\mathbf{p}\text{ as }\left(  v_{0}\gtrdot v_{1}\gtrdot\cdots\gtrdot
v_{k}\right)  \right) \\
&  =\sum_{\substack{u\in\widehat{P};\\s\gtrdot u}}\ \ \underbrace{\sum
_{\substack{\left(  v_{0}\gtrdot v_{1}\gtrdot\cdots\gtrdot v_{k}\right)
\text{ is a path from }s\text{ to }t;\\v_{1}=u}}\downslack_{\ell}^{\left(
v_{0}\gtrdot v_{1}\gtrdot\cdots\gtrdot v_{k}\right)  }}_{\substack{=\sum
_{\mathbf{p}\text{ is a path from }u\text{ to }t}\downslack_{\ell}%
^{s}\downslack_{\ell}^{\mathbf{p}}\\\text{(by (\ref{pf.prop.slacks.rec.1}))}%
}}\\
&  =\sum_{\substack{u\in\widehat{P};\\s\gtrdot u}}\ \ \sum_{\mathbf{p}\text{
is a path from }u\text{ to }t}\downslack_{\ell}^{s}\downslack_{\ell
}^{\mathbf{p}}=\downslack_{\ell}^{s}\sum_{\substack{u\in\widehat{P};\\s\gtrdot
u}}\ \ \underbrace{\sum_{\mathbf{p}\text{ is a path from }u\text{ to }%
t}\downslack_{\ell}^{\mathbf{p}}}_{\substack{=\downslack_{\ell}^{u\rightarrow
t}\\\text{(by the definition of }\downslack_{\ell}^{u\rightarrow t}\text{)}%
}}\\
&  =\downslack_{\ell}^{s}\sum_{\substack{u\in\widehat{P};\\s\gtrdot
u}}\downslack_{\ell}^{u\rightarrow t}.
\end{align*}
This proves (\ref{eq.prop.slacks.rec.down1}). The same argument (but with each
$\downslack$ symbol replaced by an $\upslack$ symbol) proves
(\ref{eq.prop.slacks.rec.up1}).

Let us now prove (\ref{eq.prop.slacks.rec.down2}). This proof will be very
similar to the above proof of (\ref{eq.prop.slacks.rec.down1}), but we will
now classify paths from $s$ to $t$ according to their second-to-last vertex
instead of their second vertex. Here are the details:

Any path $\left(  v_{0}\gtrdot v_{1}\gtrdot\cdots\gtrdot v_{k}\right)  $ from
$s$ to $t$ has at least two vertices (since $s\neq t$), and thus has a
well-defined second-to-last vertex $v_{k-1}$. This second-to-last vertex
$v_{k-1}$ satisfies $v_{k-1}\gtrdot t$ (since $v_{k-1}\gtrdot v_{k}=t$). In
other words, this second-to-last vertex $v_{k-1}$ is an element $u\in
\widehat{P}$ satisfying $u\gtrdot t$.

Now, fix an element $u\in\widehat{P}$ satisfying $u\gtrdot t$. If $\left(
v_{0}\gtrdot v_{1}\gtrdot\cdots\gtrdot v_{k}\right)  $ is a path from $s$ to
$t$ satisfying $v_{k-1}=u$, then $\left(  v_{0}\gtrdot v_{1}\gtrdot
\cdots\gtrdot v_{k-1}\right)  $ is a path from $s$ to $u$ (since $v_{0}=s$ and
$v_{k-1}=u$). Hence, we have found a map
\begin{align*}
&  \text{from }\left\{  \text{paths }\left(  v_{0}\gtrdot v_{1}\gtrdot
\cdots\gtrdot v_{k}\right)  \text{ from }s\text{ to }t\text{ satisfying
}v_{k-1}=u\right\} \\
&  \text{to }\left\{  \text{paths from }s\text{ to }u\right\}
\end{align*}
that sends each path $\left(  v_{0}\gtrdot v_{1}\gtrdot\cdots\gtrdot
v_{k}\right)  $ to $\left(  v_{0}\gtrdot v_{1}\gtrdot\cdots\gtrdot
v_{k-1}\right)  $. This map is injective\footnote{because a path $\left(
v_{0}\gtrdot v_{1}\gtrdot\cdots\gtrdot v_{k}\right)  $ from $s$ to $t$ can be
reconstructed from its image $\left(  v_{0}\gtrdot v_{1}\gtrdot\cdots\gtrdot
v_{k-1}\right)  $ under this map (since its last vertex $v_{k}$ is forced to
be $t$)} and surjective\footnote{Indeed, if $\mathbf{p}=\left(  v_{0}\gtrdot
v_{1}\gtrdot\cdots\gtrdot v_{\ell}\right)  $ is a path from $s$ to $u$, then
$\left(  v_{0}\gtrdot v_{1}\gtrdot\cdots\gtrdot v_{\ell}\gtrdot t\right)  $ is
a path from $s$ to $t$ satisfying $v_{\ell}=u$ (since $v_{\ell}=u\gtrdot t$),
and it is clear that our map sends the latter path to $\mathbf{p}$.}; hence,
it is a bijection. We can use this bijection to substitute $\left(
v_{0}\gtrdot v_{1}\gtrdot\cdots\gtrdot v_{k-1}\right)  $ for $\mathbf{p}$ in
the sum $\sum_{\mathbf{p}\text{ is a path from }s\text{ to }u}\downslack_{\ell
}^{\mathbf{p}}\downslack_{\ell}^{t}$. We thus obtain
\begin{align}
&  \sum_{\mathbf{p}\text{ is a path from }s\text{ to }u}\downslack_{\ell
}^{\mathbf{p}}\downslack_{\ell}^{t}\nonumber\\
&  =\sum_{\substack{\left(  v_{0}\gtrdot v_{1}\gtrdot\cdots\gtrdot
v_{k}\right)  \text{ is a path from }s\text{ to }t;\\v_{k-1}=u}%
}\underbrace{\downslack_{\ell}^{\left(  v_{0}\gtrdot v_{1}\gtrdot\cdots\gtrdot
v_{k-1}\right)  }}_{\substack{=\downslack_{\ell}^{v_{0}}\downslack_{\ell
}^{v_{1}}\cdots\downslack_{\ell}^{v_{k-1}}\\\text{(by the definition of
}\downslack_{\ell}^{\left(  v_{0}\gtrdot v_{1}\gtrdot\cdots\gtrdot
v_{k-1}\right)  }\text{)}}}\underbrace{\downslack_{\ell}^{t}}%
_{\substack{=\downslack_{\ell}^{v_{k}}\\\text{(since }t=v_{k}\text{)}%
}}\nonumber\\
&  =\sum_{\substack{\left(  v_{0}\gtrdot v_{1}\gtrdot\cdots\gtrdot
v_{k}\right)  \text{ is a path from }s\text{ to }t;\\v_{k-1}=u}%
}\underbrace{\downslack_{\ell}^{v_{0}}\downslack_{\ell}^{v_{1}}\cdots
\downslack_{\ell}^{v_{k-1}}\downslack_{\ell}^{v_{k}}}%
_{\substack{=\downslack_{\ell}^{v_{0}}\downslack_{\ell}^{v_{1}}\cdots
\downslack_{\ell}^{v_{k}}\\=\downslack_{\ell}^{\left(  v_{0}\gtrdot
v_{1}\gtrdot\cdots\gtrdot v_{k}\right)  }\\\text{(by the definition of
}\downslack_{\ell}^{\left(  v_{0}\gtrdot v_{1}\gtrdot\cdots\gtrdot
v_{k}\right)  }\text{)}}}\nonumber\\
&  =\sum_{\substack{\left(  v_{0}\gtrdot v_{1}\gtrdot\cdots\gtrdot
v_{k}\right)  \text{ is a path from }s\text{ to }t;\\v_{k-1}=u}%
}\downslack_{\ell}^{\left(  v_{0}\gtrdot v_{1}\gtrdot\cdots\gtrdot
v_{k}\right)  }. \label{pf.prop.slacks.rec.-1}%
\end{align}

Now, forget that we fixed $u$. We thus have proved
(\ref{pf.prop.slacks.rec.-1}) for each $u\in\widehat{P}$ satisfying $u\gtrdot
t$.

The definition of $\downslack_{\ell}^{s\rightarrow t}$ yields%
\begin{align*}
\downslack_{\ell}^{s\rightarrow t}  &  =\sum_{\mathbf{p}\text{ is a path from
}s\text{ to }t}\downslack_{\ell}^{\mathbf{p}}\\
&  =\underbrace{\sum_{\left(  v_{0}\gtrdot v_{1}\gtrdot\cdots\gtrdot
v_{k}\right)  \text{ is a path from }s\text{ to }t}}_{\substack{=\sum
_{\substack{u\in\widehat{P};\\u\gtrdot t}}\ \ \sum_{\substack{\left(
v_{0}\gtrdot v_{1}\gtrdot\cdots\gtrdot v_{k}\right)  \text{ is a path from
}s\text{ to }t;\\v_{k-1}=u}}\\\text{(because any path }\left(  v_{0}\gtrdot
v_{1}\gtrdot\cdots\gtrdot v_{k}\right)  \text{ from }s\text{ to }t\text{ has a
well-defined}\\\text{second-to-last vertex }v_{k-1}\text{, and this vertex
}v_{k-1}\text{ satisfies }v_{k-1}\gtrdot t\text{)}}}\downslack_{\ell}^{\left(
v_{0}\gtrdot v_{1}\gtrdot\cdots\gtrdot v_{k}\right)  }\\
&  \ \ \ \ \ \ \ \ \ \ \ \ \ \ \ \ \ \ \ \ \left(  \text{here we have renamed
the index }\mathbf{p}\text{ as }\left(  v_{0}\gtrdot v_{1}\gtrdot\cdots\gtrdot
v_{k}\right)  \right) \\
&  =\sum_{\substack{u\in\widehat{P};\\u\gtrdot t}}\ \ \underbrace{\sum
_{\substack{\left(  v_{0}\gtrdot v_{1}\gtrdot\cdots\gtrdot v_{k}\right)
\text{ is a path from }s\text{ to }t;\\v_{k-1}=u}}\downslack_{\ell}^{\left(
v_{0}\gtrdot v_{1}\gtrdot\cdots\gtrdot v_{k}\right)  }}_{\substack{=\sum
_{\mathbf{p}\text{ is a path from }s\text{ to }u}\downslack_{\ell}%
^{\mathbf{p}}\downslack_{\ell}^{t}\\\text{(by (\ref{pf.prop.slacks.rec.-1}))}%
}}\\
&  =\sum_{\substack{u\in\widehat{P};\\u\gtrdot t}}\ \ \underbrace{\sum
_{\mathbf{p}\text{ is a path from }s\text{ to }u}\downslack_{\ell}%
^{\mathbf{p}}}_{\substack{=\downslack_{\ell}^{s\rightarrow u}\\\text{(by the
definition of }\downslack_{\ell}^{s\rightarrow u}\text{)}}}\downslack_{\ell
}^{t}=\sum_{\substack{u\in\widehat{P};\\u\gtrdot t}}\downslack_{\ell
}^{s\rightarrow u}\downslack_{\ell}^{t}.
\end{align*}
This establishes (\ref{eq.prop.slacks.rec.down2}). The same argument (but with
each $\downslack$ symbol replaced by an $\upslack$ symbol) yields
(\ref{eq.prop.slacks.rec.up2}). Thus, Proposition \ref{prop.slacks.rec} is proven.
\end{proof}
\end{verlong}

The next proposition uses the products $\upslack_{\ell}^{v}$ and
$\downslack_{\ell-1}^{v}$ to rewrite the equality (\ref{eq.vl+1=}) (which is
essentially the definition of birational rowmotion) in a slick way:

\begin{proposition}
[Transition equation in $\downslack$-$\upslack$-form]\label{prop.rect.trans}%
Let $v\in\widehat{P}$ and $\ell\geq1$ be such that $R^{\ell}f\neq\undf$.
Assume that $a$ is invertible. Then,%
\[
\upslack_{\ell}^{v}=\downslack_{\ell-1}^{v}.
\]

\end{proposition}

\begin{vershort}

\begin{proof}
If $v$ is $0$ or $1$, then the equality $\upslack_{\ell}^{v}=\downslack_{\ell
-1}^{v}$ holds because both of its sides are $1$ (by (\ref{eq.slack.1})).
Thus, we assume WLOG that $v\in P$.

Lemma \ref{lem.slacks.wd} \textbf{(a)} yields that $v_{\ell}$ is well-defined
and invertible, while Lemma \ref{lem.slacks.wd} \textbf{(c,d)} yield that
$\upslack_{\ell}^{v}$ and $\downslack_{\ell-1}^{v}$ are well-defined. Since
$\downslack_{\ell-1}^{v}$ is defined as $v_{\ell-1}\cdot\overline
{\sum\limits_{u\lessdot v}u_{\ell-1}}$, this entails that $\sum
\limits_{u\lessdot v}u_{\ell-1}$ is invertible.

If $\alpha,\beta,\gamma,\delta$ are four invertible elements of $\mathbb{K}$
satisfying $\alpha=\beta\overline{\gamma}\delta$, then%
\begin{equation}
\delta\overline{\alpha}=\delta\underbrace{\overline{\beta\overline{\gamma
}\delta}}_{\substack{=\overline{\delta}\gamma\overline{\beta}\\\text{}}} =
\underbrace{\delta\overline{\delta}}_{=1}\gamma\overline{\beta}=\gamma
\overline{\beta}. \label{pf.prop.rect.trans.short.abcd}%
\end{equation}

Applying (\ref{eq.vl+1=}) to $\ell-1$ instead of $\ell$, we find%
\[
v_{\ell}=\left(  \sum\limits_{u\lessdot v}u_{\ell-1}\right)  \cdot
\overline{v_{\ell-1}}\cdot\overline{\sum\limits_{u\gtrdot v}\overline{u_{\ell
}}}.
\]
Thus, (\ref{pf.prop.rect.trans.short.abcd}) (applied to $\alpha=v_{\ell}$,
$\beta=\sum\limits_{u\lessdot v}u_{\ell-1}$, $\gamma=v_{\ell-1}$ and
$\delta=\overline{\sum\limits_{u\gtrdot v}\overline{u_{\ell}}}$) yields%
\[
\overline{\sum\limits_{u\gtrdot v}\overline{u_{\ell}}}\cdot\overline{v_{\ell}%
}=v_{\ell-1}\cdot\overline{\sum\limits_{u\lessdot v}u_{\ell-1}}.
\]
But the left hand side of this equality is $\upslack_{\ell}^{v}$ (by the
definition of $\upslack_{\ell}^{v}$), whereas the right hand side is
$\downslack_{\ell-1}^{v}$. Hence, this equality simplifies to $\upslack_{\ell
}^{v}=\downslack_{\ell-1}^{v}$. This proves Proposition \ref{prop.rect.trans}.
\end{proof}
\end{vershort}

\begin{verlong}

\begin{proof}
If $v$ is $0$ or $1$, then the equality $\upslack_{\ell}^{v}=\downslack_{\ell
-1}^{v}$ holds because both of its sides are $1$ (by (\ref{eq.slack.1})).
Thus, we WLOG assume that $v$ is neither $0$ nor $1$. Hence, $v\in P$. Thus,
$P \neq\varnothing$.

Lemma \ref{lem.slacks.wd} \textbf{(a)} yields that $v_{\ell}$ is well-defined
and invertible. Lemma \ref{lem.slacks.wd} \textbf{(c)} yields that
$\downslack_{\ell-1}^{v}$ is well-defined. Lemma \ref{lem.slacks.wd}
\textbf{(d)} yields that $\upslack_{\ell}^{v}$ is well-defined.

We have $\ell-1\in\mathbb{N}$ (since $\ell\geq1$) and $R^{\ell}f\neq\undf$.
Hence, (\ref{eq.vl+1=}) (applied to $\ell-1$ instead of $\ell$) yields%
\begin{equation}
v_{\ell}=\left(  \sum\limits_{u\lessdot v}u_{\ell-1}\right)  \cdot
\overline{v_{\ell-1}}\cdot\overline{\sum\limits_{u\gtrdot v}\overline{u_{\ell
}}}. \label{pf.prop.rect.trans.1}%
\end{equation}
As in the proof of Lemma \ref{lem.slacks.wd} \textbf{(c)}, we can see that
$\sum\limits_{u\lessdot v}u_{\ell-1}$ is invertible. Taking reciprocals on
both sides of (\ref{pf.prop.rect.trans.1}), we obtain%
\[
\overline{v_{\ell}}=\overline{\left(  \sum\limits_{u\lessdot v}u_{\ell
-1}\right)  \cdot\overline{v_{\ell-1}}\cdot\overline{\sum\limits_{u\gtrdot
v}\overline{u_{\ell}}}}=\left(  \sum\limits_{u\gtrdot v}\overline{u_{\ell}%
}\right)  \cdot v_{\ell-1}\cdot\overline{\sum\limits_{u\lessdot v}u_{\ell-1}}%
\]
(by Proposition~\ref{prop.inverses.ab} \textbf{(c)}). Multiplying this
equality by $\overline{\sum\limits_{u\gtrdot v}\overline{u_{\ell}}}$ on the
left, we obtain%
\[
\overline{\sum\limits_{u\gtrdot v}\overline{u_{\ell}}}\cdot\overline{v_{\ell}%
}=\underbrace{\overline{\sum\limits_{u\gtrdot v}\overline{u_{\ell}}}%
\cdot\left(  \sum\limits_{u\gtrdot v}\overline{u_{\ell}}\right)  }_{=1}\cdot
v_{\ell-1}\cdot\overline{\sum\limits_{u\lessdot v}u_{\ell-1}}=v_{\ell-1}%
\cdot\overline{\sum\limits_{u\lessdot v}u_{\ell-1}}.
\]
But the left hand side of this equality is $\upslack_{\ell}^{v}$ (by the
definition of $\upslack_{\ell}^{v}$), whereas the right hand side is
$\downslack_{\ell-1}^{v}$. Hence, this equality simplifies to $\upslack_{\ell
}^{v}=\downslack_{\ell-1}^{v}$. This proves Proposition \ref{prop.rect.trans}.
\end{proof}
\end{verlong}

As a consequence of Proposition \ref{prop.rect.trans}, we have:

\begin{corollary}
\label{cor.rect.trans-path}Let $\mathbf{p}$ be a path. Let $\ell\geq1$ be such
that $R^{\ell}f\neq\undf$. Assume that $a$ is invertible. Then,
\[
\upslack_{\ell}^{\mathbf{p}}=\downslack_{\ell-1}^{\mathbf{p}}.
\]

\end{corollary}

\begin{verlong}

\begin{proof}
Write the path $\mathbf{p}$ as $\mathbf{p}=\left(  v_{0}\gtrdot v_{1}%
\gtrdot\cdots\gtrdot v_{k}\right)  $. The definition of $\upslack_{\ell
}^{\mathbf{p}}$ thus yields%
\begin{align*}
\upslack_{\ell}^{\mathbf{p}}  &  =\underbrace{\upslack_{\ell}^{v_{0}}%
}_{\substack{=\downslack_{\ell-1}^{v_{0}}\\\text{(by Proposition
\ref{prop.rect.trans})}}}\ \ \underbrace{\upslack_{\ell}^{v_{1}}%
}_{\substack{=\downslack_{\ell-1}^{v_{1}}\\\text{(by Proposition
\ref{prop.rect.trans})}}}\cdots\underbrace{\upslack_{\ell}^{v_{k}}%
}_{\substack{=\downslack_{\ell-1}^{v_{k}}\\\text{(by Proposition
\ref{prop.rect.trans})}}}\\
&  =\downslack_{\ell-1}^{v_{0}}\downslack_{\ell-1}^{v_{1}}\cdots
\downslack_{\ell-1}^{v_{k}}.
\end{align*}
However, the definition of $\downslack_{\ell-1}^{\mathbf{p}}$ yields%
\[
\downslack_{\ell-1}^{\mathbf{p}}=\downslack_{\ell-1}^{v_{0}}\downslack_{\ell
-1}^{v_{1}}\cdots\downslack_{\ell-1}^{v_{k}}\ \ \ \ \ \ \ \ \ \ \left(
\text{since }\mathbf{p}=\left(  v_{0}\gtrdot v_{1}\gtrdot\cdots\gtrdot
v_{k}\right)  \right)  .
\]
Comparing these two equalities, we obtain $\upslack_{\ell}^{\mathbf{p}%
}=\downslack_{\ell-1}^{\mathbf{p}}$. This proves Corollary
\ref{cor.rect.trans-path}.
\end{proof}
\end{verlong}

\begin{corollary}
\label{cor.rect.trans-uv}Let $u,v\in\widehat{P}$. Let $\ell\in\mathbb{N}$ be
such that $\ell\geq1$ and $R^{\ell}f\neq\undf$. Assume that $a$ is invertible.
Then,%
\begin{equation}
\upslack_{\ell}^{u\rightarrow v}=\downslack_{\ell-1}^{u\rightarrow v}.
\label{eq.prop.rect.trans.uv}%
\end{equation}

\end{corollary}

\begin{verlong}

\begin{proof}
The definition of $\upslack_{\ell}^{u\rightarrow v}$ yields%
\[
\upslack_{\ell}^{u\rightarrow v}=\sum_{\mathbf{p}\text{ is a path from
}u\text{ to }v}\underbrace{\upslack_{\ell}^{\mathbf{p}}}%
_{\substack{=\downslack_{\ell-1}^{\mathbf{p}}\\\text{(by Corollary
\ref{cor.rect.trans-path})}}}=\sum_{\mathbf{p}\text{ is a path from }u\text{
to }v}\downslack_{\ell-1}^{\mathbf{p}}.
\]
On the other hand, the definition of $\downslack_{\ell-1}^{u\rightarrow v}$
yields%
\[
\downslack_{\ell-1}^{u\rightarrow v}=\sum_{\mathbf{p}\text{ is a path from
}u\text{ to }v}\downslack_{\ell-1}^{\mathbf{p}}.
\]
Comparing these two equalities, we obtain $\upslack_{\ell}^{u\rightarrow
v}=\downslack_{\ell-1}^{u\rightarrow v}$. This proves Corollary
\ref{cor.rect.trans-uv}.
\end{proof}
\end{verlong}

The next theorem gives ways to recover the labels $u_{\ell}=\left(  R^{\ell
}f\right)  \left(  u\right)  $ from some of the sums defined in
(\ref{eq.slack.down-uv}) and (\ref{eq.slack.up-uv}).\footnote{The condition
$\ell\geq1$ in Theorem \ref{thm.rect.path} \textbf{(a)} and \textbf{(c)} is
meant to ensure that $\upslack_{\ell}^{1\rightarrow u}$ and $\upslack_{\ell
}^{\left(  p,q\right)  \rightarrow u}$ are invertible. It can be replaced by
directly requiring the latter.}

\begin{theorem}
[path formulas for rectangle]\label{thm.rect.path}Let $\ell\in\mathbb{N}$.
Assume that $a$ is invertible. Then:

\begin{enumerate}
\item[\textbf{(a)}] If $R^{\ell}f\neq\undf$ and $\ell\geq1$, then each $u\in
P$ satisfies%
\[
u_{\ell}=\overline{\upslack_{\ell}^{1\rightarrow u}}\cdot b
\]
(and the inverse $\overline{\upslack_{\ell}^{1\rightarrow u}}$ is well-defined).

\item[\textbf{(b)}] If $R^{\ell+1}f\neq\undf$, then each $u\in P$ satisfies%
\[
u_{\ell}=\downslack_{\ell}^{u\rightarrow0}\cdot a.
\]

\item[\textbf{(c)}] If $R^{\ell}f\neq\undf$ and $\ell\geq1$, then each $u\in
P$ satisfies%
\[
u_{\ell}=\overline{\upslack_{\ell}^{\left(  p,q\right)  \rightarrow u}}\cdot
b
\]
(and the inverse $\overline{\upslack_{\ell}^{\left(  p,q\right)  \rightarrow
u}}$ is well-defined).

\item[\textbf{(d)}] If $R^{\ell+1}f\neq\undf$, then each $u\in P$ satisfies%
\[
u_{\ell}=\downslack_{\ell}^{u\rightarrow\left(  1,1\right)  }\cdot a.
\]

\end{enumerate}
\end{theorem}

\begin{proof}
[Proof of Theorem \ref{thm.rect.path}.]\textbf{(a)} Assume that $R^{\ell}%
f\neq\undf$ and $\ell\geq1$. Then, Lemma \ref{lem.slacks.wd} \textbf{(d)}
yields that the element $\upslack_{\ell}^{v}$ is well-defined and invertible
for each $v\in\widehat{P}$. Hence, the element $\upslack_{\ell}^{\mathbf{p}}$
is well-defined for each path $\mathbf{p}$. Therefore, the element
$\upslack_{\ell}^{1\rightarrow u}$ is well-defined for each $u\in P$.

Next, we will prove the equality
\begin{equation}
\upslack_{\ell}^{1\rightarrow u}=b\overline{u_{\ell}}%
\ \ \ \ \ \ \ \ \ \ \text{for each }u\in P. \label{pf.thm.rect.path.a.1}%
\end{equation}
(The $\overline{u_{\ell}}$ on the right hand side here is well-defined, since
Lemma \ref{lem.slacks.wd} \textbf{(a)} (applied to $v=u$) shows that $u_{\ell
}$ is well-defined and invertible.)

\begin{proof}
[Proof of (\ref{pf.thm.rect.path.a.1}).]We utilize downwards induction on $u$.
This is a version of strong induction in which we fix an element $v\in P$ and
assume (as the induction hypothesis) that (\ref{pf.thm.rect.path.a.1}) holds
for all $u\in P$ satisfying $u>v$. We will then prove that
(\ref{pf.thm.rect.path.a.1}) also holds for $u=v$. Since the poset $P$ is
finite, this will entail that (\ref{pf.thm.rect.path.a.1}) holds for all $u\in
P$.

\begin{verlong}
So let us prove (\ref{pf.thm.rect.path.a.1}) by downwards induction on $u$:
\end{verlong}

Let $v\in P$. Assume (as the induction hypothesis) that
(\ref{pf.thm.rect.path.a.1}) holds for all $u\in P$ satisfying $u>v$. In other
words, we have $\upslack_{\ell}^{1\rightarrow u}=b\overline{u_{\ell}}$ for
each $u\in P$ satisfying $u>v$. Thus, in particular, we have%
\begin{equation}
\upslack_{\ell}^{1\rightarrow u}=b\overline{u_{\ell}}%
\ \ \ \ \ \ \ \ \ \ \text{for each }u\in P\text{ satisfying }u\gtrdot v.
\label{pf.thm.rect.path.a.1.pf.IH2}%
\end{equation}

Note also that the only path from $1$ to $1$ is the trivial path $\left(
1\right)  $. Hence,
\begin{equation}
\upslack_{\ell}^{1\rightarrow1}=\upslack_{\ell}^{\left(  1\right)
}=\upslack_{\ell}^{1}=1=b\overline{1_{\ell}}
\label{pf.thm.rect.path.a.1.pf.IH3}%
\end{equation}
(since $1_{\ell}=b$).

However, $1\neq v$ (since $1\notin P$ and $v\in P$). Thus,
(\ref{eq.prop.slacks.rec.up2}) (applied to $s=1$ and $t=v$) yields%
\begin{align*}
\upslack_{\ell}^{1\rightarrow v}  &  =\underbrace{\sum_{\substack{u\in
\widehat{P};\\u\gtrdot v}}}_{\substack{=\sum_{u\gtrdot v}\\\text{(since our
sums}\\\text{range over }\widehat{P}\text{ by}\\\text{default)}}%
}\underbrace{\upslack_{\ell}^{1\rightarrow u}}_{\substack{=b\overline{u_{\ell
}}\\\text{(indeed, this follows from (\ref{pf.thm.rect.path.a.1.pf.IH2}) when
}u\in P\text{,}\\\text{and follows from (\ref{pf.thm.rect.path.a.1.pf.IH3})
when }u=1\text{;}\\\text{and there are no other possibilities, since }u\gtrdot
v\text{ rules out }u=0\text{)}}}\upslack_{\ell}^{v}\\
&  =\sum_{u\gtrdot v}b\overline{u_{\ell}}\upslack_{\ell}^{v}=b\left(
\sum_{u\gtrdot v}\overline{u_{\ell}}\right)  \underbrace{\upslack_{\ell}^{v}%
}_{\substack{=\overline{\sum\limits_{u\gtrdot v}\overline{u_{\ell}}}%
\cdot\overline{v_{\ell}}\\\text{(by the definition of }\upslack_{\ell}%
^{v}\text{)}}}=b\underbrace{\left(  \sum_{u\gtrdot v}\overline{u_{\ell}%
}\right)  \overline{\sum\limits_{u\gtrdot v}\overline{u_{\ell}}}}_{=1}%
\cdot\overline{v_{\ell}}=b\overline{v_{\ell}}.
\end{align*}
In other words, (\ref{pf.thm.rect.path.a.1}) holds for $u=v$. This completes
the induction step. Thus, we have proved (\ref{pf.thm.rect.path.a.1}) by induction.
\end{proof}

Note that $1_{\ell}$ is invertible (by Lemma \ref{lem.slacks.wd} \textbf{(a)},
applied to $v=1$). In other words, $b$ is invertible (since $1_{\ell}=b$).

\begin{vershort}
Now, let $u\in P$. Then, $b\overline{u_{\ell}}$ is invertible (since $b$ and
$\overline{u_{\ell}}$ are). In view of (\ref{pf.thm.rect.path.a.1}), this
means that $\upslack_{\ell}^{1\rightarrow u}$ is invertible. Hence,
$\overline{\upslack_{\ell}^{1\rightarrow u}}$ is well-defined. Solving
(\ref{pf.thm.rect.path.a.1}) for $u_{\ell}$, we thus obtain $u_{\ell
}=\overline{\upslack_{\ell}^{1\rightarrow u}}\cdot b$. This proves Theorem
\ref{thm.rect.path} \textbf{(a)}. \medskip
\end{vershort}

\begin{verlong}
Now, let $u\in P$. Then, $b$ is invertible (as we just saw), and
$\overline{u_{\ell}}$ is invertible (since any inverse is invertible). Thus,
$b\overline{u_{\ell}}$ is invertible (since a product of two invertible
elements is invertible). In other words, $\upslack_{\ell}^{1\rightarrow u}$ is
invertible (since (\ref{pf.thm.rect.path.a.1}) says that $\upslack_{\ell
}^{1\rightarrow u}=b\overline{u_{\ell}}$). Hence, $\overline{\upslack_{\ell
}^{1\rightarrow u}}$ is well-defined. Furthermore, we have $u_{\ell}%
=\overline{\upslack_{\ell}^{1\rightarrow u}}\cdot b$, since%
\begin{align*}
\overline{\upslack_{\ell}^{1\rightarrow u}}\cdot b  &  =\underbrace{\overline
{b\overline{u_{\ell}}}}_{\substack{=u_{\ell}\overline{b}\\\text{(by
Proposition \ref{prop.inverses.ab} \textbf{(b)})}}}\cdot
b\ \ \ \ \ \ \ \ \ \ \left(  \text{by (\ref{pf.thm.rect.path.a.1})}\right) \\
&  =u_{\ell}\underbrace{\overline{b}b}_{=1}=u_{\ell}.
\end{align*}
This proves Theorem \ref{thm.rect.path} \textbf{(a)}. \medskip
\end{verlong}

\textbf{(b)} This proof is rather similar to that of part \textbf{(a)}, but
uses upwards induction instead of downwards induction (and applies
(\ref{eq.prop.slacks.rec.down1}) instead of (\ref{eq.prop.slacks.rec.up2})).

\begin{verlong}
Here are the details:

Assume that $R^{\ell+1}f\neq\undf$. Then, Lemma \ref{lem.slacks.wd}
\textbf{(c)} (applied to $\ell+1$ instead of $\ell$) yields that the element
$\downslack_{\ell}^{v}$ is well-defined and invertible for each $v\in
\widehat{P}$. Hence, the element $\downslack_{\ell}^{\mathbf{p}}$ is
well-defined for each path $\mathbf{p}$. Therefore, the element
$\downslack_{\ell}^{u\rightarrow0}$ is well-defined for each $u\in P$.

Next, we will prove the equality
\begin{equation}
\downslack_{\ell}^{u\rightarrow0}=u_{\ell}\overline{a}%
\ \ \ \ \ \ \ \ \ \ \text{for each }u\in P. \label{pf.thm.rect.path.b.1}%
\end{equation}
(The $\overline{a}$ on the right hand side here is well-defined, since we
assumed that $a$ is invertible.)

\begin{proof}
[Proof of (\ref{pf.thm.rect.path.b.1}).]We prove the equality
(\ref{pf.thm.rect.path.b.1}) by upwards induction on $u$. This is a version of
strong induction in which we fix an element $v\in P$ and assume (as the
induction hypothesis) that (\ref{pf.thm.rect.path.b.1}) holds for all $u\in P$
satisfying $u<v$. We will then prove that (\ref{pf.thm.rect.path.b.1}) also
holds for $u=v$. Since the poset $P$ is finite, this will entail that
(\ref{pf.thm.rect.path.b.1}) holds for all $u\in P$.

So let us prove (\ref{pf.thm.rect.path.b.1}) by upwards induction on $u$:

Let $v\in P$. Assume (as the induction hypothesis) that
(\ref{pf.thm.rect.path.b.1}) holds for all $u\in P$ satisfying $u<v$. In other
words, we have $\downslack_{\ell}^{u\rightarrow0}=u_{\ell}\overline{a}$ for
each $u\in P$ satisfying $u<v$. Thus, in particular, we have%
\begin{equation}
\downslack_{\ell}^{u\rightarrow0}=u_{\ell}\overline{a}%
\ \ \ \ \ \ \ \ \ \ \text{for each }u\in P\text{ satisfying }u\lessdot v.
\label{pf.thm.rect.path.b.1.pf.IH2}%
\end{equation}

Note also that the only path from $0$ to $0$ is the trivial path $\left(
0\right)  $. Hence,
\begin{equation}
\downslack_{\ell}^{0\rightarrow0}=\downslack_{\ell}^{\left(  0\right)
}=\downslack_{\ell}^{0}=1=0_{\ell}\overline{a}
\label{pf.thm.rect.path.b.1.pf.IH3}%
\end{equation}
(since $0_{\ell}=a$).

However, $v\neq0$ (since $v\in P$ and $0\notin P$). Thus,
(\ref{eq.prop.slacks.rec.down1}) (applied to $s=v$ and $t=0$) yields%
\begin{align*}
\downslack_{\ell}^{v\rightarrow0}  &  =\downslack_{\ell}^{v}\underbrace{\sum
_{\substack{u\in\widehat{P};\\v\gtrdot u}}}_{\substack{=\sum_{\substack{u\in
\widehat{P};\\u\lessdot v}}=\sum_{u\lessdot v}\\\text{(since our
sums}\\\text{range over }\widehat{P}\text{ by}\\\text{default)}}%
}\underbrace{\downslack_{\ell}^{u\rightarrow0}}_{\substack{=u_{\ell}%
\overline{a}\\\text{(indeed, this follows from
(\ref{pf.thm.rect.path.b.1.pf.IH2}) when }u\in P\text{,}\\\text{and follows
from (\ref{pf.thm.rect.path.b.1.pf.IH3}) when }u=0\text{;}\\\text{and there
are no other possibilities, since }v\gtrdot u\text{ rules out }u=1\text{)}}}\\
&  =\underbrace{\downslack_{\ell}^{v}}_{\substack{=v_{\ell}\cdot\overline
{\sum\limits_{u\lessdot v}u_{\ell}}\\\text{(by the definition of
}\downslack_{\ell}^{v}\text{)}}}\sum_{u\lessdot v}u_{\ell}\overline{a}%
=v_{\ell}\cdot\underbrace{\overline{\sum\limits_{u\lessdot v}u_{\ell}}%
\cdot\sum_{u\lessdot v}u_{\ell}}_{=1}\overline{a}=v_{\ell}\overline{a}.
\end{align*}
In other words, (\ref{pf.thm.rect.path.b.1}) holds for $u=v$. This completes
the induction step, and (\ref{pf.thm.rect.path.b.1}) is proven.
\end{proof}

Now, for each $u\in P$, we have $u_{\ell}=\downslack_{\ell}^{u\rightarrow
0}\cdot a$, since%
\begin{align*}
\downslack_{\ell}^{u\rightarrow0}\cdot a  &  =u_{\ell}\underbrace{\overline
{a}\cdot a}_{=1}\ \ \ \ \ \ \ \ \ \ \left(  \text{by
(\ref{pf.thm.rect.path.b.1})}\right) \\
&  =u_{\ell}.
\end{align*}
This proves Theorem \ref{thm.rect.path} \textbf{(b)}. \medskip
\end{verlong}

\begin{vershort}
\medskip\textbf{(c)} Let $u\in P$. Recall that $\left(  p,q\right)  $ is the
unique maximal element of $P$. Therefore, each path from $1$ to $u$ begins
with the step $1\gtrdot\left(  p,q\right)  $. Thus, $\upslack_{\ell
}^{1\rightarrow u}=\upslack_{\ell}^{\left(  p,q\right)  \rightarrow u}$ (since
$\upslack_{\ell}^{1}=1$). Hence, part \textbf{(c)} follows from \textbf{(a)}.
\medskip
\end{vershort}

\begin{verlong}
\textbf{(c)} Assume that $R^{\ell}f\neq\undf$ and $\ell\geq1$. Let $t\in P$.
Every element of $\widehat{P}$ distinct from $1$ is $\leq\left(  p,q\right)
$. Therefore, the only element $u\in\widehat{P}$ satisfying $1\gtrdot u$ is
the maximal element $\left(  p,q\right)  $ of $P$. Hence, $\sum
_{\substack{u\in\widehat{P};\\1\gtrdot u}}\upslack_{\ell}^{u\rightarrow
t}=\upslack_{\ell}^{\left(  p,q\right)  \rightarrow t}$. Now,
(\ref{eq.prop.slacks.rec.up1}) (applied to $s=1$) yields%
\[
\upslack_{\ell}^{1\rightarrow t}=\underbrace{\upslack_{\ell}^{1}%
}_{\substack{=1\\\text{(by (\ref{eq.slack.1}))}}}\underbrace{\sum
_{\substack{u\in\widehat{P};\\1\gtrdot u}}\upslack_{\ell}^{u\rightarrow t}%
}_{=\upslack_{\ell}^{\left(  p,q\right)  \rightarrow t}}=\upslack_{\ell
}^{\left(  p,q\right)  \rightarrow t}.
\]

Forget that we fixed $t$. We thus have proved that $\upslack_{\ell
}^{1\rightarrow t}=\upslack_{\ell}^{\left(  p,q\right)  \rightarrow t}$ for
each $t\in P$. Renaming the index $t$ as $u$ in this statement, we obtain the
following:%
\begin{equation}
\upslack_{\ell}^{1\rightarrow u}=\upslack_{\ell}^{\left(  p,q\right)
\rightarrow u}\ \ \ \ \ \ \ \ \ \ \text{for each }u\in P.
\label{pf.thm.rect.path.c.3}%
\end{equation}

Now, let $u\in P$. Then, Theorem \ref{thm.rect.path} \textbf{(a)} yields
\[
u_{\ell}=\overline{\upslack_{\ell}^{1\rightarrow u}}\cdot b=\overline
{\upslack_{\ell}^{\left(  p,q\right)  \rightarrow u}}\cdot
b\ \ \ \ \ \ \ \ \ \ \left(  \text{by (\ref{pf.thm.rect.path.c.3})}\right)  .
\]
This proves Theorem \ref{thm.rect.path} \textbf{(c)}. \medskip
\end{verlong}

\begin{vershort}
Similarly, part \textbf{(d)} follows from \textbf{(b)}. \qedhere

\end{vershort}

\begin{verlong}
\textbf{(d)} Assume that $R^{\ell+1}f\neq\undf$. Let $s\in P$. Every element
of $\widehat{P}$ distinct from $0$ is $\geq\left(  1,1\right)  $. Thus, the
only element $u\in\widehat{P}$ satisfying $u\gtrdot0$ is the minimal element
$\left(  1,1\right)  $ of $P$. Hence,
\[
\sum_{\substack{u\in\widehat{P};\\u\gtrdot0}}\downslack_{\ell}^{s\rightarrow
u}\downslack_{\ell}^{0}=\downslack_{\ell}^{s\rightarrow\left(  1,1\right)
}\underbrace{\downslack_{\ell}^{0}}_{\substack{=1\\\text{(by (\ref{eq.slack.1}%
))}}}=\downslack_{\ell}^{s\rightarrow\left(  1,1\right)  }.
\]
Now, (\ref{eq.prop.slacks.rec.down2}) (applied to $t=0$) yields%
\[
\downslack_{\ell}^{s\rightarrow0}=\sum_{\substack{u\in\widehat{P};\\u\gtrdot
0}}\downslack_{\ell}^{s\rightarrow u}\downslack_{\ell}^{0}=\downslack_{\ell
}^{s\rightarrow\left(  1,1\right)  }.
\]

Forget that we fixed $s$. We thus have proved that $\downslack_{\ell
}^{s\rightarrow0}=\downslack_{\ell}^{s\rightarrow\left(  1,1\right)  }$ for
each $s\in P$. Renaming the index $s$ as $u$ in this statement, we obtain the
following:%
\begin{equation}
\downslack_{\ell}^{u\rightarrow0}=\downslack_{\ell}^{u\rightarrow\left(
1,1\right)  }\ \ \ \ \ \ \ \ \ \ \text{for each }u\in P.
\label{pf.thm.rect.path.d.3}%
\end{equation}

Now, let $u\in P$. Then, Theorem \ref{thm.rect.path} \textbf{(b)} yields
\[
u_{\ell}=\downslack_{\ell}^{u\rightarrow0}\cdot a=\downslack_{\ell
}^{u\rightarrow\left(  1,1\right)  }\cdot a\ \ \ \ \ \ \ \ \ \ \left(
\text{by (\ref{pf.thm.rect.path.d.3})}\right)  .
\]
This proves Theorem \ref{thm.rect.path} \textbf{(d)}. \qedhere

\end{verlong}
\end{proof}

\begin{remark}
\label{rmk.slacks.gen} Corollary \ref{cor.rect.trans-uv}, Proposition
\ref{prop.slacks.rec} and parts \textbf{(a)} and \textbf{(b)} of Theorem
\ref{thm.rect.path} hold more generally if $P$ is replaced by any finite poset
(not necessarily a rectangle). The proofs we gave above work in that
generality. Parts \textbf{(c)} and \textbf{(d)} of Theorem \ref{thm.rect.path}
can be similarly generalized as long as the poset $P$ has a global maximum
(for part \textbf{(c)}) and a global minimum (for part \textbf{(d)}); all we
need to do is to replace $\left(  p,q\right)  $ by the global maximum and
$\left(  1,1\right)  $ by the global minimum. We will have no need for this
generality, though.
\end{remark}

\section{\label{sec.ij=11}Proof of reciprocity: the case $\left(  i,j\right)
=\left(  1,1\right)  $}

Now, we are mostly ready to prove that Theorem \ref{thm.rect.antip} holds in
the case when $\left(  i,j\right)  =\left(  1,1\right)  $. For reasons both
technical and pedagogical, it is useful for us to dispose of this case now in
order to have less work to do later. First, we prove Theorem
\ref{thm.rect.antip} for $\left(  i,j\right)  =\left(  1,1\right)  $ under the
extra assumption that $a$ is invertible:

\begin{lemma}
\label{lem.rect.antip.11inv}Assume that $P$ is the $p\times q$-rectangle
$\left[  p\right]  \times\left[  q\right]  $. Let $\ell\in\mathbb{N}$ be such
that $\ell\geq1$. Let $f\in\mathbb{K}^{\widehat{P}}$ be a $\mathbb{K}%
$-labeling such that $R^{\ell}f\neq\undf$. Let $a=f\left(  0\right)  $ and
$b=f\left(  1\right)  $. Assume that $a$ is invertible. Then,%
\[
\left(  R^{\ell}f\right)  \left(  1,1\right)  =a\cdot\overline{\left(
R^{\ell-1}f\right)  \left(  p,q\right)  }\cdot b.
\]

\end{lemma}

\begin{proof}
We use the notations from Section \ref{sec.proof-nots}. Thus, $\left(
R^{\ell}f\right)  \left(  1,1\right)  =\left(  1,1\right)  _{\ell}$ and%
\[
\left(  R^{\ell-1}f\right)  \left(  p,q\right)  =\left(  p,q\right)  _{\ell
-1}=\downslack_{\ell-1}^{\left(  p,q\right)  \rightarrow\left(  1,1\right)
}\cdot a
\]
(by Theorem \ref{thm.rect.path} \textbf{(d)}, applied to $\ell-1$ and $\left(
p,q\right)  $ instead of $\ell$ and $u$). Solving this equation for
$\downslack_{\ell-1}^{\left(  p,q\right)  \rightarrow\left(  1,1\right)  }$,
we obtain%
\begin{equation}
\downslack_{\ell-1}^{\left(  p,q\right)  \rightarrow\left(  1,1\right)
}=\left(  R^{\ell-1}f\right)  \left(  p,q\right)  \cdot\overline{a}
\label{pf.lem.rect.antip.11inv.1}%
\end{equation}
(since $a$ is invertible). Note also that $R\left(  R^{\ell-1}f\right)
=R^{\ell}f\neq\undf$, and thus $\left(  R^{\ell-1}f\right)  \left(
p,q\right)  $ is invertible (by Lemma \ref{lem.R.inv}, applied to $R^{\ell
-1}f$ and $\left(  p,q\right)  $ instead of $f$ and $v$).

Now,%
\begin{align*}
\left(  R^{\ell}f\right)  \left(  1,1\right)   &  =\left(  1,1\right)  _{\ell
}=\overline{\upslack_{\ell}^{\left(  p,q\right)  \rightarrow\left(
1,1\right)  }}\cdot b\ \ \ \ \ \ \ \ \ \ \left(  \text{by Theorem
\ref{thm.rect.path} \textbf{(c)}, applied to }u=\left(  1,1\right)  \right) \\
&  =\overline{\downslack_{\ell-1}^{\left(  p,q\right)  \rightarrow\left(
1,1\right)  }}\cdot b\ \ \ \ \ \ \ \ \ \ \left(  \text{since
(\ref{eq.prop.rect.trans.uv}) yields }\upslack_{\ell}^{\left(  p,q\right)
\rightarrow\left(  1,1\right)  }=\downslack_{\ell-1}^{\left(  p,q\right)
\rightarrow\left(  1,1\right)  }\right) \\
&  =\underbrace{\overline{\left(  R^{\ell-1}f\right)  \left(  p,q\right)
\cdot\overline{a}}}_{\substack{=a\cdot\overline{\left(  R^{\ell-1}f\right)
\left(  p,q\right)  }\\\text{(since }\left(  R^{\ell-1}f\right)  \left(
p,q\right)  \text{ and }\overline{a}\\\text{are invertible)}}}\cdot
\,b\ \ \ \ \ \ \ \ \ \ \left(  \text{by (\ref{pf.lem.rect.antip.11inv.1}%
)}\right) \\
&  =a\cdot\overline{\left(  R^{\ell-1}f\right)  \left(  p,q\right)  }\cdot b.
\end{align*}
This proves Lemma \ref{lem.rect.antip.11inv}.
\end{proof}

Unfortunately, our proof of Lemma \ref{lem.rect.antip.11inv} made use of the
requirement that $a$ be invertible, since $\upslack_{\ell}^{\left(
p,q\right)  \rightarrow\left(  1,1\right)  }$ and $\downslack_{\ell
-1}^{\left(  p,q\right)  \rightarrow\left(  1,1\right)  }$ would not be
well-defined otherwise. In order to remove this requirement, we make use of a
trick, in which we \textquotedblleft temporarily\textquotedblright\ set the
label $f\left(  0\right)  $ to $1$ and then argue that this has a predictable
effect on $\left(  Rf\right)  \left(  1,1\right)  $. This trick relies on the following:

\begin{lemma}
\label{lem.ato1}Let $P$ be an arbitrary finite poset (not necessarily $\left[
p\right]  \times\left[  q\right]  $). Let $f,g\in\mathbb{K}^{\widehat{P}}$ be
two $\mathbb{K}$-labelings such that $Rf\neq\undf$. Assume that%
\begin{equation}
g\left(  x\right)  =f\left(  x\right)  \ \ \ \ \ \ \ \ \ \ \text{for each
}x\in\widehat{P}\setminus\left\{  0\right\}  . \label{eq.lem.ato1.ass1}%
\end{equation}
Assume furthermore that $g\left(  0\right)  =1$. Set $a=f\left(  0\right)  $. Then:

\begin{enumerate}
\item[\textbf{(a)}] We have $Rg\neq\undf$.

\item[\textbf{(b)}] If $v\in P$ is not a minimal element of $P$, then $\left(
Rf\right)  \left(  v\right)  =\left(  Rg\right)  \left(  v\right)  $.

\item[\textbf{(c)}] If $v\in P$ is a minimal element of $P$, then $\left(
Rf\right)  \left(  v\right)  =a\cdot\left(  Rg\right)  \left(  v\right)  $.
\end{enumerate}
\end{lemma}

\begin{vershort}

\begin{proof}
[Proof of Lemma \ref{lem.ato1} (sketched).]Our assumption
(\ref{eq.lem.ato1.ass1}) shows that the labels of $f$ equal the corresponding
labels of $g$ at all elements of $\widehat{P}$ other than at $0$. Only the
labels at $0$ can differ.

Compute the labelings $Rf$ and $Rg$ recursively, as we did in Example
\ref{ex.rowmotion.2x2}, making sure to pick a linear extension of $P$ that
starts with all minimal elements of $P$ (so that the toggles at these minimal
elements all happen at the very end of our computation). The computation for
$Rf$ proceeds identically with the computation for $Rg$ until we
\textquotedblleft interact with\textquotedblright\ the different labels at $0$
-- that is, until the labels $f\left(  0\right)  $ and $g\left(  0\right)  $
make an appearance in the sums $\sum\limits_{\substack{u\in\widehat{P}%
;\\u\lessdot v}}f\left(  u\right)  $ and $\sum\limits_{\substack{u\in
\widehat{P};\\u\lessdot v}}g\left(  u\right)  $, respectively (because all
other labels of $f$ equal the corresponding labels of $g$). However, this
\textquotedblleft interaction\textquotedblright\ only happens when we toggle
at a minimal element of $P$ (since $v$ has to be minimal in order for
$f\left(  0\right)  $ to be an addend of the sum $\sum\limits_{\substack{u\in
\widehat{P};\\u\lessdot v}}f\left(  u\right)  $). Furthermore, when we do
toggle at a minimal element $v$ of $P$, the relevant sums $\sum
\limits_{\substack{u\in\widehat{P};\\u\lessdot v}}f\left(  u\right)  $ and
$\sum\limits_{\substack{u\in\widehat{P};\\u\lessdot v}}g\left(  u\right)  $
simplify to $f\left(  0\right)  =a$ and $g\left(  0\right)  =1$, respectively
(because $0$ is the \textbf{only} element $u\in\widehat{P}$ satisfying
$u\lessdot v$). Therefore, the labels of $Rf$ and $Rg$ at $v$ end up differing
by a factor of $a$ (more precisely, the value of $Rf$ at $v$ ends up being $a$
times the label of $Rg$ at $v$). This proves Lemma \ref{lem.ato1}.
\end{proof}
\end{vershort}

\begin{verlong}

\begin{proof}
[Proof of Lemma \ref{lem.ato1}.]Pick a linear extension $\left(  v_{1}%
,v_{2},\ldots,v_{m}\right)  $ of $P$. (We know from Theorem
\ref{thm.linext.ex} that such a linear extension exists.)

For each $i\in\left\{  0,1,\ldots,m\right\}  $, define a partial map%
\[
R_{i}:=T_{v_{i+1}}\circ T_{v_{i+2}}\circ\cdots\circ T_{v_{m}}:\mathbb{K}%
^{\widehat{P}}\dashrightarrow\mathbb{K}^{\widehat{P}}.
\]
Thus, in particular,%
\[
R_{m}=T_{v_{m+1}}\circ T_{v_{m+2}}\circ\cdots\circ T_{v_{m}}=\left(
\text{empty composition}\right)  =\operatorname*{id}%
\]
and
\[
R_{0}=T_{v_{0+1}}\circ T_{v_{0+2}}\circ\cdots\circ T_{v_{m}}=T_{v_{1}}\circ
T_{v_{2}}\circ\cdots\circ T_{v_{m}}=R
\]
(by the definition of $R$).

For each $i\in\left\{  0,1,\ldots,m\right\}  $, we set%
\[
f^{\left(  i\right)  }:=R_{m-i}f\ \ \ \ \ \ \ \ \ \ \text{and}%
\ \ \ \ \ \ \ \ \ \ g^{\left(  i\right)  }:=R_{m-i}g.
\]
Each of $f^{\left(  i\right)  }$ and $g^{\left(  i\right)  }$ is either a
$\mathbb{K}$-labeling in $\mathbb{K}^{\widehat{P}}$ or $\undf$; we will soon
see that it is a $\mathbb{K}$-labeling.

The tuple $\left(  v_{1},v_{2},\ldots,v_{m}\right)  $ is a linear extension of
$P$. Thus, for each $x\in P$, there exists a unique $i\in\left\{
1,2,\ldots,m\right\}  $ that satisfies $x=v_{i}$. Let us denote this $i$ by
$\rho\left(  x\right)  $. Thus, the map%
\begin{align*}
\rho:P  &  \rightarrow\left\{  1,2,\ldots,m\right\}  ,\\
x  &  \mapsto\rho\left(  x\right)
\end{align*}
is a bijection.

Let $M$ denote the set of all minimal elements of $P$.

We now shall prove the following:

\begin{statement}
\textit{Claim 1:} Let $i\in\left\{  0,1,\ldots,m\right\}  $. Then, $f^{\left(
i\right)  }\neq\undf$ and $g^{\left(  i\right)  }\neq\undf$. Moreover,
$g^{\left(  i\right)  }\left(  1\right)  =f^{\left(  i\right)  }\left(
1\right)  $ and $f^{\left(  i\right)  }\left(  0\right)  =a\cdot g^{\left(
i\right)  }\left(  0\right)  $. Furthermore, each $v\in P$ satisfies%
\[
f^{\left(  i\right)  }\left(  v\right)  =%
\begin{cases}
g^{\left(  i\right)  }\left(  v\right)  , & \text{if }v\notin M\text{ or }%
\rho\left(  v\right)  \leq m-i;\\
a\cdot g^{\left(  i\right)  }\left(  v\right)  , & \text{otherwise.}%
\end{cases}
\]

\end{statement}

\begin{proof}
[Proof of Claim 1.]We proceed by induction on $i$:

\textit{Base case:} The definition of $f^{\left(  0\right)  }$ yields
$f^{\left(  0\right)  }=\underbrace{R_{m-0}}_{=R_{m}=\operatorname*{id}%
}f=f\neq\undf$. The definition of $g^{\left(  0\right)  }$ yields $g^{\left(
0\right)  }=\underbrace{R_{m-0}}_{=R_{m}=\operatorname*{id}}g=g\neq\undf$.

From $g^{\left(  0\right)  }=g$, we obtain $g^{\left(  0\right)  }\left(
1\right)  =g\left(  1\right)  =f\left(  1\right)  $ (by
(\ref{eq.lem.ato1.ass1}), applied to $x=1$). In other words, $g^{\left(
0\right)  }\left(  1\right)  =f^{\left(  0\right)  }\left(  1\right)  $ (since
$f^{\left(  0\right)  }=f$).

From $g^{\left(  0\right)  }=g$, we obtain $g^{\left(  0\right)  }\left(
0\right)  =g\left(  0\right)  =1$, so that $a\cdot g^{\left(  0\right)
}\left(  0\right)  =a\cdot1=a=f\left(  0\right)  =f^{\left(  0\right)
}\left(  0\right)  $ (since $f=f^{\left(  0\right)  }$). In other words,
$f^{\left(  0\right)  }\left(  0\right)  =a\cdot g^{\left(  0\right)  }\left(
0\right)  $.

Now, let $v\in P$. Then, (\ref{eq.lem.ato1.ass1}) (applied to $x=v$) yields
$g\left(  v\right)  =f\left(  v\right)  $ (since $v\in P\subseteq
\widehat{P}\setminus\left\{  0\right\}  $). In view of $g^{\left(  0\right)
}=g$ and $f^{\left(  0\right)  }=f$, we can rewrite this as $g^{\left(
0\right)  }\left(  v\right)  =f^{\left(  0\right)  }\left(  v\right)  $.
However, we have $\rho\left(  v\right)  \in\left\{  1,2,\ldots,m\right\}  $
(by the definition of $\rho\left(  v\right)  $) and therefore $\rho\left(
v\right)  \leq m=m-0$. Thus,%
\begin{align*}
&
\begin{cases}
g^{\left(  0\right)  }\left(  v\right)  , & \text{if }v\notin M\text{ or }%
\rho\left(  v\right)  \leq m-0;\\
a\cdot g^{\left(  0\right)  }\left(  v\right)  , & \text{otherwise}%
\end{cases}
\\
&  =g^{\left(  0\right)  }\left(  v\right)  \ \ \ \ \ \ \ \ \ \ \left(
\text{since }v\notin M\text{ or }\rho\left(  v\right)  \leq m-0\text{ (because
}\rho\left(  v\right)  \leq m-0\text{)}\right) \\
&  =f^{\left(  0\right)  }\left(  v\right)  .
\end{align*}
Hence,
\[
f^{\left(  0\right)  }\left(  v\right)  =%
\begin{cases}
g^{\left(  0\right)  }\left(  v\right)  , & \text{if }v\notin M\text{ or }%
\rho\left(  v\right)  \leq m-0;\\
a\cdot g^{\left(  0\right)  }\left(  v\right)  , & \text{otherwise.}%
\end{cases}
\]
Forget that we fixed $v$. We thus have shown that each $v\in P$ satisfies%
\[
f^{\left(  0\right)  }\left(  v\right)  =%
\begin{cases}
g^{\left(  0\right)  }\left(  v\right)  , & \text{if }v\notin M\text{ or }%
\rho\left(  v\right)  \leq m-0;\\
a\cdot g^{\left(  0\right)  }\left(  v\right)  , & \text{otherwise.}%
\end{cases}
\]
Since we also know that $f^{\left(  0\right)  }\neq\undf$ and $g^{\left(
0\right)  }\neq\undf$ and $g^{\left(  0\right)  }\left(  1\right)  =f^{\left(
0\right)  }\left(  1\right)  $ and $f^{\left(  0\right)  }\left(  0\right)
=a\cdot g^{\left(  0\right)  }\left(  0\right)  $, we have thus finished
proving that Claim 1 holds for $i=0$.

\textit{Induction step:} Let $j\in\left\{  0,1,\ldots,m-1\right\}  $. Assume
(as the induction hypothesis) that Claim 1 holds for $i=j$. We must prove that
Claim 1 holds for $i=j+1$. In other words, we must prove that $f^{\left(
j+1\right)  }\neq\undf$ and $g^{\left(  j+1\right)  }\neq\undf$ and
$g^{\left(  j+1\right)  }\left(  1\right)  =f^{\left(  j+1\right)  }\left(
1\right)  $ and $f^{\left(  j+1\right)  }\left(  0\right)  =a\cdot g^{\left(
j+1\right)  }\left(  0\right)  $ and that each $v\in P$ satisfies%
\begin{equation}
f^{\left(  j+1\right)  }\left(  v\right)  =%
\begin{cases}
g^{\left(  j+1\right)  }\left(  v\right)  , & \text{if }v\notin M\text{ or
}\rho\left(  v\right)  \leq m-\left(  j+1\right)  ;\\
a\cdot g^{\left(  j+1\right)  }\left(  v\right)  , & \text{otherwise.}%
\end{cases}
\label{pf.lem.ato1.c1.IG}%
\end{equation}

Our induction hypothesis tells us that Claim 1 holds for $i=j$. In other
words, we have $f^{\left(  j\right)  }\neq\undf$ and $g^{\left(  j\right)
}\neq\undf$ and $g^{\left(  j\right)  }\left(  1\right)  =f^{\left(  j\right)
}\left(  1\right)  $ and $f^{\left(  j\right)  }\left(  0\right)  =a\cdot
g^{\left(  j\right)  }\left(  0\right)  $, and each $v\in P$ satisfies%
\begin{equation}
f^{\left(  j\right)  }\left(  v\right)  =%
\begin{cases}
g^{\left(  j\right)  }\left(  v\right)  , & \text{if }v\notin M\text{ or }%
\rho\left(  v\right)  \leq m-j;\\
a\cdot g^{\left(  j\right)  }\left(  v\right)  , & \text{otherwise.}%
\end{cases}
\label{pf.lem.ato1.c1.IH}%
\end{equation}

Let $y:=v_{m-j}$. Thus, $y$ is an element of $P$. The definition of
$\rho\left(  y\right)  $ yields that $\rho\left(  y\right)  $ is the unique
$i\in\left\{  1,2,\ldots,m\right\}  $ that satisfies $y=v_{i}$. Thus,
$\rho\left(  y\right)  =m-j$ (since $y=v_{m-j}$).

It is easy to see that%
\begin{equation}
R_{m-\left(  j+1\right)  }=T_{y}\circ R_{m-j} \label{pf.lem.ato1.c1.Rj+1}%
\end{equation}
\footnote{\textit{Proof.} The definition of $R_{m-j}$ yields $R_{m-j}%
=T_{v_{m-j+1}}\circ T_{v_{m-j+2}}\circ\cdots\circ T_{v_{m}}$. The definition
of $R_{m-\left(  j+1\right)  }$ yields%
\begin{align*}
R_{m-\left(  j+1\right)  }  &  =\underbrace{T_{v_{m-\left(  j+1\right)  +1}}%
}_{=T_{v_{m-j}}}\circ\underbrace{T_{v_{m-\left(  j+1\right)  +2}}%
}_{=T_{v_{m-j+1}}}\circ\cdots\circ T_{v_{m}}=T_{v_{m-j}}\circ T_{v_{m-j+1}%
}\circ\cdots\circ T_{v_{m}}\\
&  =\underbrace{T_{v_{m-j}}}_{\substack{=T_{y}\\\text{(since }v_{m-j}%
=y\text{)}}}\circ\underbrace{T_{v_{m-j+1}}\circ T_{v_{m-j+2}}\circ\cdots\circ
T_{v_{m}}}_{=R_{m-j}}=T_{y}\circ R_{m-j}.
\end{align*}
This proves (\ref{pf.lem.ato1.c1.Rj+1}).}.

It is easy to see that $f^{\left(  j+1\right)  }\neq\undf$%
\ \ \ \ \footnote{\textit{Proof.} The definition of $R_{m-\left(  j+1\right)
}$ yields%
\[
R_{m-\left(  j+1\right)  }=\underbrace{T_{v_{m-\left(  j+1\right)  +1}}%
}_{=T_{v_{m-j}}}\circ\underbrace{T_{v_{m-\left(  j+1\right)  +2}}%
}_{=T_{v_{m-j+1}}}\circ\cdots\circ T_{v_{m}}=T_{v_{m-j}}\circ T_{v_{m-j+1}%
}\circ\cdots\circ T_{v_{m}}.
\]
However, the definition of $R$ yields%
\begin{align*}
R  &  =T_{v_{1}}\circ T_{v_{2}}\circ\cdots\circ T_{v_{m}}=\left(  T_{v_{1}%
}\circ T_{v_{2}}\circ\cdots\circ T_{v_{m-j-1}}\right)  \circ
\underbrace{\left(  T_{v_{m-j}}\circ T_{v_{m-j+1}}\circ\cdots\circ T_{v_{m}%
}\right)  }_{=R_{m-\left(  j+1\right)  }}\\
&  =\left(  T_{v_{1}}\circ T_{v_{2}}\circ\cdots\circ T_{v_{m-j-1}}\right)
\circ R_{m-\left(  j+1\right)  }.
\end{align*}
Thus, if we had $R_{m-\left(  j+1\right)  }f=\undf$, then we would have%
\begin{align*}
\underbrace{R}_{=\left(  T_{v_{1}}\circ T_{v_{2}}\circ\cdots\circ
T_{v_{m-j-1}}\right)  \circ R_{m-\left(  j+1\right)  }}f  &  =\left(  \left(
T_{v_{1}}\circ T_{v_{2}}\circ\cdots\circ T_{v_{m-j-1}}\right)  \circ
R_{m-\left(  j+1\right)  }\right)  f\\
&  =\left(  T_{v_{1}}\circ T_{v_{2}}\circ\cdots\circ T_{v_{m-j-1}}\right)
\left(  \underbrace{R_{m-\left(  j+1\right)  }f}_{=\undf}\right) \\
&  =\left(  T_{v_{1}}\circ T_{v_{2}}\circ\cdots\circ T_{v_{m-j-1}}\right)
\left(  \undf\right)  =\undf,
\end{align*}
which would contradict $Rf\neq\undf$. Hence, we cannot have $R_{m-\left(
j+1\right)  }f=\undf$. Thus, $R_{m-\left(  j+1\right)  }f\neq\undf$. However,
the definition of $f^{\left(  j+1\right)  }$ yields $f^{\left(  j+1\right)
}=R_{m-\left(  j+1\right)  }f\neq\undf$.}. Furthermore, we have
\[
f^{\left(  j+1\right)  }=T_{y}f^{\left(  j\right)  }%
\]
\footnote{\textit{Proof.} The definition of $f^{\left(  j\right)  }$ yields
$f^{\left(  j\right)  }=R_{m-j}f$. Hence, $R_{m-j}f=f^{\left(  j\right)  }$.
The definition of $f^{\left(  j+1\right)  }$ yields%
\[
f^{\left(  j+1\right)  }=\underbrace{R_{m-\left(  j+1\right)  }}%
_{\substack{=T_{y}\circ R_{m-j}\\\text{(by (\ref{pf.lem.ato1.c1.Rj+1}))}%
}}f=\left(  T_{y}\circ R_{m-j}\right)  f=T_{y}\underbrace{\left(
R_{m-j}f\right)  }_{=f^{\left(  j\right)  }}=T_{y}f^{\left(  j\right)  }.
\]
}. Similarly,%
\[
g^{\left(  j+1\right)  }=T_{y}g^{\left(  j\right)  }.
\]

Next, we observe that for each $u\in\widehat{P}$ satisfying $u\gtrdot y$, we
have%
\begin{equation}
g^{\left(  j\right)  }\left(  u\right)  =f^{\left(  j\right)  }\left(
u\right)  \label{pf.lem.ato1.c1.s3}%
\end{equation}
\footnote{\textit{Proof of (\ref{pf.lem.ato1.c1.s3}):} Let $u\in\widehat{P}$
be such that $u\gtrdot y$. Then, $u\gtrdot y$, so that $u>y$. In other words,
$y<u$. Thus, there exists an element of $P$ that is smaller than $u$ (namely,
$y$). Hence, $u$ cannot be a minimal element of $P$. In other words, $u\notin
M$ (since $M$ is the set of all minimal elements of $P$).
\par
We must prove that $g^{\left(  j\right)  }\left(  u\right)  =f^{\left(
j\right)  }\left(  u\right)  $. If $u=1$, then this follows directly from
$g^{\left(  j\right)  }\left(  1\right)  =f^{\left(  j\right)  }\left(
1\right)  $. Thus, we WLOG assume that $u\neq1$. Moreover, $u\neq0$ (because
if we had $u=0$, then we would have $0=u>y$, which would contradict the fact
that $0$ is not larger than any element of $\widehat{P}$). Combining
$u\in\widehat{P}$ with $u\neq0$ and $u\neq1$, we obtain $u\in\widehat{P}%
\setminus\left\{  0,1\right\}  =P$. Hence, $\rho\left(  u\right)  $ is
well-defined.
\par
We have $u\notin M$ or $\rho\left(  u\right)  \leq m-j$ (since $u\notin M$).
Now, (\ref{pf.lem.ato1.c1.IH}) (applied to $v=u$) yields
\begin{align*}
f^{\left(  j\right)  }\left(  u\right)   &  =%
\begin{cases}
g^{\left(  j\right)  }\left(  u\right)  , & \text{if }u\notin M\text{ or }%
\rho\left(  u\right)  \leq m-j;\\
a\cdot g^{\left(  j\right)  }\left(  u\right)  , & \text{otherwise}%
\end{cases}
\\
&  =g^{\left(  j\right)  }\left(  u\right)  \ \ \ \ \ \ \ \ \ \ \left(
\text{since }u\notin M\text{ or }\rho\left(  u\right)  \leq m-j\right)  .
\end{align*}
In other words, $g^{\left(  j\right)  }\left(  u\right)  =f^{\left(  j\right)
}\left(  u\right)  $. This proves (\ref{pf.lem.ato1.c1.s3}).}. Furthermore, we
have%
\begin{equation}
g^{\left(  j\right)  }\left(  y\right)  =f^{\left(  j\right)  }\left(
y\right)  \label{pf.lem.ato1.c1.s4}%
\end{equation}
\footnote{\textit{Proof of (\ref{pf.lem.ato1.c1.s4}):} We have $\rho\left(
y\right)  =m-j\leq m-j$. Therefore, $y\notin M$ or $\rho\left(  y\right)  \leq
m-j$. Now, (\ref{pf.lem.ato1.c1.IH}) (applied to $v=y$) yields
\begin{align*}
f^{\left(  j\right)  }\left(  y\right)   &  =%
\begin{cases}
g^{\left(  j\right)  }\left(  y\right)  , & \text{if }y\notin M\text{ or }%
\rho\left(  y\right)  \leq m-j;\\
a\cdot g^{\left(  j\right)  }\left(  y\right)  , & \text{otherwise}%
\end{cases}
\\
&  =g^{\left(  j\right)  }\left(  y\right)  \ \ \ \ \ \ \ \ \ \ \left(
\text{since }y\notin M\text{ or }\rho\left(  y\right)  \leq m-j\right)  .
\end{align*}
In other words, $g^{\left(  j\right)  }\left(  y\right)  =f^{\left(  j\right)
}\left(  y\right)  $. This proves (\ref{pf.lem.ato1.c1.s4}).}.

However, recall that $f^{\left(  j+1\right)  }=T_{y}f^{\left(  j\right)  }$,
so that $T_{y}f^{\left(  j\right)  }=f^{\left(  j+1\right)  }\neq\undf$. Thus,
the expression
\begin{equation}
\left(  \sum\limits_{\substack{u\in\widehat{P};\\u\lessdot y}}f^{\left(
j\right)  }\left(  u\right)  \right)  \cdot\overline{f^{\left(  j\right)
}\left(  y\right)  }\cdot\overline{\sum\limits_{\substack{u\in\widehat{P}%
;\\u\gtrdot y}}\overline{f^{\left(  j\right)  }\left(  u\right)  }}
\label{pf.lem.ato1.c1.wd-expr-f}%
\end{equation}
is well-defined (because if this expression was not well-defined, then the
definition of the $y$-toggle $T_{y}$ (Definition \ref{def.Tv}) would dictate
that $T_{y}f^{\left(  j\right)  }=\undf$; but this would contradict
$T_{y}f^{\left(  j\right)  }\neq\undf$). As a consequence, the expressions
$\overline{f^{\left(  j\right)  }\left(  y\right)  }$ and $\overline
{\sum\limits_{\substack{u\in\widehat{P};\\u\gtrdot y}}\overline{f^{\left(
j\right)  }\left(  u\right)  }}$ are also well-defined (since they are parts
of the well-defined expression (\ref{pf.lem.ato1.c1.wd-expr-f})). In view of
(\ref{pf.lem.ato1.c1.s3}) and (\ref{pf.lem.ato1.c1.s4}), we can rewrite this
as follows: The expressions $\overline{g^{\left(  j\right)  }\left(  y\right)
}$ and $\overline{\sum\limits_{\substack{u\in\widehat{P};\\u\gtrdot
y}}\overline{g^{\left(  j\right)  }\left(  u\right)  }}$ are well-defined.
Thus, the expression%
\begin{equation}
\left(  \sum\limits_{\substack{u\in\widehat{P};\\u\lessdot y}}g^{\left(
j\right)  }\left(  u\right)  \right)  \cdot\overline{g^{\left(  j\right)
}\left(  y\right)  }\cdot\overline{\sum\limits_{\substack{u\in\widehat{P}%
;\\u\gtrdot y}}\overline{g^{\left(  j\right)  }\left(  u\right)  }}
\label{pf.lem.ato1.c1.wd-expr-g}%
\end{equation}
is well-defined as well (since the expression $\sum\limits_{\substack{u\in
\widehat{P};\\u\lessdot y}}g^{\left(  j\right)  }\left(  u\right)  $ is
clearly well-defined\footnote{because $g^{\left(  j\right)  }\neq\undf$}).
Consequently, the definition of the $y$-toggle $T_{y}$ (Definition
\ref{def.Tv}) yields $T_{y}g^{\left(  j\right)  }\neq\undf$. In other words,
$g^{\left(  j+1\right)  }\neq\undf$ (since $g^{\left(  j+1\right)  }%
=T_{y}g^{\left(  j\right)  }$).

Next, it is easy to see that $g^{\left(  j+1\right)  }\left(  1\right)
=f^{\left(  j+1\right)  }\left(  1\right)  $\ \ \ \ \footnote{\textit{Proof.}
Recall that $g^{\left(  j\right)  }\left(  1\right)  =f^{\left(  j\right)
}\left(  1\right)  $. However, $1\neq y$ (since $1\notin P$ but $y\in P$).
Thus, Proposition \ref{prop.Tv} \textbf{(a)} (applied to $y$, $g^{\left(
j\right)  }$ and $1$ instead of $v$, $f$ and $w$) yields $\left(
T_{y}g^{\left(  j\right)  }\right)  \left(  1\right)  =g^{\left(  j\right)
}\left(  1\right)  $ (since $T_{y}g^{\left(  j\right)  }\neq\undf$). In view
of $g^{\left(  j+1\right)  }=T_{y}g^{\left(  j\right)  }$, we can rewrite this
as $g^{\left(  j+1\right)  }\left(  1\right)  =g^{\left(  j\right)  }\left(
1\right)  $. Also, Proposition \ref{prop.Tv} \textbf{(a)} (applied to $y$,
$f^{\left(  j\right)  }$ and $1$ instead of $v$, $f$ and $w$) yields $\left(
T_{y}f^{\left(  j\right)  }\right)  \left(  1\right)  =f^{\left(  j\right)
}\left(  1\right)  $ (since $T_{y}f^{\left(  j\right)  }=f^{\left(
j+1\right)  }\neq\undf$). In view of $f^{\left(  j+1\right)  }=T_{y}f^{\left(
j\right)  }$, we can rewrite this as $f^{\left(  j+1\right)  }\left(
1\right)  =f^{\left(  j\right)  }\left(  1\right)  $. Hence, $f^{\left(
j\right)  }\left(  1\right)  =f^{\left(  j+1\right)  }\left(  1\right)  $.
Combining what we have shown so far, we obtain%
\[
g^{\left(  j+1\right)  }\left(  1\right)  =g^{\left(  j\right)  }\left(
1\right)  =f^{\left(  j\right)  }\left(  1\right)  =f^{\left(  j+1\right)
}\left(  1\right)  .
\]
} and $f^{\left(  j+1\right)  }\left(  0\right)  =a\cdot g^{\left(
j+1\right)  }\left(  0\right)  $\ \ \ \ \footnote{\textit{Proof.} Recall that
$f^{\left(  j\right)  }\left(  0\right)  =a\cdot g^{\left(  j\right)  }\left(
0\right)  $. However, $0\neq y$ (since $0\notin P$ but $y\in P$). Thus,
Proposition \ref{prop.Tv} \textbf{(a)} (applied to $y$, $g^{\left(  j\right)
}$ and $0$ instead of $v$, $f$ and $w$) yields $\left(  T_{y}g^{\left(
j\right)  }\right)  \left(  0\right)  =g^{\left(  j\right)  }\left(  0\right)
$ (since $T_{y}g^{\left(  j\right)  }\neq\undf$). In view of $g^{\left(
j+1\right)  }=T_{y}g^{\left(  j\right)  }$, we can rewrite this as $g^{\left(
j+1\right)  }\left(  0\right)  =g^{\left(  j\right)  }\left(  0\right)  $.
Also, Proposition \ref{prop.Tv} \textbf{(a)} (applied to $y$, $f^{\left(
j\right)  }$ and $0$ instead of $v$, $f$ and $w$) yields $\left(
T_{y}f^{\left(  j\right)  }\right)  \left(  0\right)  =f^{\left(  j\right)
}\left(  0\right)  $ (since $T_{y}f^{\left(  j\right)  }=f^{\left(
j+1\right)  }\neq\undf$). In view of $f^{\left(  j+1\right)  }=T_{y}f^{\left(
j\right)  }$, we can rewrite this as $f^{\left(  j+1\right)  }\left(
0\right)  =f^{\left(  j\right)  }\left(  0\right)  $. Hence,%
\[
f^{\left(  j+1\right)  }\left(  0\right)  =f^{\left(  j\right)  }\left(
0\right)  =a\cdot\underbrace{g^{\left(  j\right)  }\left(  0\right)
}_{\substack{=g^{\left(  j+1\right)  }\left(  0\right)  \\\text{(since
}g^{\left(  j+1\right)  }\left(  0\right)  =g^{\left(  j\right)  }\left(
0\right)  \text{)}}}=a\cdot g^{\left(  j+1\right)  }\left(  0\right)  .
\]
}.

We now prove that each $v\in P$ satisfies (\ref{pf.lem.ato1.c1.IG}).

\begin{proof}
[Proof of (\ref{pf.lem.ato1.c1.IG}).]Let $v\in P$. We must prove
(\ref{pf.lem.ato1.c1.IG}). We are in one of the following two cases:

\textit{Case 1:} We have $v\neq y$.

\textit{Case 2:} We have $v=y$.

Let us first consider Case 1. In this case, we have $v\neq y$. Hence, the
statement \textquotedblleft$\rho\left(  v\right)  \leq m-j$\textquotedblright%
\ is equivalent to \textquotedblleft$\rho\left(  v\right)  \leq m-\left(
j+1\right)  $\textquotedblright\ \ \ \ \footnote{\textit{Proof.} Recall that
$\rho\left(  v\right)  $ is the unique $i\in\left\{  1,2,\ldots,m\right\}  $
that satisfies $v=v_{i}$ (by the definition of $\rho\left(  v\right)  $).
Hence, $v=v_{\rho\left(  v\right)  }$. Therefore, $v_{\rho\left(  v\right)
}=v\neq y=v_{m-j}$, so that $\rho\left(  v\right)  \neq m-j$ (because if we
had $\rho\left(  v\right)  =m-j$, then we would have $v_{\rho\left(  v\right)
}=v_{m-j}$, which would contradict $v_{\rho\left(  v\right)  }\neq v_{m-j}$).
In other words, we don't have $\rho\left(  v\right)  =m-j$.
\par
Now, we have the following chain of equivalences:%
\begin{align*}
\left(  \rho\left(  v\right)  \leq m-j\right)  \  &  \Longleftrightarrow
\ \left(  \rho\left(  v\right)  <m-j\text{ or }\rho\left(  v\right)
=m-j\right) \\
&  \Longleftrightarrow\ \left(  \rho\left(  v\right)  <m-j\right)
\ \ \ \ \ \ \ \ \ \ \left(  \text{since we don't have }\rho\left(  v\right)
=m-j\right) \\
&  \Longleftrightarrow\ \left(  \rho\left(  v\right)  \leq\left(  m-j\right)
-1\right)  \ \ \ \ \ \ \ \ \ \ \left(  \text{since }\rho\left(  v\right)
\text{ and }m-j\text{ are integers}\right) \\
&  \Longleftrightarrow\ \left(  \rho\left(  v\right)  \leq m-\left(
j+1\right)  \right)  \ \ \ \ \ \ \ \ \ \ \left(  \text{since }\left(
m-j\right)  -1=m-\left(  j+1\right)  \right)  .
\end{align*}
In other words, the statement \textquotedblleft$\rho\left(  v\right)  \leq
m-j$\textquotedblright\ is equivalent to \textquotedblleft$\rho\left(
v\right)  \leq m-\left(  j+1\right)  $\textquotedblright.}.

However, Proposition \ref{prop.Tv} \textbf{(a)} (applied to $y$, $g^{\left(
j\right)  }$ and $v$ instead of $v$, $f$ and $w$) yields $\left(
T_{y}g^{\left(  j\right)  }\right)  \left(  v\right)  =g^{\left(  j\right)
}\left(  v\right)  $ (since $v\neq y$ and $T_{y}g^{\left(  j\right)  }%
\neq\undf$). In view of $g^{\left(  j+1\right)  }=T_{y}g^{\left(  j\right)  }%
$, we can rewrite this as $g^{\left(  j+1\right)  }\left(  v\right)
=g^{\left(  j\right)  }\left(  v\right)  $. In other words, $g^{\left(
j\right)  }\left(  v\right)  =g^{\left(  j+1\right)  }\left(  v\right)  $.
Also, Proposition \ref{prop.Tv} \textbf{(a)} (applied to $y$, $f^{\left(
j\right)  }$ and $v$ instead of $v$, $f$ and $w$) yields $\left(
T_{y}f^{\left(  j\right)  }\right)  \left(  v\right)  =f^{\left(  j\right)
}\left(  v\right)  $ (since $v\neq y$ and $T_{y}f^{\left(  j\right)  }%
\neq\undf$). In view of $f^{\left(  j+1\right)  }=T_{y}f^{\left(  j\right)  }%
$, we can rewrite this as $f^{\left(  j+1\right)  }\left(  v\right)
=f^{\left(  j\right)  }\left(  v\right)  $. Hence,%
\begin{align*}
f^{\left(  j+1\right)  }\left(  v\right)   &  =f^{\left(  j\right)  }\left(
v\right)  =%
\begin{cases}
g^{\left(  j\right)  }\left(  v\right)  , & \text{if }v\notin M\text{ or }%
\rho\left(  v\right)  \leq m-j;\\
a\cdot g^{\left(  j\right)  }\left(  v\right)  , & \text{otherwise}%
\end{cases}
\ \ \ \ \ \ \ \ \ \ \left(  \text{by (\ref{pf.lem.ato1.c1.IH})}\right) \\
&  =%
\begin{cases}
g^{\left(  j+1\right)  }\left(  v\right)  , & \text{if }v\notin M\text{ or
}\rho\left(  v\right)  \leq m-j;\\
a\cdot g^{\left(  j+1\right)  }\left(  v\right)  , & \text{otherwise}%
\end{cases}
\ \ \ \ \ \ \ \ \ \ \left(  \text{since }g^{\left(  j\right)  }\left(
v\right)  =g^{\left(  j+1\right)  }\left(  v\right)  \right) \\
&  =%
\begin{cases}
g^{\left(  j+1\right)  }\left(  v\right)  , & \text{if }v\notin M\text{ or
}\rho\left(  v\right)  \leq m-\left(  j+1\right)  ;\\
a\cdot g^{\left(  j+1\right)  }\left(  v\right)  , & \text{otherwise}%
\end{cases}
\end{align*}
(since the statement \textquotedblleft$\rho\left(  v\right)  \leq
m-j$\textquotedblright\ is equivalent to \textquotedblleft$\rho\left(
v\right)  \leq m-\left(  j+1\right)  $\textquotedblright). In other words,
(\ref{pf.lem.ato1.c1.IG}) holds. Thus, (\ref{pf.lem.ato1.c1.IG}) is proved in
Case 1.

Let us now consider Case 2. In this case, we have $v=y$. Hence, $\rho\left(
v\right)  =\rho\left(  y\right)  =m-j>m-j-1=m-\left(  j+1\right)  $. Thus, we
do not have $\rho\left(  v\right)  \leq m-\left(  j+1\right)  $. Hence, the
statement \textquotedblleft$v\notin M$ or $\rho\left(  v\right)  \leq
m-\left(  j+1\right)  $\textquotedblright\ is equivalent to \textquotedblleft%
$v\notin M$\textquotedblright.

Recall that $T_{y}f^{\left(  j\right)  }\neq\undf$ and $T_{y}g^{\left(
j\right)  }\neq\undf$. Thus, Proposition \ref{prop.Tv} \textbf{(b)} (applied
to $f^{\left(  j\right)  }$ and $y$ instead of $f$ and $v$) yields%
\[
\left(  T_{y}f^{\left(  j\right)  }\right)  \left(  y\right)  =\left(
\sum\limits_{\substack{u\in\widehat{P};\\u\lessdot y}}f^{\left(  j\right)
}\left(  u\right)  \right)  \cdot\overline{f^{\left(  j\right)  }\left(
y\right)  }\cdot\overline{\sum\limits_{\substack{u\in\widehat{P};\\u\gtrdot
y}}\overline{f^{\left(  j\right)  }\left(  u\right)  }}.
\]
In view of $f^{\left(  j+1\right)  }=T_{y}f^{\left(  j\right)  }$, we can
rewrite this as%
\begin{equation}
f^{\left(  j+1\right)  }\left(  y\right)  =\left(  \sum\limits_{\substack{u\in
\widehat{P};\\u\lessdot y}}f^{\left(  j\right)  }\left(  u\right)  \right)
\cdot\overline{f^{\left(  j\right)  }\left(  y\right)  }\cdot\overline
{\sum\limits_{\substack{u\in\widehat{P};\\u\gtrdot y}}\overline{f^{\left(
j\right)  }\left(  u\right)  }}. \label{pf.lem.ato1.c1.IG.pf.3f}%
\end{equation}

Also, recall that $T_{y}g^{\left(  j\right)  }\neq\undf$. Hence, Proposition
\ref{prop.Tv} \textbf{(b)} (applied to $g^{\left(  j\right)  }$ and $y$
instead of $f$ and $v$) yields%
\begin{align*}
\left(  T_{y}g^{\left(  j\right)  }\right)  \left(  y\right)   &  =\left(
\sum\limits_{\substack{u\in\widehat{P};\\u\lessdot y}}g^{\left(  j\right)
}\left(  u\right)  \right)  \cdot\underbrace{\overline{g^{\left(  j\right)
}\left(  y\right)  }}_{\substack{=\overline{f^{\left(  j\right)  }\left(
y\right)  }\\\text{(by (\ref{pf.lem.ato1.c1.s4}))}}}\cdot\overline
{\sum\limits_{\substack{u\in\widehat{P};\\u\gtrdot y}}\underbrace{\overline
{g^{\left(  j\right)  }\left(  u\right)  }}_{\substack{=\overline{f^{\left(
j\right)  }\left(  u\right)  }\\\text{(by (\ref{pf.lem.ato1.c1.s3}))}}}}\\
&  =\left(  \sum\limits_{\substack{u\in\widehat{P};\\u\lessdot y}}g^{\left(
j\right)  }\left(  u\right)  \right)  \cdot\overline{f^{\left(  j\right)
}\left(  y\right)  }\cdot\overline{\sum\limits_{\substack{u\in\widehat{P}%
;\\u\gtrdot y}}\overline{f^{\left(  j\right)  }\left(  u\right)  }}.
\end{align*}
In view of $g^{\left(  j+1\right)  }=T_{y}g^{\left(  j\right)  }$, we can
rewrite this as
\begin{equation}
g^{\left(  j+1\right)  }\left(  y\right)  =\left(  \sum\limits_{\substack{u\in
\widehat{P};\\u\lessdot y}}g^{\left(  j\right)  }\left(  u\right)  \right)
\cdot\overline{f^{\left(  j\right)  }\left(  y\right)  }\cdot\overline
{\sum\limits_{\substack{u\in\widehat{P};\\u\gtrdot y}}\overline{f^{\left(
j\right)  }\left(  u\right)  }}. \label{pf.lem.ato1.c1.IG.pf.3g}%
\end{equation}

Now, we are in one of the following two subcases:

\textit{Subcase 2.1:} We have $v\in M$.

\textit{Subcase 2.2:} We have $v\notin M$.

Let us first consider Subcase 2.1. In this subcase, we have $v\in M$. In other
words, $y\in M$ (since $v=y$). In other words, $y$ is a minimal element of $P$
(since $M$ is the set of all minimal elements of $P$). Hence, the only
$u\in\widehat{P}$ that satisfies $u\lessdot y$ is the element $0$ of
$\widehat{P}$. Thus,
\[
\sum\limits_{\substack{u\in\widehat{P};\\u\lessdot y}}g^{\left(  j\right)
}\left(  u\right)  =g^{\left(  j\right)  }\left(  0\right)
\ \ \ \ \ \ \ \ \ \ \text{and}\ \ \ \ \ \ \ \ \ \ \sum\limits_{\substack{u\in
\widehat{P};\\u\lessdot y}}f^{\left(  j\right)  }\left(  u\right)  =f^{\left(
j\right)  }\left(  0\right)  .
\]
Now, from $v=y$, we obtain%
\begin{align}
f^{\left(  j+1\right)  }\left(  v\right)   &  =f^{\left(  j+1\right)  }\left(
y\right) \nonumber\\
&  =\underbrace{\left(  \sum\limits_{\substack{u\in\widehat{P};\\u\lessdot
y}}f^{\left(  j\right)  }\left(  u\right)  \right)  }_{\substack{=f^{\left(
j\right)  }\left(  0\right)  \\=a\cdot g^{\left(  j\right)  }\left(  0\right)
}}\cdot\overline{f^{\left(  j\right)  }\left(  y\right)  }\cdot\overline
{\sum\limits_{\substack{u\in\widehat{P};\\u\gtrdot y}}\overline{f^{\left(
j\right)  }\left(  u\right)  }}\ \ \ \ \ \ \ \ \ \ \left(  \text{by
(\ref{pf.lem.ato1.c1.IG.pf.3f})}\right) \nonumber\\
&  =a\cdot g^{\left(  j\right)  }\left(  0\right)  \cdot\overline{f^{\left(
j\right)  }\left(  y\right)  }\cdot\overline{\sum\limits_{\substack{u\in
\widehat{P};\\u\gtrdot y}}\overline{f^{\left(  j\right)  }\left(  u\right)  }%
}. \label{pf.lem.ato1.c1.IG.pf.5}%
\end{align}

On the other hand, we don't have $v\notin M$ (since $v\in M$). Now, recall
that the statement \textquotedblleft$v\notin M$ or $\rho\left(  v\right)  \leq
m-\left(  j+1\right)  $\textquotedblright\ is equivalent to \textquotedblleft%
$v\notin M$\textquotedblright. Hence, we don't have \textquotedblleft$v\notin
M$ or $\rho\left(  v\right)  \leq m-\left(  j+1\right)  $\textquotedblright%
\ (since we don't have $v\notin M$). Therefore,%
\begin{align*}
&
\begin{cases}
g^{\left(  j+1\right)  }\left(  v\right)  , & \text{if }v\notin M\text{ or
}\rho\left(  v\right)  \leq m-\left(  j+1\right)  ;\\
a\cdot g^{\left(  j+1\right)  }\left(  v\right)  , & \text{otherwise}%
\end{cases}
\\
&  =a\cdot g^{\left(  j+1\right)  }\left(  v\right)  =a\cdot g^{\left(
j+1\right)  }\left(  y\right)  \ \ \ \ \ \ \ \ \ \ \left(  \text{since
}v=y\right) \\
&  =a\cdot\underbrace{\left(  \sum\limits_{\substack{u\in\widehat{P}%
;\\u\lessdot y}}g^{\left(  j\right)  }\left(  u\right)  \right)
}_{=g^{\left(  j\right)  }\left(  0\right)  }\cdot\overline{f^{\left(
j\right)  }\left(  y\right)  }\cdot\overline{\sum\limits_{\substack{u\in
\widehat{P};\\u\gtrdot y}}\overline{f^{\left(  j\right)  }\left(  u\right)  }%
}\ \ \ \ \ \ \ \ \ \ \left(  \text{by (\ref{pf.lem.ato1.c1.IG.pf.3g})}\right)
\\
&  =a\cdot g^{\left(  j\right)  }\left(  0\right)  \cdot\overline{f^{\left(
j\right)  }\left(  y\right)  }\cdot\overline{\sum\limits_{\substack{u\in
\widehat{P};\\u\gtrdot y}}\overline{f^{\left(  j\right)  }\left(  u\right)  }%
}.
\end{align*}
Comparing this with (\ref{pf.lem.ato1.c1.IG.pf.5}), we obtain%
\[
f^{\left(  j+1\right)  }\left(  v\right)  =%
\begin{cases}
g^{\left(  j+1\right)  }\left(  v\right)  , & \text{if }v\notin M\text{ or
}\rho\left(  v\right)  \leq m-\left(  j+1\right)  ;\\
a\cdot g^{\left(  j+1\right)  }\left(  v\right)  , & \text{otherwise.}%
\end{cases}
\]
In other words, (\ref{pf.lem.ato1.c1.IG}) holds. Thus, we have proved
(\ref{pf.lem.ato1.c1.IG}) in Subcase 2.1.

Now, let us consider Subcase 2.2. In this subcase, we have $v\notin M$. In
other words, $v$ is not a minimal element of $P$ (since $M$ is the set of all
minimal elements of $P$). Thus, we don't have $0\lessdot v$ in $\widehat{P}$.
In other words, we don't have $0\lessdot y$ in $\widehat{P}$ (since $v=y$).

For each $u\in\widehat{P}$ satisfying $u\lessdot y$, we have%
\begin{equation}
g^{\left(  j\right)  }\left(  u\right)  =f^{\left(  j\right)  }\left(
u\right)  \label{pf.lem.ato1.c1.s5c22}%
\end{equation}
\footnote{\textit{Proof of (\ref{pf.lem.ato1.c1.s5c22}):} Let $u\in
\widehat{P}$ be such that $u\lessdot y$. Then, $u\lessdot y$, so that $u<y$.
Thus, $u\neq1$ (because if we had $u=1$, then we would have $1=u<y$, which
would contradict the fact that $1$ is not smaller than any element of
$\widehat{P}$). Moreover, $u\neq0$ (because if we had $u=0$, then we would
have $0=u\lessdot y$, which would contradict the fact that we don't have
$0\lessdot y$ in $\widehat{P}$). Combining $u\in\widehat{P}$ with $u\neq0$ and
$u\neq1$, we obtain $u\in\widehat{P}\setminus\left\{  0,1\right\}  =P$. Hence,
$\rho\left(  u\right)  $ is well-defined.
\par
The definition of $\rho\left(  u\right)  $ shows that $\rho\left(  u\right)  $
is the unique $i\in\left\{  1,2,\ldots,m\right\}  $ that satisfies $u=v_{i}$.
Hence, $u=v_{\rho\left(  u\right)  }$.
\par
Recall that $\left(  v_{1},v_{2},\ldots,v_{m}\right)  $ is a linear extension
of $P$. Thus, any $k\in\left\{  1,2,\ldots,m\right\}  $ and $\ell\in\left\{
1,2,\ldots,m\right\}  $ satisfying $v_{k}<v_{\ell}$ must satisfy $k<\ell$ (by
the definition of a linear extension). We can apply this to $k=\rho\left(
u\right)  $ and $\ell=m-j$ (since $v_{\rho\left(  u\right)  }=u<y=v_{m-j}$),
and thus obtain $\rho\left(  u\right)  <m-j$. Hence, $\rho\left(  u\right)
\leq m-j$. Therefore, $u\notin M$ or $\rho\left(  u\right)  \leq m-j$. Now,
(\ref{pf.lem.ato1.c1.IH}) (applied to $u$ instead of $v$) yields
\begin{align*}
f^{\left(  j\right)  }\left(  u\right)   &  =%
\begin{cases}
g^{\left(  j\right)  }\left(  u\right)  , & \text{if }u\notin M\text{ or }%
\rho\left(  u\right)  \leq m-j;\\
a\cdot g^{\left(  j\right)  }\left(  u\right)  , & \text{otherwise}%
\end{cases}
\\
&  =g^{\left(  j\right)  }\left(  u\right)  \ \ \ \ \ \ \ \ \ \ \left(
\text{since }u\notin M\text{ or }\rho\left(  u\right)  \leq m-j\right)  .
\end{align*}
In other words, $g^{\left(  j\right)  }\left(  u\right)  =f^{\left(  j\right)
}\left(  u\right)  $. This proves (\ref{pf.lem.ato1.c1.s5c22}).}. Now, from
$v=y$, we obtain%
\begin{align}
f^{\left(  j+1\right)  }\left(  v\right)   &  =f^{\left(  j+1\right)  }\left(
y\right) \nonumber\\
&  =\left(  \sum\limits_{\substack{u\in\widehat{P};\\u\lessdot y}}f^{\left(
j\right)  }\left(  u\right)  \right)  \cdot\overline{f^{\left(  j\right)
}\left(  y\right)  }\cdot\overline{\sum\limits_{\substack{u\in\widehat{P}%
;\\u\gtrdot y}}\overline{f^{\left(  j\right)  }\left(  u\right)  }%
}\ \ \ \ \ \ \ \ \ \ \left(  \text{by (\ref{pf.lem.ato1.c1.IG.pf.3f})}\right)
.\nonumber
\end{align}
Comparing this with%
\begin{align*}
&
\begin{cases}
g^{\left(  j+1\right)  }\left(  v\right)  , & \text{if }v\notin M\text{ or
}\rho\left(  v\right)  \leq m-\left(  j+1\right)  ;\\
a\cdot g^{\left(  j+1\right)  }\left(  v\right)  , & \text{otherwise}%
\end{cases}
\\
&  =g^{\left(  j+1\right)  }\left(  v\right)  \ \ \ \ \ \ \ \ \ \ \left(
\text{since }v\notin M\text{ or }\rho\left(  v\right)  \leq m-\left(
j+1\right)  \text{ (because }v\notin M\text{)}\right) \\
&  =g^{\left(  j+1\right)  }\left(  y\right)  \ \ \ \ \ \ \ \ \ \ \left(
\text{since }v=y\right) \\
&  =\left(  \sum\limits_{\substack{u\in\widehat{P};\\u\lessdot y}%
}\underbrace{g^{\left(  j\right)  }\left(  u\right)  }_{\substack{=f^{\left(
j\right)  }\left(  u\right)  \\\text{(by (\ref{pf.lem.ato1.c1.s5c22}))}%
}}\right)  \cdot\overline{f^{\left(  j\right)  }\left(  y\right)  }%
\cdot\overline{\sum\limits_{\substack{u\in\widehat{P};\\u\gtrdot y}%
}\overline{f^{\left(  j\right)  }\left(  u\right)  }}%
\ \ \ \ \ \ \ \ \ \ \left(  \text{by (\ref{pf.lem.ato1.c1.IG.pf.3g})}\right)
\\
&  =\left(  \sum\limits_{\substack{u\in\widehat{P};\\u\lessdot y}}f^{\left(
j\right)  }\left(  u\right)  \right)  \cdot\overline{f^{\left(  j\right)
}\left(  y\right)  }\cdot\overline{\sum\limits_{\substack{u\in\widehat{P}%
;\\u\gtrdot y}}\overline{f^{\left(  j\right)  }\left(  u\right)  }},
\end{align*}
we obtain%
\[
f^{\left(  j+1\right)  }\left(  v\right)  =%
\begin{cases}
g^{\left(  j+1\right)  }\left(  v\right)  , & \text{if }v\notin M\text{ or
}\rho\left(  v\right)  \leq m-\left(  j+1\right)  ;\\
a\cdot g^{\left(  j+1\right)  }\left(  v\right)  , & \text{otherwise.}%
\end{cases}
\]
Thus, we have proved (\ref{pf.lem.ato1.c1.IG}) in Subcase 2.2.

We have now proved (\ref{pf.lem.ato1.c1.IG}) in both Subcases 2.1 and 2.2.
Since these two Subcases cover all of Case 2, we thus have proved
(\ref{pf.lem.ato1.c1.IG}) in Case 2.

We have now proved (\ref{pf.lem.ato1.c1.IG}) in both Cases 1 and 2. Therefore,
(\ref{pf.lem.ato1.c1.IG}) always holds. This completes the proof of
(\ref{pf.lem.ato1.c1.IG}).
\end{proof}

Altogether, we have now proved that $f^{\left(  j+1\right)  }\neq\undf$ and
$g^{\left(  j+1\right)  }\neq\undf$ and $g^{\left(  j+1\right)  }\left(
1\right)  =f^{\left(  j+1\right)  }\left(  1\right)  $ and $f^{\left(
j+1\right)  }\left(  0\right)  =a\cdot g^{\left(  j+1\right)  }\left(
0\right)  $ and that each $v\in P$ satisfies (\ref{pf.lem.ato1.c1.IG}). In
other words, Claim 1 holds for $i=j+1$. This completes the induction step.
Thus, Claim 1 is proven.
\end{proof}

In order to finish our proof of Lemma \ref{lem.ato1}, we now apply Claim 1 to
$i=m$:

Claim 1 (applied to $i=m$) shows that $f^{\left(  m\right)  }\neq\undf$ and
$g^{\left(  m\right)  }\neq\undf$ and $g^{\left(  m\right)  }\left(  1\right)
=f^{\left(  m\right)  }\left(  1\right)  $ and $f^{\left(  m\right)  }\left(
0\right)  =a\cdot g^{\left(  m\right)  }\left(  0\right)  $, and that each
$v\in P$ satisfies%
\begin{equation}
f^{\left(  m\right)  }\left(  v\right)  =%
\begin{cases}
g^{\left(  m\right)  }\left(  v\right)  , & \text{if }v\notin M\text{ or }%
\rho\left(  v\right)  \leq m-m;\\
a\cdot g^{\left(  m\right)  }\left(  v\right)  , & \text{otherwise.}%
\end{cases}
\label{pf.lem.ato1.fmv=}%
\end{equation}
The definition of $g^{\left(  m\right)  }$ yields $g^{\left(  m\right)
}=\underbrace{R_{m-m}}_{=R_{0}=R}g=Rg$. The definition of $f^{\left(
m\right)  }$ yields $f^{\left(  m\right)  }=\underbrace{R_{m-m}}_{=R_{0}%
=R}f=Rf$.

Now, the three parts of Lemma \ref{lem.ato1} easily follow: \medskip

\textbf{(a)} From $g^{\left(  m\right)  }=Rg$, we obtain $Rg=g^{\left(
m\right)  }\neq\undf$. This proves Lemma \ref{lem.ato1} \textbf{(a)}. \medskip

\textbf{(b)} Let $v\in P$ be not a minimal element of $P$. Thus, $v\notin M$
(since $M$ is the set of all minimal elements of $P$). Moreover, from
$f^{\left(  m\right)  }=Rf$, we obtain $Rf=f^{\left(  m\right)  }$, and thus%
\begin{align*}
\left(  Rf\right)  \left(  v\right)   &  =f^{\left(  m\right)  }\left(
v\right)  =%
\begin{cases}
g^{\left(  m\right)  }\left(  v\right)  , & \text{if }v\notin M\text{ or }%
\rho\left(  v\right)  \leq m-m;\\
a\cdot g^{\left(  m\right)  }\left(  v\right)  , & \text{otherwise}%
\end{cases}
\ \ \ \ \ \ \ \ \ \ \left(  \text{by (\ref{pf.lem.ato1.fmv=})}\right) \\
&  =\underbrace{g^{\left(  m\right)  }}_{=Rg}\left(  v\right)
\ \ \ \ \ \ \ \ \ \ \left(  \text{since }v\notin M\text{ or }\rho\left(
v\right)  \leq m-m\text{ (because }v\notin M\text{)}\right) \\
&  =\left(  Rg\right)  \left(  v\right)  .
\end{align*}
This proves Lemma \ref{lem.ato1} \textbf{(b)}. \medskip

\textbf{(c)} Let $v\in P$ be a minimal element of $P$. Thus, $v\in M$ (since
$M$ is the set of all minimal elements of $P$). Hence, we don't have $v\notin
M$.

Moreover, $\rho\left(  v\right)  \in\left\{  1,2,\ldots,m\right\}  $ (by the
definition of $\rho\left(  v\right)  $) and therefore $\rho\left(  v\right)
\geq1>0=m-m$. Thus, we don't have $\rho\left(  v\right)  \leq m-m$.

From $f^{\left(  m\right)  }=Rf$, we obtain $Rf=f^{\left(  m\right)  }$, and
thus%
\begin{align*}
\left(  Rf\right)  \left(  v\right)   &  =f^{\left(  m\right)  }\left(
v\right)  =%
\begin{cases}
g^{\left(  m\right)  }\left(  v\right)  , & \text{if }v\notin M\text{ or }%
\rho\left(  v\right)  \leq m-m;\\
a\cdot g^{\left(  m\right)  }\left(  v\right)  , & \text{otherwise}%
\end{cases}
\ \ \ \ \ \ \ \ \ \ \left(  \text{by (\ref{pf.lem.ato1.fmv=})}\right) \\
&  =a\cdot\underbrace{g^{\left(  m\right)  }}_{=Rg}\left(  v\right)
\ \ \ \ \ \ \ \ \ \ \left(
\begin{array}
[c]{c}%
\text{since we don't have \textquotedblleft}v\notin M\text{ or }\rho\left(
v\right)  \leq m-m\text{\textquotedblright}\\
\text{(because we don't have }v\notin M\text{,}\\
\text{and we don't have }\rho\left(  v\right)  \leq m-m\text{)}%
\end{array}
\right) \\
&  =a\cdot\left(  Rg\right)  \left(  v\right)  .
\end{align*}
This proves Lemma \ref{lem.ato1} \textbf{(c)}.
\end{proof}
\end{verlong}

Let us now get rid of the \textquotedblleft$a$ is invertible\textquotedblright%
\ requirement in Lemma \ref{lem.rect.antip.11inv}:

\begin{lemma}
\label{lem.rect.antip.11full}Assume that $P$ is the $p\times q$-rectangle
$\left[  p\right]  \times\left[  q\right]  $. Let $\ell\in\mathbb{N}$ be such
that $\ell\geq1$. Let $f\in\mathbb{K}^{\widehat{P}}$ be a $\mathbb{K}%
$-labeling such that $R^{\ell}f\neq\undf$. Let $a=f\left(  0\right)  $ and
$b=f\left(  1\right)  $. Then,%
\[
\left(  R^{\ell}f\right)  \left(  1,1\right)  =a\cdot\overline{\left(
R^{\ell-1}f\right)  \left(  p,q\right)  }\cdot b.
\]

\end{lemma}

\begin{vershort}

\begin{proof}
If $R^{2}f\neq\undf$, then Lemma \ref{lem.R.01inv} yields that $a$ and $b$ are
invertible (since $a=f\left(  0\right)  $ and $b=f\left(  1\right)  $), and
therefore our claim follows directly from Lemma \ref{lem.rect.antip.11inv}.
For this reason, we WLOG assume that $R^{2}f=\undf$. If we had $\ell\geq2$,
then we would thus conclude that $R^{\ell}f=\undf$ as well, which would
contradict $R^{\ell}f\neq\undf$. Hence, we must have $\ell<2$, so that
$\ell=1$. Therefore, $R^{\ell-1}=R^{1-1}=R^{0}=\operatorname*{id}$ and
consequently $\left(  R^{\ell-1}f\right)  \left(  p,q\right)  =f\left(
p,q\right)  $. Also, $R^{\ell}=R$ (since $\ell=1$). Hence, $R=R^{\ell}$, so
that $Rf=R^{\ell}f\neq\undf$.

Now, let $g\in\mathbb{K}^{\widehat{P}}$ be the $\mathbb{K}$-labeling that is
obtained from $f$ by replacing the label $f\left(  0\right)  $ by $1$. Thus,
we have
\begin{equation}
g\left(  x\right)  =f\left(  x\right)  \ \ \ \ \ \ \ \ \ \ \text{for each
}x\in\widehat{P}\setminus\left\{  0\right\}  ,
\label{pf.lem.rect.antip.11full.short.4}%
\end{equation}
and we have $g\left(  0\right)  =1$. Then, Lemma \ref{lem.ato1} \textbf{(a)}
yields $Rg\neq\undf$. In other words, $R^{1}g\neq\undf$.

We have $\left(  p,q\right)  \in P\subseteq\widehat{P}\setminus\left\{
0\right\}  $. Hence, (\ref{pf.lem.rect.antip.11full.short.4}) yields $g\left(
p,q\right)  =f\left(  p,q\right)  $.

We also have $1\in\widehat{P}\setminus\left\{  0\right\}  $. Thus,
(\ref{pf.lem.rect.antip.11full.short.4}) yields $g\left(  1\right)  =f\left(
1\right)  =b$, so that $b=g\left(  1\right)  $. Also, $1=g\left(  0\right)  $,
and clearly $1$ is invertible. Hence, Lemma \ref{lem.rect.antip.11inv}
(applied to $1$, $g$ and $1$ instead of $\ell$, $f$ and $a$) yields
\[
\left(  R^{1}g\right)  \left(  1,1\right)  =1\cdot\overline{\left(
R^{1-1}g\right)  \left(  p,q\right)  }\cdot b=\overline{\left(  R^{1-1}%
g\right)  \left(  p,q\right)  }\cdot b.
\]
In view of $R^{1}=R$ and $R^{1-1}=\operatorname*{id}$, we can rewrite this as
\[
\left(  Rg\right)  \left(  1,1\right)  =\overline{g\left(  p,q\right)  }\cdot
b.
\]

However, $\left(  1,1\right)  $ is a minimal element of $P$. Thus, Lemma
\ref{lem.ato1} \textbf{(c)} (applied to $v=\left(  1,1\right)  $) yields
\[
\left(  Rf\right)  \left(  1,1\right)  =a\cdot\underbrace{\left(  Rg\right)
\left(  1,1\right)  }_{=\overline{g\left(  p,q\right)  }\cdot b}%
=a\cdot\overline{g\left(  p,q\right)  }\cdot b=a\cdot\overline{f\left(
p,q\right)  }\cdot b\ \ \ \ \ \ \ \ \ \ \left(  \text{since }g\left(
p,q\right)  =f\left(  p,q\right)  \right)  .
\]
In view of $R^{\ell}=R$ and $\left(  R^{\ell-1}f\right)  \left(  p,q\right)
=f\left(  p,q\right)  $, we can rewrite this as%
\[
\left(  R^{\ell}f\right)  \left(  1,1\right)  =a\cdot\overline{\left(
R^{\ell-1}f\right)  \left(  p,q\right)  }\cdot b.
\]
Thus, Lemma \ref{lem.rect.antip.11full} is proven.
\end{proof}
\end{vershort}

\begin{verlong}

\begin{proof}
If $\ell\geq2$, then we can easily see that $a$ is
invertible\footnote{\textit{Proof.} Assume that $\ell\geq2$. Thus, $2\leq\ell
$. Hence, from $R^{\ell}f\neq\undf$, we obtain $R^{2}f\neq\undf$ (by Lemma
\ref{lem.R.wd-triv}). Hence, Lemma \ref{lem.R.01inv} yields that $f\left(
0\right)  $ and $f\left(  1\right)  $ are invertible. In other words, $a$ and
$b$ are invertible (since $a=f\left(  0\right)  $ and $b=f\left(  1\right)
$). This proves that $a$ is invertible.}. Hence, if $\ell\geq2$, then Lemma
\ref{lem.rect.antip.11full} follows immediately from Lemma
\ref{lem.rect.antip.11inv}. Thus, for the rest of this proof, we WLOG assume
that we \textbf{don't} have $\ell\geq2$. Hence, $\ell=1$ (since $\ell\geq1$).
Therefore, $R^{\ell-1}=R^{1-1}=R^{0}=\operatorname*{id}$, so that $R^{\ell
-1}f=\operatorname*{id}f=f$ and therefore $\left(  R^{\ell-1}f\right)  \left(
p,q\right)  =f\left(  p,q\right)  $. Also, $R^{\ell}=R$ (since $\ell=1$).
Hence, $R=R^{\ell}$, so that $Rf=R^{\ell}f\neq\undf$.

Now, let $g\in\mathbb{K}^{\widehat{P}}$ be the $\mathbb{K}$-labeling that is
obtained from $f$ by replacing the label $f\left(  0\right)  $ by $1$. Thus,
we have
\begin{equation}
g\left(  x\right)  =f\left(  x\right)  \ \ \ \ \ \ \ \ \ \ \text{for each
}x\in\widehat{P}\setminus\left\{  0\right\}  ,
\label{pf.lem.rect.antip.11full.4}%
\end{equation}
and we have $g\left(  0\right)  =1$. Then, Lemma \ref{lem.ato1} \textbf{(a)}
yields $Rg\neq\undf$. In other words, $R^{1}g\neq\undf$.

Note that $\left(  p,q\right)  \in P\subseteq\widehat{P}\setminus\left\{
0\right\}  $. Hence, (\ref{pf.lem.rect.antip.11full.4}) (applied to $x=\left(
p,q\right)  $) yields $g\left(  p,q\right)  =f\left(  p,q\right)  $.

We have $1\in\widehat{P}\setminus\left\{  0\right\}  $. Thus, applying
(\ref{pf.lem.rect.antip.11full.4}) to $x=1$, we obtain $g\left(  1\right)
=f\left(  1\right)  =b$, so that $b=g\left(  1\right)  $. Also, $1=g\left(
0\right)  $, and clearly $1$ is invertible. Hence, Lemma
\ref{lem.rect.antip.11inv} (applied to $1$, $g$ and $1$ instead of $\ell$, $f$
and $a$) yields
\[
\left(  R^{1}g\right)  \left(  1,1\right)  =1\cdot\overline{\left(
R^{1-1}g\right)  \left(  p,q\right)  }\cdot b=\overline{\left(  R^{1-1}%
g\right)  \left(  p,q\right)  }\cdot b.
\]
In view of $R^{1}=R$ and $\underbrace{R^{1-1}}_{=R^{0}=\operatorname*{id}%
}g=\operatorname*{id}g=g$, we can rewrite this as
\[
\left(  Rg\right)  \left(  1,1\right)  =\overline{g\left(  p,q\right)  }\cdot
b.
\]

However, $\left(  1,1\right)  $ is a minimal element of $P$. Thus, Lemma
\ref{lem.ato1} \textbf{(c)} (applied to $v=\left(  1,1\right)  $) yields
\[
\left(  Rf\right)  \left(  1,1\right)  =a\cdot\underbrace{\left(  Rg\right)
\left(  1,1\right)  }_{=\overline{g\left(  p,q\right)  }\cdot b}%
=a\cdot\overline{g\left(  p,q\right)  }\cdot b=a\cdot\overline{f\left(
p,q\right)  }\cdot b\ \ \ \ \ \ \ \ \ \ \left(  \text{since }g\left(
p,q\right)  =f\left(  p,q\right)  \right)  .
\]
In view of $R^{\ell}=R$ and $\left(  R^{\ell-1}f\right)  \left(  p,q\right)
=f\left(  p,q\right)  $, we can rewrite this as%
\[
\left(  R^{\ell}f\right)  \left(  1,1\right)  =a\cdot\overline{\left(
R^{\ell-1}f\right)  \left(  p,q\right)  }\cdot b.
\]
Thus, Lemma \ref{lem.rect.antip.11full} is proven.
\end{proof}
\end{verlong}

This settles the easiest case of Theorem \ref{thm.rect.antip} -- namely, the
case $\left(  i,j\right)  =\left(  1,1\right)  $. To get a grip on the general
case, we need more lemmas.

\section{\label{sec.conversion}The conversion lemma}

We continue using the notations from Section \ref{sec.proof-nots}.


\begin{lemma}
[Four neighbors lemma]\label{lem.rect.four}Let $u,v,w,d$ be four adjacent
elements of $P$ that are arranged as follows on the Hasse diagram of $P$:
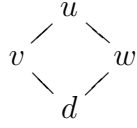
\[%
\xymatrix@R=0.6pc@C=0.6pc{
& u \are[dl] \are[dr] \\
v \are[dr] & & w \are[dl] \\
& d
}%
\]
(i.e., we have $d=\left(  i,j\right)  $, $v=\left(  i+1,j\right)  $,
$w=\left(  i,j+1\right)  $ and $u=\left(  i+1,j+1\right)  $ for some
$i\in\left[  p-1\right]  $ and some $j\in\left[  q-1\right]  $).

Assume that $a$ is invertible. Let $\ell\geq1$ be such that $R^{\ell+1}%
f\neq\undf$. Then:

\begin{enumerate}
\item[\textbf{(a)}] We have%
\[
\overline{v_{\ell}}\cdot\upslack_{\ell}^{d}\cdot d_{\ell}=\overline{u_{\ell}%
}\cdot\downslack_{\ell}^{u}\cdot w_{\ell}.
\]

\item[\textbf{(b)}] We have%
\[
\overline{w_{\ell}}\cdot\upslack_{\ell}^{d}\cdot d_{\ell}=\overline{u_{\ell}%
}\cdot\downslack_{\ell}^{u}\cdot v_{\ell}.
\]

\end{enumerate}


\end{lemma}

\begin{proof}
\textbf{(a)} We have $R\left(  R^{\ell}f\right)  =R^{\ell+1}f\neq
\undf=R\left(  \undf\right)  $ and thus $R^{\ell}f\neq\undf$. Hence, Lemma
\ref{lem.slacks.wd} \textbf{(a)} yields that $v_{\ell}$ is invertible.
Similarly, $w_{\ell}$ and $u_{\ell}$ and $d_{\ell}$ are invertible. Also,
Lemma \ref{lem.slacks.wd} \textbf{(d)} (applied to $d$ instead of $v$) yields
that the element $\upslack_{\ell}^{d}$ is well-defined and invertible.
Moreover, Lemma \ref{lem.slacks.wd} \textbf{(c)} (applied to $u$ and $\ell+1$
instead of $v$ and $\ell$) yields that the element $\downslack_{\ell}^{u}$ is
well-defined and invertible.

The elements $s\in\widehat{P}$ that satisfy $s\gtrdot d$ are $v$ and $w$.
Hence, $\sum\limits_{s\gtrdot d}\overline{s_{\ell}}=\overline{v_{\ell}%
}+\overline{w_{\ell}}$ (where, of course, the sum ranges over $s\in
\widehat{P}$). Now, the definition of $\upslack_{\ell}^{d}$ yields%
\begin{equation}
\upslack_{\ell}^{d}=\overline{\sum\limits_{s\gtrdot d}\overline{s_{\ell}}%
}\cdot\overline{d_{\ell}}=\overline{\overline{v_{\ell}}+\overline{w_{\ell}}%
}\cdot\overline{d_{\ell}} \label{pf.lem.rect.four.1}%
\end{equation}
(since $\sum\limits_{s\gtrdot d}\overline{s_{\ell}}=\overline{v_{\ell}%
}+\overline{w_{\ell}}$).

The elements $s\in\widehat{P}$ that satisfy $s\lessdot u$ are $v$ and $w$.
Hence, $\sum\limits_{s\lessdot u}s_{\ell}=v_{\ell}+w_{\ell}$. Now, the
definition of $\downslack_{\ell}^{u}$ yields%
\begin{equation}
\downslack_{\ell}^{u}=u_{\ell}\cdot\overline{\sum\limits_{s\lessdot u}s_{\ell
}}=u_{\ell}\cdot\overline{v_{\ell}+w_{\ell}} \label{pf.lem.rect.four.2}%
\end{equation}
(since $\sum\limits_{s\lessdot u}s_{\ell}=v_{\ell}+w_{\ell}$). Since this is
well-defined, the element $v_{\ell}+w_{\ell}$ of $\mathbb{K}$ must be
invertible. Also, we already know that $v_{\ell}$ and $w_{\ell}$ are
invertible. Hence, Proposition \ref{prop.inverses.a+b} \textbf{(b)} (applied
to $v_{\ell}$ and $w_{\ell}$ instead of $a$ and $b$) yields that
$\overline{v_{\ell}}+\overline{w_{\ell}}$ is invertible as well and its
inverse is%
\[
\overline{\overline{v_{\ell}}+\overline{w_{\ell}}}=v_{\ell}\cdot
\overline{v_{\ell}+w_{\ell}}\cdot w_{\ell}.
\]

Now,%
\[
\overline{v_{\ell}}\cdot\underbrace{\upslack_{\ell}^{d}}_{\substack{=\overline
{\overline{v_{\ell}}+\overline{w_{\ell}}}\cdot\overline{d_{\ell}}\\\text{(by
(\ref{pf.lem.rect.four.1}))}}}\cdot\,d_{\ell}=\overline{v_{\ell}}%
\cdot\underbrace{\overline{\overline{v_{\ell}}+\overline{w_{\ell}}}}%
_{=v_{\ell}\cdot\overline{v_{\ell}+w_{\ell}}\cdot w_{\ell}}\cdot
\,\underbrace{\overline{d_{\ell}}\cdot d_{\ell}}_{=1}=\underbrace{\overline
{v_{\ell}}\cdot v_{\ell}}_{=1}\cdot\,\overline{v_{\ell}+w_{\ell}}\cdot
w_{\ell}=\overline{v_{\ell}+w_{\ell}}\cdot w_{\ell}.
\]
Comparing this with%
\[
\overline{u_{\ell}}\cdot\underbrace{\downslack_{\ell}^{u}}_{\substack{=u_{\ell
}\cdot\overline{v_{\ell}+w_{\ell}}\\\text{(by (\ref{pf.lem.rect.four.2}))}%
}}\cdot\,w_{\ell}=\underbrace{\overline{u_{\ell}}\cdot u_{\ell}}_{=1}%
\cdot\,\overline{v_{\ell}+w_{\ell}}\cdot w_{\ell}=\overline{v_{\ell}+w_{\ell}%
}\cdot w_{\ell},
\]
we obtain $\overline{v_{\ell}}\cdot\upslack_{\ell}^{d}\cdot d_{\ell}%
=\overline{u_{\ell}}\cdot\downslack_{\ell}^{u}\cdot w_{\ell}$. Thus, Lemma
\ref{lem.rect.four} \textbf{(a)} is proved. \medskip

\textbf{(b)} This can be proved by the same argument that we used to prove
part \textbf{(a)} (with the roles of $v$ and $w$ interchanged).
\end{proof}

We recall our conventions for drawing the $p\times q$-rectangle $P=\left[
p\right]  \times\left[  q\right]  $. In light of these conventions, we shall
refer to the set $\left\{  \left(  k,q\right)  \ \mid\ k\in\left[  p\right]
\right\}  $ as the \emph{northeastern edge} of $P$, and to the set $\left\{
\left(  i,1\right)  \ \mid\ i\in\left[  p\right]  \right\}  $ as the
\emph{southwestern edge} of $P$.

The next lemma is crucial, as it allows us to \textquotedblleft
convert\textquotedblright\ between $\downslack$'s and $\upslack$'s without
changing the subscript.

\begin{lemma}
[Conversion lemma]\label{lem.rect.conv}Let $u$ and $u^{\prime}$ be two
elements of the northeastern edge of $P$ satisfying $u\gtrdot u^{\prime}$
(that is, let $u=\left(  k,q\right)  $ and $u^{\prime}=\left(  k-1,q\right)  $
for some $k\in\left\{  2,3,\ldots,p\right\}  $). Let $d$ and $d^{\prime}$ be
two elements of the southwestern edge of $P$ satisfying $d\gtrdot d^{\prime}$
(that is, let $d=\left(  i,1\right)  $ and $d^{\prime}=\left(  i-1,1\right)  $
for some $i\in\left\{  2,3,\ldots,p\right\}  $).

Assume that $a$ is invertible. Let $\ell\geq1$ be such that $R^{\ell+1}%
f\neq\undf$. Then we have:
\[
\downslack_{\ell}^{u\rightarrow d}=\upslack_{\ell}^{u^{\prime}\rightarrow
d^{\prime}}.
\]

\end{lemma}

Here is an illustration for this lemma:%
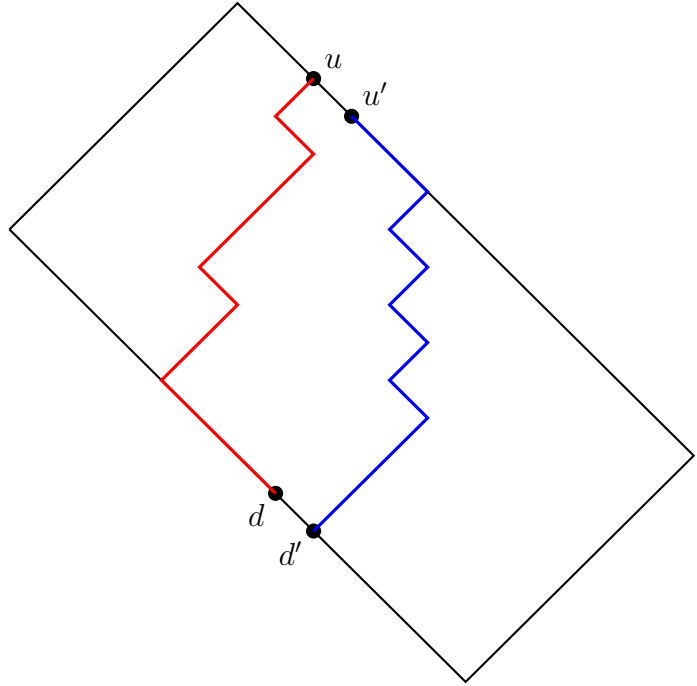
\[
\begin{tikzpicture}
\coordinate (W) at (-3, 0);
\coordinate (N) at (0, 3);
\coordinate (E) at (6, -3);
\coordinate (S) at (3, -6);
\draw[thick] (W) -- (N) -- (E) -- (S) -- (W);
\coordinate [label=80:{$u$}] (u) at (1, 2);
\coordinate [label=80:{$u'$}] (u') at (1.5, 1.5);
\coordinate [label=190:{$d$}] (d) at (0.5, -3.5);
\coordinate [label=190:{$d'$}] (d') at (1, -4);
\fill (u) circle [radius=0.1];
\fill (u') circle [radius=0.1];
\fill (d) circle [radius=0.1];
\fill (d') circle [radius=0.1];
\draw[red, very thick] (u) -- (0.5, 1.5) -- (1, 1) -- (-0.5, -0.5) -- (0, -1) -- (-1, -2) -- (d);
\draw[blue, very thick] (u') -- (2.5, 0.5) -- (2, 0) -- (2.5, -0.5) -- (2, -1) -- (2.5, -1.5) -- (2, -2) -- (2.5, -2.5) -- (2, -3) -- (d');
\end{tikzpicture}
\]
(the red path indexes one addend in the sum $\downslack_{\ell}^{u\rightarrow
d}=\sum_{\mathbf{p}\text{ is a path from }u\text{ to }d}\downslack_{\ell
}^{\mathbf{p}}$, while the blue path contributes to the sum $\upslack_{\ell
}^{u^{\prime}\rightarrow d^{\prime}}=\sum_{\mathbf{p}\text{ is a path from
}u^{\prime}\text{ to }d^{\prime}}\upslack_{\ell}^{\mathbf{p}}$).

In the case when $\mathbb{K}$ is commutative, Lemma \ref{lem.rect.conv} was
independently discovered by Johnson and Liu \cite{JohLiu22}. More precisely,
\cite[Lemma 4.1]{JohLiu22} extends it from sums over paths (such as
$\downslack_{\ell}^{u\rightarrow d}$ and $\upslack_{\ell}^{u^{\prime
}\rightarrow d^{\prime}}$) to sums over $k$-tuples of non-intersecting paths.
It is unclear whether this extension can still be made when $\mathbb{K}$ is
not commutative (what order should the $\downslack_{\ell}^{v}$'s along
different paths be multiplied in?), but the use of determinants likely
precludes any noncommutative generalization of the proof in \cite{JohLiu22}.

\begin{proof}
[Proof of Lemma \ref{lem.rect.conv}.]Let $\ell\in\mathbb{N}$. We
\textquotedblleft interpolate\textquotedblright\ between the paths from $u$ to
$d$ and the paths from $u^{\prime}$ to $d^{\prime}$ using what we call
\textquotedblleft path-jump-paths\textquotedblright. To define these formally,
we introduce some more basic notations.

The first coordinate of any $x\in P$ will be denoted by $\first x$. Thus,
$\first \left(  i,j\right)  =i$ for any $\left(  i,j\right)  \in P$.

Furthermore, for any $x=\left(  i,j\right)  \in P$, we define the \emph{rank}
of $x$ to be the positive integer $i+j-1$. This rank will be denoted by
$\rank x$.

We define a new binary relation $\blacktriangleright$ on the set $P$ as
follows: If $x$ and $y$ are two elements of $P$, then the relation
$x\blacktriangleright y$ means \textquotedblleft$\rank x=\rank y+1$ and
$\first x>\first y$\textquotedblright. In other words, the relation
$x\blacktriangleright y$ means that%
\[
\text{if }x=\left(  i,j\right)  \text{, then }y=\left(  i-k,j+k-1\right)
\text{ for some }k>0.
\]
Visually speaking, it means that $y$ is one step southeast and a (nonnegative)
amount of steps east of $x$ (on the Hasse diagram).

We define a \emph{path-jump-path} to be a tuple $\mathbf{p}=\left(
v_{0},v_{1},\ldots,v_{k}\right)  $ of elements of $P$ along with a chosen
number $i\in\left\{  0,1,\ldots,k-1\right\}  $ such that the chain of
relations%
\[
v_{0}\gtrdot v_{1}\gtrdot\cdots\gtrdot v_{i}\blacktriangleright v_{i+1}\gtrdot
v_{i+2}\gtrdot\cdots\gtrdot v_{k}%
\]
holds. We denote this path-jump-path simply by%
\begin{equation}
\mathbf{p}=\left(  v_{0}\gtrdot v_{1}\gtrdot\cdots\gtrdot v_{i}%
\blacktriangleright v_{i+1}\gtrdot v_{i+2}\gtrdot\cdots\gtrdot v_{k}\right)  ,
\label{eq.lem.rect.conv.pjp}%
\end{equation}
and we say that this path-jump-path $\mathbf{p}$ has \emph{jump at }$i$. The
elements $v_{0},v_{1},\ldots,v_{k}$ are called the \emph{vertices} of this
path-jump-path. The pairs $\left(  v_{j},v_{j+1}\right)  $ of consecutive
vertices are called the \emph{steps} of this path-jump-path. Such a step
$\left(  v_{j},v_{j+1}\right)  $ is said to be a $\gtrdot$\emph{-step} if
$j\neq i$, and it is said to be a $\blacktriangleright$\emph{-step} if $j=i$.

Here is an example of a path-jump-path, where the red edge is the
$\blacktriangleright$-step:
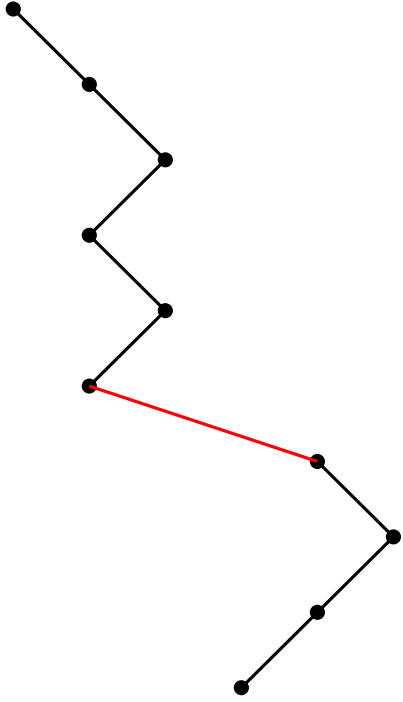
\[
\begin{tikzpicture}
\fill (0, 5) circle [radius=0.1];
\fill (1, 4) circle [radius=0.1];
\fill (2, 3) circle [radius=0.1];
\fill (1, 2) circle [radius=0.1];
\fill (2, 1) circle [radius=0.1];
\fill (1, 0) circle [radius=0.1];
\fill (4, -1) circle [radius=0.1];
\fill (5, -2) circle [radius=0.1];
\fill (4, -3) circle [radius=0.1];
\fill (3, -4) circle [radius=0.1];
\draw[very thick] (0, 5) -- (1, 4) -- (2, 3) -- (1, 2) -- (2, 1) -- (1, 0);
\draw[red, very thick] (1, 0) -- (4, -1);
\draw[very thick] (4, -1) -- (5, -2) -- (4, -3) -- (3, -4);
\end{tikzpicture}
\]

(Note that two vertices $x$ and $y$ can satisfy $x\blacktriangleright y$ and
$x\gtrdot y$ simultaneously. Thus, it can happen that several path-jump-paths
with jumps at different $i$'s contain the same vertices. We nevertheless do
not consider these path-jump-paths to be identical, because we understand a
path-jump-path like (\ref{eq.lem.rect.conv.pjp}) to \textquotedblleft
remember\textquotedblright\ not only its vertices $v_{0},v_{1},\ldots,v_{k}$
but also the value of $i$.)

A \emph{path-jump-path from }$u$\emph{ to }$d^{\prime}$ will mean a
path-jump-path \newline$\left(  v_{0}\gtrdot v_{1}\gtrdot\cdots\gtrdot
v_{i}\blacktriangleright v_{i+1}\gtrdot v_{i+2}\gtrdot\cdots\gtrdot
v_{k}\right)  $ such that $v_{0}=u$ and $v_{k}=d^{\prime}$.

We note that if two elements $x$ and $y$ of $P$ satisfy $x\gtrdot y$ or
$x\blacktriangleright y$, then
\begin{equation}
\rank y=\rank x-1. \label{eq.pf.lem.rect.conv.rank-1}%
\end{equation}

As a consequence of this fact, successive entries $v_{j-1}$ and $v_{j}$ in a
path-jump-path $\left(  v_{0}\gtrdot v_{1}\gtrdot\cdots\gtrdot v_{i}%
\blacktriangleright v_{i+1}\gtrdot v_{i+2}\gtrdot\cdots\gtrdot v_{k}\right)  $
always satisfy $\rank \left(  v_{j}\right)  =\rank \left(  v_{j-1}\right)  -1$
for each $j\in\left[  k\right]  $. In other words, the ranks of the vertices
of a path-jump-path decrease by $1$ at each step.

Hence, the difference in ranks between the first and final entries of a
path-jump-path $\left(  v_{0}\gtrdot v_{1}\gtrdot\cdots\gtrdot v_{i}%
\blacktriangleright v_{i+1}\gtrdot v_{i+2}\gtrdot\cdots\gtrdot v_{k}\right)  $
is one less than its number of entries:
\begin{equation}
\rank \left(  v_{0}\right)  -\rank \left(  v_{k}\right)  =k.
\label{pf.lem.rect.conv.1st.rank-diff}%
\end{equation}

\begin{verlong}
[\textit{Proof of (\ref{pf.lem.rect.conv.1st.rank-diff}):} Let $\left(
v_{0}\gtrdot v_{1}\gtrdot\cdots\gtrdot v_{i}\blacktriangleright v_{i+1}\gtrdot
v_{i+2}\gtrdot\cdots\gtrdot v_{k}\right)  $ be a path-jump-path. Then, each
$j\in\left[  k\right]  $ satisfies $v_{j-1}\gtrdot v_{j}$ or $v_{j-1}%
\blacktriangleright v_{j}$ (by the definition of a \textquotedblleft
path-jump-path\textquotedblright). Hence, each $j\in\left[  k\right]  $
satisfies $\rank \left(  v_{j}\right)  =\rank \left(  v_{j-1}\right)  -1$ (by
(\ref{eq.pf.lem.rect.conv.rank-1}), applied to $x=v_{j-1}$ and $y=v_{j}$). In
other words, each $j\in\left[  k\right]  $ satisfies $1=\rank \left(
v_{j-1}\right)  -\rank \left(  v_{j}\right)  $. Summing this equality over all
$j\in\left[  k\right]  $, we obtain%
\begin{align*}
\sum_{j\in\left[  k\right]  }1  &  =\sum_{j\in\left[  k\right]  }\left(
\rank \left(  v_{j-1}\right)  -\rank \left(  v_{j}\right)  \right) \\
&  =\left(  \rank \left(  v_{0}\right)  -\rank \left(  v_{1}\right)  \right)
+\left(  \rank \left(  v_{1}\right)  -\rank \left(  v_{2}\right)  \right)
+\cdots+\left(  \rank \left(  v_{k-1}\right)  -\rank \left(  v_{k}\right)
\right) \\
&  =\rank \left(  v_{0}\right)  -\rank \left(  v_{k}\right)
\ \ \ \ \ \ \ \ \ \ \left(  \text{by the telescope principle}\right)  .
\end{align*}
Hence, $\rank \left(  v_{0}\right)  -\rank \left(  v_{k}\right)  =\sum
_{j\in\left[  k\right]  }1=\left\vert \left[  k\right]  \right\vert
\cdot1=\left\vert \left[  k\right]  \right\vert =k$. This proves
(\ref{pf.lem.rect.conv.1st.rank-diff}).]
\end{verlong}

\begin{vershort}
Let $r:=\rank u-\rank \left(  d^{\prime}\right)  $. Thus, any path-jump-path
from $u$ to $d^{\prime}$ must contain exactly $r+1$ vertices (by
(\ref{pf.lem.rect.conv.1st.rank-diff})). In other words, any path-jump-path
from $u$ to $d^{\prime}$ must have the form $\left(  v_{0}\gtrdot v_{1}%
\gtrdot\cdots\gtrdot v_{i}\blacktriangleright v_{i+1}\gtrdot v_{i+2}%
\gtrdot\cdots\gtrdot v_{r}\right)  $.
\end{vershort}

\begin{verlong}
Let $r:=\rank u-\rank \left(  d^{\prime}\right)  $. Thus, any path-jump-path
from $u$ to $d^{\prime}$ must contain exactly $r+1$
vertices\footnote{\textit{Proof.} Let $\mathbf{p}$ be a path-jump-path from
$u$ to $d^{\prime}$. We must prove that $\mathbf{p}$ contains exactly $r+1$
vertices.
\par
Write $\mathbf{p}$ in the form $\mathbf{p}=\left(  v_{0}\gtrdot v_{1}%
\gtrdot\cdots\gtrdot v_{i}\blacktriangleright v_{i+1}\gtrdot v_{i+2}%
\gtrdot\cdots\gtrdot v_{k}\right)  $. Since $\mathbf{p}$ is a path-jump-path
from $u$ to $d^{\prime}$, we thus have $v_{0}=u$ and $v_{k}=d^{\prime}$.
However, (\ref{pf.lem.rect.conv.1st.rank-diff}) yields $\rank \left(
v_{0}\right)  -\rank \left(  v_{k}\right)  =k$. Hence,%
\[
k=\rank \left(  \underbrace{v_{0}}_{=u}\right)  -\rank \left(
\underbrace{v_{k}}_{=d^{\prime}}\right)  =\rank u-\rank \left(  d^{\prime
}\right)  =r.
\]
\par
However, $\mathbf{p}=\left(  v_{0}\gtrdot v_{1}\gtrdot\cdots\gtrdot
v_{i}\blacktriangleright v_{i+1}\gtrdot v_{i+2}\gtrdot\cdots\gtrdot
v_{k}\right)  $ shows that $\mathbf{p}$ contains exactly $k+1$ vertices. In
other words, $\mathbf{p}$ contains exactly $r+1$ vertices (since $k=r$),
qed.}. In other words, any path-jump-path from $u$ to $d^{\prime}$ must have
the form $\left(  v_{0}\gtrdot v_{1}\gtrdot\cdots\gtrdot v_{i}%
\blacktriangleright v_{i+1}\gtrdot v_{i+2}\gtrdot\cdots\gtrdot v_{r}\right)  $.
\end{verlong}

We have $R\left(  R^{\ell}f\right)  =R^{\ell+1}f\neq\undf=R\left(
\undf\right)  $ and thus $R^{\ell}f\neq\undf$. Hence, Lemma
\ref{lem.slacks.wd} \textbf{(a)} yields that $v_{\ell}$ is well-defined and
invertible for each $v\in P$. Also, Lemma \ref{lem.slacks.wd} \textbf{(d)}
yields that $\upslack_{\ell}^{v}$ is well-defined and invertible for each
$v\in P$. Moreover, Lemma \ref{lem.slacks.wd} \textbf{(c)} (applied to
$\ell+1$ instead of $\ell$) yields that $\downslack_{\ell}^{v}$ is
well-defined and invertible for each $v\in P$.

In this proof, we will not consider any $\mathbb{K}$-labelings other than
$R^{\ell}f$. Thus, the only labels we will be using are the labels $v_{\ell
}=\left(  R^{\ell}f\right)  \left(  v\right)  $ for $v\in\widehat{P}$. Thus,
we agree to use the following shorthand notation: If $v\in\widehat{P}$, then
the elements $v_{\ell}$, $\upslack_{\ell}^{v}$ and $\downslack_{\ell}^{v}$ of
$\mathbb{K}$ will be denoted simply by $v$, $\upslack^{v}$ and $\downslack^{v}%
$, respectively. In other words, \textbf{we shall omit subscripts when these
subscripts are $\ell$.} For instance, the product $\downslack_{\ell}%
^{u}u_{\ell}\overline{u_{\ell}^{\prime}}$ will thus be abbreviated as
$\downslack^{u}u\overline{u^{\prime}}$.

For any path-jump-path
\[
\mathbf{p}=\left(  v_{0}\gtrdot v_{1}\gtrdot\cdots\gtrdot v_{i}%
\blacktriangleright v_{i+1}\gtrdot v_{i+2}\gtrdot\cdots\gtrdot v_{r}\right)
\]
that contains $r+1$ vertices, we set%
\[
E_{\mathbf{p}}:=\downslack^{v_{0}}\downslack^{v_{1}}\cdots\downslack^{v_{i-1}%
}v_{i}\overline{v_{i+1}}\upslack^{v_{i+2}}\upslack^{v_{i+3}}\cdots
\upslack^{v_{r}}\in\mathbb{K}.
\]

\begin{verlong}
\noindent(Here, as we have already announced, we are omitting a subscript
under each symbol. All the omitted subscripts are $\ell$ -- for example,
\textquotedblleft$v_{i}$\textquotedblright\ means $\left(  v_{i}\right)
_{\ell}$, and \textquotedblleft$v_{i+1}$\textquotedblright\ means $\left(
v_{i+1}\right)  _{\ell}$, and \textquotedblleft$\downslack^{v_{0}}%
$\textquotedblright\ means $\downslack_{\ell}^{v_{0}}$, and so on. We will do
the same with all expressions that follow.)
\end{verlong}

Now we claim the following (again omitting subscripts that are $\ell$):

\begin{statement}
\textit{Claim 1:} We have
\[
\downslack^{u\rightarrow d}=\sum_{\substack{\mathbf{p}\text{ is a
path-jump-path}\\\text{from }u\text{ to }d^{\prime}\\\text{with jump at }%
r-1}}E_{\mathbf{p}}.
\]

\end{statement}

\begin{statement}
\textit{Claim 2:} We have%
\[
\upslack^{u^{\prime}\rightarrow d^{\prime}}=\sum_{\substack{\mathbf{p}\text{
is a path-jump-path}\\\text{from }u\text{ to }d^{\prime}\\\text{with jump at
}0}}E_{\mathbf{p}}.
\]

\end{statement}

\begin{statement}
\textit{Claim 3:} For each $j\in\left\{  0,1,\ldots,r-2\right\}  $, we have%
\[
\sum_{\substack{\mathbf{p}\text{ is a path-jump-path}\\\text{from }u\text{ to
}d^{\prime}\\\text{with jump at }j}}E_{\mathbf{p}}=\sum_{\substack{\mathbf{p}%
\text{ is a path-jump-path}\\\text{from }u\text{ to }d^{\prime}\\\text{with
jump at }j+1}}E_{\mathbf{p}}.
\]

\end{statement}

Before we prove these three claims, let us explain how Lemma
\ref{lem.rect.conv} will follow from them:%
\begin{align*}
\upslack_{\ell}^{u^{\prime}\rightarrow d^{\prime}}  &  =\upslack^{u^{\prime
}\rightarrow d^{\prime}}\\
&  =\sum_{\substack{\mathbf{p}\text{ is a path-jump-path}\\\text{from }u\text{
to }d^{\prime}\\\text{with jump at }0}}E_{\mathbf{p}}%
\ \ \ \ \ \ \ \ \ \ \left(  \text{by Claim 2}\right) \\
&  =\sum_{\substack{\mathbf{p}\text{ is a path-jump-path}\\\text{from }u\text{
to }d^{\prime}\\\text{with jump at }1}}E_{\mathbf{p}}%
\ \ \ \ \ \ \ \ \ \ \left(  \text{by Claim 3, applied to }j=0\right) \\
&  =\sum_{\substack{\mathbf{p}\text{ is a path-jump-path}\\\text{from }u\text{
to }d^{\prime}\\\text{with jump at }2}}E_{\mathbf{p}}%
\ \ \ \ \ \ \ \ \ \ \left(  \text{by Claim 3, applied to }j=1\right) \\
&  =\cdots\\
&  =\sum_{\substack{\mathbf{p}\text{ is a path-jump-path}\\\text{from }u\text{
to }d^{\prime}\\\text{with jump at }r-1}}E_{\mathbf{p}}%
\ \ \ \ \ \ \ \ \ \ \left(  \text{by Claim 3, applied to }j=r-2\right) \\
&  =\downslack^{u\rightarrow d}\ \ \ \ \ \ \ \ \ \ \left(  \text{by Claim
1}\right) \\
&  =\downslack_{\ell}^{u\rightarrow d}.
\end{align*}
Hence, Lemma \ref{lem.rect.conv} will follow once Claims 1, 2 and 3 have been
proved. Let us now prove these three claims: \medskip

\begin{proof}
[Proof of Claim 1.]We know that $d$ lies on the southwestern edge of $P$.
Hence, the only $s\in\widehat{P}$ satisfying $s\lessdot d$ is $d^{\prime}$
(since $d\gtrdot d^{\prime}$). Therefore, $\sum_{\substack{s\in\widehat{P}%
;\\s\lessdot d}}s_{\ell}=d_{\ell}^{\prime}$. However, the definition of
$\downslack_{\ell}^{d}$ shows that $\downslack_{\ell}^{d}=d_{\ell}%
\cdot\overline{\sum_{\substack{s\in\widehat{P};\\s\lessdot d}}s_{\ell}%
}=d_{\ell}\overline{d_{\ell}^{\prime}}$ (since $\sum_{\substack{s\in
\widehat{P};\\s\lessdot d}}s_{\ell}=d_{\ell}^{\prime}$). Since we omit
subscripts (when these subscripts are $\ell$), we can rewrite this as
\begin{equation}
\downslack^{d}=d\overline{d^{\prime}}. \label{pf.lem.rect.conv.c1.pf.1}%
\end{equation}

\begin{vershort}
We know that any path-jump-path from $u$ to $d^{\prime}$ must have the form
\newline$\left(  v_{0}\gtrdot v_{1}\gtrdot\cdots\gtrdot v_{i}%
\blacktriangleright v_{i+1}\gtrdot v_{i+2}\gtrdot\cdots\gtrdot v_{r}\right)
$. If such a path-jump-path has jump at $r-1$, then it must have the form
$\left(  v_{0}\gtrdot v_{1}\gtrdot\cdots\gtrdot v_{r-1}\blacktriangleright
v_{r}\right)  $; that is, its last step $\left(  v_{r-1},v_{r}\right)  $ is an
$\blacktriangleright$-step. However, since it ends at $d^{\prime}$, we must
have $v_{r}=d^{\prime}$ and thus $v_{r-1}\blacktriangleright v_{r}=d^{\prime}%
$. This entails $v_{r-1}=d$ (since the only $g\in P$ satisfying
$g\blacktriangleright d^{\prime}$ is $d$\ \ \ \ \footnote{This follows easily
from the geographical positions of $d$ and $d^{\prime}$ on the southwestern
edge of $P$.}), and therefore $\left(  v_{r-1},v_{r}\right)  =\left(
d,d^{\prime}\right)  $ (since $v_{r}=d^{\prime}$). In other words, the last
step of this path-jump-path is $\left(  d,d^{\prime}\right)  $.
\end{vershort}

\begin{verlong}
We know that any path-jump-path from $u$ to $d^{\prime}$ must have the form
\newline$\left(  v_{0}\gtrdot v_{1}\gtrdot\cdots\gtrdot v_{i}%
\blacktriangleright v_{i+1}\gtrdot v_{i+2}\gtrdot\cdots\gtrdot v_{r}\right)
$. If such a path-jump-path has jump at $r-1$, then it must have the form
$\left(  v_{0}\gtrdot v_{1}\gtrdot\cdots\gtrdot v_{r-1}\blacktriangleright
v_{r}\right)  $; that is, its last step $\left(  v_{r-1},v_{r}\right)  $ is an
$\blacktriangleright$-step. However, since it ends at $d^{\prime}$, we must
have $v_{r}=d^{\prime}$ and thus $v_{r-1}\blacktriangleright v_{r}=d^{\prime}%
$. This entails $v_{r-1}=d$ (since the only $g\in P$ satisfying
$g\blacktriangleright d^{\prime}$ is $d$\ \ \ \ \footnote{\textit{Proof.}
Recall that $d^{\prime}$ lies on the southwestern edge of $P$. In other words,
$d^{\prime}=\left(  i,1\right)  $ for some $i\in\left[  p\right]  $. Consider
this $i$. Therefore, $d=\left(  i+1,1\right)  $ (since $d$ also lies on the
southwestern edge of $P$ and satisfies $d\gtrdot d^{\prime}$).
\par
From $d^{\prime}=\left(  i,1\right)  $, we obtain $\first \left(  d^{\prime
}\right)  =i$ and $\rank \left(  d^{\prime}\right)  =i+1-1=i$.
\par
Now, let $g\in P$ be such that $g\blacktriangleright d^{\prime}$. By the
definition of the relation $\blacktriangleright$, we thus have
$\rank g=\rank \left(  d^{\prime}\right)  +1$ and $\first g>\first \left(
d^{\prime}\right)  $. Hence, $\first g>\first \left(  d^{\prime}\right)  =i$
and $\rank g=\underbrace{\rank \left(  d^{\prime}\right)  }_{=i}+1=i+1$.
\par
Write $g$ in the form $g=\left(  i^{\prime},j^{\prime}\right)  $ for some
$i^{\prime}\in\left[  p\right]  $ and some $j^{\prime}\in\left[  q\right]  $.
Thus, $\first g=i^{\prime}$ and $\rank g=i^{\prime}+j^{\prime}-1$. Hence,
$i^{\prime}=\first g>i$ and $i+1=\rank g=\underbrace{i^{\prime}}%
_{>i}+j^{\prime}-1>i+j^{\prime}-1$. Subtracting $i$ from both sides of the
latter inequality, we obtain $1>j^{\prime}-1$. Thus, $j^{\prime}<1+1=2$, so
that $j^{\prime}=1$ (since $j^{\prime}\in\left[  q\right]  $). Now, from
$i+1=i^{\prime}+\underbrace{j^{\prime}}_{=1}-1=i^{\prime}+1-1=i^{\prime}$, we
obtain $i^{\prime}=i+1$. Hence, $g=\left(  i^{\prime},j^{\prime}\right)
=\left(  i+1,1\right)  $ (since $i^{\prime}=i+1$ and $j^{\prime}=1$).
Comparing this with $d=\left(  i+1,1\right)  $, we find $g=d$.
\par
Forget that we fixed $g$. We thus have shown that if $g\in P$ satisfies
$g\blacktriangleright d^{\prime}$, then $g=d$. In other words, the only $g\in
P$ satisfying $g\blacktriangleright d^{\prime}$ is $d$ (since it is easy to
see that $d$ does indeed satisfy $d\blacktriangleright d^{\prime}$).}), and
therefore $\left(  v_{r-1},v_{r}\right)  =\left(  d,d^{\prime}\right)  $
(since $v_{r}=d^{\prime}$). In other words, the last step of this
path-jump-path is $\left(  d,d^{\prime}\right)  $.
\end{verlong}

\begin{vershort}
We have thus shown that if a path-jump-path from $u$ to $d^{\prime}$ has jump
at $r-1$, then its last step is $\left(  d,d^{\prime}\right)  $. Hence, any
path-jump-path from $u$ to $d^{\prime}$ with jump at $r-1$ must have the form%
\[
\left(  v_{0}\gtrdot v_{1}\gtrdot\cdots\gtrdot v_{r-1}\blacktriangleright
d^{\prime}\right)  ,
\]
where $\left(  v_{0}\gtrdot v_{1}\gtrdot\cdots\gtrdot v_{r-1}\right)  $ is a
path from $u$ to $d$. Conversely, any tuple of the latter form is a
path-jump-path from $u$ to $d^{\prime}$ with jump at $r-1$ (since
$d\blacktriangleright d^{\prime}$). Therefore,%
\begin{align*}
\sum_{\substack{\mathbf{p}\text{ is a path-jump-path}\\\text{from }u\text{ to
}d^{\prime}\\\text{with jump at }r-1}}E_{\mathbf{p}}  &  =\sum
_{\substack{\left(  v_{0}\gtrdot v_{1}\gtrdot\cdots\gtrdot v_{r-1}\right)
\\\text{is a path from }u\text{ to }d}}\underbrace{E_{\left(  v_{0}\gtrdot
v_{1}\gtrdot\cdots\gtrdot v_{r-1}\blacktriangleright d^{\prime}\right)  }%
}_{\substack{=\downslack^{v_{0}}\downslack^{v_{1}}\cdots\downslack^{v_{r-2}%
}v_{r-1}\overline{d^{\prime}}\\\text{(by the definition of }E_{\left(
v_{0}\gtrdot v_{1}\gtrdot\cdots\gtrdot v_{r-1}\blacktriangleright d^{\prime
}\right)  }\text{)}}}\\
&  =\sum_{\substack{\left(  v_{0}\gtrdot v_{1}\gtrdot\cdots\gtrdot
v_{r-1}\right)  \\\text{is a path from }u\text{ to }d}}\downslack^{v_{0}%
}\downslack^{v_{1}}\cdots\downslack^{v_{r-2}}\underbrace{v_{r-1}%
}_{\substack{=d}}\overline{d^{\prime}}\\
&  =\sum_{\substack{\left(  v_{0}\gtrdot v_{1}\gtrdot\cdots\gtrdot
v_{r-1}\right)  \\\text{is a path from }u\text{ to }d}}\downslack^{v_{0}%
}\downslack^{v_{1}}\cdots\downslack^{v_{r-2}}\underbrace{d\overline{d^{\prime
}}}_{\substack{=\downslack^{d}\\\text{(by (\ref{pf.lem.rect.conv.c1.pf.1}))}%
}}\\
&  =\sum_{\substack{\left(  v_{0}\gtrdot v_{1}\gtrdot\cdots\gtrdot
v_{r-1}\right)  \\\text{is a path from }u\text{ to }d}}\downslack^{v_{0}%
}\downslack^{v_{1}}\cdots\downslack^{v_{r-2}}\underbrace{\downslack^{d}%
}_{\substack{=\downslack^{v_{r-1}}\\\text{(because }d=v_{r-1}\text{)}}}\\
&  =\sum_{\substack{\left(  v_{0}\gtrdot v_{1}\gtrdot\cdots\gtrdot
v_{r-1}\right)  \\\text{is a path from }u\text{ to }d}%
}\underbrace{\downslack^{v_{0}}\downslack^{v_{1}}\cdots\downslack^{v_{r-2}%
}\downslack^{v_{r-1}}}_{\substack{=\downslack^{\left(  v_{0}\gtrdot
v_{1}\gtrdot\cdots\gtrdot v_{r-1}\right)  }\\\text{(by the definition of
}\downslack^{\left(  v_{0}\gtrdot v_{1}\gtrdot\cdots\gtrdot v_{r-1}\right)
}\text{)}}}\\
&  =\sum_{\substack{\left(  v_{0}\gtrdot v_{1}\gtrdot\cdots\gtrdot
v_{r-1}\right)  \\\text{is a path from }u\text{ to }d}}\downslack^{\left(
v_{0}\gtrdot v_{1}\gtrdot\cdots\gtrdot v_{r-1}\right)  }=\sum_{\mathbf{p}%
\text{ is a path from }u\text{ to }d}\downslack^{\mathbf{p}}\\
&  =\downslack^{u\rightarrow d}\ \ \ \ \ \ \ \ \ \ \left(  \text{by the
definition of }\downslack^{u\rightarrow d}\right)  .
\end{align*}

\end{vershort}

\begin{verlong}
We have thus shown that if a path-jump-path from $u$ to $d^{\prime}$ has jump
at $r-1$, then its last step is $\left(  d,d^{\prime}\right)  $. Hence, any
path-jump-path from $u$ to $d^{\prime}$ with jump at $r-1$ must have the form%
\[
\left(  v_{0}\gtrdot v_{1}\gtrdot\cdots\gtrdot v_{r-1}\blacktriangleright
d^{\prime}\right)  ,
\]
where $\left(  v_{0}\gtrdot v_{1}\gtrdot\cdots\gtrdot v_{r-1}\right)  $ is a
path from $u$ to $d$. Conversely, any tuple of the latter form is a
path-jump-path from $u$ to $d^{\prime}$ with jump at $r-1$ (since
$d\blacktriangleright d^{\prime}$). Therefore, we can substitute $\left(
v_{0}\gtrdot v_{1}\gtrdot\cdots\gtrdot v_{r-1}\blacktriangleright d^{\prime
}\right)  $ for $\mathbf{p}$ in the sum $\sum_{\substack{\mathbf{p}\text{ is a
path-jump-path}\\\text{from }u\text{ to }d^{\prime}\\\text{with jump at }%
r-1}}E_{\mathbf{p}}$. We thus obtain%
\begin{align*}
\sum_{\substack{\mathbf{p}\text{ is a path-jump-path}\\\text{from }u\text{ to
}d^{\prime}\\\text{with jump at }r-1}}E_{\mathbf{p}}  &  =\sum
_{\substack{\left(  v_{0}\gtrdot v_{1}\gtrdot\cdots\gtrdot v_{r-1}\right)
\\\text{is a path from }u\text{ to }d}}\underbrace{E_{\left(  v_{0}\gtrdot
v_{1}\gtrdot\cdots\gtrdot v_{r-1}\blacktriangleright d^{\prime}\right)  }%
}_{\substack{=\downslack^{v_{0}}\downslack^{v_{1}}\cdots\downslack^{v_{r-2}%
}v_{r-1}\overline{d^{\prime}}\\\text{(by the definition of }E_{\left(
v_{0}\gtrdot v_{1}\gtrdot\cdots\gtrdot v_{r-1}\blacktriangleright d^{\prime
}\right)  }\text{)}}}\\
&  =\sum_{\substack{\left(  v_{0}\gtrdot v_{1}\gtrdot\cdots\gtrdot
v_{r-1}\right)  \\\text{is a path from }u\text{ to }d}}\downslack^{v_{0}%
}\downslack^{v_{1}}\cdots\downslack^{v_{r-2}}\underbrace{v_{r-1}%
}_{\substack{=d\\\text{(since }\left(  v_{0}\gtrdot v_{1}\gtrdot\cdots\gtrdot
v_{r-1}\right)  \\\text{is a path from }u\text{ to }d\text{)}}}\overline
{d^{\prime}}\\
&  =\sum_{\substack{\left(  v_{0}\gtrdot v_{1}\gtrdot\cdots\gtrdot
v_{r-1}\right)  \\\text{is a path from }u\text{ to }d}}\downslack^{v_{0}%
}\downslack^{v_{1}}\cdots\downslack^{v_{r-2}}\underbrace{d\overline{d^{\prime
}}}_{\substack{=\downslack^{d}\\\text{(by (\ref{pf.lem.rect.conv.c1.pf.1}))}%
}}\\
&  =\sum_{\substack{\left(  v_{0}\gtrdot v_{1}\gtrdot\cdots\gtrdot
v_{r-1}\right)  \\\text{is a path from }u\text{ to }d}}\downslack^{v_{0}%
}\downslack^{v_{1}}\cdots\downslack^{v_{r-2}}\underbrace{\downslack^{d}%
}_{\substack{=\downslack^{v_{r-1}}\\\text{(because }d=v_{r-1}\\\text{(again
since }\left(  v_{0}\gtrdot v_{1}\gtrdot\cdots\gtrdot v_{r-1}\right)
\\\text{is a path from }u\text{ to }d\text{))}}}\\
&  =\sum_{\substack{\left(  v_{0}\gtrdot v_{1}\gtrdot\cdots\gtrdot
v_{r-1}\right)  \\\text{is a path from }u\text{ to }d}%
}\underbrace{\downslack^{v_{0}}\downslack^{v_{1}}\cdots\downslack^{v_{r-2}%
}\downslack^{v_{r-1}}}_{\substack{=\downslack^{v_{0}}\downslack^{v_{1}}%
\cdots\downslack^{v_{r-1}}\\=\downslack^{\left(  v_{0}\gtrdot v_{1}%
\gtrdot\cdots\gtrdot v_{r-1}\right)  }\\\text{(by the definition of
}\downslack^{\left(  v_{0}\gtrdot v_{1}\gtrdot\cdots\gtrdot v_{r-1}\right)
}\text{)}}}\\
&  =\sum_{\substack{\left(  v_{0}\gtrdot v_{1}\gtrdot\cdots\gtrdot
v_{r-1}\right)  \\\text{is a path from }u\text{ to }d}}\downslack^{\left(
v_{0}\gtrdot v_{1}\gtrdot\cdots\gtrdot v_{r-1}\right)  }\\
&  =\sum_{\mathbf{p}\text{ is a path from }u\text{ to }d}%
\underbrace{\downslack^{\mathbf{p}}}_{=\downslack_{\ell}^{\mathbf{p}}}\\
&  \ \ \ \ \ \ \ \ \ \ \ \ \ \ \ \ \ \ \ \ \left(
\begin{array}
[c]{c}%
\text{here we have renamed the}\\
\text{summation index }\left(  v_{0}\gtrdot v_{1}\gtrdot\cdots\gtrdot
v_{r-1}\right)  \text{ as }\mathbf{p}%
\end{array}
\right) \\
&  =\sum_{\mathbf{p}\text{ is a path from }u\text{ to }d}\downslack_{\ell
}^{\mathbf{p}}\\
&  =\downslack_{\ell}^{u\rightarrow d}\ \ \ \ \ \ \ \ \ \ \left(  \text{by the
definition of }\downslack_{\ell}^{u\rightarrow d}\right) \\
&  =\downslack^{u\rightarrow d}.
\end{align*}

\end{verlong}

\noindent This proves Claim 1.
\end{proof}

\begin{vershort}

\begin{proof}
[Proof of Claim 2.]This is analogous to the proof of Claim 1. This time, we
need to argue that if a path-jump-path from $u$ to $d^{\prime}$ has jump at
$0$, then its first step is $\left(  u,u^{\prime}\right)  $ (since the only
$g\in P$ satisfying $u\blacktriangleright g$ is $u^{\prime}$).
\end{proof}
\end{vershort}

\begin{verlong}

\begin{proof}
[Proof of Claim 2.]This is mostly analogous to the above proof of Claim 1, but
we nevertheless present the full argument for the sake of completeness.

We know that $u^{\prime}$ lies on the northeastern edge of $P$. Hence, the
only $s\in\widehat{P}$ satisfying $s\gtrdot u^{\prime}$ is $u$ (since
$u\gtrdot u^{\prime}$). Therefore, $\sum_{\substack{s\in\widehat{P};\\s\gtrdot
u^{\prime}}}\overline{s_{\ell}}=\overline{u_{\ell}}$. Therefore,%
\[
\overline{\sum_{\substack{s\in\widehat{P};\\s\gtrdot u^{\prime}}%
}\overline{s_{\ell}}}=\overline{\overline{u_{\ell}}}=u_{\ell}.
\]
However, the definition of $\upslack_{\ell}^{u^{\prime}}$ shows that
$\upslack_{\ell}^{u^{\prime}}=\overline{\sum_{\substack{s\in\widehat{P}%
;\\s\gtrdot u^{\prime}}}\overline{s_{\ell}}}\cdot\overline{u_{\ell}^{\prime}%
}=u_{\ell}\overline{u_{\ell}^{\prime}}$ (since $\overline{\sum_{\substack{s\in
\widehat{P};\\s\gtrdot u^{\prime}}}\overline{s_{\ell}}}=u_{\ell}$). Since we
omit subscripts (when these subscripts are $\ell$), we can rewrite this as
\begin{equation}
\upslack^{u^{\prime}}=u\overline{u^{\prime}}. \label{pf.lem.rect.conv.c2.pf.1}%
\end{equation}

We know that any path-jump-path from $u$ to $d^{\prime}$ must have the form
\newline$\left(  v_{0}\gtrdot v_{1}\gtrdot\cdots\gtrdot v_{i}%
\blacktriangleright v_{i+1}\gtrdot v_{i+2}\gtrdot\cdots\gtrdot v_{r}\right)
$. If such a path-jump-path has jump at $0$, then it must have the form
$\left(  v_{0}\blacktriangleright v_{1}\gtrdot v_{2}\gtrdot\cdots\gtrdot
v_{r}\right)  $; that is, its first step $\left(  v_{0},v_{1}\right)  $ is an
$\blacktriangleright$-step. However, since it starts at $u$, we must have
$v_{0}=u$ and thus $u=v_{0}\blacktriangleright v_{1}$. This entails
$v_{1}=u^{\prime}$ (since the only $g\in P$ satisfying $u\blacktriangleright
g$ is $u^{\prime}$\ \ \ \ \footnote{\textit{Proof.} Recall that $u$ lies on
the northeastern edge of $P$. In other words, $u=\left(  i,q\right)  $ for
some $i\in\left[  p\right]  $. Consider this $i$. Therefore, $u^{\prime
}=\left(  i-1,q\right)  $ (since $u^{\prime}$ also lies on the northeastern
edge of $P$ and satisfies $u\gtrdot u^{\prime}$).
\par
From $u=\left(  i,q\right)  $, we obtain $\first u=i$ and $\rank u=i+q-1$.
\par
Now, let $g\in P$ be such that $u\blacktriangleright g$. By the definition of
the relation $\blacktriangleright$, we thus have $\rank u=\rank g+1$ and
$\first u>\first g$. Hence, $\first g<\first u=i$.
\par
Write $g$ in the form $g=\left(  i^{\prime},j^{\prime}\right)  $ for some
$i^{\prime}\in\left[  p\right]  $ and some $j^{\prime}\in\left[  q\right]  $.
Thus, $\first g=i^{\prime}$ and $\rank g=i^{\prime}+j^{\prime}-1$. Hence,
$i^{\prime}=\first g<i$ and%
\[
i+q-1=\rank u=\underbrace{\rank g}_{=i^{\prime}+j^{\prime}-1}+1=i^{\prime
}+j^{\prime}-1+1=\underbrace{i^{\prime}}_{<i}+j^{\prime}<i+j^{\prime}.
\]
Subtracting $i$ from both sides of the latter inequality, we obtain
$q-1<j^{\prime}$. Thus, $j^{\prime}>q-1$, so that $j^{\prime}=q$ (since
$j^{\prime}\in\left[  q\right]  $). Now, subtracting $q$ from both sides of
the equality $i+q-1=i^{\prime}+\underbrace{j^{\prime}}_{=q}=i^{\prime}+q$, we
obtain $i-1=i^{\prime}$. In other words, $i^{\prime}=i-1$. Hence, $g=\left(
i^{\prime},j^{\prime}\right)  =\left(  i-1,q\right)  $ (since $i^{\prime}=i-1$
and $j^{\prime}=q$). Comparing this with $u^{\prime}=\left(  i-1,q\right)  $,
we find $g=u^{\prime}$.
\par
Forget that we fixed $g$. We thus have shown that if $g\in P$ satisfies
$u\blacktriangleright g$, then $g=u^{\prime}$. In other words, the only $g\in
P$ satisfying $u\blacktriangleright g$ is $u^{\prime}$ (since it is easy to
see that $u^{\prime}$ does indeed satisfy $u\blacktriangleright u^{\prime}$%
).}), and therefore $\left(  v_{0},v_{1}\right)  =\left(  u,u^{\prime}\right)
$ (since $v_{0}=u$). In other words, the first step of this path-jump-path is
$\left(  u,u^{\prime}\right)  $.

We have thus shown that if a path-jump-path from $u$ to $d^{\prime}$ has jump
at $0$, then its first step is $\left(  u,u^{\prime}\right)  $. Hence, any
path-jump-path from $u$ to $d^{\prime}$ with jump at $0$ must have the form%
\[
\left(  u\blacktriangleright v_{1}\gtrdot v_{2}\gtrdot\cdots\gtrdot
v_{r}\right)  ,
\]
where $\left(  v_{1}\gtrdot v_{2}\gtrdot\cdots\gtrdot v_{r}\right)  $ is a
path from $u^{\prime}$ to $d^{\prime}$. Conversely, any tuple of the latter
form is a path-jump-path from $u$ to $d^{\prime}$ with jump at $0$ (since
$u\blacktriangleright u^{\prime}$). Therefore, we can substitute $\left(
u\blacktriangleright v_{1}\gtrdot v_{2}\gtrdot\cdots\gtrdot v_{r}\right)  $
for $\mathbf{p}$ in the sum $\sum_{\substack{\mathbf{p}\text{ is a
path-jump-path}\\\text{from }u\text{ to }d^{\prime}\\\text{with jump at }%
0}}E_{\mathbf{p}}$. We thus obtain%
\begin{align*}
\sum_{\substack{\mathbf{p}\text{ is a path-jump-path}\\\text{from }u\text{ to
}d^{\prime}\\\text{with jump at }0}}E_{\mathbf{p}}  &  =\sum
_{\substack{\left(  v_{1}\gtrdot v_{2}\gtrdot\cdots\gtrdot v_{r}\right)
\\\text{is a path from }u^{\prime}\text{ to }d^{\prime}}%
}\underbrace{E_{\left(  u\blacktriangleright v_{1}\gtrdot v_{2}\gtrdot
\cdots\gtrdot v_{r}\right)  }}_{\substack{=u\overline{v_{1}}\upslack^{v_{2}%
}\upslack^{v_{3}}\cdots\upslack^{v_{r}}\\\text{(by the definition of
}E_{\left(  u\blacktriangleright v_{1}\gtrdot v_{2}\gtrdot\cdots\gtrdot
v_{r}\right)  }\text{)}}}\\
&  =\sum_{\substack{\left(  v_{1}\gtrdot v_{2}\gtrdot\cdots\gtrdot
v_{r}\right)  \\\text{is a path from }u^{\prime}\text{ to }d^{\prime}%
}}u\underbrace{\overline{v_{1}}}_{\substack{=\overline{u^{\prime}%
}\\\text{(because }v_{1}=u^{\prime}\\\text{(since }\left(  v_{1}\gtrdot
v_{2}\gtrdot\cdots\gtrdot v_{r}\right)  \\\text{is a path from }u^{\prime
}\text{ to }d^{\prime}\text{))}}}\upslack^{v_{2}}\upslack^{v_{3}}%
\cdots\upslack^{v_{r}}\\
&  =\sum_{\substack{\left(  v_{1}\gtrdot v_{2}\gtrdot\cdots\gtrdot
v_{r}\right)  \\\text{is a path from }u^{\prime}\text{ to }d^{\prime}%
}}\underbrace{u\overline{u^{\prime}}}_{\substack{=\upslack^{u^{\prime}%
}\\\text{(by (\ref{pf.lem.rect.conv.c2.pf.1}))}}}\upslack^{v_{2}%
}\upslack^{v_{3}}\cdots\upslack^{v_{r}}\\
&  =\sum_{\substack{\left(  v_{1}\gtrdot v_{2}\gtrdot\cdots\gtrdot
v_{r}\right)  \\\text{is a path from }u^{\prime}\text{ to }d^{\prime}%
}}\underbrace{\upslack^{u^{\prime}}}_{\substack{=\upslack^{v_{1}%
}\\\text{(because }u^{\prime}=v_{1}\\\text{(since }\left(  v_{1}\gtrdot
v_{2}\gtrdot\cdots\gtrdot v_{r}\right)  \\\text{is a path from }u^{\prime
}\text{ to }d^{\prime}\text{))}}}\upslack^{v_{2}}\upslack^{v_{3}}%
\cdots\upslack^{v_{r}}\\
&  =\sum_{\substack{\left(  v_{1}\gtrdot v_{2}\gtrdot\cdots\gtrdot
v_{r}\right)  \\\text{is a path from }u^{\prime}\text{ to }d^{\prime}%
}}\underbrace{\upslack^{v_{1}}\upslack^{v_{2}}\upslack^{v_{3}}\cdots
\upslack^{v_{r}}}_{\substack{=\upslack^{v_{1}}\upslack^{v_{2}}\cdots
\upslack^{v_{r}}\\=\upslack^{\left(  v_{1}\gtrdot v_{2}\gtrdot\cdots\gtrdot
v_{r}\right)  }\\\text{(by the definition of }\upslack^{\left(  v_{1}\gtrdot
v_{2}\gtrdot\cdots\gtrdot v_{r}\right)  }\text{)}}}\\
&  =\sum_{\substack{\left(  v_{1}\gtrdot v_{2}\gtrdot\cdots\gtrdot
v_{r}\right)  \\\text{is a path from }u^{\prime}\text{ to }d^{\prime}%
}}\upslack^{\left(  v_{1}\gtrdot v_{2}\gtrdot\cdots\gtrdot v_{r}\right)  }\\
&  =\sum_{\substack{\left(  v_{1}\gtrdot v_{2}\gtrdot\cdots\gtrdot
v_{r}\right)  \\\text{is a path from }u^{\prime}\text{ to }d^{\prime}%
}}\underbrace{\upslack^{\mathbf{p}}}_{=\upslack_{\ell}^{\mathbf{p}}}\\
&  \ \ \ \ \ \ \ \ \ \ \ \ \ \ \ \ \ \ \ \ \left(
\begin{array}
[c]{c}%
\text{here we have renamed the}\\
\text{summation index }\left(  v_{1}\gtrdot v_{2}\gtrdot\cdots\gtrdot
v_{r}\right)  \text{ as }\mathbf{p}%
\end{array}
\right) \\
&  =\sum_{\mathbf{p}\text{ is a path from }u^{\prime}\text{ to }d^{\prime}%
}\upslack_{\ell}^{\mathbf{p}}\\
&  =\upslack_{\ell}^{u^{\prime}\rightarrow d^{\prime}}%
\ \ \ \ \ \ \ \ \ \ \left(  \text{by the definition of }\upslack_{\ell
}^{u^{\prime}\rightarrow d^{\prime}}\right) \\
&  =\upslack^{u^{\prime}\rightarrow d^{\prime}}.
\end{align*}
This proves Claim 2.
\end{proof}
\end{verlong}

Proving Claim 3 is a bit trickier. As an auxiliary result, we first show the following:

\begin{statement}
\textit{Claim 4:} Let $s$ and $t$ be two elements of $P$. Then,%
\begin{equation}
\sum_{\substack{x\in P;\\s\blacktriangleright x\gtrdot t}}s\overline
{x}\upslack^{t}=\sum_{\substack{x\in P;\\s\gtrdot x\blacktriangleright
t}}\downslack^{s}x\overline{t}. \label{pf.lem.rect.conv.c4}%
\end{equation}

\end{statement}

\begin{proof}
[Proof of Claim 4.]

\begin{vershort}
First, we observe that an $x\in P$ satisfying $s\blacktriangleright x\gtrdot
t$ cannot exist unless $\rank t=\rank s-2$ (because
(\ref{eq.pf.lem.rect.conv.rank-1}) yields that such an $x$ must satisfy
$\rank
x=\rank s-1$ and $\rank t=\rank x-1$, whence $\rank t=\rank x-1=\left(
\rank
s-1\right)  -1=\rank s-2$). Hence, the left hand side of the desired equality
(\ref{pf.lem.rect.conv.c4}) is an empty sum unless $\rank t=\rank s-2$.
Similarly, the same can be said about the right hand side. Thus,
(\ref{pf.lem.rect.conv.c4}) boils down to $0=0$ unless $\rank t=\rank s-2$. We
therefore assume WLOG that $\rank t=\rank s-2$. In other words,
$\rank s-\rank
t=2$. In terms of the way that we draw our poset $P$, this means that the
point $s$ lies two rows above the point $t$.
\end{vershort}

\begin{verlong}
We first observe that Claim 4 trivially holds if $\rank s-\rank t\neq
2$\ \ \ \ \footnote{\textit{Proof.} Assume that $\rank s-\rank t\neq2$.
\par
We claim that there exists no $x\in P$ such that $s\blacktriangleright
x\gtrdot t$. Indeed, assume the contrary. Thus, such an $x$ does exist.
Consider this $x$. Then, (\ref{eq.pf.lem.rect.conv.rank-1}) (applied to $s$
and $x$ instead of $x$ and $y$) yields $\rank x=\rank s-1$ (since
$s\blacktriangleright x$). Also, (\ref{eq.pf.lem.rect.conv.rank-1}) (applied
to $t$ instead of $y$) yields $\rank t=\rank x-1$ (since $x\gtrdot t$). Thus,%
\[
\rank t=\underbrace{\rank x}_{=\rank s-1}-1=\rank s-1-1=\rank s-2,
\]
so that $\rank s-\rank t=2$; but this contradicts $\rank s-\rank t\neq2$. This
contradiction shows that our assumption was false. Hence, we have shown that
there exists no $x\in P$ such that $s\blacktriangleright x\gtrdot t$. Thus,
the sum $\sum_{\substack{x\in P;\\s\blacktriangleright x\gtrdot t}%
}s\overline{x}\upslack^{t}$ is empty. Similarly, we can see that the sum
$\sum_{\substack{x\in P;\\s\gtrdot x\blacktriangleright t}}\downslack^{s}%
x\overline{t}$ is empty. Thus, these two sums are both empty and therefore
both equal $0$. Hence, $\sum_{\substack{x\in P;\\s\blacktriangleright x\gtrdot
t}}s\overline{x}\upslack^{t}=\sum_{\substack{x\in P;\\s\gtrdot
x\blacktriangleright t}}\downslack^{s}x\overline{t}$. We have thus proved
Claim 4 in the case when $\rank s-\rank t\neq2$.}. Thus, for the rest of this
proof, we WLOG assume that $\rank s-\rank t=2$. In terms of the way that we
draw our poset $P$, this means that the point $s$ lies two rows above the
point $t$.
\end{verlong}

The definition of $\upslack_{\ell}^{t}$ yields $\upslack_{\ell}^{t}%
=\overline{\sum\limits_{x\gtrdot t}\overline{x_{\ell}}}\cdot\overline{t_{\ell
}}$. Omitting the subscripts, we can rewrite this as%
\begin{equation}
\upslack^{t}=\overline{\sum\limits_{x\gtrdot t}\overline{x}}\cdot\overline{t}.
\label{pf.lem.rect.conv.c4.pf.1}%
\end{equation}

The definition of $\downslack_{\ell}^{s}$ yields $\downslack_{\ell}%
^{s}=s_{\ell}\cdot\overline{\sum\limits_{x\lessdot s}x_{\ell}}$. Omitting the
subscripts, we can rewrite this as%
\begin{equation}
\downslack^{s}=s\cdot\overline{\sum\limits_{x\lessdot s}x}.
\label{pf.lem.rect.conv.c4.pf.2}%
\end{equation}

\begin{vershort}
Write the elements $s,t\in P$ in the forms $s=\left(  i,j\right)  $ and
$t=\left(  i^{\prime},j^{\prime}\right)  $. Then, $\rank s=i+j-1$ and $\rank
t=i^{\prime}+j^{\prime}-1$. Hence, $\rank s-\rank t=i+j-i^{\prime}-j^{\prime}%
$, so that $i+j-i^{\prime}-j^{\prime}=\rank s-\rank t=2$. Thus, $j^{\prime
}=i+j-i^{\prime}-2$.
\end{vershort}

\begin{verlong}
Write $s\in P$ in the form $s=\left(  i,j\right)  $. Write $t\in P$ in the
form $t=\left(  i^{\prime},j^{\prime}\right)  $. From $s=\left(  i,j\right)
$, we obtain $\rank s=i+j-1$. Likewise, $\rank t=i^{\prime}+j^{\prime}-1$.
Hence,%
\[
\rank s-\rank t=\left(  i+j-1\right)  -\left(  i^{\prime}+j^{\prime}-1\right)
=i+j-i^{\prime}-j^{\prime},
\]
so that $i+j-i^{\prime}-j^{\prime}=\rank s-\rank t=2$. Thus, $j^{\prime
}=i+j-i^{\prime}-2$.
\end{verlong}

We are in one of the following three cases:

\textit{Case 1:} We have $i^{\prime}<i-1$.

\textit{Case 2:} We have $i^{\prime}=i-1$.

\textit{Case 3:} We have $i^{\prime}>i-1$.

Representative examples for these three cases are illustrated in the following
pictures:%
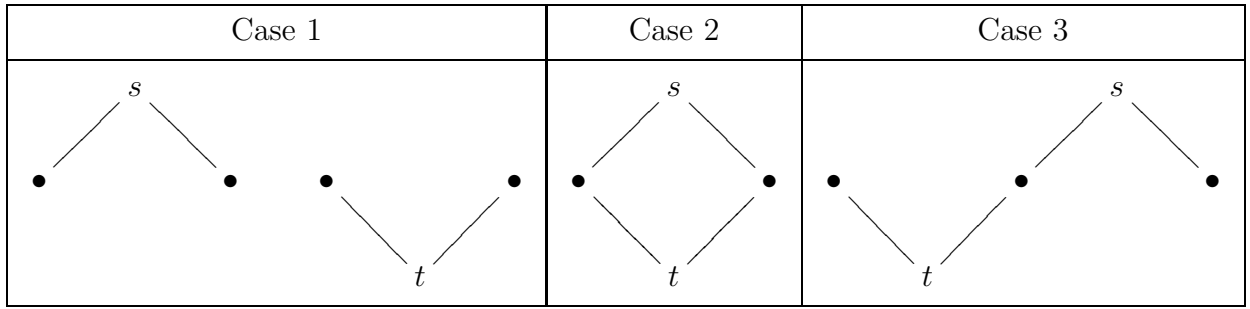
\[%
\begin{tabular}
[c]{|c|c|c|}\hline
Case 1 & \multicolumn{1}{|c|}{Case 2} & Case 3\\\hline
$%
\xymatrix{
& s \are[dl] \are[dr] \\
\bullet& & \bullet& \bullet& & \bullet\\
& & & & t \are[ur] \are[ul]
}%
$ & \multicolumn{1}{|c|}{$%
\xymatrix{
& s \are[dl] \are[dr] \\
\bullet& & \bullet\\
& t \are[ur] \are[ul]
}%
$} & $%
\xymatrix{
& & & s \are[dl] \are[dr] \\
\bullet& & \bullet& & \bullet\\
& t \are[ur] \are[ul]
}%
$\\\hline
\end{tabular}
\ \
\]
(the bullets signify the positions of potential neighbors of $s$ and $t$; some
of these positions may fall outside of $P$, but this does not disturb our
argument). In terms of the way we draw our poset $P$, the three cases can be
reformulated as \textquotedblleft the point $s$ lies further west than
$t$\textquotedblright\ (Case 1), \textquotedblleft the point $s$ lies due
north of $t$\textquotedblright\ (Case 2) and \textquotedblleft the point $s$
lies further east than $t$\textquotedblright\ (Case 3). Note that two elements
$x,y\in P$ satisfy $x\blacktriangleright y$ if and only if $y$ lies one step
south and some arbitrary distance east of $x$ in our pictures.

\begin{vershort}
Let us first consider Case 1. In this case, the point $s$ lies further west
than $t$. Thus, $s$ lies further west than any neighbor of $t$ as
well\footnote{This becomes fairly clear if you draw the configuration and
recall that $s$ lies two rows above $t$ (so that $P$ has points further east
than $s$ but further west than $t$). A rigorous version of this argument
(without reference to pictures) can be found in the detailed version of the
present paper.}. Hence, each element $x$ of $P$ that satisfies $x\gtrdot t$
must satisfy $s\blacktriangleright x$ automatically. Therefore, the summation
sign $\sum_{\substack{x\in P;\\s\blacktriangleright x\gtrdot t}}$ can be
simplified to $\sum_{\substack{x\in P;\\x\gtrdot t}}$, and even further to
$\sum_{x\gtrdot t}$ (because any $x\in\widehat{P}$ that satisfies $x\gtrdot t$
must belong to $P$ automatically\footnote{Indeed, the rank of any such $x$
must lie between the ranks of $s$ and $t$, and thus $x$ cannot be $0$ or $1$%
.}). Hence,%
\begin{align}
\sum_{\substack{x\in P;\\s\blacktriangleright x\gtrdot t}}s\overline
{x}\upslack^{t}  &  =\sum_{x\gtrdot t}s\overline{x}\upslack^{t}=s\left(
\sum_{x\gtrdot t}\overline{x}\right)  \underbrace{\upslack^{t}}%
_{\substack{=\overline{\sum\limits_{x\gtrdot t}\overline{x}}\cdot\overline
{t}\\\text{(by (\ref{pf.lem.rect.conv.c4.pf.1}))}}}=s\underbrace{\left(
\sum_{x\gtrdot t}\overline{x}\right)  \overline{\sum\limits_{x\gtrdot
t}\overline{x}}}_{=1}\cdot\,\overline{t}\nonumber\\
&  =s\overline{t}. \label{pf.lem.rect.conv.c4.c1.pf.short.1}%
\end{align}

Recall again that the point $s$ lies further west than $t$. Thus, any neighbor
of $s$ lies further west than $t$ as well (since $s$ lies two rows above $t$).
Hence, each element $x$ of $P$ that satisfies $s\gtrdot x$ must satisfy
$x\blacktriangleright t$ automatically. Therefore, the summation sign
$\sum_{\substack{x\in P;\\s\gtrdot x\blacktriangleright t}}$ can be simplified
to $\sum_{\substack{x\in P;\\s\gtrdot x}}=\sum_{\substack{x\in P;\\x\lessdot
s}}$, and even further to $\sum_{x\lessdot s}$ (because any $x\in\widehat{P}$
that satisfies $x\lessdot s$ must belong to $P$ automatically\footnote{Indeed,
the rank of any such $x$ must lie between the ranks of $s$ and $t$, and thus
$x$ cannot be $0$ or $1$.}). Hence,%
\[
\sum_{\substack{x\in P;\\s\gtrdot x\blacktriangleright t}}\downslack^{s}%
x\overline{t}=\sum_{x\lessdot s}\downslack^{s}x\overline{t}%
=\underbrace{\downslack^{s}}_{\substack{=s\cdot\overline{\sum
\limits_{x\lessdot s}x}\\\text{(by (\ref{pf.lem.rect.conv.c4.pf.2}))}}}\left(
\sum_{x\lessdot s}x\right)  \overline{t}=s\cdot\underbrace{\overline
{\sum\limits_{x\lessdot s}x}\left(  \sum_{x\lessdot s}x\right)  }%
_{=1}\overline{t}=s\overline{t}.
\]
Comparing this with (\ref{pf.lem.rect.conv.c4.c1.pf.short.1}), we obtain
$\sum_{\substack{x\in P;\\s\blacktriangleright x\gtrdot t}}s\overline
{x}\upslack^{t}=\sum_{\substack{x\in P;\\s\gtrdot x\blacktriangleright
t}}\downslack^{s}x\overline{t}$. Thus, Claim 4 is proved in Case 1.
\end{vershort}

\begin{verlong}
Let us first consider Case 1. In this case, we have $i^{\prime}<i-1$. Thus,
$i^{\prime}+1<i$. Hence, each element $x$ of $P$ that satisfies $x\gtrdot t$
must satisfy $s\blacktriangleright x$ automatically\footnote{\textit{Proof.}
Let $x\in P$ be such that $x\gtrdot t$. We must show that
$s\blacktriangleright x$.
\par
From $x\gtrdot t$, we obtain $\rank t=\rank x-1$ (by
(\ref{eq.pf.lem.rect.conv.rank-1}), applied to $y=t$) and thus
$\rank x-\rank
t=1$. Now,%
\[
\rank s-\rank x=\underbrace{\left(  \rank s-\rank t\right)  }_{=2}%
-\underbrace{\left(  \rank x-\rank t\right)  }_{=1}=2-1=1.
\]
Thus, $\rank s=\rank x+1$.
\par
Furthermore, $x$ is either $\left(  i^{\prime}+1,j^{\prime}\right)  $ or
$\left(  i^{\prime},j^{\prime}+1\right)  $ (since $x\gtrdot t=\left(
i^{\prime},j^{\prime}\right)  $). Hence, $\first x\leq i^{\prime}+1<i=\first
s$ (since $s=\left(  i,j\right)  $), so that $\first s>\first x$. Combining
this with $\rank s=\rank x+1$, we obtain $s\blacktriangleright x$ (by the
definition of the relation $\blacktriangleright$).}. Therefore, the summation
sign $\sum_{\substack{x\in P;\\s\blacktriangleright x\gtrdot t}}$ can be
simplified to $\sum_{\substack{x\in P;\\x\gtrdot t}}$. In turn, the latter
summation sign $\sum_{\substack{x\in P;\\x\gtrdot t}}$ can be rewritten as
$\sum_{\substack{x\in\widehat{P};\\x\gtrdot t}}$ (because any $x\in
\widehat{P}$ that satisfies $x\gtrdot t$ must automatically belong to
$P$\ \ \ \ \footnote{\textit{Proof.} From $t=\left(  i^{\prime},j^{\prime
}\right)  $, we obtain $\first t=i^{\prime}<i-1<i\leq p=\first\left(
p,q\right)  $. Hence, $\first t\neq\first\left(  p,q\right)  $, so that
$t\neq\left(  p,q\right)  $. This shows that $t$ is not a maximal element of
$P$ (since the only maximal element of $P$ is $\left(  p,q\right)  $). In
other words, we don't have $t\lessdot1$ in $\widehat{P}$. In other words, we
don't have $1\gtrdot t$ in $\widehat{P}$. Hence, any $x\in\widehat{P}$ that
satisfies $x\gtrdot t$ must automatically belong to $P$ (since we have neither
$1\gtrdot t$ nor $0\gtrdot t$).}). Hence, we obtain the following equality of
summation signs:%
\[
\sum_{\substack{x\in P;\\s\blacktriangleright x\gtrdot t}}=\sum
_{\substack{x\in P;\\x\gtrdot t}}=\sum_{\substack{x\in\widehat{P};\\x\gtrdot
t}}=\sum_{x\gtrdot t}%
\]
(since our sums are understood to range over $\widehat{P}$ by default). Thus,%
\begin{align}
\sum_{\substack{x\in P;\\s\blacktriangleright x\gtrdot t}}s\overline
{x}\upslack^{t}  &  =\sum_{x\gtrdot t}s\overline{x}\upslack^{t}=s\left(
\sum_{x\gtrdot t}\overline{x}\right)  \underbrace{\upslack^{t}}%
_{\substack{=\overline{\sum\limits_{x\gtrdot t}\overline{x}}\cdot\overline
{t}\\\text{(by (\ref{pf.lem.rect.conv.c4.pf.1}))}}}=s\underbrace{\left(
\sum_{x\gtrdot t}\overline{x}\right)  \overline{\sum\limits_{x\gtrdot
t}\overline{x}}}_{=1}\cdot\,\overline{t}\nonumber\\
&  =s\overline{t}. \label{pf.lem.rect.conv.c4.c1.pf.1}%
\end{align}

Furthermore, $i^{\prime}<i-1$, so that $i-1>i^{\prime}$. Hence, each element
$x$ of $P$ that satisfies $s\gtrdot x$ must satisfy $x\blacktriangleright t$
automatically\footnote{\textit{Proof.} Let $x\in P$ be such that $s\gtrdot x$.
We must show that $x\blacktriangleright t$.
\par
From $s\gtrdot x$, we obtain $\rank x=\rank s-1$ (by
(\ref{eq.pf.lem.rect.conv.rank-1}), applied to $s$ and $x$ instead of $x$ and
$y$). Hence, $\rank s-\rank x=1$, so that%
\[
\rank x-\rank t=\underbrace{\left(  \rank s-\rank t\right)  }_{=2}%
-\underbrace{\left(  \rank s-\rank x\right)  }_{=1}=2-1=1.
\]
Thus, $\rank x=\rank t+1$.
\par
Furthermore, $x$ is either $\left(  i-1,j\right)  $ or $\left(  i,j-1\right)
$ (since $\left(  i,j\right)  =s\gtrdot x$). Hence, $\first x\geq
i-1>i^{\prime}=\first t$ (since $t=\left(  i^{\prime},j^{\prime}\right)  $).
Combining this with $\rank x=\rank t+1$, we obtain $x\blacktriangleright t$
(by the definition of the relation $\blacktriangleright$).}. Therefore, the
summation sign $\sum_{\substack{x\in P;\\s\gtrdot x\blacktriangleright t}}$
can be simplified to $\sum_{\substack{x\in P;\\s\gtrdot x}}$. In turn, the
latter summation sign $\sum_{\substack{x\in P;\\s\gtrdot x}}$ can be rewritten
as $\sum_{\substack{x\in\widehat{P};\\s\gtrdot x}}$ (because any
$x\in\widehat{P}$ that satisfies $s\gtrdot x$ must automatically belong to
$P$\ \ \ \ \footnote{\textit{Proof.} From $s=\left(  i,j\right)  $, we obtain
$\first s=i>i-1>i^{\prime}\geq1=\first\left(  1,1\right)  $. Hence, $\first
s\neq\first\left(  1,1\right)  $, so that $s\neq\left(  1,1\right)  $. This
shows that $s$ is not a minimal element of $P$ (since the only minimal element
of $P$ is $\left(  1,1\right)  $). In other words, we don't have $s\gtrdot0$
in $\widehat{P}$. Hence, any $x\in\widehat{P}$ that satisfies $s\gtrdot x$
must automatically belong to $P$ (since we have neither $s\gtrdot0$ nor
$s\gtrdot1$).}). Hence, we obtain the following equality of summation signs:%
\[
\sum_{\substack{x\in P;\\s\gtrdot x\blacktriangleright t}}=\sum
_{\substack{x\in P;\\s\gtrdot x}}=\sum_{\substack{x\in\widehat{P};\\s\gtrdot
x}}=\sum_{\substack{x\in\widehat{P};\\x\lessdot s}}=\sum_{x\lessdot s}%
\]
(since our sums are understood to range over $\widehat{P}$ by default). Thus,%
\[
\sum_{\substack{x\in P;\\s\gtrdot x\blacktriangleright t}}\downslack^{s}%
x\overline{t}=\sum_{x\lessdot s}\downslack^{s}x\overline{t}%
=\underbrace{\downslack^{s}}_{\substack{=s\cdot\overline{\sum
\limits_{x\lessdot s}x}\\\text{(by (\ref{pf.lem.rect.conv.c4.pf.2}))}}}\left(
\sum_{x\lessdot s}x\right)  \overline{t}=s\cdot\underbrace{\overline
{\sum\limits_{x\lessdot s}x}\left(  \sum_{x\lessdot s}x\right)  }%
_{=1}\overline{t}=s\overline{t}.
\]
Comparing this with (\ref{pf.lem.rect.conv.c4.c1.pf.1}), we obtain
$\sum_{\substack{x\in P;\\s\blacktriangleright x\gtrdot t}}s\overline
{x}\upslack^{t}=\sum_{\substack{x\in P;\\s\gtrdot x\blacktriangleright
t}}\downslack^{s}x\overline{t}$. Thus, Claim 4 is proved in Case 1.
\end{verlong}

Let us now consider Case 2. In this case, we have $i^{\prime}=i-1$. Hence,
$j^{\prime}=i+j-\underbrace{i^{\prime}}_{=i-1}-2=i+j-\left(  i-1\right)
-2=j-1$. Thus, $t=\left(  i^{\prime},j^{\prime}\right)  =\left(
i-1,j-1\right)  $ (since $i^{\prime}=i-1$ and $j^{\prime}=j-1$). Let
$v:=\left(  i,j-1\right)  $ and $w:=\left(  i-1,j\right)  $. In our coordinate
system, the four points%
\[
s=\left(  i,j\right)  ,\ \ \ \ \ \ \ \ \ \ t=\left(  i-1,j-1\right)
,\ \ \ \ \ \ \ \ \ \ v=\left(  i,j-1\right)  ,\ \ \ \ \ \ \ \ \ \ w=\left(
i-1,j\right)
\]
are arranged in a $1\times1$-square, which looks as follows:%
\begin{equation}%
\xymatrix@R=0.6pc@C=0.6pc{
& s \are[dl] \are[dr] \\
v \are[dr] & & w \are[dl] \\
& t & & .
}
\label{pf.lem.rect.conv.c4.c2.diag}%
\end{equation}
Hence, $v$ and $w$ belong to $P$ (since $s$ and $t$ belong to $P$), and
furthermore, Lemma \ref{lem.rect.four} \textbf{(b)} (applied to $s$ and $t$
instead of $u$ and $d$) yields
\[
\overline{w_{\ell}}\cdot\upslack_{\ell}^{t}\cdot t_{\ell}=\overline{s_{\ell}%
}\cdot\downslack_{\ell}^{s}\cdot v_{\ell}.
\]
Since we are omitting subscripts, we can rewrite this as follows:%
\[
\overline{w}\cdot\upslack^{t}\cdot t=\overline{s}\cdot\downslack^{s}\cdot v.
\]

The picture (\ref{pf.lem.rect.conv.c4.c2.diag}) shows that we have
$s\blacktriangleright w$ but not $s\blacktriangleright v$. Hence, there is
only one element $x\in P$ that satisfies $s\blacktriangleright x\gtrdot t$;
namely, this element $x$ is $w$. Hence,%
\begin{align}
\sum_{\substack{x\in P;\\s\blacktriangleright x\gtrdot t}}s\overline
{x}\upslack^{t}  &  =s\overline{w}\upslack^{t}=s\cdot\overline{w}%
\cdot\upslack^{t}\cdot\underbrace{1}_{=t\cdot\overline{t}}=s\cdot
\underbrace{\overline{w}\cdot\upslack^{t}\cdot t}_{=\overline{s}%
\cdot\downslack^{s}\cdot v}\cdot\,\overline{t}=\underbrace{s\cdot\overline{s}%
}_{=1}\cdot\downslack^{s}\cdot v\cdot\overline{t}\nonumber\\
&  =\downslack^{s}\cdot v\cdot\overline{t}. \label{pf.lem.rect.conv.c4.c2.L}%
\end{align}
On the other hand, the picture (\ref{pf.lem.rect.conv.c4.c2.diag}) shows that
we have $v\blacktriangleright t$ but not $w\blacktriangleright t$. Hence,
there is only one element $x\in P$ that satisfies $s\gtrdot
x\blacktriangleright t$; namely, this element $x$ is $v$. Hence,%
\[
\sum_{\substack{x\in P;\\s\gtrdot x\blacktriangleright t}}\downslack^{s}%
x\overline{t}=\downslack^{s}v\overline{t}=\downslack^{s}\cdot v\cdot
\overline{t}.
\]
Comparing this with (\ref{pf.lem.rect.conv.c4.c2.L}), we obtain $\sum
_{\substack{x\in P;\\s\blacktriangleright x\gtrdot t}}s\overline
{x}\upslack^{t}=\sum_{\substack{x\in P;\\s\gtrdot x\blacktriangleright
t}}\downslack^{s}x\overline{t}$. Thus, Claim 4 is proved in Case 2.

Let us finally consider Case 3. In this case, we have $i^{\prime}>i-1$. Thus,
$i^{\prime}\geq i$ (since $i^{\prime}$ and $i$ are integers), so that $i\leq
i^{\prime}$. Note that $i=\first s$ (since $s=\left(  i,j\right)  $) and
$i^{\prime}=\first t$ (since $t=\left(  i^{\prime},j^{\prime}\right)  $).

\begin{vershort}
There exists no $x\in P$ satisfying $s\blacktriangleright x\gtrdot t$ (because
if $x\in P$ satisfies $s\blacktriangleright x\gtrdot t$, then $x\gtrdot
t=\left(  i^{\prime},j^{\prime}\right)  $ entails $\first x\geq i^{\prime}\geq
i=\first s$, but this clearly contradicts $s\blacktriangleright x$). Hence,
the sum $\sum_{\substack{x\in P;\\s\blacktriangleright x\gtrdot t}%
}s\overline{x}\upslack^{t}$ is empty. Thus, $\sum_{\substack{x\in
P;\\s\blacktriangleright x\gtrdot t}}s\overline{x}\upslack^{t}=0$.
\end{vershort}

\begin{verlong}
There exists no $x\in P$ satisfying $s\blacktriangleright x\gtrdot
t$\ \ \ \ \footnote{\textit{Proof.} Assume the contrary. Thus, there exists an
$x\in P$ satisfying $s\blacktriangleright x\gtrdot t$. Consider this $x$.
\par
We have $x\gtrdot t=\left(  i^{\prime},j^{\prime}\right)  $; thus, $x$ is
either $\left(  i^{\prime}+1,j^{\prime}\right)  $ or $\left(  i^{\prime
},j^{\prime}+1\right)  $. Hence, $\first x\geq i^{\prime}\geq i=\first s$.
However, from $s\blacktriangleright x$, we obtain $\first s>\first x$ (by the
definition of the relation $\blacktriangleright$). This contradicts $\first
x\geq\first s$. This contradiction shows that our assumption was false, qed.}.
Hence, the sum $\sum_{\substack{x\in P;\\s\blacktriangleright x\gtrdot
t}}s\overline{x}\upslack^{t}$ is empty. Thus, $\sum_{\substack{x\in
P;\\s\blacktriangleright x\gtrdot t}}s\overline{x}\upslack^{t}=0$.
\end{verlong}

\begin{vershort}
Furthermore, there exists no $x\in P$ satisfying $s\gtrdot
x\blacktriangleright t$ (because if $x\in P$ satisfies $s\gtrdot
x\blacktriangleright t$, then $\left(  i,j\right)  =s\gtrdot x$ entails
$\first x\leq i\leq i^{\prime}=\first t$; but this clearly contradicts
$x\blacktriangleright t$). Hence, the sum $\sum_{\substack{x\in P;\\s\gtrdot
x\blacktriangleright t}}\downslack^{s}x\overline{t}$ is empty. Thus,
$\sum_{\substack{x\in P;\\s\gtrdot x\blacktriangleright t}}\downslack^{s}%
x\overline{t}=0$.
\end{vershort}

\begin{verlong}
Furthermore, there exists no $x\in P$ satisfying $s\gtrdot
x\blacktriangleright t$\ \ \ \ \footnote{\textit{Proof.} Assume the contrary.
Thus, there exists an $x\in P$ satisfying $s\gtrdot x\blacktriangleright t$.
Consider this $x$.
\par
We have $s\gtrdot x$, so that $x\lessdot s=\left(  i,j\right)  $; thus, $x$ is
either $\left(  i-1,j\right)  $ or $\left(  i,j-1\right)  $. Hence, $\first
x\leq i\leq i^{\prime}=\first t$. However, from $x\blacktriangleright t$, we
obtain $\first x>\first t$ (by the definition of the relation
$\blacktriangleright$). This contradicts $\first x\leq\first t$. This
contradiction shows that our assumption was false, qed.}. Hence, the sum
$\sum_{\substack{x\in P;\\s\gtrdot x\blacktriangleright t}}\downslack^{s}%
x\overline{t}$ is empty. Thus, $\sum_{\substack{x\in P;\\s\gtrdot
x\blacktriangleright t}}\downslack^{s}x\overline{t}=0$.
\end{verlong}

Comparing this with $\sum_{\substack{x\in P;\\s\blacktriangleright x\gtrdot
t}}s\overline{x}\upslack^{t}=0$, we obtain $\sum_{\substack{x\in
P;\\s\blacktriangleright x\gtrdot t}}s\overline{x}\upslack^{t}=\sum
_{\substack{x\in P;\\s\gtrdot x\blacktriangleright t}}\downslack^{s}%
x\overline{t}$. Thus, Claim 4 is proved in Case 3.

We have now proved Claim 4 in all three cases.
\end{proof}

We can now step to the proof of Claim 3:

\begin{proof}
[Proof of Claim 3.]Let $j\in\left\{  0,1,\ldots,r-2\right\}  $.

\begin{vershort}
We know that any path-jump-path from $u$ to $d^{\prime}$ must have the form
\newline$\left(  v_{0}\gtrdot v_{1}\gtrdot\cdots\gtrdot v_{i}%
\blacktriangleright v_{i+1}\gtrdot v_{i+2}\gtrdot\cdots\gtrdot v_{r}\right)
$. If such a path-jump-path has jump at $j$, then it must have the form
$\left(  v_{0}\gtrdot v_{1}\gtrdot\cdots\gtrdot v_{j}\blacktriangleright
v_{j+1}\gtrdot v_{j+2}\gtrdot\cdots\gtrdot v_{r}\right)  $. Thus,%
\begin{align*}
&  \sum_{\substack{\mathbf{p}\text{ is a path-jump-path}\\\text{from }u\text{
to }d^{\prime}\\\text{with jump at }j}}E_{\mathbf{p}}\\
&  =\sum_{\substack{\left(  v_{0}\gtrdot v_{1}\gtrdot\cdots\gtrdot
v_{j}\blacktriangleright v_{j+1}\gtrdot v_{j+2}\gtrdot\cdots\gtrdot
v_{r}\right)  \\\text{is a path-jump-path}\\\text{from }u\text{ to }d^{\prime
}\\\text{with jump at }j}}\underbrace{E_{\left(  v_{0}\gtrdot v_{1}%
\gtrdot\cdots\gtrdot v_{j}\blacktriangleright v_{j+1}\gtrdot v_{j+2}%
\gtrdot\cdots\gtrdot v_{r}\right)  }}_{\substack{=\downslack^{v_{0}%
}\downslack^{v_{1}}\cdots\downslack^{v_{j-1}}v_{j}\overline{v_{j+1}%
}\upslack^{v_{j+2}}\upslack^{v_{j+3}}\cdots\upslack^{v_{r}}\\\text{(by the
definition of }E_{\left(  v_{0}\gtrdot v_{1}\gtrdot\cdots\gtrdot
v_{j}\blacktriangleright v_{j+1}\gtrdot v_{j+2}\gtrdot\cdots\gtrdot
v_{r}\right)  }\text{)}}}\\
&  =\sum_{\substack{\left(  v_{0}\gtrdot v_{1}\gtrdot\cdots\gtrdot
v_{j}\blacktriangleright v_{j+1}\gtrdot v_{j+2}\gtrdot\cdots\gtrdot
v_{r}\right)  \\\text{is a path-jump-path}\\\text{from }u\text{ to }d^{\prime
}\\\text{with jump at }j}}\downslack^{v_{0}}\downslack^{v_{1}}\cdots
\downslack^{v_{j-1}}v_{j}\overline{v_{j+1}}\upslack^{v_{j+2}}\upslack^{v_{j+3}%
}\cdots\upslack^{v_{r}}\\
&  =\sum_{\substack{\left(  v_{0}\gtrdot v_{1}\gtrdot\cdots\gtrdot
v_{j}\right)  \\\text{is a path starting at }u}}\ \ \sum_{\substack{\left(
v_{j+2}\gtrdot v_{j+3}\gtrdot\cdots\gtrdot v_{r}\right)  \\\text{is a path
ending at }d^{\prime}}}\ \ \sum_{\substack{v_{j+1}\in P;\\v_{j}%
\blacktriangleright v_{j+1}\gtrdot v_{j+2}}}\downslack^{v_{0}}%
\downslack^{v_{1}}\cdots\downslack^{v_{j-1}}v_{j}\overline{v_{j+1}%
}\upslack^{v_{j+2}}\upslack^{v_{j+3}}\cdots\upslack^{v_{r}}\\
&  \ \ \ \ \ \ \ \ \ \ \ \ \ \ \ \ \ \ \ \ \left(
\begin{array}
[c]{c}%
\text{here, we have broken up our}\\
\text{path-jump-path }\left(  v_{0}\gtrdot v_{1}\gtrdot\cdots\gtrdot
v_{j}\blacktriangleright v_{j+1}\gtrdot v_{j+2}\gtrdot\cdots\gtrdot
v_{r}\right) \\
\text{into two paths }\left(  v_{0}\gtrdot v_{1}\gtrdot\cdots\gtrdot
v_{j}\right)  \text{ and }\left(  v_{j+2}\gtrdot v_{j+3}\gtrdot\cdots\gtrdot
v_{r}\right) \\
\text{and an intermediate vertex }v_{j+1}\text{ satisfying }v_{j}%
\blacktriangleright v_{j+1}\gtrdot v_{j+2}%
\end{array}
\right) \\
&  =\sum_{\substack{\left(  v_{0}\gtrdot v_{1}\gtrdot\cdots\gtrdot
v_{j}\right)  \\\text{is a path starting at }u}}\ \ \sum_{\substack{\left(
v_{j+2}\gtrdot v_{j+3}\gtrdot\cdots\gtrdot v_{r}\right)  \\\text{is a path
ending at }d^{\prime}}}\ \ \sum_{\substack{x\in P;\\v_{j}\blacktriangleright
x\gtrdot v_{j+2}}}\downslack^{v_{0}}\downslack^{v_{1}}\cdots
\downslack^{v_{j-1}}v_{j}\overline{x}\upslack^{v_{j+2}}\upslack^{v_{j+3}%
}\cdots\upslack^{v_{r}}\\
&  \ \ \ \ \ \ \ \ \ \ \ \ \ \ \ \ \ \ \ \ \left(  \text{here we have renamed
}v_{j+1}\text{ as }x\text{ in the inner sum}\right) \\
&  =\sum_{\substack{\left(  v_{0}\gtrdot v_{1}\gtrdot\cdots\gtrdot
v_{j}\right)  \\\text{is a path starting at }u}}\ \ \sum_{\substack{\left(
v_{j+2}\gtrdot v_{j+3}\gtrdot\cdots\gtrdot v_{r}\right)  \\\text{is a path
ending at }d^{\prime}}}\downslack^{v_{0}}\downslack^{v_{1}}\cdots
\downslack^{v_{j-1}}\underbrace{\sum_{\substack{x\in P;\\v_{j}%
\blacktriangleright x\gtrdot v_{j+2}}}v_{j}\overline{x}\upslack^{v_{j+2}}%
}_{\substack{=\sum_{\substack{x\in P;\\v_{j}\gtrdot x\blacktriangleright
v_{j+2}}}\downslack^{v_{j}}x\overline{v_{j+2}}\\\text{(by Claim 4,}%
\\\text{applied to }s=v_{j}\\\text{and }t=v_{j+2}\text{)}}}\upslack^{v_{j+3}%
}\upslack^{v_{j+4}}\cdots\upslack^{v_{r}}\\
&  =\sum_{\substack{\left(  v_{0}\gtrdot v_{1}\gtrdot\cdots\gtrdot
v_{j}\right)  \\\text{is a path starting at }u}}\ \ \sum_{\substack{\left(
v_{j+2}\gtrdot v_{j+3}\gtrdot\cdots\gtrdot v_{r}\right)  \\\text{is a path
ending at }d^{\prime}}}\downslack^{v_{0}}\downslack^{v_{1}}\cdots
\downslack^{v_{j-1}}\sum_{\substack{x\in P;\\v_{j}\gtrdot x\blacktriangleright
v_{j+2}}}\downslack^{v_{j}}x\overline{v_{j+2}}\upslack^{v_{j+3}}%
\upslack^{v_{j+4}}\cdots\upslack^{v_{r}}.
\end{align*}

We know that any path-jump-path from $u$ to $d^{\prime}$ must have the form
\newline$\left(  v_{0}\gtrdot v_{1}\gtrdot\cdots\gtrdot v_{i}%
\blacktriangleright v_{i+1}\gtrdot v_{i+2}\gtrdot\cdots\gtrdot v_{r}\right)
$. If such a path-jump-path has jump at $j+1$, then it must have the form
$\left(  v_{0}\gtrdot v_{1}\gtrdot\cdots\gtrdot v_{j+1}\blacktriangleright
v_{j+2}\gtrdot v_{j+3}\gtrdot\cdots\gtrdot v_{r}\right)  $. Thus,%
\begin{align*}
&  \sum_{\substack{\mathbf{p}\text{ is a path-jump-path}\\\text{from }u\text{
to }d^{\prime}\\\text{with jump at }j+1}}E_{\mathbf{p}}\\
&  =\sum_{\substack{\left(  v_{0}\gtrdot v_{1}\gtrdot\cdots\gtrdot
v_{j+1}\blacktriangleright v_{j+2}\gtrdot v_{j+3}\gtrdot\cdots\gtrdot
v_{r}\right)  \\\text{is a path-jump-path}\\\text{from }u\text{ to }d^{\prime
}\\\text{with jump at }j+1}}\underbrace{E_{\left(  v_{0}\gtrdot v_{1}%
\gtrdot\cdots\gtrdot v_{j+1}\blacktriangleright v_{j+2}\gtrdot v_{j+3}%
\gtrdot\cdots\gtrdot v_{r}\right)  }}_{\substack{=\downslack^{v_{0}%
}\downslack^{v_{1}}\cdots\downslack^{v_{j}}v_{j+1}\overline{v_{j+2}%
}\upslack^{v_{j+3}}\upslack^{v_{j+4}}\cdots\upslack^{v_{r}}\\\text{(by the
definition of }E_{\left(  v_{0}\gtrdot v_{1}\gtrdot\cdots\gtrdot
v_{j+1}\blacktriangleright v_{j+2}\gtrdot v_{j+3}\gtrdot\cdots\gtrdot
v_{r}\right)  }\text{)}}}\\
&  =\sum_{\substack{\left(  v_{0}\gtrdot v_{1}\gtrdot\cdots\gtrdot
v_{j+1}\blacktriangleright v_{j+2}\gtrdot v_{j+3}\gtrdot\cdots\gtrdot
v_{r}\right)  \\\text{is a path-jump-path}\\\text{from }u\text{ to }d^{\prime
}\\\text{with jump at }j+1}}\downslack^{v_{0}}\downslack^{v_{1}}%
\cdots\downslack^{v_{j}}v_{j+1}\overline{v_{j+2}}\upslack^{v_{j+3}%
}\upslack^{v_{j+4}}\cdots\upslack^{v_{r}}\\
&  =\sum_{\substack{\left(  v_{0}\gtrdot v_{1}\gtrdot\cdots\gtrdot
v_{j}\right)  \\\text{is a path starting at }u}}\ \ \sum_{\substack{\left(
v_{j+2}\gtrdot v_{j+3}\gtrdot\cdots\gtrdot v_{r}\right)  \\\text{is a path
ending at }d^{\prime}}}\ \ \sum_{\substack{v_{j+1}\in P;\\v_{j}\gtrdot
v_{j+1}\blacktriangleright v_{j+2}}}\downslack^{v_{0}}\downslack^{v_{1}}%
\cdots\downslack^{v_{j}}v_{j+1}\overline{v_{j+2}}\upslack^{v_{j+3}%
}\upslack^{v_{j+4}}\cdots\upslack^{v_{r}}\\
&  =\sum_{\substack{\left(  v_{0}\gtrdot v_{1}\gtrdot\cdots\gtrdot
v_{j}\right)  \\\text{is a path starting at }u}}\ \ \sum_{\substack{\left(
v_{j+2}\gtrdot v_{j+3}\gtrdot\cdots\gtrdot v_{r}\right)  \\\text{is a path
ending at }d^{\prime}}}\ \ \sum_{\substack{x\in P;\\v_{j}\gtrdot
x\blacktriangleright v_{j+2}}}\downslack^{v_{0}}\downslack^{v_{1}}%
\cdots\downslack^{v_{j}}x\overline{v_{j+2}}\upslack^{v_{j+3}}\upslack^{v_{j+4}%
}\cdots\upslack^{v_{r}}\\
&  \ \ \ \ \ \ \ \ \ \ \ \ \ \ \ \ \ \ \ \ \left(  \text{here we have renamed
}v_{j+1}\text{ as }x\text{ in the inner sum}\right) \\
&  =\sum_{\substack{\left(  v_{0}\gtrdot v_{1}\gtrdot\cdots\gtrdot
v_{j}\right)  \\\text{is a path starting at }u}}\ \ \sum_{\substack{\left(
v_{j+2}\gtrdot v_{j+3}\gtrdot\cdots\gtrdot v_{r}\right)  \\\text{is a path
ending at }d^{\prime}}}\downslack^{v_{0}}\downslack^{v_{1}}\cdots
\downslack^{v_{j-1}}\sum_{\substack{x\in P;\\v_{j}\gtrdot x\blacktriangleright
v_{j+2}}}\downslack^{v_{j}}x\overline{v_{j+2}}\upslack^{v_{j+3}}%
\upslack^{v_{j+4}}\cdots\upslack^{v_{r}}.
\end{align*}

\end{vershort}

\begin{verlong}
We know that any path-jump-path from $u$ to $d^{\prime}$ must have the form
\newline$\left(  v_{0}\gtrdot v_{1}\gtrdot\cdots\gtrdot v_{i}%
\blacktriangleright v_{i+1}\gtrdot v_{i+2}\gtrdot\cdots\gtrdot v_{r}\right)
$. If such a path-jump-path has jump at $j$, then it must have the form
$\left(  v_{0}\gtrdot v_{1}\gtrdot\cdots\gtrdot v_{j}\blacktriangleright
v_{j+1}\gtrdot v_{j+2}\gtrdot\cdots\gtrdot v_{r}\right)  $. Thus,%
\begin{align*}
&  \sum_{\substack{\mathbf{p}\text{ is a path-jump-path}\\\text{from }u\text{
to }d^{\prime}\\\text{with jump at }j}}E_{\mathbf{p}}\\
&  =\sum_{\substack{\left(  v_{0}\gtrdot v_{1}\gtrdot\cdots\gtrdot
v_{j}\blacktriangleright v_{j+1}\gtrdot v_{j+2}\gtrdot\cdots\gtrdot
v_{r}\right)  \\\text{is a path-jump-path}\\\text{from }u\text{ to }d^{\prime
}\\\text{with jump at }j}}\underbrace{E_{\left(  v_{0}\gtrdot v_{1}%
\gtrdot\cdots\gtrdot v_{j}\blacktriangleright v_{j+1}\gtrdot v_{j+2}%
\gtrdot\cdots\gtrdot v_{r}\right)  }}_{\substack{=\downslack^{v_{0}%
}\downslack^{v_{1}}\cdots\downslack^{v_{j-1}}v_{j}\overline{v_{j+1}%
}\upslack^{v_{j+2}}\upslack^{v_{j+3}}\cdots\upslack^{v_{r}}\\\text{(by the
definition of }E_{\left(  v_{0}\gtrdot v_{1}\gtrdot\cdots\gtrdot
v_{j}\blacktriangleright v_{j+1}\gtrdot v_{j+2}\gtrdot\cdots\gtrdot
v_{r}\right)  }\text{)}}}\\
&  =\sum_{\substack{\left(  v_{0}\gtrdot v_{1}\gtrdot\cdots\gtrdot
v_{j}\blacktriangleright v_{j+1}\gtrdot v_{j+2}\gtrdot\cdots\gtrdot
v_{r}\right)  \\\text{is a path-jump-path}\\\text{from }u\text{ to }d^{\prime
}\\\text{with jump at }j}}\downslack^{v_{0}}\downslack^{v_{1}}\cdots
\downslack^{v_{j-1}}v_{j}\overline{v_{j+1}}\upslack^{v_{j+2}}\upslack^{v_{j+3}%
}\cdots\upslack^{v_{r}}\\
&  =\sum_{\substack{\left(  v_{0}\gtrdot v_{1}\gtrdot\cdots\gtrdot
v_{j}\right)  \\\text{is a path starting at }u}}\ \ \sum_{\substack{\left(
v_{j+2}\gtrdot v_{j+3}\gtrdot\cdots\gtrdot v_{r}\right)  \\\text{is a path
ending at }d^{\prime}}}\ \ \sum_{\substack{v_{j+1}\in P;\\v_{j}%
\blacktriangleright v_{j+1}\gtrdot v_{j+2}}}\downslack^{v_{0}}%
\downslack^{v_{1}}\cdots\downslack^{v_{j-1}}v_{j}\overline{v_{j+1}%
}\upslack^{v_{j+2}}\upslack^{v_{j+3}}\cdots\upslack^{v_{r}}\\
&  \ \ \ \ \ \ \ \ \ \ \ \ \ \ \ \ \ \ \ \ \left(
\begin{array}
[c]{c}%
\text{here, we have broken up our}\\
\text{path-jump-path }\left(  v_{0}\gtrdot v_{1}\gtrdot\cdots\gtrdot
v_{j}\blacktriangleright v_{j+1}\gtrdot v_{j+2}\gtrdot\cdots\gtrdot
v_{r}\right) \\
\text{into two paths }\left(  v_{0}\gtrdot v_{1}\gtrdot\cdots\gtrdot
v_{j}\right)  \text{ and }\left(  v_{j+2}\gtrdot v_{j+3}\gtrdot\cdots\gtrdot
v_{r}\right) \\
\text{and an intermediate vertex }v_{j+1}\text{ satisfying }v_{j}%
\blacktriangleright v_{j+1}\gtrdot v_{j+2}%
\end{array}
\right) \\
&  =\sum_{\substack{\left(  v_{0}\gtrdot v_{1}\gtrdot\cdots\gtrdot
v_{j}\right)  \\\text{is a path starting at }u}}\ \ \sum_{\substack{\left(
v_{j+2}\gtrdot v_{j+3}\gtrdot\cdots\gtrdot v_{r}\right)  \\\text{is a path
ending at }d^{\prime}}}\ \ \underbrace{\sum_{\substack{x\in P;\\v_{j}%
\blacktriangleright x\gtrdot v_{j+2}}}\downslack^{v_{0}}\downslack^{v_{1}%
}\cdots\downslack^{v_{j-1}}v_{j}\overline{x}\upslack^{v_{j+2}}%
\upslack^{v_{j+3}}\cdots\upslack^{v_{r}}}_{\substack{=\sum_{\substack{x\in
P;\\v_{j}\blacktriangleright x\gtrdot v_{j+2}}}\downslack^{v_{0}%
}\downslack^{v_{1}}\cdots\downslack^{v_{j-1}}v_{j}\overline{x}%
\upslack^{v_{j+2}}\upslack^{v_{j+3}}\upslack^{v_{j+4}}\cdots\upslack^{v_{r}%
}\\=\downslack^{v_{0}}\downslack^{v_{1}}\cdots\downslack^{v_{j-1}}%
\sum_{\substack{x\in P;\\v_{j}\blacktriangleright x\gtrdot v_{j+2}}%
}v_{j}\overline{x}\upslack^{v_{j+2}}\upslack^{v_{j+3}}\upslack^{v_{j+4}}%
\cdots\upslack^{v_{r}}}}\\
&  \ \ \ \ \ \ \ \ \ \ \ \ \ \ \ \ \ \ \ \ \left(  \text{here we have renamed
}v_{j+1}\text{ as }x\text{ in the inner sum}\right) \\
&  =\sum_{\substack{\left(  v_{0}\gtrdot v_{1}\gtrdot\cdots\gtrdot
v_{j}\right)  \\\text{is a path starting at }u}}\ \ \sum_{\substack{\left(
v_{j+2}\gtrdot v_{j+3}\gtrdot\cdots\gtrdot v_{r}\right)  \\\text{is a path
ending at }d^{\prime}}}\downslack^{v_{0}}\downslack^{v_{1}}\cdots
\downslack^{v_{j-1}}\underbrace{\sum_{\substack{x\in P;\\v_{j}%
\blacktriangleright x\gtrdot v_{j+2}}}v_{j}\overline{x}\upslack^{v_{j+2}}%
}_{\substack{=\sum_{\substack{x\in P;\\v_{j}\gtrdot x\blacktriangleright
v_{j+2}}}\downslack^{v_{j}}x\overline{v_{j+2}}\\\text{(by Claim 4,}%
\\\text{applied to }s=v_{j}\\\text{and }t=v_{j+2}\text{)}}}\upslack^{v_{j+3}%
}\upslack^{v_{j+4}}\cdots\upslack^{v_{r}}\\
&  =\sum_{\substack{\left(  v_{0}\gtrdot v_{1}\gtrdot\cdots\gtrdot
v_{j}\right)  \\\text{is a path starting at }u}}\ \ \sum_{\substack{\left(
v_{j+2}\gtrdot v_{j+3}\gtrdot\cdots\gtrdot v_{r}\right)  \\\text{is a path
ending at }d^{\prime}}}\downslack^{v_{0}}\downslack^{v_{1}}\cdots
\downslack^{v_{j-1}}\sum_{\substack{x\in P;\\v_{j}\gtrdot x\blacktriangleright
v_{j+2}}}\downslack^{v_{j}}x\overline{v_{j+2}}\upslack^{v_{j+3}}%
\upslack^{v_{j+4}}\cdots\upslack^{v_{r}}.
\end{align*}

We know that any path-jump-path from $u$ to $d^{\prime}$ must have the form
\newline$\left(  v_{0}\gtrdot v_{1}\gtrdot\cdots\gtrdot v_{i}%
\blacktriangleright v_{i+1}\gtrdot v_{i+2}\gtrdot\cdots\gtrdot v_{r}\right)
$. If such a path-jump-path has jump at $j+1$, then it must have the form
$\left(  v_{0}\gtrdot v_{1}\gtrdot\cdots\gtrdot v_{j+1}\blacktriangleright
v_{j+2}\gtrdot v_{j+3}\gtrdot\cdots\gtrdot v_{r}\right)  $. Thus,%
\begin{align*}
&  \sum_{\substack{\mathbf{p}\text{ is a path-jump-path}\\\text{from }u\text{
to }d^{\prime}\\\text{with jump at }j+1}}E_{\mathbf{p}}\\
&  =\sum_{\substack{\left(  v_{0}\gtrdot v_{1}\gtrdot\cdots\gtrdot
v_{j+1}\blacktriangleright v_{j+2}\gtrdot v_{j+3}\gtrdot\cdots\gtrdot
v_{r}\right)  \\\text{is a path-jump-path}\\\text{from }u\text{ to }d^{\prime
}\\\text{with jump at }j+1}}\underbrace{E_{\left(  v_{0}\gtrdot v_{1}%
\gtrdot\cdots\gtrdot v_{j+1}\blacktriangleright v_{j+2}\gtrdot v_{j+3}%
\gtrdot\cdots\gtrdot v_{r}\right)  }}_{\substack{=\downslack^{v_{0}%
}\downslack^{v_{1}}\cdots\downslack^{v_{j}}v_{j+1}\overline{v_{j+2}%
}\upslack^{v_{j+3}}\upslack^{v_{j+4}}\cdots\upslack^{v_{r}}\\\text{(by the
definition of }E_{\left(  v_{0}\gtrdot v_{1}\gtrdot\cdots\gtrdot
v_{j+1}\blacktriangleright v_{j+2}\gtrdot v_{j+3}\gtrdot\cdots\gtrdot
v_{r}\right)  }\text{)}}}\\
&  =\sum_{\substack{\left(  v_{0}\gtrdot v_{1}\gtrdot\cdots\gtrdot
v_{j+1}\blacktriangleright v_{j+2}\gtrdot v_{j+3}\gtrdot\cdots\gtrdot
v_{r}\right)  \\\text{is a path-jump-path}\\\text{from }u\text{ to }d^{\prime
}\\\text{with jump at }j+1}}\downslack^{v_{0}}\downslack^{v_{1}}%
\cdots\downslack^{v_{j}}v_{j+1}\overline{v_{j+2}}\upslack^{v_{j+3}%
}\upslack^{v_{j+4}}\cdots\upslack^{v_{r}}\\
&  =\sum_{\substack{\left(  v_{0}\gtrdot v_{1}\gtrdot\cdots\gtrdot
v_{j}\right)  \\\text{is a path starting at }u}}\ \ \sum_{\substack{\left(
v_{j+2}\gtrdot v_{j+3}\gtrdot\cdots\gtrdot v_{r}\right)  \\\text{is a path
ending at }d^{\prime}}}\ \ \sum_{\substack{v_{j+1}\in P;\\v_{j}\gtrdot
v_{j+1}\blacktriangleright v_{j+2}}}\downslack^{v_{0}}\downslack^{v_{1}}%
\cdots\downslack^{v_{j}}v_{j+1}\overline{v_{j+2}}\upslack^{v_{j+3}%
}\upslack^{v_{j+4}}\cdots\upslack^{v_{r}}\\
&  \ \ \ \ \ \ \ \ \ \ \ \ \ \ \ \ \ \ \ \ \left(
\begin{array}
[c]{c}%
\text{here, we have broken up our}\\
\text{path-jump-path }\left(  v_{0}\gtrdot v_{1}\gtrdot\cdots\gtrdot
v_{j+1}\blacktriangleright v_{j+2}\gtrdot v_{j+3}\gtrdot\cdots\gtrdot
v_{r}\right) \\
\text{into two paths }\left(  v_{0}\gtrdot v_{1}\gtrdot\cdots\gtrdot
v_{j}\right)  \text{ and }\left(  v_{j+2}\gtrdot v_{j+3}\gtrdot\cdots\gtrdot
v_{r}\right) \\
\text{and an intermediate vertex }v_{j+1}\text{ satisfying }v_{j}\gtrdot
v_{j+1}\blacktriangleright v_{j+2}%
\end{array}
\right) \\
&  =\sum_{\substack{\left(  v_{0}\gtrdot v_{1}\gtrdot\cdots\gtrdot
v_{j}\right)  \\\text{is a path starting at }u}}\ \ \sum_{\substack{\left(
v_{j+2}\gtrdot v_{j+3}\gtrdot\cdots\gtrdot v_{r}\right)  \\\text{is a path
ending at }d^{\prime}}}\ \ \underbrace{\sum_{\substack{x\in P;\\v_{j}\gtrdot
x\blacktriangleright v_{j+2}}}\downslack^{v_{0}}\downslack^{v_{1}}%
\cdots\downslack^{v_{j}}x\overline{v_{j+2}}\upslack^{v_{j+3}}\upslack^{v_{j+4}%
}\cdots\upslack^{v_{r}}}_{\substack{=\sum_{\substack{x\in P;\\v_{j}\gtrdot
x\blacktriangleright v_{j+2}}}\downslack^{v_{0}}\downslack^{v_{1}}%
\cdots\downslack^{v_{j-1}}\downslack^{v_{j}}x\overline{v_{j+2}}%
\upslack^{v_{j+3}}\upslack^{v_{j+4}}\cdots\upslack^{v_{r}}\\=\downslack^{v_{0}%
}\downslack^{v_{1}}\cdots\downslack^{v_{j-1}}\sum_{\substack{x\in
P;\\v_{j}\gtrdot x\blacktriangleright v_{j+2}}}\downslack^{v_{j}}%
x\overline{v_{j+2}}\upslack^{v_{j+3}}\upslack^{v_{j+4}}\cdots\upslack^{v_{r}}%
}}\\
&  \ \ \ \ \ \ \ \ \ \ \ \ \ \ \ \ \ \ \ \ \left(  \text{here we have renamed
}v_{j+1}\text{ as }x\text{ in the inner sum}\right) \\
&  =\sum_{\substack{\left(  v_{0}\gtrdot v_{1}\gtrdot\cdots\gtrdot
v_{j}\right)  \\\text{is a path starting at }u}}\ \ \sum_{\substack{\left(
v_{j+2}\gtrdot v_{j+3}\gtrdot\cdots\gtrdot v_{r}\right)  \\\text{is a path
ending at }d^{\prime}}}\downslack^{v_{0}}\downslack^{v_{1}}\cdots
\downslack^{v_{j-1}}\sum_{\substack{x\in P;\\v_{j}\gtrdot x\blacktriangleright
v_{j+2}}}\downslack^{v_{j}}x\overline{v_{j+2}}\upslack^{v_{j+3}}%
\upslack^{v_{j+4}}\cdots\upslack^{v_{r}}.
\end{align*}

\end{verlong}

Comparing our last two equalities, we obtain%
\[
\sum_{\substack{\mathbf{p}\text{ is a path-jump-path}\\\text{from }u\text{ to
}d^{\prime}\\\text{with jump at }j}}E_{\mathbf{p}}=\sum_{\substack{\mathbf{p}%
\text{ is a path-jump-path}\\\text{from }u\text{ to }d^{\prime}\\\text{with
jump at }j+1}}E_{\mathbf{p}}.
\]
Thus, Claim 3 is proven.
\end{proof}

We have now proved all three Claims 1, 2 and 3. As we explained, this
completes the proof of Lemma \ref{lem.rect.conv}.
\end{proof}

\begin{remark}
\label{rmk.lem.rect.conv.2}Parts of the above proof of Lemma
\ref{lem.rect.conv} can be rewritten in a more abstract (although probably not
shorter) manner, avoiding the notion of a \textquotedblleft
path-jump-path\textquotedblright\ and the nested sums that appeared in our
proof of Claim 3.

To rewrite the proof, we need the notion of $P\times P$-matrices.
A $P\times P$\emph{-matrix} is a matrix whose rows and columns are indexed not
by integers but by elements of $P$. (That is, it is a family of elements of
$\mathbb{K}$ indexed by pairs $\left(  i,j\right)  \in P\times P$.) If $C$ is
any $P\times P$-matrix, and if $i$ and $j$ are two elements of $P$, then the
$\left(  i,j\right)  $-th entry of $C$ is denoted by $C_{i,j}$. Addition and
multiplication are defined for $P\times P$-matrices in the same way as they
are for usual matrices. That is, for any $P\times P$-matrices $C$ and $D$ and
any $\left(  i,j\right)  \in P\times P$, we have%
\[
\left(  C+D\right)  _{i,j}=C_{i,j}+D_{i,j}\ \ \ \ \ \ \ \ \ \ \text{and}%
\ \ \ \ \ \ \ \ \ \ \left(  CD\right)  _{i,j}=\sum_{k\in P}C_{i,k}D_{k,j}.
\]

For any statement $\mathcal{A}$, we let $\left[  \mathcal{A}\right]  $ be the
Iverson bracket (i.e., truth value) of $\mathcal{A}$. That is, $\left[
\mathcal{A}\right]  =1$ if $\mathcal{A}$ is true, and $\left[  \mathcal{A}%
\right]  =0$ if $\mathcal{A}$ is false.

Now, let $\ell\in\mathbb{N}$. Define three $P\times P$-matrices $\bfdownslack$%
, $\bfupslack$ and $\mathbf{U}$ by
\begin{align*}
\bfdownslack_{x,y}  &  :=\downslack^{x}\left[  x\gtrdot y\right]  ,\\
\bfupslack_{x,y}  &  :=\upslack^{y}\left[  x\gtrdot y\right]  ,\\
\mathbf{U}_{x,y}  &  :=x\overline{y}\left[  x\blacktriangleright y\right]
\ \ \ \ \ \ \ \ \ \ \ \text{for all }x,y\in P.
\end{align*}
Here, the relation $x\blacktriangleright y$ is defined as in the above proof
of Lemma \ref{lem.rect.conv}, and we are again omitting the \textquotedblleft%
$\ell$\textquotedblright\ subscripts, so (for instance) \textquotedblleft%
$x\overline{y}$\textquotedblright\ actually means $x_{\ell}\overline{y_{\ell}%
}$.

Now, Claim 4 in our above proof of Lemma \ref{lem.rect.conv} can be rewritten
in a nice and compact form as the equality%
\[
\bfdownslack \mathbf{U}=\mathbf{U}\bfupslack.
\]
From this, we easily obtain%
\begin{equation}
\bfdownslack^{k}\mathbf{U}=\mathbf{U}\bfupslack^{k}%
\ \ \ \ \ \ \ \ \ \ \text{for any }k\in\mathbb{N}.
\label{eq.rmk.lem.rect.conv.2.k}%
\end{equation}
This equality essentially replaces Claim 3 in the above proof.

Setting $k=\rank u-\rank d$ in (\ref{eq.rmk.lem.rect.conv.2.k}), and comparing
the $\left(  u,d^{\prime}\right)  $-entries of both sides, we quickly obtain
$\downslack^{u\rightarrow d}=\upslack^{u^{\prime}\rightarrow d^{\prime}}$
(since $x\blacktriangleright d^{\prime}$ holds only for $x=d$, and since
$u\blacktriangleright x$ holds only for $x=u^{\prime}$). This proves Lemma
\ref{lem.rect.conv} again.
\end{remark}

\section{\label{sec.Rfi1}Proof of reciprocity: the case $j=1$}

Using the conversion lemma, we can now easily prove Theorem
\ref{thm.rect.antip} in the case when $j=1$:

\begin{lemma}
\label{lem.rect.i1}Assume that $P$ is the $p\times q$-rectangle $\left[
p\right]  \times\left[  q\right]  $. Let $i\in\left[  p\right]  $. Let
$\ell\in\mathbb{N}$ satisfy $\ell\geq i$. Let $f\in\mathbb{K}^{\widehat{P}}$
be a $\mathbb{K}$-labeling such that $R^{\ell}f\neq\undf$. Let $a=f\left(
0\right)  $ and $b=f\left(  1\right)  $. Then, using the notations from
Section \ref{sec.proof-nots}, we have%
\[
\left(  i,1\right)  _{\ell}=a\cdot\overline{\left(  p+1-i,\ q\right)
_{\ell-i}}\cdot b.
\]

\end{lemma}

\begin{vershort}

\begin{proof}
We have $\ell\geq i\geq1$. Hence, Lemma \ref{lem.rect.antip.11full} yields
that $\left(  R^{\ell}f\right)  \left(  1,1\right)  =a\cdot\overline{\left(
R^{\ell-1}f\right)  \left(  p,q\right)  }\cdot b$. In other words, $\left(
1,1\right)  _{\ell}=a\cdot\overline{\left(  p,q\right)  _{\ell-1}}\cdot b$.
This proves Lemma \ref{lem.rect.i1} in the case when $i=1$.

Hence, for the rest of this proof, we WLOG assume that $i\neq1$. Thus,
$i\geq2$, so that $\ell\geq i\geq2$, and therefore $R^{2}f\neq\undf$ (by Lemma
\ref{lem.R.wd-triv}, since $R^{\ell}f\neq\undf$). Hence, Lemma
\ref{lem.R.01inv} yields that $a$ and $b$ are invertible (since $a=f\left(
0\right)  $ and $b=f\left(  1\right)  $).

We have $\ell-i+1\geq1$ (since $\ell\geq i$) and $\ell-i+1\leq\ell$ (since
$i\geq1$). The latter inequality entails $R^{\ell-i+1}f\neq\undf$ (by Lemma
\ref{lem.R.wd-triv}, since $R^{\ell}f\neq\undf$). Thus, Lemma
\ref{lem.slacks.wd} \textbf{(b)} (applied to $\left(  p-i+1,\ q\right)  $ and
$\ell-i+1$ instead of $v$ and $\ell$) yields that the element $\left(
p-i+1,\ q\right)  _{\ell-i}$ is well-defined and invertible.

Theorem \ref{thm.rect.path} \textbf{(d)} (applied to $\left(
p-i+1,\ q\right)  $ and $\ell-i$ instead of $u$ and $\ell$) yields
\[
\left(  p-i+1,\ q\right)  _{\ell-i}=\downslack_{\ell-i}^{\left(
p-i+1,\ q\right)  \rightarrow\left(  1,\ 1\right)  }\cdot a.
\]
Solving this for $\downslack_{\ell-i}^{\left(  p-i+1,\ q\right)
\rightarrow\left(  1,\ 1\right)  }$, we obtain $\downslack_{\ell-i}^{\left(
p-i+1,\ q\right)  \rightarrow\left(  1,\ 1\right)  }=\left(  p-i+1,\ q\right)
_{\ell-i}\cdot\overline{a}$, and thus%
\begin{equation}
\overline{\downslack_{\ell-i}^{\left(  p-i+1,\ q\right)  \rightarrow\left(
1,\ 1\right)  }}=\overline{\left(  p-i+1,\ q\right)  _{\ell-i}\cdot
\overline{a}}=a\cdot\overline{\left(  p-i+1,\ q\right)  _{\ell-i}}.
\label{pf.lem.rect.i1.short.1}%
\end{equation}

For each $k\in\left\{  0,1,\ldots,i-2\right\}  $, we have%
\begin{align}
\upslack_{\ell-k}^{\left(  p-k,\ q\right)  \rightarrow\left(  i-k,\ 1\right)
}  &  =\downslack_{\ell-k-1}^{\left(  p-k,\ q\right)  \rightarrow\left(
i-k,\ 1\right)  }\ \ \ \ \ \ \ \ \ \ \left(
\begin{array}
[c]{c}%
\text{by (\ref{eq.prop.rect.trans.uv}), since we can easily}\\
\text{find }\ell-k\geq2\geq1\text{ and }R^{\ell-k}f\neq\undf
\end{array}
\right) \nonumber\\
&  =\upslack_{\ell-k-1}^{\left(  p-k-1,\ q\right)  \rightarrow\left(
i-k-1,\ 1\right)  }\ \ \ \ \ \ \ \ \ \ \left(
\begin{array}
[c]{c}%
\text{by Lemma \ref{lem.rect.conv}, applied to}\\
\left(  p-k,\ q\right)  \text{, }\left(  p-k-1,\ q\right)  \text{,}\\
\left(  i-k,\ 1\right)  \text{, }\left(  i-k-1,\ 1\right) \\
\text{and }\ell-k-1\\
\text{instead of }u\text{, }u^{\prime}\text{, }d\text{, }d^{\prime}\text{ and
}\ell
\end{array}
\right) \nonumber\\
&  =\upslack_{\ell-\left(  k+1\right)  }^{\left(  p-\left(  k+1\right)
,\ q\right)  \rightarrow\left(  i-\left(  k+1\right)  ,\ 1\right)  }.
\label{pf.lem.rect.i1.short.2}%
\end{align}
Now,%
\begin{align}
\upslack_{\ell}^{\left(  p,\ q\right)  \rightarrow\left(  i,\ 1\right)  }  &
=\upslack_{\ell-0}^{\left(  p-0,\ q\right)  \rightarrow\left(  i-0,\ 1\right)
}\nonumber\\
&  =\upslack_{\ell-1}^{\left(  p-1,\ q\right)  \rightarrow\left(
i-1,\ 1\right)  }\ \ \ \ \ \ \ \ \ \ \left(  \text{by
(\ref{pf.lem.rect.i1.short.2}), applied to }k=0\right) \nonumber\\
&  =\upslack_{\ell-2}^{\left(  p-2,\ q\right)  \rightarrow\left(
i-2,\ 1\right)  }\ \ \ \ \ \ \ \ \ \ \left(  \text{by
(\ref{pf.lem.rect.i1.short.2}), applied to }k=1\right) \nonumber\\
&  =\cdots\nonumber\\
&  =\upslack_{\ell-\left(  i-1\right)  }^{\left(  p-\left(  i-1\right)
,\ q\right)  \rightarrow\left(  i-\left(  i-1\right)  ,\ 1\right)
}\ \ \ \ \ \ \ \ \ \ \left(  \text{by (\ref{pf.lem.rect.i1.short.2}), applied
to }k=i-2\right) \nonumber\\
&  =\upslack_{\ell-i+1}^{\left(  p-i+1,\ q\right)  \rightarrow\left(
1,\ 1\right)  }\ \ \ \ \ \ \ \ \ \ \left(
\begin{array}
[c]{c}%
\text{since }p-\left(  i-1\right)  =p-i+1\\
\text{and }i-\left(  i-1\right)  =1\\
\text{and }\ell-\left(  i-1\right)  =\ell-i+1
\end{array}
\right) \nonumber\\
&  =\downslack_{\ell-i}^{\left(  p-i+1,\ q\right)  \rightarrow\left(
1,\ 1\right)  }\ \ \ \ \ \ \ \ \ \ \left(  \text{by
(\ref{eq.prop.rect.trans.uv})}\right)  . \label{pf.lem.rect.i1.short.3}%
\end{align}
However, Theorem \ref{thm.rect.path} \textbf{(c)} (applied to $u=\left(
i,\ 1\right)  $) yields
\begin{align*}
\left(  i,\ 1\right)  _{\ell}  &  =\overline{\upslack_{\ell}^{\left(
p,\ q\right)  \rightarrow\left(  i,\ 1\right)  }}\cdot b=\overline
{\downslack_{\ell-i}^{\left(  p-i+1,\ q\right)  \rightarrow\left(
1,\ 1\right)  }}\cdot b\ \ \ \ \ \ \ \ \ \ \left(  \text{by
(\ref{pf.lem.rect.i1.short.3})}\right) \\
&  =a\cdot\overline{\left(  p-i+1,\ q\right)  _{\ell-i}}\cdot
b\ \ \ \ \ \ \ \ \ \ \left(  \text{by (\ref{pf.lem.rect.i1.short.1})}\right)
\\
&  =a\cdot\overline{\left(  p+1-i,\ q\right)  _{\ell-i}}\cdot b.
\end{align*}
This proves Lemma \ref{lem.rect.i1}.
\end{proof}
\end{vershort}

\begin{verlong}

\begin{proof}
We have $\ell\geq i\geq1$ (since $i\in\left[  p\right]  $). Also, from
$i\in\left[  p\right]  $ and $1\in\left[  q\right]  $, we obtain $\left(
i,1\right)  \in\left[  p\right]  \times\left[  q\right]  =P$. Furthermore,
from $i\in\left[  p\right]  $, we obtain $p+1-i\in\left[  p\right]  $.
Combined with $q\in\left[  q\right]  $, this leads to $\left(
p+1-i,\ q\right)  \in\left[  p\right]  \times\left[  q\right]  =P$.

If $\ell=1$, then the claim of Lemma \ref{lem.rect.i1} easily follows from
Lemma \ref{lem.rect.antip.11full}\footnote{\textit{Proof.} Assume that
$\ell=1$. Then, $1=\ell\geq i$, so that $i\leq1$ and therefore $i=1$ (since
$i\in\left[  p\right]  $). However, Lemma \ref{lem.rect.antip.11full} yields%
\begin{equation}
\left(  R^{\ell}f\right)  \left(  1,1\right)  =a\cdot\overline{\left(
R^{\ell-1}f\right)  \left(  p,q\right)  }\cdot b.
\label{pf.lem.rect.i1.exclude.1}%
\end{equation}
Since we are using the notations from Section \ref{sec.proof-nots}, we have
\[
\left(  i,1\right)  _{\ell}=\left(  R^{\ell}f\right)  \left(  i,1\right)
=\left(  R^{\ell}f\right)  \left(  1,1\right)  \ \ \ \ \ \ \ \ \ \ \left(
\text{since }i=1\right)
\]
and
\begin{align*}
\left(  p+1-i,\ q\right)  _{\ell-i}  &  =\left(  R^{\ell-i}f\right)  \left(
p+1-i,\ q\right)  =\left(  R^{\ell-1}f\right)  \left(  \underbrace{p+1-1}%
_{=p},\ q\right)  \ \ \ \ \ \ \ \ \ \ \left(  \text{since }i=1\right) \\
&  =\left(  R^{\ell-1}f\right)  \left(  p,q\right)  .
\end{align*}
Thus, we can rewrite (\ref{pf.lem.rect.i1.exclude.1}) as%
\[
\left(  i,1\right)  _{\ell}=a\cdot\overline{\left(  p+1-i,\ q\right)
_{\ell-i}}\cdot b.
\]
Hence, Lemma \ref{lem.rect.i1} is proved under the assumption that $\ell=1$.}.
Thus, for the rest of this proof, we WLOG assume that $\ell\neq1$. Hence,
$\ell\geq2$ (since $\ell\geq1$). Thus, $2\leq\ell$. Hence, from $R^{\ell}%
f\neq\undf$, we obtain $R^{2}f\neq\undf$ (by Lemma \ref{lem.R.wd-triv}).
Hence, Lemma \ref{lem.R.01inv} yields that $f\left(  0\right)  $ and $f\left(
1\right)  $ are invertible. In other words, $a$ and $b$ are invertible (since
$a=f\left(  0\right)  $ and $b=f\left(  1\right)  $).

We have $\underbrace{\ell}_{\geq i}-i+1\geq i-i+1=1$. Furthermore,
$\ell-\underbrace{i}_{\geq1}+1\leq\ell-1+1=\ell$ and thus $R^{\ell-i+1}%
f\neq\undf$ (by Lemma \ref{lem.R.wd-triv}, since $R^{\ell}f\neq\undf$). Thus,
Lemma \ref{lem.slacks.wd} \textbf{(b)} (applied to $\left(  p-i+1,\ q\right)
$ and $\ell-i+1$ instead of $v$ and $\ell$) yields that the element $\left(
p-i+1,\ q\right)  _{\ell-i}$ is well-defined and invertible.

Furthermore, Theorem \ref{thm.rect.path} \textbf{(d)} (applied to $\ell-i$ and
$\left(  p-i+1,\ q\right)  $ instead of $\ell$ and $u$) yields $\left(
p-i+1,\ q\right)  _{\ell-i}=\downslack_{\ell-i}^{\left(  p-i+1,\ q\right)
\rightarrow\left(  1,1\right)  }\cdot a$ (since $R^{\ell-i+1}f\neq\undf$).
Solving this for $\downslack_{\ell-i}^{\left(  p-i+1,\ q\right)
\rightarrow\left(  1,1\right)  }$, we obtain
\begin{equation}
\downslack_{\ell-i}^{\left(  p-i+1,\ q\right)  \rightarrow\left(  1,1\right)
}=\left(  p-i+1,\ q\right)  _{\ell-i}\cdot\overline{a}.
\label{pf.lem.rect.i1.0}%
\end{equation}
The right hand side of this equality is a product of two invertible elements
(since both $\left(  p-i+1,\ q\right)  _{\ell-i}$ and $\overline{a}$ are
invertible), and thus is invertible. Hence, the left hand side is invertible
as well. Taking reciprocals on both sides of (\ref{pf.lem.rect.i1.0}), we now
obtain%
\begin{align}
\overline{\downslack_{\ell-i}^{\left(  p-i+1,\ q\right)  \rightarrow\left(
1,1\right)  }}  &  =\overline{\left(  p-i+1,\ q\right)  _{\ell-i}%
\cdot\overline{a}}\nonumber\\
&  =a\cdot\overline{\left(  p-i+1,\ q\right)  _{\ell-i}}.
\label{pf.lem.rect.i1.1}%
\end{align}

Now, using Lemma \ref{lem.rect.conv} and Proposition \ref{prop.rect.trans}, we
can easily see the following: For each $k\in\left\{  0,1,\ldots,i-2\right\}
$, we have%
\begin{equation}
\upslack_{\ell-k}^{\left(  p-k,\ q\right)  \rightarrow\left(  i-k,\ 1\right)
}=\upslack_{\ell-\left(  k+1\right)  }^{\left(  p-\left(  k+1\right)
,\ q\right)  \rightarrow\left(  i-\left(  k+1\right)  ,\ 1\right)  }.
\label{pf.lem.rect.i1.2}%
\end{equation}

[\textit{Proof of (\ref{pf.lem.rect.i1.2}):} Let $k\in\left\{  0,1,\ldots
,i-2\right\}  $. Then, $k\leq i-2<\underbrace{i}_{\leq\ell}-1\leq\ell-1$, so
that $\ell-1>k$ and thus $\ell-k>1$.

From $k<i-1$, we also obtain $1<i-k$, so that $1\leq i-k-1$ (since $1$ and
$i-k$ are integers). Also, $i\in\left[  p\right]  $, so that $i\leq p$ and
thus $i-k\leq p-k$. Furthermore, $k\geq0$, so that $p-k\leq p$.

Now, we have $i-k-1\in\left[  p\right]  $ (since $1\leq i-k-1$ and $i-k-1\leq
i-k\leq p-k\leq p$), so that $\left(  i-k-1,\ 1\right)  \in\left[  p\right]
\times\left[  q\right]  =P$.

Furthermore, we have $i-k\in\left[  p\right]  $ (since $1\leq i-k-1\leq i-k$
and $i-k\leq p-k\leq p$), so that $\left(  i-k,\ 1\right)  \in\left[
p\right]  \times\left[  q\right]  =P$.

Furthermore, we have $p-k-1\in\left[  p\right]  $ (since $1\leq\underbrace{i}%
_{\leq p}-k-1\leq p-k-1$ and $p-k-1\leq p-k\leq p$), so that $\left(
p-k-1,\ q\right)  \in\left[  p\right]  \times\left[  q\right]  =P$.

Furthermore, we have $p-k\in\left[  p\right]  $ (since $1\leq i-k-1\leq
i-k\leq p-k$ and $p-k\leq p$), so that $\left(  p-k,\ q\right)  \in\left[
p\right]  \times\left[  q\right]  =P$.

Also, $\ell-\underbrace{k}_{\geq0}\leq\ell$ and therefore $R^{\ell-k}%
f\neq\undf$ (by Lemma \ref{lem.R.wd-triv}, since $R^{\ell}f\neq\undf$). Hence,
(\ref{eq.prop.rect.trans.uv}) (applied to $\left(  p-k,\ q\right)  $, $\left(
i-k,\ 1\right)  $ and $\ell-k$ instead of $u$, $v$ and $\ell$) yields%
\[
\upslack_{\ell-k}^{\left(  p-k,\ q\right)  \rightarrow\left(  i-k,\ 1\right)
}=\downslack_{\ell-k-1}^{\left(  p-k,\ q\right)  \rightarrow\left(
i-k,\ 1\right)  }.
\]
However, $\left(  p-k,\ q\right)  $ and $\left(  p-k-1,\ q\right)  $ are two
elements of the northeastern edge of $P$ satisfying $\left(  p-k,\ q\right)
\gtrdot\left(  p-k-1,\ q\right)  $, whereas $\left(  i-k,\ 1\right)  $ and
$\left(  i-k-1,\ 1\right)  $ are two elements of the southwestern edge of $P$
satisfying $\left(  i-k,\ 1\right)  \gtrdot\left(  i-k-1,\ 1\right)  $. We
furthermore have $\ell-k-1\geq1$ (since $\ell-k>1$) and $R^{\ell
-k-1+1}f=R^{\ell-k}f\neq\undf$. Thus, Lemma \ref{lem.rect.conv} (applied to
$\left(  p-k,\ q\right)  $, $\left(  p-k-1,\ q\right)  $, $\left(
i-k,\ 1\right)  $, $\left(  i-k-1,\ 1\right)  $ and $\ell-k-1$ instead of $u$,
$u^{\prime}$, $d$, $d^{\prime}$ and $\ell$) yields%
\[
\downslack_{\ell-k-1}^{\left(  p-k,\ q\right)  \rightarrow\left(
i-k,\ 1\right)  }=\upslack_{\ell-k-1}^{\left(  p-k-1,\ q\right)
\rightarrow\left(  i-k-1,\ 1\right)  }.
\]
Combining what we have shown, we now obtain
\[
\upslack_{\ell-k}^{\left(  p-k,\ q\right)  \rightarrow\left(  i-k,\ 1\right)
}=\downslack_{\ell-k-1}^{\left(  p-k,\ q\right)  \rightarrow\left(
i-k,\ 1\right)  }=\upslack_{\ell-k-1}^{\left(  p-k-1,\ q\right)
\rightarrow\left(  i-k-1,\ 1\right)  }=\upslack_{\ell-\left(  k+1\right)
}^{\left(  p-\left(  k+1\right)  ,\ q\right)  \rightarrow\left(  i-\left(
k+1\right)  ,\ 1\right)  }%
\]
(since $\ell-k-1=\ell-\left(  k+1\right)  $ and $p-k-1=p-\left(  k+1\right)  $
and $i-k-1=i-\left(  k+1\right)  $). This proves (\ref{pf.lem.rect.i1.2}).]
\medskip

Now,%
\begin{align}
\upslack_{\ell}^{\left(  p,\ q\right)  \rightarrow\left(  i,\ 1\right)  }  &
=\upslack_{\ell-0}^{\left(  p-0,\ q\right)  \rightarrow\left(  i-0,\ 1\right)
}\ \ \ \ \ \ \ \ \ \ \left(  \text{since }p=p-0\text{ and }i=i-0\text{ and
}\ell=\ell-0\right) \nonumber\\
&  =\upslack_{\ell-1}^{\left(  p-1,\ q\right)  \rightarrow\left(
i-1,\ 1\right)  }\ \ \ \ \ \ \ \ \ \ \left(  \text{by (\ref{pf.lem.rect.i1.2}%
), applied to }k=0\right) \nonumber\\
&  =\upslack_{\ell-2}^{\left(  p-2,\ q\right)  \rightarrow\left(
i-2,\ 1\right)  }\ \ \ \ \ \ \ \ \ \ \left(  \text{by (\ref{pf.lem.rect.i1.2}%
), applied to }k=1\right) \nonumber\\
&  =\cdots\nonumber\\
&  =\upslack_{\ell-\left(  i-1\right)  }^{\left(  p-\left(  i-1\right)
,\ q\right)  \rightarrow\left(  i-\left(  i-1\right)  ,\ 1\right)
}\ \ \ \ \ \ \ \ \ \ \left(  \text{by (\ref{pf.lem.rect.i1.2}), applied to
}k=i-2\right) \nonumber\\
&  =\upslack_{\ell-i+1}^{\left(  p-i+1,\ q\right)  \rightarrow\left(
1,\ 1\right)  }\ \ \ \ \ \ \ \ \ \ \left(
\begin{array}
[c]{c}%
\text{since }p-\left(  i-1\right)  =p-i+1\\
\text{and }i-\left(  i-1\right)  =1\\
\text{and }\ell-\left(  i-1\right)  =\ell-i+1
\end{array}
\right) \nonumber\\
&  =\downslack_{\ell-i}^{\left(  p-i+1,\ q\right)  \rightarrow\left(
1,\ 1\right)  } \label{pf.lem.rect.i1.3}%
\end{align}
(by (\ref{eq.prop.rect.trans.uv}), applied to $\ell-i+1$, $\left(
p-i+1,\ q\right)  $ and $\left(  1,\ 1\right)  $ instead of $\ell$, $u$ and
$v$).

However, Theorem \ref{thm.rect.path} \textbf{(c)} (applied to $u=\left(
i,1\right)  $) yields
\begin{align*}
\left(  i,1\right)  _{\ell}  &  =\overline{\upslack_{\ell}^{\left(
p,\ q\right)  \rightarrow\left(  i,\ 1\right)  }}\cdot b=\overline
{\downslack_{\ell-i}^{\left(  p-i+1,\ q\right)  \rightarrow\left(
1,\ 1\right)  }}\cdot b\ \ \ \ \ \ \ \ \ \ \left(  \text{by
(\ref{pf.lem.rect.i1.3})}\right) \\
&  =a\cdot\overline{\left(  p-i+1,\ q\right)  _{\ell-i}}\cdot
b\ \ \ \ \ \ \ \ \ \ \left(  \text{by (\ref{pf.lem.rect.i1.1})}\right) \\
&  =a\cdot\overline{\left(  p+1-i,\ q\right)  _{\ell-i}}\cdot
b\ \ \ \ \ \ \ \ \ \ \left(  \text{since }p-i+1=p+1-i\right)  .
\end{align*}
This proves Lemma \ref{lem.rect.i1}.
\end{proof}
\end{verlong}

In analogy to Lemma \ref{lem.rect.i1}, we have the following:

\begin{lemma}
\label{lem.rect.1j}Assume that $P$ is the $p\times q$-rectangle $\left[
p\right]  \times\left[  q\right]  $. Let $j\in\left[  q\right]  $. Let
$\ell\in\mathbb{N}$ satisfy $\ell\geq j$. Let $f\in\mathbb{K}^{\widehat{P}}$
be a $\mathbb{K}$-labeling such that $R^{\ell}f\neq\undf$. Let $a=f\left(
0\right)  $ and $b=f\left(  1\right)  $. Then, using the notations from
Section \ref{sec.proof-nots}, we have%
\[
\left(  1,j\right)  _{\ell}=a\cdot\overline{\left(  p,\ q+1-j\right)
_{\ell-j}}\cdot b.
\]

\end{lemma}

\begin{proof}
The two coordinates $u$ and $v$ of an element $\left(  u,v\right)  \in P$ play
symmetric roles. Lemma \ref{lem.rect.1j} is just Lemma \ref{lem.rect.i1} with
the roles of these two coordinates interchanged. Thus, the proof of Lemma
\ref{lem.rect.1j} is analogous to the proof of Lemma \ref{lem.rect.i1}.
\end{proof}

\section{\label{sec.gencase}Proof of reciprocity: the general case}

Somewhat surprisingly, the general case of Theorem \ref{thm.rect.antip}
follows by a fairly straightforward induction argument from Lemma
\ref{lem.rect.i1}:

\begin{vershort}

\begin{proof}
[Proof of Theorem \ref{thm.rect.antip}.]We again use the notations from
Section \ref{sec.proof-nots}.

For any $\left(  i,j\right)  \in P$, we define $\tilt \left(  i,j\right)  $ to
be the positive integer $i+2j$. Our goal is to prove
(\ref{pf.thm.rect.antip.restate1}) for each $x=\left(  i,j\right)  \in P$ and
$\ell\in\mathbb{N}$ satisfying $\ell-i-j+1\geq0$ and $R^{\ell}f\neq\undf$. We
will now prove this by strong induction on $\tilt x$.

\textit{Induction step:} Fix $N\in\mathbb{N}$. Assume (as the induction
hypothesis) that

\begin{statement}
(\ref{pf.thm.rect.antip.restate1}) holds for each $x=\left(  i,j\right)  \in
P$ satisfying $\tilt x<N$ and each $\ell\in\mathbb{N}$ satisfying
$\ell-i-j+1\geq0$ and $R^{\ell}f\neq\undf$.
\end{statement}

We now fix an element $v=\left(  i,j\right)  \in P$ satisfying $\tilt v=N$ and
an $\ell\in\mathbb{N}$ satisfying $\ell-i-j+1\geq0$ and $R^{\ell}f\neq\undf$.
Our goal is to prove that (\ref{pf.thm.rect.antip.restate1}) holds for $x=v$.
In other words, our goal is to prove that $v_{\ell}=a\cdot\overline
{v_{\ell-i-j+1}^{\sim}}\cdot b$. \medskip

We have $N=\tilt v=i+2j$ (since $v=\left(  i,j\right)  $). We are in one of
the following six cases:

\textit{Case 1:} We have $i=1$.

\textit{Case 2:} We have $j=1$.

\textit{Case 3:} We have $j=2$ and $1<i<p$.

\textit{Case 4:} We have $j=2$ and $i=p>1$.

\textit{Case 5:} We have $j>2$ and $1<i<p$.

\textit{Case 6:} We have $j>2$ and $i=p>1$. \medskip

Let us first consider Case 1. In this case, we have $i=1$. Thus, $v=\left(
i,j\right)  =\left(  1,j\right)  $ (since $i=1$). The definition of $v^{\sim}$
thus yields $v^{\sim}=\left(  p+1-1,\ q+1-j\right)  =\left(  p,\ q+1-j\right)
$. Also, $\ell-\underbrace{i}_{=1}-j+1=\ell-1-j+1=\ell-j$, so that
$\ell-j=\ell-i-j+1\geq0$. In other words, $\ell\geq j$. Hence, Lemma
\ref{lem.rect.1j} yields%
\[
\left(  1,j\right)  _{\ell}=a\cdot\overline{\left(  p,\ q+1-j\right)
_{\ell-j}}\cdot b.
\]
In view of $v=\left(  1,j\right)  $ and $v^{\sim}=\left(  p,\ q+1-j\right)  $
and $\ell-i-j+1=\ell-j$, we can rewrite this as $v_{\ell}=a\cdot
\overline{v_{\ell-i-j+1}^{\sim}}\cdot b$. Thus, $v_{\ell}=a\cdot
\overline{v_{\ell-i-j+1}^{\sim}}\cdot b$ is proved in Case 1. \medskip

Similarly (but using Lemma \ref{lem.rect.i1} instead of Lemma
\ref{lem.rect.1j}), we can obtain the same result (viz., $v_{\ell}%
=a\cdot\overline{v_{\ell-i-j+1}^{\sim}}\cdot b$) in Case 2. \medskip

Next, let us analyze the four remaining cases: Cases 3, 4, 5 and 6. The most
complex of these four cases is Case 5, so it is this case that we start with.

In this case, we have $j>2$ and $1<i<p$. Recall that $v=\left(  i,j\right)  $.
Define the four further pairs%
\begin{align*}
m  &  :=\left(  i,j-1\right)  ,\ \ \ \ \ \ \ \ \ \ u:=\left(  i+1,j-1\right)
,\\
s  &  :=\left(  i,j-2\right)  ,\ \ \ \ \ \ \ \ \ \ t:=\left(  i-1,j-1\right)
.
\end{align*}
The conditions $j>2$ and $1<i<p$ entail that all these four pairs $m$, $u$,
$s$ and $t$ belong to $\left[  p\right]  \times\left[  q\right]  =P$. Here is
how the five elements $v,m,u,s,t$ of $P$ are aligned on the Hasse diagram of
$P$:
\begin{equation}%
\xymatrix@R=0.6pc@C=1pc{
u \are[rd] & & v \are[ld] \\
& m \are[rd] \are[ld] &\\
s & & t & & .
}
\label{pf.thm.rect.antip.short.diagram1}%
\end{equation}
In particular, the two elements of $\widehat{P}$ that cover $m$ are $u$ and
$v$, whereas the two elements of $\widehat{P}$ that are covered by $m$ are $s$
and $t$.

The map $P\rightarrow P,\ x\mapsto x^{\sim}$ (which can be visualized as
\textquotedblleft reflecting\textquotedblright\ each point in $P$ around the
center of the rectangle $\left[  p\right]  \times\left[  q\right]  $)
\textquotedblleft reverses\textquotedblright\ covering relations (i.e., if
$x,y\in P$ satisfy $x\gtrdot y$, then $x^{\sim}\lessdot y^{\sim}$). Hence,
applying this map to the diagram (\ref{pf.thm.rect.antip.short.diagram1})
yields%
\[%
\xymatrix@R=0.7pc@C=1pc{
t^\sim\are[rd] & & s^\sim\are[ld] \\
& m^\sim\are[ld] \are[rd] \\
v^\sim& & u^\sim& & .
}%
\]
In particular, the two elements of $\widehat{P}$ that are covered by $m^{\sim
}$ are $u^{\sim}$ and $v^{\sim}$, whereas the two elements of $\widehat{P}$
that cover $m^{\sim}$ are $s^{\sim}$ and $t^{\sim}$.

From $\ell-i-j+1\geq0$, we obtain $\ell\geq\underbrace{i}_{>1}+\underbrace{j}%
_{>2}-1>1+2-1=2$, so that $\ell\geq2$. Therefore, $\ell-1\in\mathbb{N}$ and
$2\leq\ell$.

Hence, from $R^{\ell}f\neq\undf$, we obtain $R^{2}f\neq\undf$ (by Lemma
\ref{lem.R.wd-triv}). Therefore, Lemma \ref{lem.R.01inv} yields that $a$ and
$b$ are invertible (since $a=f\left(  0\right)  $ and $b=f\left(  1\right)
$). Also, we have $R^{\ell-1}f\neq\undf$ (since $R\left(  R^{\ell-1}f\right)
=R^{\ell}f\neq\undf=R\left(  \undf\right)  $).

Set $k:=i+j-2$. Then, $k\geq0$ (since $i\geq1$ and $j\geq1$), so that
$k\in\mathbb{N}$.

Now, straightforward computations show that the four elements $m$, $u$, $s$
and $t$ of $P$ satisfy
\[
\tilt m<N,\ \ \ \ \ \ \ \ \ \ \tilt u<N,\ \ \ \ \ \ \ \ \ \ \tilt s<N,\ \ \ \ \ \ \ \ \ \ \tilt t<N
\]
(since $i+2j=N$). Hence, using the induction hypothesis, it is easy to see
that the five equalities%
\begin{align}
m_{\ell}  &  =a\cdot\overline{m_{\ell-k}^{\sim}}\cdot
b,\label{pf.thm.rect.antip.short.3}\\
s_{\ell-1}  &  =a\cdot\overline{s_{\ell-k}^{\sim}}\cdot
b,\label{pf.thm.rect.antip.short.4}\\
t_{\ell-1}  &  =a\cdot\overline{t_{\ell-k}^{\sim}}\cdot
b,\label{pf.thm.rect.antip.short.4b}\\
m_{\ell-1}  &  =a\cdot\overline{m_{\ell-k-1}^{\sim}}\cdot
b,\label{pf.thm.rect.antip.short.5}\\
u_{\ell}  &  =a\cdot\overline{u_{\ell-k-1}^{\sim}}\cdot b
\label{pf.thm.rect.antip.short.5b}%
\end{align}
hold\footnote{In more detail: The induction hypothesis tells us that ...
\par
\begin{itemize}
\item ... we can apply (\ref{pf.thm.rect.antip.restate1}) to $m$ and $\left(
i,j-1\right)  $ instead of $x$ and $\left(  i,j\right)  $ (since $m=\left(
i,j-1\right)  \in P$ and $\tilt m<N$ and $\ell\in\mathbb{N}$ and
$\ell-i-\underbrace{\left(  j-1\right)  }_{\leq j}+1\geq\ell-i-j+1\geq0$ and
$R^{\ell}f\neq\undf$). This yields (\ref{pf.thm.rect.antip.short.3}) (since an
easy computation shows that $\ell-i-\left(  j-1\right)  +1=\ell-k$).
\par
\item ... we can apply (\ref{pf.thm.rect.antip.restate1}) to $s$ and $\left(
i,j-2\right)  $ and $\ell-1$ instead of $x$ and $\left(  i,j\right)  $ and
$\ell$ (the reader can easily verify that the requirements for this are
satisfied). This yields (\ref{pf.thm.rect.antip.short.4}) (since an easy
computation shows that $\left(  \ell-1\right)  -i-\left(  j-2\right)
+1=\ell-k$).
\par
\item ... we can apply (\ref{pf.thm.rect.antip.restate1}) to $t$ and $\left(
i-1,j-1\right)  $ and $\ell-1$ instead of $x$ and $\left(  i,j\right)  $ and
$\ell$. This yields (\ref{pf.thm.rect.antip.short.4b}).
\par
\item ... we can apply (\ref{pf.thm.rect.antip.restate1}) to $m$ and $\left(
i,j-1\right)  $ and $\ell-1$ instead of $x$ and $\left(  i,j\right)  $ and
$\ell$. This yields (\ref{pf.thm.rect.antip.short.5}).
\par
\item ... we can apply (\ref{pf.thm.rect.antip.restate1}) to $u$ and $\left(
i+1,j-1\right)  $ instead of $x$ and $\left(  i,j\right)  $. This yields
(\ref{pf.thm.rect.antip.short.5b}).
\end{itemize}
}.

We have $\ell-1\in\mathbb{N}$ and $R^{\ell}f\neq\undf$. Hence, the transition
equation (\ref{eq.vl+1=}) (applied to $m$ and $\ell-1$ instead of $v$ and
$\ell$) yields%
\[
m_{\ell}=\left(  \sum_{x\lessdot m}x_{\ell-1}\right)  \cdot\overline
{m_{\ell-1}}\cdot\overline{\sum_{x\gtrdot m}\overline{x_{\ell}}}%
\]
(here we have renamed the summation indices $u$ from (\ref{eq.vl+1=}) as $x$,
since the letter $u$ is already being used for something else in our current
setting). Thus,%
\begin{align}
m_{\ell}  &  =\underbrace{\left(  \sum_{x\lessdot m}x_{\ell-1}\right)
}_{\substack{=s_{\ell-1}+t_{\ell-1}\\\text{(since the two elements of
}\widehat{P}\text{ that}\\\text{are covered by }m\text{ are }s\text{ and
}t\text{)}}}\cdot\,\overline{m_{\ell-1}}\,\cdot\underbrace{\overline
{\sum_{x\gtrdot m}\overline{x_{\ell}}}}_{\substack{=\overline{\overline
{u_{\ell}}+\overline{v_{\ell}}}\\\text{(since the two elements of }%
\widehat{P}\\\text{that cover }m\text{ are }u\text{ and }v\text{)}%
}}\nonumber\\
&  =\left(  s_{\ell-1}+t_{\ell-1}\right)  \cdot\overline{m_{\ell-1}}%
\cdot\overline{\overline{u_{\ell}}+\overline{v_{\ell}}}.
\label{pf.thm.rect.antip.short.6}%
\end{align}

On the other hand, from $k=i+j-2$, we obtain $\ell-k-1=\ell-i-j+1\geq0$. Thus,
$\ell-k-1\in\mathbb{N}$. Also, $\ell-\underbrace{k}_{\geq0}\leq\ell$, so that
$R^{\ell-k}f\neq\undf$ (by Lemma \ref{lem.R.wd-triv}, since $R^{\ell}%
f\neq\undf$). Hence, the transition equation (\ref{eq.vl+1=}) (applied to
$m^{\sim}$ and $\ell-k-1$ instead of $v$ and $\ell$) yields%
\begin{align}
m_{\ell-k}^{\sim}  &  =\underbrace{\left(  \sum_{x\lessdot m^{\sim}}%
x_{\ell-k-1}\right)  }_{\substack{=u_{\ell-k-1}^{\sim}+v_{\ell-k-1}^{\sim
}\\\text{(since the two elements of }\widehat{P}\text{ that}\\\text{are
covered by }m^{\sim}\text{ are }u^{\sim}\text{ and }v^{\sim}\text{)}}%
}\cdot\,\overline{m_{\ell-k-1}^{\sim}}\,\cdot\underbrace{\overline
{\sum_{x\gtrdot m^{\sim}}\overline{x_{\ell-k}}}}_{\substack{=\overline
{\overline{s_{\ell-k}^{\sim}}+\overline{t_{\ell-k}^{\sim}}}\\\text{(since the
two elements of }\widehat{P}\\\text{that cover }m^{\sim}\text{ are }s^{\sim
}\text{ and }t^{\sim}\text{)}}}\nonumber\\
&  =\left(  u_{\ell-k-1}^{\sim}+v_{\ell-k-1}^{\sim}\right)  \cdot
\overline{m_{\ell-k-1}^{\sim}}\cdot\overline{\overline{s_{\ell-k}^{\sim}%
}+\overline{t_{\ell-k}^{\sim}}}. \label{pf.thm.rect.antip.short.7}%
\end{align}
This entails that the elements $\overline{s_{\ell-k}^{\sim}}+\overline
{t_{\ell-k}^{\sim}}$ and $m_{\ell-k-1}^{\sim}$ of $\mathbb{K}$ are invertible
(since their inverses appear on the right hand side of this equality). Hence,
their product $\left(  \overline{s_{\ell-k}^{\sim}}+\overline{t_{\ell-k}%
^{\sim}}\right)  \cdot m_{\ell-k-1}^{\sim}$ is invertible as well.

Also, $\ell-k\geq1$ (since $\ell-k-1\geq0$) and $R^{\ell-k}f\neq\undf$. Hence,
Lemma \ref{lem.slacks.wd} \textbf{(a)} (applied to $\ell-k$ and $m^{\sim}$
instead of $\ell$ and $v$) shows that $m_{\ell-k}^{\sim}$ is well-defined and
invertible. Now, solving (\ref{pf.thm.rect.antip.short.7}) for $u_{\ell
-k-1}^{\sim}+v_{\ell-k-1}^{\sim}$, we obtain%
\[
u_{\ell-k-1}^{\sim}+v_{\ell-k-1}^{\sim}=m_{\ell-k}^{\sim}\cdot\left(
\overline{s_{\ell-k}^{\sim}}+\overline{t_{\ell-k}^{\sim}}\right)  \cdot
m_{\ell-k-1}^{\sim}.
\]
This shows that $u_{\ell-k-1}^{\sim}+v_{\ell-k-1}^{\sim}$ is invertible (since
the three factors $m_{\ell-k}^{\sim}$ and $\overline{s_{\ell-k}^{\sim}%
}+\overline{t_{\ell-k}^{\sim}}$ and $m_{\ell-k-1}^{\sim}$ on the right hand
side are invertible).

Taking reciprocals on both sides of (\ref{pf.thm.rect.antip.short.7}), we
obtain%
\begin{align}
\overline{m_{\ell-k}^{\sim}}  &  =\overline{\left(  u_{\ell-k-1}^{\sim
}+v_{\ell-k-1}^{\sim}\right)  \cdot\overline{m_{\ell-k-1}^{\sim}}%
\cdot\overline{\overline{s_{\ell-k}^{\sim}}+\overline{t_{\ell-k}^{\sim}}}%
}\nonumber\\
&  =\left(  \overline{s_{\ell-k}^{\sim}}+\overline{t_{\ell-k}^{\sim}}\right)
\cdot m_{\ell-k-1}^{\sim}\cdot\overline{u_{\ell-k-1}^{\sim}+v_{\ell-k-1}%
^{\sim}} \label{pf.thm.rect.antip.short.8}%
\end{align}
(by Proposition \ref{prop.inverses.ab} \textbf{(c)}).

Comparing (\ref{pf.thm.rect.antip.short.6}) with
(\ref{pf.thm.rect.antip.short.3}), we obtain%
\begin{align*}
a\cdot\overline{m_{\ell-k}^{\sim}}\cdot b  &  =\left(  \underbrace{s_{\ell-1}%
}_{\substack{=a\cdot\overline{s_{\ell-k}^{\sim}}\cdot b\\\text{(by
(\ref{pf.thm.rect.antip.short.4}))}}}+\underbrace{t_{\ell-1}}%
_{\substack{=a\cdot\overline{t_{\ell-k}^{\sim}}\cdot b\\\text{(by
(\ref{pf.thm.rect.antip.short.4b}))}}}\right)  \cdot\underbrace{\overline
{m_{\ell-1}}}_{\substack{=\overline{a\cdot\overline{m_{\ell-k-1}^{\sim}}\cdot
b}\\\text{(by (\ref{pf.thm.rect.antip.short.5}))}}}\cdot\,\overline
{\overline{u_{\ell}}+\overline{v_{\ell}}}\\
&  =\underbrace{\left(  a\cdot\overline{s_{\ell-k}^{\sim}}\cdot b+a\cdot
\overline{t_{\ell-k}^{\sim}}\cdot b\right)  }_{=a\cdot\left(  \overline
{s_{\ell-k}^{\sim}}+\overline{t_{\ell-k}^{\sim}}\right)  \cdot b}%
\cdot\underbrace{\overline{a\cdot\overline{m_{\ell-k-1}^{\sim}}\cdot b}%
}_{\substack{=\overline{b}\cdot m_{\ell-k-1}^{\sim}\cdot\overline
{a}\\\text{(by Proposition \ref{prop.inverses.ab} \textbf{(c)},}\\\text{since
}a\text{ and }\overline{m_{\ell-k-1}^{\sim}}\text{ and }b\text{ are
invertible)}}}\cdot\,\overline{\overline{u_{\ell}}+\overline{v_{\ell}}}\\
&  =a\cdot\left(  \overline{s_{\ell-k}^{\sim}}+\overline{t_{\ell-k}^{\sim}%
}\right)  \cdot\underbrace{b\cdot\overline{b}}_{=1}\cdot\,m_{\ell-k-1}^{\sim
}\cdot\overline{a}\cdot\overline{\overline{u_{\ell}}+\overline{v_{\ell}}}\\
&  =a\cdot\left(  \overline{s_{\ell-k}^{\sim}}+\overline{t_{\ell-k}^{\sim}%
}\right)  \cdot m_{\ell-k-1}^{\sim}\cdot\overline{a}\cdot\overline
{\overline{u_{\ell}}+\overline{v_{\ell}}}.
\end{align*}
Multiplying both sides of this equality by $\overline{a}$ on the left and by
$\overline{b}$ on the right (this is allowed, since $a$ and $b$ are
invertible), we obtain%
\begin{align*}
\overline{m_{\ell-k}^{\sim}}  &  =\left(  \overline{s_{\ell-k}^{\sim}%
}+\overline{t_{\ell-k}^{\sim}}\right)  \cdot m_{\ell-k-1}^{\sim}%
\cdot\underbrace{\overline{a}\cdot\overline{\overline{u_{\ell}}+\overline
{v_{\ell}}}\cdot\overline{b}}_{\substack{=\overline{b\cdot\left(
\overline{u_{\ell}}+\overline{v_{\ell}}\right)  \cdot a}\\\text{(by
Proposition \ref{prop.inverses.ab} \textbf{(c)})}}}\\
&  =\left(  \overline{s_{\ell-k}^{\sim}}+\overline{t_{\ell-k}^{\sim}}\right)
\cdot m_{\ell-k-1}^{\sim}\cdot\overline{b\cdot\left(  \overline{u_{\ell}%
}+\overline{v_{\ell}}\right)  \cdot a}.
\end{align*}
Comparing this with (\ref{pf.thm.rect.antip.short.8}), we obtain
\[
\left(  \overline{s_{\ell-k}^{\sim}}+\overline{t_{\ell-k}^{\sim}}\right)
\cdot m_{\ell-k-1}^{\sim}\cdot\overline{b\cdot\left(  \overline{u_{\ell}%
}+\overline{v_{\ell}}\right)  \cdot a}=\left(  \overline{s_{\ell-k}^{\sim}%
}+\overline{t_{\ell-k}^{\sim}}\right)  \cdot m_{\ell-k-1}^{\sim}\cdot
\overline{u_{\ell-k-1}^{\sim}+v_{\ell-k-1}^{\sim}}.
\]
Cancelling the $\left(  \overline{s_{\ell-k}^{\sim}}+\overline{t_{\ell
-k}^{\sim}}\right)  \cdot m_{\ell-k-1}^{\sim}$ factors on the left of this
equality (this is allowed, since $\left(  \overline{s_{\ell-k}^{\sim}%
}+\overline{t_{\ell-k}^{\sim}}\right)  \cdot m_{\ell-k-1}^{\sim}$ is
invertible), we obtain%
\[
\overline{b\cdot\left(  \overline{u_{\ell}}+\overline{v_{\ell}}\right)  \cdot
a}=\overline{u_{\ell-k-1}^{\sim}+v_{\ell-k-1}^{\sim}}.
\]
Taking reciprocals on both sides, we find%
\[
b\cdot\left(  \overline{u_{\ell}}+\overline{v_{\ell}}\right)  \cdot
a=u_{\ell-k-1}^{\sim}+v_{\ell-k-1}^{\sim}.
\]
Expanding the left hand side by distributivity, we rewrite this as%
\begin{equation}
b\cdot\overline{u_{\ell}}\cdot a+b\cdot\overline{v_{\ell}}\cdot a=u_{\ell
-k-1}^{\sim}+v_{\ell-k-1}^{\sim}. \label{pf.thm.rect.antip.short.14}%
\end{equation}

However, (\ref{pf.thm.rect.antip.short.5b}) yields%
\[
\overline{u_{\ell}}=\overline{a\cdot\overline{u_{\ell-k-1}^{\sim}}\cdot
b}=\overline{b}\cdot u_{\ell-k-1}^{\sim}\cdot\overline{a}%
\ \ \ \ \ \ \ \ \ \ \left(  \text{by Proposition \ref{prop.inverses.ab}
\textbf{(c)}}\right)  .
\]
Multiplying both sides by $b$ from the left and by $a$ from the right, we can
transform this into%
\[
b\cdot\overline{u_{\ell}}\cdot a=u_{\ell-k-1}^{\sim}.
\]
Subtracting this equality from (\ref{pf.thm.rect.antip.short.14}), we obtain%
\begin{equation}
b\cdot\overline{v_{\ell}}\cdot a=v_{\ell-k-1}^{\sim}.
\label{pf.thm.rect.antip.short.16}%
\end{equation}
This equality expresses $v_{\ell-k-1}^{\sim}$ as a product of three invertible
elements (namely, $b$, $\overline{v_{\ell}}$ and $a$). Thus, $v_{\ell
-k-1}^{\sim}$ is itself invertible.

Taking reciprocals on both sides of (\ref{pf.thm.rect.antip.short.16}), we now
obtain $\overline{b\cdot\overline{v_{\ell}}\cdot a}=\overline{v_{\ell
-k-1}^{\sim}}$. Hence,%
\[
\overline{v_{\ell-k-1}^{\sim}}=\overline{b\cdot\overline{v_{\ell}}\cdot
a}=\overline{a}\cdot v_{\ell}\cdot\overline{b}\ \ \ \ \ \ \ \ \ \ \left(
\text{by Proposition \ref{prop.inverses.ab} \textbf{(c)}}\right)  .
\]
Solving this for $v_{\ell}$, we obtain%
\[
v_{\ell}=a\cdot\overline{v_{\ell-k-1}^{\sim}}\cdot b=a\cdot\overline
{v_{\ell-i-j+1}^{\sim}}\cdot b\ \ \ \ \ \ \ \ \ \ \left(  \text{since }%
\ell-k-1=\ell-i-j+1\right)  .
\]
Thus, $v_{\ell}=a\cdot\overline{v_{\ell-i-j+1}^{\sim}}\cdot b$ is proved in
Case 5. \medskip

The arguments required to prove $v_{\ell}=a\cdot\overline{v_{\ell-i-j+1}%
^{\sim}}\cdot b$ in the Cases 3, 4 and 6 are similar to the one we have used
in Case 5, but simpler:

\begin{itemize}
\item In Case 3, we have $s\notin P$. The \textquotedblleft
neighborhood\textquotedblright\ of $m$ thus looks as follows:%
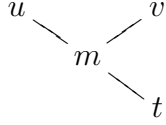
\[%
\xymatrix@R=0.6pc@C=1pc{
u \are[rd] & & v \are[ld] \\
& m \are[rd] \\
& & t
}%
\]
(instead of looking as in (\ref{pf.thm.rect.antip.short.diagram1})). This
necessitates some changes to the proof; in particular, all addends that
involve $s$ or $s^{\sim}$ in any way need to be removed, along with the
equality (\ref{pf.thm.rect.antip.short.4}).

\item Case 6 is similar, but now we have $u\notin P$ instead. (Subtraction is
no longer required in this case.)

\item In Case 4, we have both $s\notin P$ and $u\notin P$.
\end{itemize}

Thus, we have proved the equality $v_{\ell}=a\cdot\overline{v_{\ell
-i-j+1}^{\sim}}\cdot b$ in all six Cases 1, 2, 3, 4, 5 and 6. Hence, this
equality always holds. In other words, (\ref{pf.thm.rect.antip.restate1})
holds for $x=v$. This completes the induction step. Thus,
(\ref{pf.thm.rect.antip.restate1}) is proved by induction. In other words,
Theorem \ref{thm.rect.antip} is proven.
\end{proof}
\end{vershort}

\begin{verlong}
\begin{proof}
[Proof of Theorem \ref{thm.rect.antip}.]We again use the notations from
Section \ref{sec.proof-nots}.

For any $\left(  i,j\right)  \in P$, we define $\tilt \left(  i,j\right)  $ to
be the positive integer $i+2j$.

Our goal is to prove Theorem \ref{thm.rect.antip}. In other words, our goal is
to prove (\ref{pf.thm.rect.antip.restate1}) for each $x=\left(  i,j\right)
\in P$ and $\ell\in\mathbb{N}$ satisfying $\ell-i-j+1\geq0$ and $R^{\ell}%
f\neq\undf$. We will now prove this by strong induction on $\tilt x$.

\textit{Induction step:} Fix $N\in\mathbb{N}$. Assume (as the induction
hypothesis) that

\begin{statement}
(\ref{pf.thm.rect.antip.restate1}) holds for each $x=\left(  i,j\right)  \in
P$ satisfying $\tilt x<N$ and each $\ell\in\mathbb{N}$ satisfying
$\ell-i-j+1\geq0$ and $R^{\ell}f\neq\undf$.
\end{statement}

We now fix an element $v=\left(  i,j\right)  \in P$ satisfying $\tilt v=N$ and
an $\ell\in\mathbb{N}$ satisfying $\ell-i-j+1\geq0$ and $R^{\ell}f\neq\undf$.
Our goal is to prove that (\ref{pf.thm.rect.antip.restate1}) holds for $x=v$.
In other words, our goal is to prove that $v_{\ell}=a\cdot\overline
{v_{\ell-i-j+1}^{\sim}}\cdot b$. \medskip

We have $N=\tilt \underbrace{v}_{=\left(  i,j\right)  }=\tilt \left(
i,j\right)  =i+2j$ (by the definition of $\tilt \left(  i,j\right)  $). We are
in one of the following six cases:

\textit{Case 1:} We have $i=1$.

\textit{Case 2:} We have $j=1$.

\textit{Case 3:} We have $j=2$ and $1<i<p$.

\textit{Case 4:} We have $j=2$ and $i=p>1$.

\textit{Case 5:} We have $j>2$ and $1<i<p$.

\textit{Case 6:} We have $j>2$ and $i=p>1$. \medskip

Let us first consider Case 1. In this case, we have $i=1$. Thus, $v=\left(
i,j\right)  =\left(  1,j\right)  $ (since $i=1$). The definition of $v^{\sim}$
thus yields $v^{\sim}=\left(  p+1-1,\ q+1-j\right)  =\left(  p,\ q+1-j\right)
$. Also, $\ell-\underbrace{i}_{=1}-j+1=\ell-1-j+1=\ell-j$, so that
$\ell-j=\ell-i-j+1\geq0$. In other words, $\ell\geq j$. Hence, Lemma
\ref{lem.rect.1j} yields%
\[
\left(  1,j\right)  _{\ell}=a\cdot\overline{\left(  p,\ q+1-j\right)
_{\ell-j}}\cdot b.
\]
In view of $v=\left(  1,j\right)  $ and $v^{\sim}=\left(  p,\ q+1-j\right)  $
and $\ell-i-j+1=\ell-j$, we can rewrite this as $v_{\ell}=a\cdot
\overline{v_{\ell-i-j+1}^{\sim}}\cdot b$. Thus, $v_{\ell}=a\cdot
\overline{v_{\ell-i-j+1}^{\sim}}\cdot b$ is proved in Case 1. \medskip

Similarly (but using Lemma \ref{lem.rect.i1} instead of Lemma
\ref{lem.rect.1j}), we can obtain the same result (viz., $v_{\ell}%
=a\cdot\overline{v_{\ell-i-j+1}^{\sim}}\cdot b$) in Case 2. \medskip

Next, let us analyze the four remaining cases: Cases 3, 4, 5 and 6. The most
complex of these four cases is Case 5, so it is this case that we start with.

Thus, let us consider Case 5. In this case, we have $j>2$ and $1<i<p$. Recall
that $v=\left(  i,j\right)  $. Define the four further pairs%
\begin{align*}
m  &  :=\left(  i,j-1\right)  ,\ \ \ \ \ \ \ \ \ \ u:=\left(  i+1,j-1\right)
,\\
s  &  :=\left(  i,j-2\right)  ,\ \ \ \ \ \ \ \ \ \ t:=\left(  i-1,j-1\right)
.
\end{align*}
The conditions $j>2$ and $1<i<p$ entail that all these four pairs $m$, $u$,
$s$ and $t$ belong to $\left[  p\right]  \times\left[  q\right]  =P$. Here is
how the five elements $v,m,u,s,t$ of $P$ are aligned on the Hasse diagram of
$P$:
\begin{equation}%
\xymatrix@R=0.6pc@C=1pc{
u \are[rd] & & v \are[ld] \\
& m \are[rd] \are[ld] &\\
s & & t
}%
\ \ . \label{pf.thm.rect.antip.diagram1}%
\end{equation}
In particular, the two elements of $P$ that cover $m$ are $u$ and $v$, whereas
the two elements of $P$ that are covered by $m$ are $s$ and $t$. Clearly, we
can replace $P$ by $\widehat{P}$ in this sentence (since $1$ only covers those
elements of $P$ that are not covered by any element of $P$, and since $0$ is
covered only by those elements of $P$ that do not cover any element of $P$).
Thus, we obtain the following: The two elements of $\widehat{P}$ that cover
$m$ are $u$ and $v$, whereas the two elements of $\widehat{P}$ that are
covered by $m$ are $s$ and $t$.

Moreover, the map $P\rightarrow P,\ x\mapsto x^{\sim}$ (which can be
visualized as \textquotedblleft reflecting\textquotedblright\ each point in
$P$ around the center of the rectangle $\left[  p\right]  \times\left[
q\right]  $) \textquotedblleft reverses\textquotedblright\ covering relations:
That is, if two elements $x$ and $y$ of $P$ satisfy $x\gtrdot y$, then
$x^{\sim}\lessdot y^{\sim}$. Hence, the two elements of $P$ that are covered
by $m^{\sim}$ are $u^{\sim}$ and $v^{\sim}$ (since the two elements of $P$
that cover $m$ are $u$ and $v$), whereas the two elements of $P$ that cover
$m^{\sim}$ are $s^{\sim}$ and $t^{\sim}$ (since the two elements of $P$ that
are covered by $m$ are $s$ and $t$). Clearly, we can replace $P$ by
$\widehat{P}$ in this sentence (since $1$ only covers those elements of $P$
that are not covered by any element of $P$, and since $0$ is covered only by
those elements of $P$ that do not cover any element of $P$). Thus, we obtain
the following: The two elements of $\widehat{P}$ that are covered by $m^{\sim
}$ are $u^{\sim}$ and $v^{\sim}$, whereas the two elements of $\widehat{P}$
that cover $m^{\sim}$ are $s^{\sim}$ and $t^{\sim}$.

All in all, applying the map $P\rightarrow P,\ x\mapsto x^{\sim}$ to the
diagram (\ref{pf.thm.rect.antip.diagram1}) yields%
\[%
\xymatrix@R=0.7pc@C=1pc{
t^\sim\are[rd] & & s^\sim\are[ld] \\
& m^\sim\are[ld] \are[rd] \\
v^\sim& & u^\sim\ \ .
}%
\ \ .
\]

From $\ell-i-j+1\geq0$, we obtain $\ell\geq\underbrace{i}_{>1}+\underbrace{j}%
_{>2}-1>1+2-1=2$, so that $\ell\geq2$. Therefore, $\ell-1\geq1\geq0$ and thus
$\ell-1\in\mathbb{N}$. Also, $\ell\geq2$ entails $2\leq\ell$.

Hence, from $R^{\ell}f\neq\undf$, we obtain $R^{2}f\neq\undf$ (by Lemma
\ref{lem.R.wd-triv}). Therefore, Lemma \ref{lem.R.01inv} yields that $f\left(
0\right)  $ and $f\left(  1\right)  $ are invertible. In other words, $a$ and
$b$ are invertible (since $a=f\left(  0\right)  $ and $b=f\left(  1\right)
$). Also, we have $R^{\ell-1}f\neq\undf$ (since $R\left(  R^{\ell-1}f\right)
=R^{\ell}f\neq\undf=R\left(  \undf\right)  $).

Set $k:=i+j-2$. Then, $k=\underbrace{i}_{\geq1}+\underbrace{j}_{\geq1}%
-2\geq1+1-2=0$, so that $k\in\mathbb{N}$.

Now, it is easy to see that the four elements $m$, $u$, $s$ and $t$ of $P$
satisfy
\[
\tilt m<N,\ \ \ \ \ \ \ \ \ \ \tilt u<N,\ \ \ \ \ \ \ \ \ \ \tilt s<N,\ \ \ \ \ \ \ \ \ \ \tilt t<N
\]
\footnote{\textit{Proof.} Recall that $\tilt \left(  i^{\prime},j^{\prime
}\right)  =i^{\prime}+2j^{\prime}$ for each $\left(  i^{\prime},j^{\prime
}\right)  \in P$ (by the definition of $\tilt \left(  i^{\prime},j^{\prime
}\right)  $). Thus:
\par
\begin{itemize}
\item From $m=\left(  i,j-1\right)  $, we obtain $\tilt m=i+2\left(
j-1\right)  =i+2j-2<i+2j=N$.
\par
\item From $u=\left(  i+1,j-1\right)  $, we obtain $\tilt u=i+1+2\left(
j-1\right)  =i+2j-1<i+2j=N$.
\par
\item From $s=\left(  i,j-2\right)  $, we obtain $\tilt s=i+2\left(
j-2\right)  =i+2j-4<i+2j=N$.
\par
\item From $t=\left(  i-1,j-1\right)  $, we obtain $\tilt t=i-1+2\left(
j-1\right)  =i+2j-3<i+2j=N$.
\end{itemize}
}. Hence, using the induction hypothesis, it is easy to see that the five
equalities%
\begin{align}
m_{\ell}  &  =a\cdot\overline{m_{\ell-k}^{\sim}}\cdot
b,\label{pf.thm.rect.antip.3}\\
s_{\ell-1}  &  =a\cdot\overline{s_{\ell-k}^{\sim}}\cdot
b,\label{pf.thm.rect.antip.4}\\
t_{\ell-1}  &  =a\cdot\overline{t_{\ell-k}^{\sim}}\cdot
b,\label{pf.thm.rect.antip.4b}\\
m_{\ell-1}  &  =a\cdot\overline{m_{\ell-k-1}^{\sim}}\cdot
b,\label{pf.thm.rect.antip.5}\\
u_{\ell}  &  =a\cdot\overline{u_{\ell-k-1}^{\sim}}\cdot b
\label{pf.thm.rect.antip.5b}%
\end{align}
hold\footnote{\textit{Proof.} The induction hypothesis tells us that we can
apply (\ref{pf.thm.rect.antip.restate1}) to $m$ and $\left(  i,j-1\right)  $
instead of $x$ and $\left(  i,j\right)  $ (since $m=\left(  i,j-1\right)  \in
P$ and $\tilt m<N$ and $\ell\in\mathbb{N}$ and $\ell-i-\underbrace{\left(
j-1\right)  }_{\leq j}+1\geq\ell-i-j+1\geq0$ and $R^{\ell}f\neq\undf$). Thus,
we obtain%
\[
m_{\ell}=a\cdot\overline{m_{\ell-i-\left(  j-1\right)  +1}^{\sim}}\cdot
b=a\cdot\overline{m_{\ell-k}^{\sim}}\cdot b
\]
(since $\ell-i-\left(  j-1\right)  +1=\ell-\underbrace{\left(  i+j-2\right)
}_{=k}=\ell-k$). This proves (\ref{pf.thm.rect.antip.3}).
\par
The induction hypothesis tells us that we can apply
(\ref{pf.thm.rect.antip.restate1}) to $s$ and $\left(  i,j-2\right)  $ and
$\ell-1$ instead of $x$ and $\left(  i,j\right)  $ and $\ell$ (since
$s=\left(  i,j-2\right)  \in P$ and $\tilt s<N$ and $\ell-1\in\mathbb{N}$ and
$\ell-1-i-\left(  j-2\right)  +1=\ell-i-j+2\geq\ell-i-j+1\geq0$ and
$R^{\ell-1}f\neq\undf$). Thus, we obtain%
\[
s_{\ell-1}=a\cdot\overline{s_{\left(  \ell-1\right)  -i-\left(  j-2\right)
+1}^{\sim}}\cdot b=a\cdot\overline{s_{\ell-k}^{\sim}}\cdot b
\]
(since $\left(  \ell-1\right)  -i-\left(  j-2\right)  +1=\ell
-\underbrace{\left(  i+j-2\right)  }_{=k}=\ell-k$). This proves
(\ref{pf.thm.rect.antip.4}).
\par
The induction hypothesis tells us that we can apply
(\ref{pf.thm.rect.antip.restate1}) to $t$ and $\left(  i-1,j-1\right)  $ and
$\ell-1$ instead of $x$ and $\left(  i,j\right)  $ and $\ell$ (since
$t=\left(  i-1,j-1\right)  \in P$ and $\tilt t<N$ and $\ell-1\in\mathbb{N}$
and $\ell-1-\left(  i-1\right)  -\left(  j-1\right)  +1=\ell-i-j+2\geq
\ell-i-j+1\geq0$ and $R^{\ell-1}f\neq\undf$). Thus, we obtain%
\[
t_{\ell-1}=a\cdot\overline{t_{\left(  \ell-1\right)  -\left(  i-1\right)
-\left(  j-1\right)  +1}^{\sim}}\cdot b=a\cdot\overline{t_{\ell-k}^{\sim}%
}\cdot b
\]
(since $\left(  \ell-1\right)  -\left(  i-1\right)  -\left(  j-1\right)
+1=\ell-\underbrace{\left(  i+j-2\right)  }_{=k}=\ell-k$). This proves
(\ref{pf.thm.rect.antip.4b}).
\par
The induction hypothesis tells us that we can apply
(\ref{pf.thm.rect.antip.restate1}) to $m$ and $\left(  i,j-1\right)  $ and
$\ell-1$ instead of $x$ and $\left(  i,j\right)  $ and $\ell$ (since
$m=\left(  i,j-1\right)  \in P$ and $\tilt m<N$ and $\ell-1\in\mathbb{N}$ and
$\ell-1-i-\left(  j-1\right)  +1=\ell-i-j+1\geq0$ and $R^{\ell-1}f\neq\undf$).
Thus, we obtain%
\[
m_{\ell-1}=a\cdot\overline{m_{\left(  \ell-1\right)  -i-\left(  j-1\right)
+1}^{\sim}}\cdot b=a\cdot\overline{m_{\ell-k-1}^{\sim}}\cdot b
\]
(since $\left(  \ell-1\right)  -i-\left(  j-1\right)  +1=\ell
-\underbrace{\left(  i+j-2\right)  }_{=k}-1=\ell-k-1$). This proves
(\ref{pf.thm.rect.antip.5}).
\par
The induction hypothesis tells us that we can apply
(\ref{pf.thm.rect.antip.restate1}) to $u$ and $\left(  i+1,j-1\right)  $
instead of $x$ and $\left(  i,j\right)  $ (since $u=\left(  i+1,j-1\right)
\in P$ and $\tilt u<N$ and $\ell\in\mathbb{N}$ and $\ell-\left(  i+1\right)
-\left(  j-1\right)  +1=\ell-i-j+1\geq0$ and $R^{\ell}f\neq\undf$). Thus, we
obtain%
\[
u_{\ell}=a\cdot\overline{u_{\ell-\left(  i+1\right)  -\left(  j-1\right)
+1}^{\sim}}\cdot b=a\cdot\overline{u_{\ell-k-1}^{\sim}}\cdot b
\]
(since $\ell-\left(  i+1\right)  -\left(  j-1\right)  +1=\ell
-\underbrace{\left(  i+j-2\right)  }_{=k}-1=\ell-k-1$). This proves
(\ref{pf.thm.rect.antip.5b}).}.

We have $\ell-1\in\mathbb{N}$ and $R^{\ell-1+1}f=R^{\ell}f\neq\undf$. Hence,
the equality (\ref{eq.vl+1=}) (applied to $m$ and $\ell-1$ instead of $v$ and
$\ell$) yields%
\[
m_{\ell-1+1}=\left(  \sum_{x\lessdot m}x_{\ell-1}\right)  \cdot\overline
{m_{\ell-1}}\cdot\overline{\sum_{x\gtrdot m}\overline{x_{\ell-1+1}}}%
\]
(here we have renamed the summation indices $u$ from (\ref{eq.vl+1=}) as $x$,
since the letter $u$ is already being used for something else in our current
setting). Since $\ell-1+1=\ell$, this can be simplified to%
\begin{equation}
m_{\ell}=\left(  \sum_{x\lessdot m}x_{\ell-1}\right)  \cdot\overline
{m_{\ell-1}}\cdot\overline{\sum_{x\gtrdot m}\overline{x_{\ell}}}.
\label{pf.thm.rect.antip.6a}%
\end{equation}

However, recall that the two elements of $\widehat{P}$ that are covered by $m$
are $s$ and $t$. In other words, the two elements $x\in\widehat{P}$ satisfying
$x\lessdot m$ are $s$ and $t$. Hence, $\sum_{x\lessdot m}x_{\ell-1}=s_{\ell
-1}+t_{\ell-1}$.

Also, recall that the two elements of $\widehat{P}$ that cover $m$ are $u$ and
$v$. In other words, the two elements $x\in\widehat{P}$ satisfying $x\gtrdot
m$ are $u$ and $v$. Hence, $\sum_{x\gtrdot m}\overline{x_{\ell}}%
=\overline{u_{\ell}}+\overline{v_{\ell}}$.

Now we know that $\sum_{x\lessdot m}x_{\ell-1}=s_{\ell-1}+t_{\ell-1}$ and
$\sum_{x\gtrdot m}\overline{x_{\ell}}=\overline{u_{\ell}}+\overline{v_{\ell}}%
$. Using these formulas, we can rewrite (\ref{pf.thm.rect.antip.6a}) as%
\begin{equation}
m_{\ell}=\left(  s_{\ell-1}+t_{\ell-1}\right)  \cdot\overline{m_{\ell-1}}%
\cdot\overline{\overline{u_{\ell}}+\overline{v_{\ell}}}.
\label{pf.thm.rect.antip.6}%
\end{equation}

On the other hand, from $k=i+j-2$, we obtain $\ell-k-1=\ell-\left(
i+j-2\right)  -1=\ell-i-j+1\geq0$. Thus, $\ell-k-1\in\mathbb{N}$. Also,
$\ell-\underbrace{k}_{\geq0}\leq\ell$, so that $R^{\ell-k}f\neq\undf$ (by
Lemma \ref{lem.R.wd-triv}, since $R^{\ell}f\neq\undf$). Hence, $R^{\ell
-k-1+1}f=R^{\ell-k}f\neq\undf$. Hence, the equality (\ref{eq.vl+1=}) (applied
to $m^{\sim}$ and $\ell-k-1$ instead of $v$ and $\ell$) yields%
\[
m_{\ell-k-1+1}^{\sim}=\left(  \sum_{x\lessdot m^{\sim}}x_{\ell-k-1}\right)
\cdot\overline{m_{\ell-k-1}^{\sim}}\cdot\overline{\sum_{x\gtrdot m^{\sim}%
}\overline{x_{\ell-k-1+1}}}.
\]
Since $\ell-k-1+1=\ell-k$, this can be simplified to%
\begin{equation}
m_{\ell-k}^{\sim}=\left(  \sum_{x\lessdot m^{\sim}}x_{\ell-k-1}\right)
\cdot\overline{m_{\ell-k-1}^{\sim}}\cdot\overline{\sum_{x\gtrdot m^{\sim}%
}\overline{x_{\ell-k}}}. \label{pf.thm.rect.antip.7a}%
\end{equation}

However, recall that the two elements of $\widehat{P}$ that are covered by
$m^{\sim}$ are $u^{\sim}$ and $v^{\sim}$. In other words, the two elements
$x\in\widehat{P}$ satisfying $x\lessdot m^{\sim}$ are $u^{\sim}$ and $v^{\sim
}$. Hence, $\sum_{x\lessdot m^{\sim}}x_{\ell-k-1}=u_{\ell-k-1}^{\sim}%
+v_{\ell-k-1}^{\sim}$.

Also, recall that the two elements of $\widehat{P}$ that cover $m^{\sim}$ are
$s^{\sim}$ and $t^{\sim}$. In other words, the two elements $x\in\widehat{P}$
satisfying $x\gtrdot m^{\sim}$ are $s^{\sim}$ and $t^{\sim}$. Hence,
$\sum_{x\gtrdot m^{\sim}}\overline{x_{\ell-k}}=\overline{s_{\ell-k}^{\sim}%
}+\overline{t_{\ell-k}^{\sim}}$.

We now know that $\sum_{x\lessdot m^{\sim}}x_{\ell-k-1}=u_{\ell-k-1}^{\sim
}+v_{\ell-k-1}^{\sim}$ and $\sum_{x\gtrdot m^{\sim}}\overline{x_{\ell-k}%
}=\overline{s_{\ell-k}^{\sim}}+\overline{t_{\ell-k}^{\sim}}$. In light of
these two equalities, we can rewrite (\ref{pf.thm.rect.antip.7a}) as%
\begin{equation}
m_{\ell-k}^{\sim}=\left(  u_{\ell-k-1}^{\sim}+v_{\ell-k-1}^{\sim}\right)
\cdot\overline{m_{\ell-k-1}^{\sim}}\cdot\overline{\overline{s_{\ell-k}^{\sim}%
}+\overline{t_{\ell-k}^{\sim}}}. \label{pf.thm.rect.antip.7}%
\end{equation}
This entails that the inverses $\overline{\overline{s_{\ell-k}^{\sim}%
}+\overline{t_{\ell-k}^{\sim}}}$ and $\overline{m_{\ell-k-1}^{\sim}}$ are
well-defined (since they appear on the right hand side of this equality). In
other words, the elements $\overline{s_{\ell-k}^{\sim}}+\overline{t_{\ell
-k}^{\sim}}$ and $m_{\ell-k-1}^{\sim}$ of $\mathbb{K}$ are invertible. Hence,
their product $\left(  \overline{s_{\ell-k}^{\sim}}+\overline{t_{\ell-k}%
^{\sim}}\right)  \cdot m_{\ell-k-1}^{\sim}$ is invertible as well.

Also, $\ell-k\geq1$ (since $\ell-k-1\geq0$) and $R^{\ell-k}f\neq\undf$. Hence,
Lemma \ref{lem.slacks.wd} \textbf{(a)} (applied to $\ell-k$ and $m^{\sim}$
instead of $\ell$ and $v$) shows that $m_{\ell-k}^{\sim}$ is well-defined and
invertible. Now, multiplying both sides of (\ref{pf.thm.rect.antip.7}) with
$\left(  \overline{s_{\ell-k}^{\sim}}+\overline{t_{\ell-k}^{\sim}}\right)
\cdot m_{\ell-k-1}^{\sim}$, on the right, we obtain%
\begin{align*}
&  m_{\ell-k}^{\sim}\cdot\left(  \overline{s_{\ell-k}^{\sim}}+\overline
{t_{\ell-k}^{\sim}}\right)  \cdot m_{\ell-k-1}^{\sim}\\
&  =\left(  u_{\ell-k-1}^{\sim}+v_{\ell-k-1}^{\sim}\right)  \cdot
\overline{m_{\ell-k-1}^{\sim}}\cdot\underbrace{\overline{\overline{s_{\ell
-k}^{\sim}}+\overline{t_{\ell-k}^{\sim}}}\cdot\left(  \overline{s_{\ell
-k}^{\sim}}+\overline{t_{\ell-k}^{\sim}}\right)  }_{=1}\cdot\,m_{\ell
-k-1}^{\sim}\\
&  =\left(  u_{\ell-k-1}^{\sim}+v_{\ell-k-1}^{\sim}\right)  \cdot
\underbrace{\overline{m_{\ell-k-1}^{\sim}}\cdot m_{\ell-k-1}^{\sim}}%
_{=1}=u_{\ell-k-1}^{\sim}+v_{\ell-k-1}^{\sim}.
\end{align*}
Hence,%
\[
u_{\ell-k-1}^{\sim}+v_{\ell-k-1}^{\sim}=m_{\ell-k}^{\sim}\cdot\left(
\overline{s_{\ell-k}^{\sim}}+\overline{t_{\ell-k}^{\sim}}\right)  \cdot
m_{\ell-k-1}^{\sim}.
\]
This shows that $u_{\ell-k-1}^{\sim}+v_{\ell-k-1}^{\sim}$ is a product of
three invertible elements (since $m_{\ell-k}^{\sim}$ and $\overline{s_{\ell
-k}^{\sim}}+\overline{t_{\ell-k}^{\sim}}$ and $m_{\ell-k-1}^{\sim}$ are
invertible). Thus, $u_{\ell-k-1}^{\sim}+v_{\ell-k-1}^{\sim}$ itself is invertible.

Taking reciprocals on both sides of (\ref{pf.thm.rect.antip.7}), we obtain%
\begin{align}
\overline{m_{\ell-k}^{\sim}}  &  =\overline{\left(  u_{\ell-k-1}^{\sim
}+v_{\ell-k-1}^{\sim}\right)  \cdot\overline{m_{\ell-k-1}^{\sim}}%
\cdot\overline{\overline{s_{\ell-k}^{\sim}}+\overline{t_{\ell-k}^{\sim}}}%
}\nonumber\\
&  =\left(  \overline{s_{\ell-k}^{\sim}}+\overline{t_{\ell-k}^{\sim}}\right)
\cdot m_{\ell-k-1}^{\sim}\cdot\overline{u_{\ell-k-1}^{\sim}+v_{\ell-k-1}%
^{\sim}} \label{pf.thm.rect.antip.8}%
\end{align}
(by Proposition \ref{prop.inverses.ab} \textbf{(c)}).

Comparing (\ref{pf.thm.rect.antip.6}) with (\ref{pf.thm.rect.antip.3}), we
obtain%
\begin{align*}
a\cdot\overline{m_{\ell-k}^{\sim}}\cdot b  &  =\left(  \underbrace{s_{\ell-1}%
}_{\substack{=a\cdot\overline{s_{\ell-k}^{\sim}}\cdot b\\\text{(by
(\ref{pf.thm.rect.antip.4}))}}}+\underbrace{t_{\ell-1}}_{\substack{=a\cdot
\overline{t_{\ell-k}^{\sim}}\cdot b\\\text{(by (\ref{pf.thm.rect.antip.4b}))}%
}}\right)  \cdot\underbrace{\overline{m_{\ell-1}}}_{\substack{=\overline
{a\cdot\overline{m_{\ell-k-1}^{\sim}}\cdot b}\\\text{(by
(\ref{pf.thm.rect.antip.5}))}}}\cdot\,\overline{\overline{u_{\ell}}%
+\overline{v_{\ell}}}\\
&  =\underbrace{\left(  a\cdot\overline{s_{\ell-k}^{\sim}}\cdot b+a\cdot
\overline{t_{\ell-k}^{\sim}}\cdot b\right)  }_{=a\cdot\left(  \overline
{s_{\ell-k}^{\sim}}+\overline{t_{\ell-k}^{\sim}}\right)  \cdot b}%
\cdot\underbrace{\overline{a\cdot\overline{m_{\ell-k-1}^{\sim}}\cdot b}%
}_{\substack{=\overline{b}\cdot m_{\ell-k-1}^{\sim}\cdot\overline
{a}\\\text{(by Proposition \ref{prop.inverses.ab} \textbf{(c)},}\\\text{since
}a\text{ and }\overline{m_{\ell-k-1}^{\sim}}\text{ and }b\text{ are
invertible)}}}\cdot\,\overline{\overline{u_{\ell}}+\overline{v_{\ell}}}\\
&  =a\cdot\left(  \overline{s_{\ell-k}^{\sim}}+\overline{t_{\ell-k}^{\sim}%
}\right)  \cdot\underbrace{b\cdot\overline{b}}_{=1}\cdot\,m_{\ell-k-1}^{\sim
}\cdot\overline{a}\cdot\overline{\overline{u_{\ell}}+\overline{v_{\ell}}}\\
&  =a\cdot\left(  \overline{s_{\ell-k}^{\sim}}+\overline{t_{\ell-k}^{\sim}%
}\right)  \cdot m_{\ell-k-1}^{\sim}\cdot\overline{a}\cdot\overline
{\overline{u_{\ell}}+\overline{v_{\ell}}}.
\end{align*}
Multiplying both sides of this equality by $\overline{a}$ on the left and by
$\overline{b}$ on the right (this is allowed, since $a$ and $b$ are
invertible), we obtain%
\begin{align*}
\overline{a}\cdot a\cdot\overline{m_{\ell-k}^{\sim}}\cdot b\cdot\overline{b}
&  =\underbrace{\overline{a}\cdot a}_{=1}\cdot\,\left(  \overline{s_{\ell
-k}^{\sim}}+\overline{t_{\ell-k}^{\sim}}\right)  \cdot m_{\ell-k-1}^{\sim
}\cdot\underbrace{\overline{a}\cdot\overline{\overline{u_{\ell}}%
+\overline{v_{\ell}}}\cdot\overline{b}}_{\substack{=\overline{b\cdot\left(
\overline{u_{\ell}}+\overline{v_{\ell}}\right)  \cdot a}\\\text{(by
Proposition \ref{prop.inverses.ab} \textbf{(c)})}}}\\
&  =\left(  \overline{s_{\ell-k}^{\sim}}+\overline{t_{\ell-k}^{\sim}}\right)
\cdot m_{\ell-k-1}^{\sim}\cdot\overline{b\cdot\left(  \overline{u_{\ell}%
}+\overline{v_{\ell}}\right)  \cdot a}.
\end{align*}
Hence,
\begin{align*}
&  \left(  \overline{s_{\ell-k}^{\sim}}+\overline{t_{\ell-k}^{\sim}}\right)
\cdot m_{\ell-k-1}^{\sim}\cdot\overline{b\cdot\left(  \overline{u_{\ell}%
}+\overline{v_{\ell}}\right)  \cdot a}\\
&  =\underbrace{\overline{a}\cdot a}_{=1}\cdot\,\overline{m_{\ell-k}^{\sim}%
}\cdot\underbrace{b\cdot\overline{b}}_{=1}=\overline{m_{\ell-k}^{\sim}}\\
&  =\left(  \overline{s_{\ell-k}^{\sim}}+\overline{t_{\ell-k}^{\sim}}\right)
\cdot m_{\ell-k-1}^{\sim}\cdot\overline{u_{\ell-k-1}^{\sim}+v_{\ell-k-1}%
^{\sim}}\ \ \ \ \ \ \ \ \ \ \left(  \text{by (\ref{pf.thm.rect.antip.8}%
)}\right)  .
\end{align*}
Cancelling the $\left(  \overline{s_{\ell-k}^{\sim}}+\overline{t_{\ell
-k}^{\sim}}\right)  \cdot m_{\ell-k-1}^{\sim}$ factors on the left of this
equality (this is allowed, since $\left(  \overline{s_{\ell-k}^{\sim}%
}+\overline{t_{\ell-k}^{\sim}}\right)  \cdot m_{\ell-k-1}^{\sim}$ is
invertible), we obtain%
\[
\overline{b\cdot\left(  \overline{u_{\ell}}+\overline{v_{\ell}}\right)  \cdot
a}=\overline{u_{\ell-k-1}^{\sim}+v_{\ell-k-1}^{\sim}}.
\]
Taking reciprocals on both sides, we find%
\[
b\cdot\left(  \overline{u_{\ell}}+\overline{v_{\ell}}\right)  \cdot
a=u_{\ell-k-1}^{\sim}+v_{\ell-k-1}^{\sim}.
\]
In other words,%
\begin{equation}
b\cdot\overline{u_{\ell}}\cdot a+b\cdot\overline{v_{\ell}}\cdot a=u_{\ell
-k-1}^{\sim}+v_{\ell-k-1}^{\sim} \label{pf.thm.rect.antip.14}%
\end{equation}
(since $b\cdot\left(  \overline{u_{\ell}}+\overline{v_{\ell}}\right)  \cdot
a=b\cdot\overline{u_{\ell}}\cdot a+b\cdot\overline{v_{\ell}}\cdot a$).

However, (\ref{pf.thm.rect.antip.5b}) yields%
\[
\overline{u_{\ell}}=\overline{a\cdot\overline{u_{\ell-k-1}^{\sim}}\cdot
b}=\overline{b}\cdot u_{\ell-k-1}^{\sim}\cdot\overline{a}%
\ \ \ \ \ \ \ \ \ \ \left(  \text{by Proposition \ref{prop.inverses.ab}
\textbf{(c)}}\right)  .
\]
Thus,%
\[
b\cdot\overline{u_{\ell}}\cdot a=\underbrace{b\cdot\overline{b}}_{=1}%
\cdot\,u_{\ell-k-1}^{\sim}\cdot\underbrace{\overline{a}\cdot a}_{=1}%
=u_{\ell-k-1}^{\sim}.
\]
Subtracting this equality from (\ref{pf.thm.rect.antip.14}), we obtain%
\begin{equation}
b\cdot\overline{v_{\ell}}\cdot a=v_{\ell-k-1}^{\sim}.
\label{pf.thm.rect.antip.16}%
\end{equation}
The left hand side of this equality is a product of three invertible elements
(since $b$, $\overline{v_{\ell}}$ and $a$ are invertible), and thus itself
invertible. Hence, the right hand side is invertible as well. In other words,
$v_{\ell-k-1}^{\sim}$ is invertible.

Taking reciprocals on both sides of (\ref{pf.thm.rect.antip.16}), we now
obtain $\overline{b\cdot\overline{v_{\ell}}\cdot a}=\overline{v_{\ell
-k-1}^{\sim}}$. Hence,%
\[
\overline{v_{\ell-k-1}^{\sim}}=\overline{b\cdot\overline{v_{\ell}}\cdot
a}=\overline{a}\cdot v_{\ell}\cdot\overline{b}\ \ \ \ \ \ \ \ \ \ \left(
\text{by Proposition \ref{prop.inverses.ab} \textbf{(c)}}\right)  .
\]
Thus,%
\[
a\cdot\overline{v_{\ell-k-1}^{\sim}}\cdot b=\underbrace{a\cdot\overline{a}%
}_{=1}\cdot\,v_{\ell}\cdot\underbrace{\overline{b}\cdot b}_{=1}=v_{\ell}.
\]
In other words,%
\[
v_{\ell}=a\cdot\overline{v_{\ell-k-1}^{\sim}}\cdot b=a\cdot\overline
{v_{\ell-i-j+1}^{\sim}}\cdot b\ \ \ \ \ \ \ \ \ \ \left(  \text{since }%
\ell-k-1=\ell-i-j+1\right)  .
\]
Thus, $v_{\ell}=a\cdot\overline{v_{\ell-i-j+1}^{\sim}}\cdot b$ is proved in
Case 5. \medskip

The arguments required to prove $v_{\ell}=a\cdot\overline{v_{\ell-i-j+1}%
^{\sim}}\cdot b$ in the Cases 3, 4 and 6 are similar to the one we have used
in Case 5, but simpler in some ways. The specific differences are as follows:

\begin{itemize}
\item In Case 3, we have $s\notin P$. The \textquotedblleft
neighborhood\textquotedblright\ of $m$ thus looks as follows:%
\[%
\xymatrix@R=0.6pc@C=1pc{
u \are[rd] & & v \are[ld] \\
& m \are[rd] \\
& & t
}%
\]
(instead of looking as in (\ref{pf.thm.rect.antip.diagram1})). This
necessitates some changes to the proof; in particular, all addends that
involve $s$ or $s^{\sim}$ in any way need to be removed, along with the
equality (\ref{pf.thm.rect.antip.4}).

\item In Case 6, we have $u\notin P$. The \textquotedblleft
neighborhood\textquotedblright\ of $m$ thus looks as follows:%
\[%
\xymatrix@R=0.6pc@C=1pc{
& & v \are[ld] \\
& m \are[rd] \are[ld]  \\
s & & t
}%
\]
(instead of looking as in (\ref{pf.thm.rect.antip.diagram1})). This
necessitates some changes to the proof; in particular, all addends that
involve $u$ or $u^{\sim}$ in any way need to be removed, along with the
equality (\ref{pf.thm.rect.antip.5b}). (Subtraction is no longer required in
this case.)

\item In Case 4, we have $s\notin P$ and $u\notin P$. The \textquotedblleft
neighborhood\textquotedblright\ of $m$ thus looks as follows:%
\[%
\xymatrix@R=0.6pc@C=1pc{
& & v \are[ld] \\
& m \are[rd] \\
& & t
}%
\]
(instead of looking as in (\ref{pf.thm.rect.antip.diagram1})). This
necessitates some changes to the proof; in particular, all addends that
involve $u$ or $u^{\sim}$ or $s$ or $s^{\sim}$ in any way need to be removed,
along with the equalities (\ref{pf.thm.rect.antip.4}) and
(\ref{pf.thm.rect.antip.5b}).
\end{itemize}

Thus, we have proved the equality $v_{\ell}=a\cdot\overline{v_{\ell
-i-j+1}^{\sim}}\cdot b$ in all six Cases 1, 2, 3, 4, 5 and 6. Hence, this
equality always holds. In other words, (\ref{pf.thm.rect.antip.restate1})
holds for $x=v$. This completes the induction step. Thus,
(\ref{pf.thm.rect.antip.restate1}) is proved by induction. In other words,
Theorem \ref{thm.rect.antip} is proven.
\end{proof}
\end{verlong}

As we have already seen (in Section \ref{sec.antip-to-ord}), this entails that
Theorem \ref{thm.rect.ord} is proven as well.

\section{\label{sec.semiring}The case of a semiring}

An attentive reader may have noticed that nowhere in the definitions of
$v$-toggles and birational rowmotion do any subtraction sign appear. This
means that all these definitions can be extended to the case when $\mathbb{K}$
is not a ring but a \emph{semiring}.

A \emph{semiring} is a set $\mathbb{K}$ equipped with a structure of an
abelian semigroup $\left(  \mathbb{K},+\right)  $ and the structure of a (not
necessarily abelian) monoid $\left(  \mathbb{K},\cdot,1\right)  $ such that
the distributive laws $\left(  a+b\right)  c=ac+bc$ and $a\left(  b+c\right)
=ab+ac$ are satisfied (where we use the shorthand notation $xy$ for $x\cdot
y$). Some standard concepts defined for rings can be straightforwardly
generalized to semirings; in particular, any nonempty finite family $\left(
a_{i}\right)  _{i\in I}$ of elements of a semiring $\mathbb{K}$ has a
well-defined sum $\sum_{i\in I}a_{i}$. Definition \ref{def.inverses.inv}, too,
applies verbatim to the case when $\mathbb{K}$ is a semiring instead of a
ring. Thus, the definition of a $v$-toggle (Definition \ref{def.Tv}) and the
definition of birational rowmotion (Definition \ref{def.rm}) can be applied to
a semiring $\mathbb{K}$ as well. We thus can wonder:

\begin{question}
\label{quest.rect.semiring}Do twisted periodicity (Theorem \ref{thm.rect.ord})
and reciprocity (Theorem \ref{thm.rect.antip}) still hold if $\mathbb{K}$ is
not a ring but merely a semiring?
\end{question}

If we assume that $\mathbb{K}$ is commutative, then the answer to this
question is positive, for fairly simple general reasons (see \cite[Remark
10]{bir-row-1}). However, no such general reasoning helps for noncommutative
$\mathbb{K}$. Indeed, there are subtraction-free identities involving inverses
that hold for all rings but fail for some semirings. One example is the
identity $a\cdot\overline{a+b}\cdot b=b\cdot\overline{a+b}\cdot a$ from
Proposition \ref{prop.inverses.a+b} \textbf{(a)}: David Speyer has constructed
an example of a semiring $\mathbb{K}$ and two elements $a$ and $b$ of
$\mathbb{K}$ such that $a+b$ is invertible (actually, $a+b=1$ in his example),
but this identity does not hold. See \cite{Speyer21} for details.

Of course, this does not mean that the answer to Question
\ref{quest.rect.semiring} is negative; we are, in fact, inclined to suspect
that the question has a positive answer. Our proofs of Lemma \ref{lem.rect.i1}
and Lemma \ref{lem.rect.1j} apply in the semiring setting (i.e., when
$\mathbb{K}$ is a semiring rather than a ring) without any need for changes;
thus, Theorem \ref{thm.rect.antip} holds over any semiring $\mathbb{K}$ at
least in the case when one of $i$ and $j$ is $1$. Unfortunately, subtraction
is used in the proof of Theorem \ref{thm.rect.antip}, and we have so far been
unable to excise it from the argument. (With a bit of thought, we can convince
ourselves that subtraction is actually unnecessary if $p=2$ or $q=2$, so the
first interesting case is obtained for $P=\left[  3\right]  \times\left[
3\right]  $.)

\section{\label{sec.otherposets}Other posets: conjectures and results}

We now proceed to discuss the behavior of $R$ on some other families of posets
$P$. We no longer use the notations introduced in Section \ref{sec.proof-nots}.

\subsection{The $\Delta$ and $\nabla$ triangles}

When $p=q$, the $p\times q$-rectangle $\left[  p\right]  \times\left[
q\right]  $ becomes a square. By cutting this square in half along its
horizontal axis, we obtain two triangles:

\begin{definition}
\label{def.DelNabtri}Let $p$ be a positive integer. Define two subsets
$\Delta\left(  p\right)  $ and $\nabla\left(  p\right)  $ of the $p\times
p$-rectangle $\left[  p\right]  \times\left[  p\right]  $ by%
\begin{align*}
\Delta\left(  p\right)   &  =\left\{  \left(  i,k\right)  \in\left[  p\right]
\times\left[  p\right]  \ \mid\ i+k>p+1\right\}  ;\\
\nabla\left(  p\right)   &  =\left\{  \left(  i,k\right)  \in\left[  p\right]
\times\left[  p\right]  \ \mid\ i+k<p+1\right\}  .
\end{align*}
Each of these two subsets $\Delta\left(  p\right)  $ and $\nabla\left(
p\right)  $ inherits a poset structure from $\left[  p\right]  \times\left[
p\right]  $. In the following, we will consider $\Delta\left(  p\right)  $ and
$\nabla\left(  p\right)  $ as posets using these structures.
\end{definition}

The Hasse diagrams of these posets $\Delta\left(  p\right)  $ and
$\nabla\left(  p\right)  $ look like triangles; if we draw $\left[  p\right]
\times\left[  p\right]  $ as agreed in Convention \ref{conv.rect.draw}, then
$\Delta\left(  p\right)  $ is the \textquotedblleft upper
half\textquotedblright\ of the square $\left[  p\right]  \times\left[
p\right]  $, whereas $\nabla\left(  p\right)  $ is the \textquotedblleft lower
half\textquotedblright\ of this square.

\Needspace{24\baselineskip}

\begin{example}
Here is the Hasse diagram of the poset $\Delta\left(  4\right)  $:%
\[
\xymatrixrowsep{0.9pc}\xymatrixcolsep{0.20pc}\xymatrix{
& & & \left(4,4\right) \ar@{-}[rd] \ar@{-}[ld] & & & \\
& & \left(4,3\right) \ar@{-}[rd] \ar@{-}[ld] & & \left(3,4\right) \ar@{-}[rd] \ar@{-}[ld] & & \\
& \left(4,2\right) & & \left(3,3\right) & & \left(2,4\right) & & & .
}
\]
Here, on the other hand, is the Hasse diagram of the poset $\nabla\left(
4\right)  $:
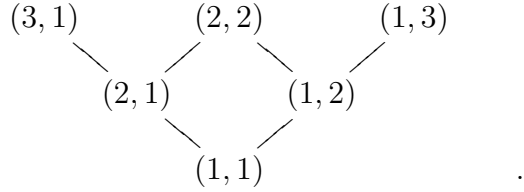
\[
\xymatrixrowsep{0.9pc}\xymatrixcolsep{0.20pc}\xymatrix{
& \left(3,1\right) \ar@{-}[rd] & & \left(2,2\right) \ar@{-}[rd] \ar@{-}[ld] & & \left(1,3\right) \ar@{-}[ld] & \\
& & \left(2,1\right) \ar@{-}[rd] & & \left(1,2\right) \ar@{-}[ld] & & \\
& & & \left(1,1\right) & & & & & .
}
\]

\end{example}

Note that $\Delta\left(  p\right)  =\varnothing$ when $p=1$.

Computations with SageMath \cite{sage} for $p=3$ have made us suspect a
periodicity-like phenomenon similar to Theorem \ref{thm.rect.ord}:

\begin{conjecture}
[periodicity conjecture for $\Delta$-triangle]\label{conj.Deltri.per}Let
$p\geq2$ be an integer. Assume that $P$ is the poset $\Delta\left(  p\right)
$. Let $f\in\mathbb{K}^{\widehat{P}}$ be a $\mathbb{K}$-labeling such that
$R^{p}f\neq\undf$. Let $a=f\left(  0\right)  $ and $b=f\left(  1\right)  $.
Let $x\in\widehat{P}$. We define an element $x^{\prime}\in\widehat{P}$ as follows:

\begin{itemize}
\item If $x=0$ or $x=1$, then we set $x^{\prime}:=x$.

\item Otherwise, we write $x$ in the form $x=\left(  i,j\right)  $, and we set
$x^{\prime}:=\left(  j,i\right)  $.
\end{itemize}

Then, $a$ and $b$ are invertible, and we have
\[
\left(  R^{p}f\right)  \left(  x\right)  =a\overline{b}\cdot f\left(
x^{\prime}\right)  \cdot\overline{a}b.
\]

\end{conjecture}

\begin{conjecture}
[periodicity conjecture for $\nabla$-triangle]\label{conj.Nabtri.per}The same
holds if $P=\nabla\left(  p\right)  $ instead of $P=\Delta\left(  p\right)  $.
\end{conjecture}

If true, these two conjectures would generalize \cite[Theorem 65]{bir-row-2},
where $\mathbb{K}$ is commutative.

\subsection{The \textquotedblleft right half\textquotedblright\ triangle}

We can also cut the square $\left[  p\right]  \times\left[  p\right]  $ along
its vertical axis:

\begin{definition}
\label{def.Leftri}Let $p$ be a positive integer. Define a subset
$\operatorname*{Tria}\left(  p\right)  $ of the $p\times p$-rectangle $\left[
p\right]  \times\left[  p\right]  $ by%
\[
\operatorname*{Tria}\left(  p\right)  :=\left\{  \left(  i,k\right)
\in\left[  p\right]  \times\left[  p\right]  \ \mid\ i\leq k\right\}  .
\]
This subset $\operatorname*{Tria}\left(  p\right)  $ inherits a poset
structure from $\left[  p\right]  \times\left[  p\right]  $.

\end{definition}

The Hasse diagram of this poset $\operatorname*{Tria}\left(  p\right)  $ has
the shape of a triangle; if we draw $\left[  p\right]  \times\left[  p\right]
$ as agreed in Convention \ref{conv.rect.draw}, then $\operatorname*{Tria}%
\left(  p\right)  $ is the \textquotedblleft right half\textquotedblright\ of
the square $\left[  p\right]  \times\left[  p\right]  $.

\Needspace{16\baselineskip}

\begin{example}
Here is the Hasse diagram of the poset $\operatorname*{Tria}\left(  4\right)
$:%
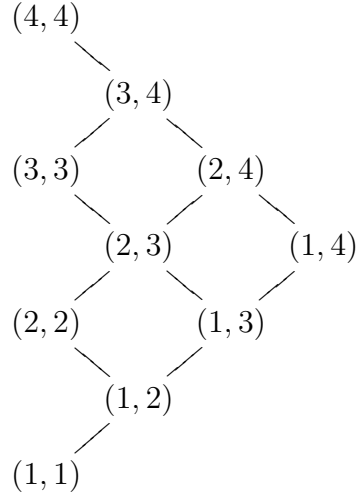
\[
\xymatrixrowsep{0.9pc}\xymatrixcolsep{0.20pc}\xymatrix{
& & & \left(4,4\right) \ar@{-}[rd] & & & \\
& & & & \left(3,4\right) \ar@{-}[rd] \ar@{-}[ld] & & \\
& & & \left(3,3\right) \ar@{-}[rd] & & \left(2,4\right) \ar@{-}[rd] \ar@{-}[ld] & \\
& & & & \left(2,3\right) \ar@{-}[rd] \ar@{-}[ld] & & \left(1,4\right) \ar@{-}[ld] \\
& & & \left(2,2\right) \ar@{-}[rd] & & \left(1,3\right) \ar@{-}[ld] & \\
& & & & \left(1,2\right) \ar@{-}[ld] & & \\
& & & \left(1,1\right) & & & & & .
}
\]

\end{example}

The inequality $i\leq k$ in Definition \ref{def.Leftri} could just as well be
replaced by the reverse inequality $i\geq k$; the resulting poset would be
isomorphic to $\operatorname*{Tria}\left(  p\right)  $. But we have to agree
on something.

Now, we again suspect a periodicity-like phenomenon:

\begin{conjecture}
[periodicity conjecture for \textquotedblleft right half\textquotedblright%
\ triangle]\label{conj.Leftri.per}Let $p$ be a positive integer. Assume that
$P$ is the poset $\operatorname*{Tria}\left(  p\right)  $. Let $f\in
\mathbb{K}^{\widehat{P}}$ be a $\mathbb{K}$-labeling such that $R^{2p}%
f\neq\undf$. Let $a=f\left(  0\right)  $ and $b=f\left(  1\right)  $. Let
$x\in\widehat{P}$. Then, $a$ and $b$ are invertible, and we have%
\[
\left(  R^{2p}f\right)  \left(  x\right)  =a\overline{b}\cdot f\left(
x\right)  \cdot\overline{a}b.
\]

\end{conjecture}

If true, this conjecture would generalize \cite[Theorem 58]{bir-row-2}, where
$\mathbb{K}$ is commutative.

In a sense, we can \textquotedblleft almost\textquotedblright\ prove
Conjecture \ref{conj.Leftri.per}: Namely, the proof of its commutative case
(\cite[Theorem 58]{bir-row-2}) given in \cite{bir-row-2} can be adapted to the
case of a general ring $\mathbb{K}$, as long as the number $2$ is invertible
in $\mathbb{K}$. The latter condition has all the earmarks of a technical
assumption that should not matter for the validity of the result;
unfortunately, however, we are not aware of a rigorous argument that would
allow us to dispose of such an assumption in the noncommutative case.

\subsection{Trapezoids}

Nathan Williams's conjecture \cite[Conjecture 75]{bir-row-2}, too, seems to
extend to the noncommutative setting:

\begin{conjecture}
[periodicity conjecture for the trapezoid]\label{conj.trapezoid.per}Let $p$ be
an integer $>1$. Let $s\in\mathbb{N}$. Assume that $P$ is the subposet
\[
\left\{  \left(  i,k\right)  \in\left[  p\right]  \times\left[  p\right]
\ \mid\ i+k>p+1\text{ and }i\leq k\text{ and }k\geq s\right\}
\]
of $\left[  p\right]  \times\left[  p\right]  $. Let $f\in\mathbb{K}%
^{\widehat{P}}$ be a $\mathbb{K}$-labeling such that $R^{p}f\neq\undf$. Let
$a=f\left(  0\right)  $ and $b=f\left(  1\right)  $. Let $x\in\widehat{P}$.
Then, $a$ and $b$ are invertible, and we have%
\[
\left(  R^{p}f\right)  \left(  x\right)  =a\overline{b}\cdot f\left(
x\right)  \cdot\overline{a}b.
\]

\end{conjecture}

Again, this has been verified using SageMath for certain values of $p$ and $s$
and some randomly chosen $\mathbb{K}$-labelings with $\mathbb{K}%
=\mathbb{Q}^{3\times3}$. Even for commutative $\mathbb{K}$, a proof is yet to
be found, although significant advances have been recently made (see
\cite[Chapter 4]{Johnso23}\footnote{See also \cite{DWYWZ20} for a proof on the
level of order ideals.}).

\subsection{Ill-behaved posets}

The above results and conjectures may suggest that every finite poset $P$ for
which birational rowmotion $R$ has finite order when $\mathbb{K}$ is
commutative must also satisfy a similar (if slightly more complicated)
property when $\mathbb{K}$ is noncommutative. In particular, one might expect
that if some positive integer $m$ satisfies $R^{m}=\operatorname*{id}$ (as
rational maps) for all fields $\mathbb{K}$, then $R^{m}f=f$ should also hold
for all noncommutative rings $\mathbb{K}$ and all $\mathbb{K}$-labelings
$f\in\mathbb{K}^{\widehat{P}}$ that satisfy $f\left(  0\right)  =f\left(
1\right)  =1$ (the latter condition ensures, e.g., that the $a\overline{b}$
and $\overline{a}b$ factors in Theorem \ref{thm.rect.ord} can be removed).
However, this expectation is foiled by the following example:

\begin{example}
\label{exa.illbehaved.claw} Let $P$ be the four-element poset $\left\{
p,q_{1},q_{2},q_{3}\right\}  $ with order relation defined by setting
$p<q_{i}$ for each $i\in\left\{  1,2,3\right\}  $. This poset has Hasse
diagram
\[
\xymatrixrowsep{0.9pc}\xymatrixcolsep{2pc}\xymatrix{
q_1 \ar@{-}[rd] & q_2 \ar@{-}[d] & q_3 \ar@{-}[ld] \\
& p & & & .
}
\]
It is known (see \cite[Example 18]{bir-row-1} or \cite[Corollary
76]{bir-row-1}) that the birational rowmotion $R$ of this poset $P$ satisfies
$R^{6}=\operatorname*{id}$ (as rational maps) if $\mathbb{K}$ is a field. In
other words, if $\mathbb{K}$ is a field, and if $f\in\mathbb{K}^{\widehat{P}}$
is a $\mathbb{K}$-labeling such that $R^{6}f\neq\undf$, then $R^{6}f=f$. But
nothing like this holds when $\mathbb{K}$ is a noncommutative ring. For
instance, if we let $\mathbb{K}$ be the matrix ring $\mathbb{Q}^{2\times2}$,
and if we define a $\mathbb{K}$-labeling $f\in\mathbb{K}^{\widehat{P}}$ by%
\begin{align*}
f\left(  0\right)   &  =I_{2}\ \ \ \ \ \ \ \ \ \ \left(  \text{the identity
matrix in }\mathbb{K}\right)  ,\\
f\left(  1\right)   &  =I_{2},\ \ \ \ \ \ \ \ \ \ f\left(  p\right)
=I_{2},\ \ \ \ \ \ \ \ \ \ f\left(  q_{1}\right)  =I_{2},\\
f\left(  q_{2}\right)   &  =\left(
\begin{array}
[c]{cc}%
1 & 0\\
0 & -1
\end{array}
\right)  ,\ \ \ \ \ \ \ \ \ \ f\left(  q_{3}\right)  =\left(
\begin{array}
[c]{cc}%
1 & 1\\
0 & 1
\end{array}
\right)  ,
\end{align*}
then $R^{m}f$ is distinct from $f$ (and also distinct from $\undf$) for all
positive integers $m$.
\end{example}

\begin{vershort}
(See the detailed version of this article for a proof.)
\end{vershort}

\begin{verlong}
\begin{proof}
[Proof sketch.]Let $\mathbb{K}$ be the matrix ring $\mathbb{Q}^{2\times2}$.
For any row vector $\left(  y,z\right)  \in\mathbb{Q}^{2}$, we define a
$\mathbb{K}$-labeling $f_{\left(  y,z\right)  }\in\mathbb{K}^{\widehat{P}}$ by
setting%
\begin{align*}
f_{\left(  y,z\right)  }\left(  0\right)   &  =f_{\left(  y,z\right)  }\left(
1\right)  =f_{\left(  y,z\right)  }\left(  p\right)  =I_{2};\\
f_{\left(  y,z\right)  }\left(  q_{2}\right)   &  =\left(
\begin{array}
[c]{cc}%
1 & 0\\
0 & -1
\end{array}
\right)  ;\ \ \ \ \ \ \ \ \ \ f_{\left(  y,z\right)  }\left(  q_{1}\right)
=\left(
\begin{array}
[c]{cc}%
1 & y\\
0 & 1
\end{array}
\right)  ;\ \ \ \ \ \ \ \ \ \ f_{\left(  y,z\right)  }\left(  q_{3}\right)
=\left(
\begin{array}
[c]{cc}%
1 & z\\
0 & 1
\end{array}
\right)  .
\end{align*}
Using direct computation (by hand or using SageMath \cite{sage}), we can see
the following:

\begin{statement}
\textit{Claim 1:} For any $y,z\in\mathbb{Q}$, we have $R^{6}f_{\left(
y,z\right)  }=f_{\left(  y^{\prime},z^{\prime}\right)  }$, where%
\[
y^{\prime}:=\dfrac{5y+4z}{9}\ \ \ \ \ \ \ \ \ \ \text{and}%
\ \ \ \ \ \ \ \ \ \ z^{\prime}:=\dfrac{4y+5z}{9}.
\]

\end{statement}

We define a $\mathbb{Q}$-linear map $\Phi:\mathbb{Q}^{2}\rightarrow
\mathbb{Q}^{2}$ that sends each row vector $\left(  y,z\right)  \in
\mathbb{Q}^{2}$ to $\left(  y^{\prime},z^{\prime}\right)  $, where $y^{\prime
}$ and $z^{\prime}$ are as in Claim 1. Then, Claim 1 says that $R^{6}%
f_{v}=f_{\Phi\left(  v\right)  }$ for each $v\in\mathbb{Q}^{2}$. Thus, for
each $v\in\mathbb{Q}^{2}$ and each $i\in\mathbb{N}$, we have $R^{6i}%
f_{v}=f_{\Phi^{i}\left(  v\right)  }$.

However, the endomorphism $\Phi$ of $\mathbb{Q}^{2}$ is diagonalizable with
eigenvalues $1$ and $\dfrac{1}{9}$. Hence, if a vector $\left(  y,z\right)
\in\mathbb{Q}^{2}$ satisfies $y\neq z$, then its iterative images $\Phi
^{0}\left(  y,z\right)  $, $\Phi^{1}\left(  y,z\right)  $, $\Phi^{2}\left(
y,z\right)  $, $\ldots$ are (pairwise) distinct. Therefore, if $y,z\in
\mathbb{Q}^{2}$ satisfy $y\neq z$, then the $\mathbb{K}$-labelings
$R^{6i}f_{\left(  y,z\right)  }=f_{\Phi^{i}\left(  y,z\right)  }$ are
(pairwise) distinct and thus, in particular, distinct from $f_{\left(
y,z\right)  }$. Therefore, in this case, $R^{m}f_{\left(  y,z\right)  }$ is
distinct from $f_{\left(  y,z\right)  }$ for any positive integer $m$ (because
if we had $R^{m}f_{\left(  y,z\right)  }=f_{\left(  y,z\right)  }$ for some
$m$, then we would also have $R^{6m}f_{\left(  y,z\right)  }=f_{\left(
y,z\right)  }$). In particular, by taking $y=0$ and $z=1$, we obtain the
specific labeling constructed in Example \ref{exa.illbehaved.claw}.
\end{proof}
\end{verlong}

\begin{example}
Let $P$ be the four-element poset $\left\{  p_{1},p_{2},q_{1},q_{2}\right\}  $
with order relation defined by setting $p_{i}<q_{j}$ for each $i,j$. It
follows from \cite[Proposition 74 \textbf{(b)} and Proposition 61]{bir-row-1}
that the birational rowmotion $R$ of this poset $P$ satisfies $R^{6}%
=\operatorname*{id}$ (as rational maps) if $\mathbb{K}$ is a field. On the
other hand, if $\mathbb{K}$ is the matrix ring $\mathbb{Q}^{2\times2}$, then
we can easily find a $\mathbb{K}$-labeling $f$ of $P$ such that $R^{m}f\neq f$
for all $1\leq m\leq10\ 000$ (and probably for all positive $m$, but we have
not verified this formally), despite $f\left(  0\right)  $ and $f\left(
1\right)  $ both being the identity matrix $I_{2}$.
\end{example}

\section{\label{sec.Rf11}A note on general posets}

We finish with some curiosities. While Theorem \ref{thm.rect.antip} is
specific to rectangles, its $\left(  i,j\right)  =\left(  1,1\right)  $ case
can be generalized to arbitrary finite posets $P$ in the following form:

\begin{proposition}
\label{prop.R.bottom-top}Let $P$ be any finite poset. Let $f\in\mathbb{K}%
^{\widehat{P}}$ be a labeling of $P$ such that $Rf\neq\undf$. Let $a=f\left(
0\right)  $ and $b=f\left(  1\right)  $. Then,%
\begin{equation}
b\cdot\sum_{\substack{u\in\widehat{P};\\u\gtrdot0}}\overline{\left(
Rf\right)  \left(  u\right)  }\cdot a=\sum_{\substack{u\in\widehat{P}%
;\\u\lessdot1}}f\left(  u\right)  , \label{eq.prop.R.bottom-top.eq}%
\end{equation}
assuming that the inverses $\overline{\left(  Rf\right)  \left(  u\right)  }$
on the left-hand side are well-defined.
\end{proposition}

\begin{vershort}
\begin{proof}
Even though we are not requiring $P$ to be a rectangle, we shall use some of
the notations introduced in Section \ref{sec.proof-nots}. Specifically, we
shall use the notation $x_{\ell}$ defined in (\ref{eq.xl=}), the notion of a
\textquotedblleft path\textquotedblright, and the notations $\downslack_{\ell
}^{v}$, $\upslack_{\ell}^{v}$, $\downslack_{\ell}^{\mathbf{p}}$,
$\upslack_{\ell}^{\mathbf{p}}$, $\downslack_{\ell}^{u\rightarrow v}$ and
$\upslack_{\ell}^{u\rightarrow v}$ defined afterwards. Hence, the equality
(\ref{eq.prop.R.bottom-top.eq}) (which we must prove) can be rewritten as
\begin{equation}
b\cdot\sum_{\substack{u\in\widehat{P};\\u\gtrdot0}}\overline{u_{1}}\cdot
a=\sum_{\substack{u\in\widehat{P};\\u\lessdot1}}u_{0}
\label{pf.prop.R.bottom-top.short.goal}%
\end{equation}
(since $u_{1}=\left(  Rf\right)  \left(  u\right)  $ and $u_{0}=f\left(
u\right)  $).

We assume that the inverses $\overline{\left(  Rf\right)  \left(  u\right)  }$
on the left-hand side of (\ref{eq.prop.R.bottom-top.eq}) are well-defined
(since the claim of Proposition \ref{prop.R.bottom-top} requires this). We
furthermore WLOG assume that $P\neq\varnothing$ (since the claim is easily
checked otherwise). Using these two assumptions, it is not hard to show that
both $a$ and $b$ are invertible. (See the detailed version for a proof.)

In Remark \ref{rmk.slacks.gen}, we have observed that Corollary
\ref{cor.rect.trans-uv}, Proposition \ref{prop.slacks.rec} and parts
\textbf{(a)} and \textbf{(b)} of Theorem \ref{thm.rect.path} hold for our
poset $P$ (even though $P$ is not necessarily a rectangle).

Now, Theorem \ref{thm.rect.path} \textbf{(a)} (applied to $\ell=1$) shows that
each $u\in P$ satisfies%
\[
u_{1}=\overline{\upslack_{1}^{1\rightarrow u}}\cdot b
\]
and thus%
\[
b\cdot\overline{u_{1}}=b\cdot\underbrace{\overline{\overline{\upslack_{1}%
^{1\rightarrow u}}\cdot b}}_{=\overline{b}\cdot\upslack_{1}^{1\rightarrow u}%
}=\underbrace{b\cdot\overline{b}}_{=1}\cdot\,\upslack_{1}^{1\rightarrow
u}=\upslack_{1}^{1\rightarrow u}.
\]
This latter equality also holds for $u=1$ (indeed, from $1_{1}=b$, we obtain
$b\cdot\overline{1_{1}}=b\cdot\overline{b}=1$; but it is easy to prove that
$\upslack_{1}^{1\rightarrow1}=1$ as well, and thus we obtain $b\cdot
\overline{1_{1}}=1=\upslack_{1}^{1\rightarrow1}$). Therefore, it holds for all
$u\in P\cup\left\{  1\right\}  $. Hence, in particular, it holds for all
$u\in\widehat{P}$ satisfying $u\gtrdot0$. Summing it over all such $u$, we
obtain%
\[
\sum_{\substack{u\in\widehat{P};\\u\gtrdot0}}b\cdot\overline{u_{1}}%
=\sum_{\substack{u\in\widehat{P};\\u\gtrdot0}} \upslack_{1}^{1\rightarrow u}
=\upslack_{1}^{1\rightarrow0}%
\]
(since (\ref{eq.prop.slacks.rec.up2}) (applied to $\ell=1$ and $s=1$ and
$t=0$) yields $\upslack_{1}^{1\rightarrow0}=\sum_{\substack{u\in
\widehat{P};\\u\gtrdot0}}\upslack_{1}^{1\rightarrow u}
\underbrace{\upslack_{1}^{0}}_{=1} =\sum_{\substack{u\in\widehat{P}%
;\\u\gtrdot0}} \upslack_{1}^{1\rightarrow u}$). Therefore,%
\begin{equation}
b\cdot\sum_{\substack{u\in\widehat{P};\\u\gtrdot0}}\overline{u_{1}}%
=\sum_{\substack{u\in\widehat{P};\\u\gtrdot0}}b\cdot\overline{u_{1}%
}=\upslack_{1}^{1\rightarrow0}=\downslack_{0}^{1\rightarrow0}
\label{pf.prop.R.bottom-top.short.5}%
\end{equation}
(by Corollary \ref{cor.rect.trans-uv}, applied to $\ell=1$ and $u=1$ and
$v=0$). Hence,%
\begin{align*}
b\cdot\sum_{\substack{u\in\widehat{P};\\u\gtrdot0}}\overline{u_{1}}  &
=\downslack_{0}^{1\rightarrow0} =\underbrace{\downslack_{0}^{1}}_{=1}%
\sum_{\substack{u\in\widehat{P};\\1\gtrdot u}}\downslack_{0}^{u\rightarrow
0}\ \ \ \ \ \ \ \ \ \ \left(
\begin{array}
[c]{c}%
\text{by (\ref{eq.prop.slacks.rec.down1}), applied to }\ell=0\\
\text{and }s=1\text{ and }t=0
\end{array}
\right) \\
&  =\sum_{\substack{u\in\widehat{P};\\1\gtrdot u}}\downslack_{0}%
^{u\rightarrow0} = \sum_{\substack{u\in\widehat{P};\\u\lessdot1}%
}\downslack_{0}^{u\rightarrow0}.
\end{align*}
Multiplying both sides of this equality by $a$ on the right, we obtain%
\begin{equation}
b\cdot\sum_{\substack{u\in\widehat{P};\\u\gtrdot0}}\overline{u_{1}}\cdot
a=\sum_{\substack{u\in\widehat{P};\\u\lessdot1}}\downslack_{0}^{u\rightarrow
0}\cdot a. \label{pf.prop.R.bottom-top.short.9}%
\end{equation}

However, Theorem \ref{thm.rect.path} \textbf{(b)} (applied to $\ell=0$) shows
that each $u\in P$ satisfies%
\[
u_{0}=\downslack_{0}^{u\rightarrow0}\cdot a.
\]
This equality also holds for $u=0$ (since $0_{0}=a$ equals
$\underbrace{\downslack_{0}^{0\rightarrow0}}_{=1}\cdot\, a=a$). Thus, it holds
for all $u\in P\cup\left\{  0\right\}  $. In particular, it therefore holds
for all $u\in\widehat{P}$ satisfying $u\lessdot1$. Summing it over all such
$u$, we obtain
\[
\sum_{\substack{u\in\widehat{P};\\u\lessdot1}}u_{0}=\sum_{\substack{u\in
\widehat{P};\\u\lessdot1}}\downslack_{0}^{u\rightarrow0}\cdot a.
\]
Comparing this with (\ref{pf.prop.R.bottom-top.short.9}), we obtain%
\[
b\cdot\sum_{\substack{u\in\widehat{P};\\u\gtrdot0}}\overline{u_{1}}\cdot
a=\sum_{\substack{u\in\widehat{P};\\u\lessdot1}}u_{0}.
\]
This proves (\ref{pf.prop.R.bottom-top.short.goal}) and, with it, Proposition
\ref{prop.R.bottom-top}.
\end{proof}
\end{vershort}

\begin{verlong}
\begin{proof}
Even though we are not requiring $P$ to be a rectangle, we shall use some of
the notations introduced in Section \ref{sec.proof-nots}. Specifically, we
shall use the notation $x_{\ell}$ defined in (\ref{eq.xl=}), the notion of a
\textquotedblleft path\textquotedblright, and the notations $\downslack_{\ell
}^{v}$, $\upslack_{\ell}^{v}$, $\downslack_{\ell}^{\mathbf{p}}$,
$\upslack_{\ell}^{\mathbf{p}}$, $\downslack_{\ell}^{u\rightarrow v}$ and
$\upslack_{\ell}^{u\rightarrow v}$ defined afterwards. Every $u\in\widehat{P}$
satisfies%
\begin{align}
u_{0}  &  =\underbrace{\left(  R^{0}f\right)  }_{\substack{=f\\\text{(since
}R^{0}=\operatorname*{id}\text{)}}}\left(  u\right)
\ \ \ \ \ \ \ \ \ \ \left(  \text{by (\ref{eq.xl=})}\right) \nonumber\\
&  =f\left(  u\right)  \label{pf.prop.R.bottom-top.u0=}%
\end{align}
and%
\begin{align}
u_{1}  &  =\left(  \underbrace{R^{1}}_{=R}f\right)  \left(  u\right)
\ \ \ \ \ \ \ \ \ \ \left(  \text{by (\ref{eq.xl=})}\right) \nonumber\\
&  =\left(  Rf\right)  \left(  u\right)  . \label{pf.prop.R.bottom-top.u1=}%
\end{align}

We assume that the inverses $\overline{\left(  Rf\right)  \left(  u\right)  }$
on the left-hand side of (\ref{eq.prop.R.bottom-top.eq}) are well-defined
(since the claim of Proposition \ref{prop.R.bottom-top} requires this). In
other words, we assume that%
\begin{equation}
\left(  Rf\right)  \left(  u\right)  \text{ is invertible for every }%
u\in\widehat{P}\text{ satisfying }u\gtrdot0. \label{pf.prop.R.bottom-top.ible}%
\end{equation}

It is easy to see that Proposition \ref{prop.R.bottom-top} holds when
$P=\varnothing$\ \ \ \ \footnote{\textit{Proof.} Assume that $P=\varnothing$.
Thus, $\widehat{P}=\left\{  0,1\right\}  $ with $0<1$. Hence, the only
$u\in\widehat{P}$ satisfying $u\gtrdot0$ is $1$. Therefore, $\sum
_{\substack{u\in\widehat{P};\\u\gtrdot0}}\overline{\left(  Rf\right)  \left(
u\right)  }=\overline{\left(  Rf\right)  \left(  1\right)  }=\overline{b}$
(since Proposition \ref{prop.R.implicit.01} yields $\left(  Rf\right)  \left(
1\right)  =f\left(  1\right)  =b$). Moreover, the only $u\in\widehat{P}$
satisfying $u\lessdot1$ is $0$ (since $\widehat{P}=\left\{  0,1\right\}  $
with $0<1$). Thus, $\sum_{\substack{u\in\widehat{P};\\u\lessdot1}}f\left(
u\right)  =f\left(  0\right)  =a$. Now,%
\[
b\cdot\underbrace{\sum_{\substack{u\in\widehat{P};\\u\gtrdot0}}\overline
{\left(  Rf\right)  \left(  u\right)  }}_{=\overline{b}}\cdot\,
a=\underbrace{b\cdot\overline{b}}_{=1}\cdot\, a=a=\sum_{\substack{u\in
\widehat{P};\\u\lessdot1}}f\left(  u\right)  .
\]
Hence, Proposition \ref{prop.R.bottom-top} is proved (under the assumption
that $P=\varnothing$).}. Thus, we WLOG assume that $P\neq\varnothing$. Hence,
Lemma \ref{lem.R.1inv} yields that $f\left(  1\right)  $ is invertible. In
other words, $b$ is invertible (since $b=f\left(  1\right)  $).

Furthermore, we can easily see that $a$ is invertible\footnote{\textit{Proof.}
Proposition \ref{prop.poset-minmax} \textbf{(a)} yields that the poset $P$ has
a minimal element. Consider such an element, and denote it by $p$. Then, $p$
is a minimal element of $P$, and therefore satisfies $0\lessdot p$ in
$\widehat{P}$ (by Remark \ref{rmk.Phat.covers} \textbf{(a)}). Hence,
$p\gtrdot0$ in $\widehat{P}$. Therefore, (\ref{pf.prop.R.bottom-top.ible})
(applied to $u=p$) yields that $\left(  Rf\right)  \left(  p\right)  $ is
invertible. However, Proposition \ref{prop.R.implicit} (applied to $v=p$)
yields%
\begin{equation}
\left(  Rf\right)  \left(  p\right)  =\left(  \sum\limits_{\substack{u\in
\widehat{P};\\u\lessdot p}}f\left(  u\right)  \right)  \cdot\overline{f\left(
p\right)  }\cdot\overline{\sum\limits_{\substack{u\in\widehat{P};\\u\gtrdot
p}}\overline{\left(  Rf\right)  \left(  u\right)  }}.
\label{pf.prop.R.bottom-top.fn2.1}%
\end{equation}
Thus, the two elements $f\left(  p\right)  $ and $\sum\limits_{\substack{u\in
\widehat{P};\\u\gtrdot p}}\overline{\left(  Rf\right)  \left(  u\right)  }$ of
$\mathbb{K}$ are invertible (since their inverses appear in
(\ref{pf.prop.R.bottom-top.fn2.1})).
\par
However, $p$ is a minimal element of $P$. Thus, no $u\in P$ satisfies
$u\lessdot p$. Therefore, the only $u\in\widehat{P}$ that satisfies $u\lessdot
p$ is $0$. Therefore, $\sum\limits_{\substack{u\in\widehat{P};\\u\lessdot
p}}f\left(  u\right)  =f\left(  0\right)  =a$. Thus, we can rewrite the
equality (\ref{pf.prop.R.bottom-top.fn2.1}) as%
\[
\left(  Rf\right)  \left(  p\right)  =a\cdot\overline{f\left(  p\right)
}\cdot\overline{\sum\limits_{\substack{u\in\widehat{P};\\u\gtrdot p}%
}\overline{\left(  Rf\right)  \left(  u\right)  }}.
\]
Multiplying both sides of this equality by $\left(  \sum
\limits_{\substack{u\in\widehat{P};\\u\gtrdot p}}\overline{\left(  Rf\right)
\left(  u\right)  }\right)  \cdot f\left(  p\right)  $ on the right, we obtain%
\begin{align*}
\left(  Rf\right)  \left(  p\right)  \cdot\left(  \sum\limits_{\substack{u\in
\widehat{P};\\u\gtrdot p}}\overline{\left(  Rf\right)  \left(  u\right)
}\right)  \cdot f\left(  p\right)   &  =a\cdot\overline{f\left(  p\right)
}\cdot\underbrace{\overline{\sum\limits_{\substack{u\in\widehat{P};\\u\gtrdot
p}}\overline{\left(  Rf\right)  \left(  u\right)  }}\cdot\left(
\sum\limits_{\substack{u\in\widehat{P};\\u\gtrdot p}}\overline{\left(
Rf\right)  \left(  u\right)  }\right)  }_{=1}\cdot f\left(  p\right) \\
&  =a\cdot\underbrace{\overline{f\left(  p\right)  }\cdot f\left(  p\right)
}_{=1} =a.
\end{align*}
The left hand side of this equality is a product of three invertible elements
of $\mathbb{K}$ (since the elements $\left(  Rf\right)  \left(  p\right)  $,
$\sum\limits_{\substack{u\in\widehat{P};\\u\gtrdot p}}\overline{\left(
Rf\right)  \left(  u\right)  }$ and $f\left(  p\right)  $ are invertible), and
thus itself must be invertible. Hence, the right hand side is invertible as
well. In other words, $a$ is invertible.}.

In Remark \ref{rmk.slacks.gen}, we have observed that Corollary
\ref{cor.rect.trans-uv}, Proposition \ref{prop.slacks.rec} and parts
\textbf{(a)} and \textbf{(b)} of Theorem \ref{thm.rect.path} hold for our
poset $P$ (even though $P$ is not necessarily a rectangle). In particular,
Corollary \ref{cor.rect.trans-uv} (applied to $\ell=1$, $u=1$ and $v=0$)
yields that
\begin{equation}
\upslack_{1}^{1\rightarrow0}=\downslack_{0}^{1\rightarrow0}
\label{pf.prop.R.bottom-top.u=d}%
\end{equation}
(since $1\geq1$ and $R^{1}f=Rf\neq\undf$).

Now, $1\geq1$ and $R^{1}f=Rf\neq\undf$. Hence, Theorem \ref{thm.rect.path}
\textbf{(a)} (applied to $\ell=1$) shows that each $u\in P$ satisfies%
\begin{equation}
u_{1}=\overline{\upslack_{1}^{1\rightarrow u}}\cdot b.
\label{pf.prop.R.bottom-top.u1=2}%
\end{equation}
Hence, for each $u\in\widehat{P}$ satisfying $u\gtrdot0$, we have%
\begin{equation}
b\cdot\overline{\left(  Rf\right)  \left(  u\right)  }=\upslack_{1}%
^{1\rightarrow u} \label{pf.prop.R.bottom-top.u1=3}%
\end{equation}
\footnote{\textit{Proof of (\ref{pf.prop.R.bottom-top.u1=3}):} Let
$u\in\widehat{P}$ satisfy $u\gtrdot0$. We must prove
(\ref{pf.prop.R.bottom-top.u1=3}). We note that $\left(  Rf\right)  \left(
u\right)  $ is invertible (by (\ref{pf.prop.R.bottom-top.ible})), so that
$\overline{\left(  Rf\right)  \left(  u\right)  }$ is well-defined.
\par
We are in one of the following two cases:
\par
\textit{Case 1:} We have $u=1$.
\par
\textit{Case 2:} We have $u\neq1$.
\par
Let us first consider Case 1. In this case, we have $u=1$. Thus, $\left(
Rf\right)  \left(  u\right)  =\left(  Rf\right)  \left(  1\right)  =f\left(
1\right)  $ (by Proposition \ref{prop.R.implicit.01}). Hence, $\left(
Rf\right)  \left(  u\right)  =f\left(  1\right)  =b$. Thus, $b\cdot
\overline{\left(  Rf\right)  \left(  u\right)  }=b\cdot\overline{b}=1$.
\par
However, the only path from $1$ to $1$ is the trivial path $\left(  1\right)
$. Thus, $\upslack_{1}^{1\rightarrow1}=\upslack_{1}^{\left(  1\right)
}=\upslack_{1}^{1}=1$ (by the definition of $\upslack_{1}^{1}$). Comparing
this with $b\cdot\overline{\left(  Rf\right)  \left(  u\right)  }=1$, we
obtain $b\cdot\overline{\left(  Rf\right)  \left(  u\right)  }=\upslack_{1}%
^{1\rightarrow1}$. In other words, $b\cdot\overline{\left(  Rf\right)  \left(
u\right)  }=\upslack_{1}^{1\rightarrow u}$ (since $1=u$). Thus,
(\ref{pf.prop.R.bottom-top.u1=3}) is proved in Case 1.
\par
Let us now consider Case 2. In this case, we have $u\neq1$. Also, we have
$u>0$ (since $u\gtrdot0$), so that $u\neq0$. Combining $u\in\widehat{P}$ with
$u\neq0$ and $u\neq1$, we obtain $u\in\widehat{P}\setminus\left\{
0,1\right\}  =P$. Hence, (\ref{pf.prop.R.bottom-top.u1=2}) yields
$u_{1}=\overline{\upslack_{1}^{1\rightarrow u}}\cdot b$. In view of
(\ref{pf.prop.R.bottom-top.u1=}), we can rewrite this as $\left(  Rf\right)
\left(  u\right)  =\overline{\upslack_{1}^{1\rightarrow u}}\cdot b$. Since
$\left(  Rf\right)  \left(  u\right)  $ is invertible, we can take inverses on
both sides of this equality. We thus obtain%
\begin{align*}
\overline{\left(  Rf\right)  \left(  u\right)  }  &  =\overline{\overline
{\upslack_{1}^{1\rightarrow u}}\cdot b}=\overline{b}\cdot\underbrace{\overline
{\overline{\upslack_{1}^{1\rightarrow u}}}}_{= \upslack_{1}^{1\rightarrow u}}
\ \ \ \ \ \ \ \ \ \ \left(  \text{since }\overline{\upslack_{1}^{1\rightarrow
u}}\text{ and }b\text{ are invertible}\right) \\
&  =\overline{b}\cdot\upslack_{1}^{1\rightarrow u}.
\end{align*}
Thus,%
\[
b\cdot\overline{\left(  Rf\right)  \left(  u\right)  }=\underbrace{b\cdot
\overline{b}}_{=1}\cdot\,\upslack_{1}^{1\rightarrow u}=\upslack_{1}%
^{1\rightarrow u}.
\]
Thus, (\ref{pf.prop.R.bottom-top.u1=3}) is proved in Case 2.
\par
We have now proved (\ref{pf.prop.R.bottom-top.u1=3}) in both Cases 1 and 2.
Hence, (\ref{pf.prop.R.bottom-top.u1=3}) always holds.}. Thus,%
\begin{align}
b\cdot\sum_{\substack{u\in\widehat{P};\\u\gtrdot0}}\overline{\left(
Rf\right)  \left(  u\right)  }\cdot a  &  =\sum_{\substack{u\in\widehat{P}%
;\\u\gtrdot0}}\underbrace{b\cdot\overline{\left(  Rf\right)  \left(  u\right)
}}_{\substack{=\upslack_{1}^{1\rightarrow u}\\\text{(by
(\ref{pf.prop.R.bottom-top.u1=3}))}}}\cdot\, a\nonumber\\
&  =\sum_{\substack{u\in\widehat{P};\\u\gtrdot0}}\upslack_{1}^{1\rightarrow
u}\cdot a. \label{pf.prop.R.bottom-top.4}%
\end{align}
However, $\upslack_{1}^{0}=1$ (by definition of $\upslack_{1}^{0}$). Now,
(\ref{eq.prop.slacks.rec.up2}) (applied to $\ell=1$ and $s=1$ and $t=0$)
yields%
\begin{equation}
\upslack_{1}^{1\rightarrow0}=\sum_{\substack{u\in\widehat{P};\\u\gtrdot
0}}\upslack_{1}^{1\rightarrow u}\underbrace{\upslack_{1}^{0}}_{=1}%
=\sum_{\substack{u\in\widehat{P};\\u\gtrdot0}}\upslack_{1}^{1\rightarrow u}.
\label{pf.prop.R.bottom-top.5}%
\end{equation}
Thus, (\ref{pf.prop.R.bottom-top.4}) becomes%
\begin{align}
b\cdot\sum_{\substack{u\in\widehat{P};\\u\gtrdot0}}\overline{\left(
Rf\right)  \left(  u\right)  }\cdot a  &  =\underbrace{\sum_{\substack{u\in
\widehat{P};\\u\gtrdot0}}\upslack_{1}^{1\rightarrow u}}%
_{\substack{=\upslack_{1}^{1\rightarrow0}\\\text{(by
(\ref{pf.prop.R.bottom-top.5}))}}}\cdot\, a=\underbrace{\upslack_{1}%
^{1\rightarrow0}}_{\substack{=\downslack_{0}^{1\rightarrow0}\\\text{(by
(\ref{pf.prop.R.bottom-top.u=d}))}}}\cdot\, a\nonumber\\
&  =\downslack_{0}^{1\rightarrow0}\cdot a. \label{pf.prop.R.bottom-top.6}%
\end{align}

However, $R^{0+1}f=R^{1}f=Rf\neq\undf$. Hence, Theorem \ref{thm.rect.path}
\textbf{(b)} (applied to $\ell=0$) shows that each $u\in P$ satisfies%
\begin{equation}
u_{0}=\downslack_{0}^{u\rightarrow0}\cdot a. \label{pf.prop.R.bottom-top.u0=2}%
\end{equation}
Hence, for each $u\in\widehat{P}$ satisfying $u\lessdot1$, we have%
\begin{equation}
f\left(  u\right)  =\downslack_{0}^{u\rightarrow0}\cdot a
\label{pf.prop.R.bottom-top.u0=3}%
\end{equation}
\footnote{\textit{Proof of (\ref{pf.prop.R.bottom-top.u0=3}):} Let
$u\in\widehat{P}$ satisfy $u\lessdot1$. We must prove
(\ref{pf.prop.R.bottom-top.u0=3}).
\par
We are in one of the following two cases:
\par
\textit{Case 1:} We have $u=0$.
\par
\textit{Case 2:} We have $u\neq0$.
\par
Let us first consider Case 1. In this case, we have $u=0$. Thus, $f\left(
u\right)  =f\left(  0\right)  =a$.
\par
However, the only path from $0$ to $0$ is the trivial path $\left(  0\right)
$. Thus, $\downslack_{0}^{0\rightarrow0}=\downslack_{0}^{\left(  0\right)
}=\downslack_{0}^{0}=1$ (by the definition of $\downslack_{0}^{0}$). From
$u=0$, we obtain $\downslack_{0}^{u\rightarrow0}\cdot
a=\underbrace{\downslack_{0}^{0\rightarrow0}}_{=1}\cdot\, a=a$. Comparing this
with $f\left(  u\right)  =a$, we obtain $f\left(  u\right)  =\downslack_{0}%
^{u\rightarrow0}\cdot a$. Thus, (\ref{pf.prop.R.bottom-top.u0=3}) is proved in
Case 1.
\par
Let us now consider Case 2. In this case, we have $u\neq0$. Also, we have
$u<1$ (since $u\lessdot1$), so that $u\neq1$. Combining $u\in\widehat{P}$ with
$u\neq0$ and $u\neq1$, we obtain $u\in\widehat{P}\setminus\left\{
0,1\right\}  =P$. Hence, (\ref{pf.prop.R.bottom-top.u0=2}) yields
$u_{0}=\downslack_{0}^{u\rightarrow0}\cdot a$. In view of
(\ref{pf.prop.R.bottom-top.u0=}), we can rewrite this as $f\left(  u\right)
=\downslack_{0}^{u\rightarrow0}\cdot a$. Thus,
(\ref{pf.prop.R.bottom-top.u0=3}) is proved in Case 2.
\par
We have now proved (\ref{pf.prop.R.bottom-top.u0=3}) in both Cases 1 and 2.
Hence, (\ref{pf.prop.R.bottom-top.u0=3}) always holds.}. Thus,%
\begin{equation}
\sum_{\substack{u\in\widehat{P};\\u\lessdot1}}\underbrace{f\left(  u\right)
}_{\substack{=\downslack_{0}^{u\rightarrow0}\cdot a\\\text{(by
(\ref{pf.prop.R.bottom-top.u0=3}))}}}=\sum_{\substack{u\in\widehat{P}%
;\\u\lessdot1}}\downslack_{0}^{u\rightarrow0}\cdot a.
\label{pf.prop.R.bottom-top.8}%
\end{equation}

However, (\ref{eq.prop.slacks.rec.down1}) (applied to $\ell=0$ and $s=1$ and
$t=0$) yields%
\[
\downslack_{0}^{1\rightarrow0}=\underbrace{\downslack_{0}^{1}}%
_{\substack{=1\\\text{(by the definition of }\downslack_{0}^{1}\text{)}}%
}\sum_{\substack{u\in\widehat{P};\\1\gtrdot u}}\downslack_{0}^{u\rightarrow
0}=\sum_{\substack{u\in\widehat{P};\\1\gtrdot u}}\downslack_{0}^{u\rightarrow
0}=\sum_{\substack{u\in\widehat{P};\\u\lessdot1}}\downslack_{0}^{u\rightarrow
0}%
\]
(since the condition \textquotedblleft$1\gtrdot u$\textquotedblright\ under
the summation sign is equivalent to \textquotedblleft$u\lessdot1$%
\textquotedblright). Thus,%
\[
\downslack_{0}^{1\rightarrow0}\cdot a=\sum_{\substack{u\in\widehat{P}%
;\\u\lessdot1}}\downslack_{0}^{u\rightarrow0}\cdot a=\sum_{\substack{u\in
\widehat{P};\\u\lessdot1}}f\left(  u\right)
\]
(by (\ref{pf.prop.R.bottom-top.8})). Therefore, (\ref{pf.prop.R.bottom-top.6})
can be rewritten as%
\[
b\cdot\sum_{\substack{u\in\widehat{P};\\u\gtrdot0}}\overline{\left(
Rf\right)  \left(  u\right)  }\cdot a=\sum_{\substack{u\in\widehat{P}%
;\\u\lessdot1}}f\left(  u\right)  .
\]
Proposition \ref{prop.R.bottom-top} is thus proven.
\end{proof}
\end{verlong}

\begin{proposition}
\label{prop.R.invariant}Let $P$ be any finite poset. Let $f\in\mathbb{K}%
^{\widehat{P}}$ be a labeling of $P$ such that $Rf\neq\undf$ and $f\left(
0\right)  =f\left(  1\right)  =1$. Then,%
\[
\sum_{\substack{u,v\in\widehat{P};\\u\lessdot v}}\left(  Rf\right)  \left(
u\right)  \cdot\overline{\left(  Rf\right)  \left(  v\right)  }=\sum
_{\substack{u,v\in\widehat{P};\\u\lessdot v}}f\left(  u\right)  \cdot
\overline{f\left(  v\right)  },
\]
assuming that the inverses $\overline{\left(  Rf\right)  \left(  v\right)  }$
on the left-hand side are well-defined.
\end{proposition}

Proposition \ref{prop.R.invariant} is essentially saying that the sum
$\sum_{\substack{u,v\in\widehat{P};\\u\lessdot v}}f\left(  u\right)
\cdot\overline{f\left(  v\right)  }$ is an invariant under birational
rowmotion $R$ when $f\left(  0\right)  =f\left(  1\right)  =1$. This is a
noncommutative analogue of the conservation of the \textquotedblleft
superpotential\textquotedblright\ $\mathcal{F}_{G}\left(  X\right)  $ of an
$R$-system (\cite[Proposition 5.2]{GalPyl19}). We do not know whether such
invariants exist in the general case.

\begin{vershort}
\begin{proof}
[Proof of Proposition \ref{prop.R.invariant}.]This follows by combining
Proposition \ref{prop.R.bottom-top} with Proposition \ref{prop.R.implicit}.
(This is elaborated in the detailed version.)
\end{proof}
\end{vershort}

\begin{verlong}
\begin{proof}
[Proof of Proposition \ref{prop.R.invariant}.]We have $f\left(  0\right)
=f\left(  1\right)  =1$. Thus, $1=f\left(  0\right)  $ and $1=f\left(
1\right)  $. Hence, Proposition \ref{prop.R.bottom-top} (applied to $a=1$ and
$b=1$) yields%
\[
1\cdot\sum_{\substack{u\in\widehat{P};\\u\gtrdot0}}\overline{\left(
Rf\right)  \left(  u\right)  }\cdot1=\sum_{\substack{u\in\widehat{P}%
;\\u\lessdot1}}f\left(  u\right)  .
\]
This obviously simplifies to
\begin{equation}
\sum_{\substack{u\in\widehat{P};\\u\gtrdot0}}\overline{\left(  Rf\right)
\left(  u\right)  }=\sum_{\substack{u\in\widehat{P};\\u\lessdot1}}f\left(
u\right)  . \label{pf.prop.R.invariant.01}%
\end{equation}

Proposition \ref{prop.R.implicit.01} yields $\left(  Rf\right)  \left(
0\right)  =f\left(  0\right)  =1$. Also, from $f\left(  1\right)  =1$, we
obtain $\overline{f\left(  1\right)  }=\overline{1}=1$.

Now, let $v\in P$. Then, Proposition \ref{prop.R.implicit} yields%
\[
\left(  Rf\right)  \left(  v\right)  =\left(  \sum\limits_{\substack{u\in
\widehat{P};\\u\lessdot v}}f\left(  u\right)  \right)  \cdot\overline{f\left(
v\right)  }\cdot\overline{\sum\limits_{\substack{u\in\widehat{P};\\u\gtrdot
v}}\overline{\left(  Rf\right)  \left(  u\right)  }}.
\]
Multiplying both sides of this equality by $\sum\limits_{\substack{u\in
\widehat{P};\\u\gtrdot v}}\overline{\left(  Rf\right)  \left(  u\right)  }$ on
the right, we obtain%
\begin{align}
\left(  Rf\right)  \left(  v\right)  \cdot\sum\limits_{\substack{u\in
\widehat{P};\\u\gtrdot v}}\overline{\left(  Rf\right)  \left(  u\right)  }  &
=\left(  \sum\limits_{\substack{u\in\widehat{P};\\u\lessdot v}}f\left(
u\right)  \right)  \cdot\overline{f\left(  v\right)  }\,\cdot
\underbrace{\overline{\sum\limits_{\substack{u\in\widehat{P};\\u\gtrdot
v}}\overline{\left(  Rf\right)  \left(  u\right)  }}\cdot\sum
\limits_{\substack{u\in\widehat{P};\\u\gtrdot v}}\overline{\left(  Rf\right)
\left(  u\right)  }}_{=1}\nonumber\\
&  =\left(  \sum\limits_{\substack{u\in\widehat{P};\\u\lessdot v}}f\left(
u\right)  \right)  \cdot\overline{f\left(  v\right)  }\nonumber\\
&  =\sum_{\substack{u\in\widehat{P};\\u\lessdot v}}f\left(  u\right)
\cdot\overline{f\left(  v\right)  }. \label{pf.prop.R.invariant.1}%
\end{align}

Forget that we fixed $v$. We thus have proved (\ref{pf.prop.R.invariant.1})
for each $v\in P$.

Now,
\begin{align*}
&  \sum_{\substack{u,v\in\widehat{P};\\u\lessdot v}}\left(  Rf\right)  \left(
u\right)  \cdot\overline{\left(  Rf\right)  \left(  v\right)  }\\
&  =\sum_{\substack{u,v\in\widehat{P};\\v\gtrdot u}}\left(  Rf\right)  \left(
u\right)  \cdot\overline{\left(  Rf\right)  \left(  v\right)  }%
\ \ \ \ \ \ \ \ \ \ \left(
\begin{array}
[c]{c}%
\text{since the condition \textquotedblleft}u\lessdot
v\text{\textquotedblright}\\
\text{is equivalent to \textquotedblleft}v\gtrdot u\text{\textquotedblright}%
\end{array}
\right) \\
&  =\underbrace{\sum_{\substack{v,u\in\widehat{P};\\u\gtrdot v}}}_{=\sum
_{v\in\widehat{P}}\ \ \sum_{\substack{u\in\widehat{P};\\u\gtrdot v}}}\left(
Rf\right)  \left(  v\right)  \cdot\overline{\left(  Rf\right)  \left(
u\right)  }\ \ \ \ \ \ \ \ \ \ \left(
\begin{array}
[c]{c}%
\text{here, we have renamed the}\\
\text{summation indices }u\text{ and }v\text{ as }v\text{ and }u
\end{array}
\right) \\
&  =\sum_{v\in\widehat{P}}\ \ \sum_{\substack{u\in\widehat{P};\\u\gtrdot
v}}\left(  Rf\right)  \left(  v\right)  \cdot\overline{\left(  Rf\right)
\left(  u\right)  }\\
&  =\sum_{v\in P\cup\left\{  0,1\right\}  }\ \ \sum_{\substack{u\in
\widehat{P};\\u\gtrdot v}}\left(  Rf\right)  \left(  v\right)  \cdot
\overline{\left(  Rf\right)  \left(  u\right)  }\ \ \ \ \ \ \ \ \ \ \left(
\text{since }\widehat{P}=P\cup\left\{  0,1\right\}  \right) \\
&  =\underbrace{\sum_{\substack{u\in\widehat{P};\\u\gtrdot0}}\left(
Rf\right)  \left(  0\right)  \cdot\overline{\left(  Rf\right)  \left(
u\right)  }}_{=\left(  Rf\right)  \left(  0\right)  \cdot\sum_{\substack{u\in
\widehat{P};\\u\gtrdot0}}\overline{\left(  Rf\right)  \left(  u\right)  }%
}+\underbrace{\sum_{\substack{u\in\widehat{P};\\u\gtrdot1}}\left(  Rf\right)
\left(  1\right)  \cdot\overline{\left(  Rf\right)  \left(  u\right)  }%
}_{\substack{=\left(  \text{empty sum}\right)  \\\text{(since there exists no
}u\in\widehat{P}\\\text{satisfying }u\gtrdot1\text{)}}}+\sum_{v\in
P}\ \ \underbrace{\sum_{\substack{u\in\widehat{P};\\u\gtrdot v}}\left(
Rf\right)  \left(  v\right)  \cdot\overline{\left(  Rf\right)  \left(
u\right)  }}_{=\left(  Rf\right)  \left(  v\right)  \cdot\sum
\limits_{\substack{u\in\widehat{P};\\u\gtrdot v}}\overline{\left(  Rf\right)
\left(  u\right)  }}\\
&  \ \ \ \ \ \ \ \ \ \ \ \ \ \ \ \ \ \ \ \ \left(
\begin{array}
[c]{c}%
\text{here, we have split off the addends}\\
\text{for }v=0\text{ and for }v=1\text{ from the sum}%
\end{array}
\right) \\
&  =\left(  Rf\right)  \left(  0\right)  \cdot\sum_{\substack{u\in
\widehat{P};\\u\gtrdot0}}\overline{\left(  Rf\right)  \left(  u\right)
}+\underbrace{\left(  \text{empty sum}\right)  }_{=0}+\sum_{v\in P}\left(
Rf\right)  \left(  v\right)  \cdot\sum\limits_{\substack{u\in\widehat{P}%
;\\u\gtrdot v}}\overline{\left(  Rf\right)  \left(  u\right)  }\\
&  =\underbrace{\left(  Rf\right)  \left(  0\right)  }_{=1}\cdot
\underbrace{\sum_{\substack{u\in\widehat{P};\\u\gtrdot0}}\overline{\left(
Rf\right)  \left(  u\right)  }}_{\substack{=\sum_{\substack{u\in
\widehat{P};\\u\lessdot1}}f\left(  u\right)  \\\text{(by
(\ref{pf.prop.R.invariant.01}))}}}+\sum_{v\in P}\underbrace{\left(  Rf\right)
\left(  v\right)  \cdot\sum\limits_{\substack{u\in\widehat{P};\\u\gtrdot
v}}\overline{\left(  Rf\right)  \left(  u\right)  }}_{\substack{=\sum
_{\substack{u\in\widehat{P};\\u\lessdot v}}f\left(  u\right)  \cdot
\overline{f\left(  v\right)  }\\\text{(by (\ref{pf.prop.R.invariant.1}))}}}\\
&  =\sum_{\substack{u\in\widehat{P};\\u\lessdot1}}f\left(  u\right)
+\sum_{v\in P}\ \ \sum_{\substack{u\in\widehat{P};\\u\lessdot v}}f\left(
u\right)  \cdot\overline{f\left(  v\right)  }.
\end{align*}
Comparing this with%
\begin{align*}
&  \underbrace{\sum_{\substack{u,v\in\widehat{P};\\u\lessdot v}}}_{=\sum
_{v\in\widehat{P}}\ \ \sum_{\substack{u\in\widehat{P};\\u\lessdot v}}}f\left(
u\right)  \cdot\overline{f\left(  v\right)  }\\
&  =\sum_{v\in\widehat{P}}\ \ \sum_{\substack{u\in\widehat{P};\\u\lessdot
v}}f\left(  u\right)  \cdot\overline{f\left(  v\right)  }\\
&  =\sum_{v\in P\cup\left\{  0,1\right\}  }\ \ \sum_{\substack{u\in
\widehat{P};\\u\lessdot v}}f\left(  u\right)  \cdot\overline{f\left(
v\right)  }\ \ \ \ \ \ \ \ \ \ \left(  \text{since }\widehat{P}=P\cup\left\{
0,1\right\}  \right) \\
&  =\underbrace{\sum_{\substack{u\in\widehat{P};\\u\lessdot0}}f\left(
u\right)  \cdot\overline{f\left(  0\right)  }}_{\substack{=\left(  \text{empty
sum}\right)  \\\text{(since there exists no }u\in\widehat{P}\\\text{satisfying
}u\lessdot0\text{)}}}+\sum_{\substack{u\in\widehat{P};\\u\lessdot1}}f\left(
u\right)  \cdot\underbrace{\overline{f\left(  1\right)  }}_{=1}+\sum_{v\in
P}\ \ \sum_{\substack{u\in\widehat{P};\\u\lessdot v}}f\left(  u\right)
\cdot\overline{f\left(  v\right)  }\\
&  \ \ \ \ \ \ \ \ \ \ \ \ \ \ \ \ \ \ \ \ \left(
\begin{array}
[c]{c}%
\text{here, we have split off the addends}\\
\text{for }v=0\text{ and for }v=1\text{ from the sum}%
\end{array}
\right) \\
&  =\underbrace{\left(  \text{empty sum}\right)  }_{=0}+\sum_{\substack{u\in
\widehat{P};\\u\lessdot1}}f\left(  u\right)  +\sum_{v\in P}\ \ \sum
_{\substack{u\in\widehat{P};\\u\lessdot v}}f\left(  u\right)  \cdot
\overline{f\left(  v\right)  }\\
&  =\sum_{\substack{u\in\widehat{P};\\u\lessdot1}}f\left(  u\right)
+\sum_{v\in P}\ \ \sum_{\substack{u\in\widehat{P};\\u\lessdot v}}f\left(
u\right)  \cdot\overline{f\left(  v\right)  },
\end{align*}
we obtain%
\[
\sum_{\substack{u,v\in\widehat{P};\\u\lessdot v}}\left(  Rf\right)  \left(
u\right)  \cdot\overline{\left(  Rf\right)  \left(  v\right)  }=\sum
_{\substack{u,v\in\widehat{P};\\u\lessdot v}}f\left(  u\right)  \cdot
\overline{f\left(  v\right)  }.
\]
This proves Proposition \ref{prop.R.invariant}.
\end{proof}
\end{verlong}


\begin{thebibliography}{99999999999}                                                                                      %


\bibitem[AyKlSc12]{ayyer-klee-schilling}%
\href{https://doi.org/10.1007/s10801-013-0470-9}{Arvind Ayyer, Steven Klee,
Anne Schilling, \textit{Combinatorial Markov chains on linear extensions},
Journal of Algebraic Combinatorics \textbf{39}, September 2013, DOI
10.1007/s10801-013-0470-9}. Also available as
\href{https://arxiv.org/abs/1205.7074v3}{arXiv:1205.7074v3}.

\bibitem[BerRet15]{BerRet15}\href{https://arxiv.org/abs/1510.02628v4}{Arkady
Berenstein, Vladimir Retakh, \textit{Noncommutative marked surfaces}, Advances
in Mathematics \textbf{328} (2018), pp. 1010--1087. Also available as
arXiv:1510.02628v4.}

\bibitem[BrSchr74]{brouwer-schrijver}%
\href{http://www.win.tue.nl/~aeb/preprints/zw24.pdf}{Andries E. Brouwer and A.
Schrijver, \textit{On the period of an operator, defined on antichains}, Math.
Centr. report ZW24, Amsterdam (Jun. 1974).}

\bibitem[DefKra21]{DefKra21}\href{https://doi.org/10.5070/C61055363}{Colin
Defant, Noah Kravitz, \textit{Friends and strangers walking on graphs},
Combinatorial Theory \textbf{1} (2021), \#6}.

\bibitem[DWYWZ20]{DWYWZ20}\href{https://doi.org/10.37236/9769}{Quang Vu Dao,
Julian Wellman, Calvin Yost-Wolff, Sylvester W. Zhang, \textit{Rowmotion
Orbits of Trapezoid Posets}, The Electronic Journal of Combinatorics
\textbf{29}, Issue 2 (2022), \#P2.29.}

\bibitem[EinPro13]{einstein-propp}%
\href{https://doi.org/10.5802/alco.139}{David Einstein, James Propp,
\textit{Combinatorial, piecewise-linear, and birational homomesy for products
of two chains}, Algebraic Combinatorics \textbf{4} (2021) no. 2, pp. 201--224}.

\bibitem[Etienn84]{etienne}%
\href{https://doi.org/10.1016/0012-365X(84)90108-0}{Gwihen Etienne,
\textit{Linear extensions of finite posets and a conjecture of G. Kreweras on
permutations}, Discrete Mathematics \textbf{52} (1984), Issue 1, pp. 107--111}.

\bibitem[GalPyl19]{GalPyl19}%
\href{https://doi.org/10.1007/s00029-019-0470-2}{Pavel Galashin,\ Pavlo
Pylyavskyy, $R$\textit{-systems}, Selecta Mathematica \textbf{25} (2019),
Article 22.}

\bibitem[GonKon21]{GonKon21}%
\href{https://arxiv.org/abs/2108.04168v3}{Alexander Goncharov, Maxim
Kontsevich, \textit{Spectral description of non-commutative local systems on
surfaces and non-commutative cluster varieties}, arXiv:2108.04168v3}.

\bibitem[GriRob14]{bir-row-arxiv}%
\href{https://arxiv.org/abs/1402.6178v7}{Darij Grinberg, Tom Roby,
\textit{Iterative properties of birational rowmotion}, 15 August 2022,
arXiv:1402.6178v7}.

\bibitem[GriRob15]{bir-row-2}\href{https://doi.org/10.37236/4335}{Darij
Grinberg, Tom Roby, \textit{Iterative Properties of Birational Rowmotion II:
Rectangles and Triangles}, Electronic Journal of Combinatorics \textbf{22}
(2015), Paper \#P3.40}.

\bibitem[GriRob16]{bir-row-1}\href{https://doi.org/10.37236/4334}{Darij
Grinberg, Tom Roby, \textit{Iterative Properties of Birational Rowmotion I:
Generalities and Skeletal Posets}, Electronic Journal of Combinatorics
\textbf{23} (2016), Paper \#P1.33}.

\bibitem[GriRob23]{this-fpsac}%
\href{https://www.cip.ifi.lmu.de/~grinberg/algebra/fps2023.pdf}{Darij
Grinberg, Tom Roby, \textit{Birational rowmotion over noncommutative rings},
S\'{e}minaire Lotharingien de Combinatoire, \textbf{89B}.50 (2023), 12 pp.
(FPSAC 2023).}.

\bibitem[Gyoja86]{Gyoja86}%
\href{https://projecteuclid.org/journals/osaka-journal-of-mathematics/volume-23/issue-4/A-q-analogue-of-Young-symmetrizer/ojm/1200779724.full}{Akihiko
Gyoja, \textit{A }$\mathit{q}$\textit{-analogue of Young symmetrizer}, Osaka
J. Math. \textbf{23} (1986), pp. 841--852.}

\bibitem[IyuShk14]{IyuShk14}%
\href{http://dx.doi.org/10.1215/00127094-3146603}{Natalia Iyudu, Stanislav
Shkarin, \textit{The proof of the Kontsevich periodicity conjecture on
noncommutative birational transformations}, Duke Math. J. \textbf{164}(13),
pp. 2539--2575 (1 October 2015).}

\bibitem[Johnso23]{Johnso23}%
\href{https://www.lib.ncsu.edu/resolver/1840.20/40802}{Joseph William Norman
Johnson, \textit{Problems in Dynamical Algebraic Combinatorics and Algebraic
Statistics}, dissertation at North Carolina State University, 2023.}

\bibitem[JohLiu22]{JohLiu22}\href{https://arxiv.org/abs/2204.04255v1}{Joseph
Johnson, Ricky Ini Liu, \textit{Birational rowmotion and the octahedron
recurrence}, arXiv:2204.04255v1.}

\bibitem[JosRob20]{JosRob20}\href{https://doi.org/10.5802/alco.125}{Michael
Joseph, Tom Roby, \textit{Birational and noncommutative lifts of antichain
toggling and rowmotion}, Algebraic Combinatorics \textbf{3} (2020), issue 4,
pp. 955--984}.

\bibitem[JosRob21]{JosRob21}\href{https://doi.org/10.46298/dmtcs.6633}{Michael
Joseph, Tom Roby, \textit{A birational lifting of the Stanley-Thomas word on
products of two chains}, Discrete Mathematics \& Theoretical Computer Science
\textbf{23} (2021) no. 1, Combinatorics (August 18, 2021) dmtcs:8367}.

\bibitem[Kirill00]{kirillov-intro}%
\href{https://doi.org/10.1142/9789812810007_0005}{Anatol N. Kirillov,
\textit{Introduction to tropical combinatorics}, Physics and combinatorics:
Proceedings of the Nagoya 2000 International Workshop, held 21 - 26 August
2000 in Nagoya University. Edited by Anatol N Kirillov (Nagoya University) \&
Nadejda Liskova. Published by World Scientific Publishing Co. Pte. Ltd., 2001.
ISBN \#9789812810007, pp. 82--150.}

\bibitem[KirBer95]{kirillov-berenstein}Anatol N. Kirillov, Arkadiy D.
Berenstein, \textit{Groups generated by involutions, Gelfand-Tsetlin patterns,
and combinatorics of Young tableaux}, Algebra i Analiz \textbf{7} (1995),
issue 1, pp. 92--152. A preprint is available at:\newline\url{https://pages.uoregon.edu/arkadiy/bk1.pdf}

\bibitem[MusRob17]{MusRob17}\href{https://doi.org/10.5802/alco.43}{Gregg
Musiker, Tom Roby, \textit{Paths to Understanding Birational Rowmotion on
Products of Two Chains}, Algebraic Combinatorics \textbf{2} (2019) no. 2, pp.
275--304}.

\bibitem[Naatz00]{Naatz00}%
\href{https://doi.org/10.1137/s0895480199352609}{Michael Naatz, \textit{The
graph of linear extensions revisited}, Siam J. Discrete Math. \textbf{13}
(2000), No. 3, pp. 354--369}.

\bibitem[ProRob13]{propp-roby}James Propp and Tom Roby, \textit{Homomesy in
products of two chains}, DMTCS proc. FPSAC 2013,\newline\url{http://www.math.uconn.edu/~troby/ceFPSAC.pdf}

\bibitem[ProRob14]{propp-roby-full}\href{https://doi.org/10.37236/3579}{James
Propp and Tom Roby, \textit{Homomesy in products of two chains}, The
Electronic Journal of Combinatorics \textbf{22}, Issue 3 (2015), Paper
\#P3.4}. A preprint appeared as
\href{http://arxiv.org/abs/1310.5201v6}{\texttt{arXiv:1310.5201v6}}.

\bibitem[Roby15]{Roby15}Tom Roby, \textit{Dynamical Algebraic Combinatorics
and the Homomesy Phenomenon}, in: A. Beveridge, J. Griggs, L. Hogben, G.
Musiker, P. Tetali, eds., Recent Trends in Combinatorics (IMA Volume in
Mathematics and its Applications), Springer, 2016.\newline\url{https://www2.math.uconn.edu/~troby/homomesyIMA2015Revised.pdf}

\bibitem[Rupel17]{Rupel17}\href{https://doi.org/10.5802/alco.81}{Dylan Rupel,
\textit{Rank Two Non-Commutative Laurent Phenomenon and Pseudo-Positivity},
Algebraic Combinatorics \textbf{2} (2019), issue 6, pp. 1239--1273.}

\bibitem[S{$^{+}$}09]{sage}W.\thinspace{}A. Stein et~al., \emph{{S}age
{M}athematics {S}oftware ({V}ersion 9.4)}, The Sage Development Team, 2021, \url{http://www.sagemath.org}.

\bibitem[Sage08]{sage-combinat}The Sage-Combinat community,
\textit{Sage-Combinat: enhancing Sage as a toolbox for computer exploration in
algebraic combinatorics}, 2008.\newline\url{http://combinat.sagemath.org}

\bibitem[Speyer21]{Speyer21}David E Speyer, \textit{MathOverflow post
\#401273}, \url{https://mathoverflow.net/q/401273} .

\bibitem[Stan86]{stanley-polytopes}%
\href{https://math.mit.edu/~rstan/pubs/pubfiles/66.pdf}{Richard P. Stanley,
\textit{Two Poset Polytopes}, Discrete \& Computational Geometry, 1986, Volume
1, Issue 1, pp. 9--23.}

\bibitem[Stan11]{Stanley-EC1}Richard P. Stanley, \textit{Enumerative
Combinatorics, volume 1}, Second edition, version of 15 July 2011. Available
at \url{http://math.mit.edu/~rstan/ec/}.\newline See
\url{http://math.mit.edu/~rstan/ec/} for errata.

\bibitem[StWi11]{striker-williams}%
\href{https://doi.org/10.1016/j.ejc.2012.05.003}{Jessica Striker, Nathan
Williams, \textit{Promotion and Rowmotion}, European Journal of Combinatorics
33 (2012), pp. 1919--1942, DOI 10.1016/j.ejc.2012.05.003.} \newline\url{https://doi.org/10.1016/j.ejc.2012.05.003}

\bibitem[ThoWil19]{row-slow}\href{https://doi.org/10.1112/plms.12251}{Hugh
Thomas and Nathan Williams, \textit{Rowmotion in slow motion}, Proceedings of
the London Mathematical Society, \textbf{119}(5) (2019) 1149--1178.}

\bibitem[Volk06]{volkov}%
\href{https://doi.org/10.1007/s00220-007-0343-y}{Alexandre Yu. Volkov,
\textit{On the Periodicity Conjecture for Y-systems}, Communications Math.
Physics 276 (2007), pp. 509--517, DOI 10.1007/s00220-007-0343-y.}\newline A
preprint of this paper is also available under the name \textit{On
Zamolodchikov's Periodicity Conjecture} as arXiv:hep-th/0606094v1:\newline\url{http://arxiv.org/abs/hep-th/0606094v1}
\end{thebibliography}
\end{document}